\documentclass[11pt,oneside]{article}
\usepackage{amssymb,amsmath,latexsym,amsfonts,amscd,amsthm}
\usepackage{fancyhdr}
\usepackage[headings]{fullpage}
\usepackage{stmaryrd}
\usepackage[all]{xy}

\usepackage{setspace}

\usepackage{hyperref}

\newcommand{\gt}[1]{\mathfrak{#1}}
\newcommand{\mc}[1]{\mathcal{#1}}
\newcommand{\comment}[1]{}

\newcommand{\RR}{{\mathbb R}}%Reals
\newcommand{\RRp}{\RR^+}%positive real
\newcommand{\CC}{{\mathbb C}}%Complex
\newcommand{\ZZ}{{\mathbb Z}}%Integers
\newcommand{\ZZp}{\ZZ^+}%Positive integers
\newcommand{\NN}{{\mathbb N}}%Natural numbers
\newcommand{\QQ}{{\mathbb Q}}%Rationals
\newcommand{\QQp}{{\mathbb Q}^+}%positive rationals
\newcommand{\HH}{{\mathbb H}}%hyperbolic plane
\newcommand{\ii}{{\bf i}}%Sqrt of -1
\newcommand{\tpi}{2\pi\ii}%The fundamental constant
\newcommand{\fps}{4\pi^2}%minus the fundamental constant squared

\newcommand{\lab}{{\langle}}    %Left angle brackets
\newcommand{\rab}{{\rangle}}    %Right angle brackets
\newcommand{\llb}{{(\!(}}
\newcommand{\rrb}{{)\!)}}

\newcommand{\Aut}{\operatorname{Aut}}

\newcommand{\Res}{\operatorname{Res}}

\newcommand{\Pdet}{P\!\det}

\newcommand{\ee}{\operatorname{e}} %Number theory exp
\newcommand{\ev}{V} %Whittaker function v
\newcommand{\ew}{W} %Whittaker function w
\newcommand{\fc}[2]{{\rm c}_{#1}^{#2}} %Formal series coefficients
\newcommand{\gc}[2]{{\rm g}_{#1}^{#2}} %Green's function Fourier coefficients

\newcommand{\sqnom}[2]{\genfrac{[}{]}{0pt}{}{#1}{#2}}
\newcommand{\lBZ}{\left\llbracket} %shorthand for right B(\ZZ) cosets
\newcommand{\rBZ}{\right\rrbracket} %shorthand for left B(\ZZ) cosets
\newcommand{\lBZh}{\left.} %empty bracket on the left
\newcommand{\rBZh}{\right.} %empty bracket on the right
\newcommand{\lBgZ}{\lBZ^g} %shorthand for right B(\ZZ) cosets
\newcommand{\lBhZ}{\lBZ^h} %shorthand for right B(\ZZ) cosets
\newcommand{\jac}{\operatorname{jac}}
\newcommand{\genus}{\operatorname{genus}}
\newcommand{\Kl}{\operatorname{Kl}} %Kloosterman term
\newcommand{\Bf}{\operatorname{Bl}} %Bessel integral - Formal series
\newcommand{\vq}{{q}}     %q variable
\newcommand{\vp}{{p}}    %p variable
\newcommand{\cpr}{{\it p}}     %cusp rep 1
\newcommand{\cqr}{{\it q}}     %cusp rep 2
\newcommand{\cp}{{\sf p}}     %cusp 1
\newcommand{\cq}{{\sf q}}     %cusp 2
\newcommand{\co}{{\sf r}}     %cusp 3
\newcommand{\cP}{{\sf P}}     %set of cusps
\newcommand{\MT}{T}      %McKay--Thompson series
\newcommand{\GD}{J}      %Twisted module graded dimension
\newcommand{\HO}{\hat{T}}    %Hecke operator
\newcommand{\IO}[2]{I_{#1}^{#2}}   %Integral operator
\newcommand{\JO}[2]{J_{#1}^{#2}}   %Doubly conjugated integral operator
\newcommand{\rads}{{\varrho}}

\newcommand{\gdim}{\operatorname{gdim}} %Graded dimension
\newcommand{\squan}{\operatorname{{\mathbb S}}}   %Second quantization

\newcommand{\Ex}{\operatorname{Ex}}
\newcommand{\sop}[1]{|^{#1}}            %Slash operator on functions
\newcommand{\cop}[2]{\|_{#1}^{#2}}    %Coset operator on integrals
\newcommand{\hop}[2]{\dag_{#1}^{#2}}      %Hecke (i.e. double coset) operator on integrals
\newcommand{\RS}[3]{R_{#1}^{#2(#3)}}         %Rademacher sum
\newcommand{\TS}[3]{T_{#1}^{#2(#3)}}         %Rademacher sum continued
\newcommand{\TSa}[3]{T_{0,#1}^{#2(#3)}}         %Rademacher sum continued
\newcommand{\CS}[3]{\bar{R}_{#1}^{#2(#3)}}   %Conjugate Rademacher sum
\newcommand{\QS}[3]{Q_{#1}^{#2(#3)}}         %Conv Rademacher sum
\newcommand{\PS}[3]{P_{#1}^{#2(#3)}}         %Holomorphic Poincare series
\newcommand{\PSa}[3]{P_{0,#1}^{#2(#3)}}      %Holomorphic Poincare series correction
\newcommand{\GS}[3]{G_{#1}^{#2(#3)}}           %Green's function coefficient
\newcommand{\FS}[3]{\tilde{F}_{#1}^{#2(#3)}} %Formal series (univariate)
\newcommand{\FR}[3]{F_{#1}^{#2(#3)}}         %Fourier series
\newcommand{\DS}[3]{Z_{#1}^{#2(#3)}} %Dirichlet series
\newcommand{\DSt}[3]{\tilde{Z}_{#1}^{#2(#3)}} %Dirichlet series opposite
\newcommand{\Rreg}[1]{\operatorname{r}^{#1}}
\newcommand{\Treg}[1]{\operatorname{t}^{#1}}

\newcommand{\PSL}{\operatorname{\textsl{PSL}}}    %PSL group
\newcommand{\SL}{\operatorname{\textsl{SL}}}      %SL group
\newcommand{\PGL}{\operatorname{\textsl{PGL}}}    %PGL group
\newcommand{\GL}{\operatorname{\textsl{GL}}}      %GL group
\newcommand{\MM}{\mathbb{M}}    %Monster
\newcommand{\Co}{\operatorname{\textsl{Co}}}    %Conway 1
\newcommand{\vn}{V^{\natural}} %moonshine module
\newcommand{\vnh}{{W}^{\natural}}
\newcommand{\LL}{\Lambda}     %Leech lattice
\newcommand{\commod}{\gt{G}}     %Groups commensurable with the modular group
\newcommand{\ogi}{\theta}     %Kummer involution

\newcommand{\mla}{\gt{m}}     %Monster Lie algebra
\newcommand{\mlav}{\mc{V}}    %Monster Lie algebra Verma module
\newcommand{\GamN}{\Gamma_0(N)}
\newcommand{\zz}{z} %Complex coordinate
\newcommand{\ww}{w} %Other complex coordinate
\newcommand{\spp}{s} %Spectral parameter
\newcommand{\atp}{\kappa}   %automorphic parameter

\newtheorem{thm}{Theorem}[subsection]
\newtheorem{cor}[thm]{Corollary}
\newtheorem{lem}[thm]{Lemma}
\newtheorem{prop}[thm]{Proposition}
\newtheorem{conj}[thm]{Conjecture}

\theoremstyle{definition}

\theoremstyle{remark}
\newtheorem{rmk}[thm]{Remark}

\numberwithin{equation}{subsection}

\begin{document}

\setstretch{1.3}

\title{
    \textsc{{R}ademacher sums, moonshine and
    gravity}
          }

\author{  John F. R. Duncan\footnote{
          Case Western Reserve University,
          Department of Mathematics, 
          %Yost Hall,
          10900 Euclid Avenue,
          Cleveland, OH 44106.%,
          %U.S.A.%;\;{\tt j.duncan@dpmms.cam.ac.uk};\;
          %{\tt http://www.dpmms.cam.ac.uk/\~{}jd513/}
               }\\
          Igor B. Frenkel\footnote{
          Yale University,
          Department of Mathematics,
          10 Hillhouse Avenue,
          New Haven, CT 06520,
          %U.S.A.
          \newline
          %The research of J.D. was conducted during an Assistant Professorship at Harvard University.\newline
          The research of I.F. was supported by NSF grant DMS-0457444.
               }
               }

\date{1 March 2012}

\maketitle

\abstract{In 1939 Rademacher derived a conditionally convergent series expression for the elliptic modular invariant, and used this expression---the first Rademacher sum---to verify its modular invariance. By generalizing Rademacher's approach we construct bases for the spaces of automorphic integrals of arbitrary even integer weight, for groups commensurable with the modular group. Our methods provide explicit expressions for the Fourier expansions of the Rademacher sums we construct at arbitrary cusps, and illuminate various aspects of the structure of the spaces of automorphic integrals, including the actions of Hecke operators.

We give a moduli interpretation for a class of groups commensurable with the modular group which includes all those that are associated to the Monster via monstrous moonshine. We show that within this class the monstrous groups can be characterized just in terms of the behavior of their Rademacher sums. In particular, the genus zero property of monstrous moonshine is encoded naturally in the properties of Rademacher sums.

Just as the ellptic modular invariant gives the graded dimension of the moonshine module, the exponential generating function of the Rademacher sums associated to the modular group furnishes the bi-graded dimension of the Verma module for the Monster Lie algebra. This result generalizes naturally to all the groups of monstrous moonshine, and recovers a certain family of monstrous Lie algebras recently introduced by Carnahan.

Our constructions suggest conjectures relating monstrous moonshine to a distinguished family of chiral three dimensional quantum gravities, and relating monstrous Lie algebras and their Verma modules to the second quantization of this family of chiral three dimensional quantum gravities.}

\tableofcontents

%------------------------------------------------------------------%
\section{Introduction}\label{sec:intro}
%------------------------------------------------------------------%

\subsection{Monstrous moonshine}\label{sec:intro:moon}

A remarkable coincidence between the coefficients of the elliptic modular invariant
\begin{gather}\label{eqn:intro:moon:Four_Exp_J}
     J(\zz)=\ee(-\zz)+\sum_{n>0}c(n)\ee(n\zz),\quad
     \ee(n\zz)=\exp(\tpi n\zz),
\end{gather}
and the dimensions of the irreducible representations of the {\em Monster group}, denoted $\MM$, led McKay and Thompson \cite{Tho_NmrlgyMonsEllModFn} to conjecture the existence of a naturally defined infinite dimensional representation
\begin{gather}
     V=V_{-1}\oplus V_1\oplus V_2\oplus \cdots
\end{gather}
for the Monster group, with the property that $\dim V_n=c(n)$. Thompson \cite{Tho_FinGpsModFns} also proposed to consider the modular properties of the series
\begin{gather}\label{eqn:intro:moon:Four_Exp_MTg}
     \MT_g(\zz)=\ee(-\zz)+\sum_{n>0}({\rm tr}|_{V_n}g)\ee(n\zz),
\end{gather}
for any $g\in \MM$, the function $J(\zz)$ of (\ref{eqn:intro:moon:Four_Exp_J}) being recovered upon taking $g=e$ in (\ref{eqn:intro:moon:Four_Exp_MTg}). It is clear from the presentation (\ref{eqn:intro:moon:Four_Exp_MTg}) that $\MT_g(\zz)$ is invariant under the {\em translation group}, which we denote $B(\ZZ)$ and which is generated by the translation $\zz\mapsto \zz+1$. In the case (\ref{eqn:intro:moon:Four_Exp_J}) that $g$ is the identity element, the invariance extends to the full modular group $\PSL_2(\ZZ)$. This leads us to the question: what is special about the invariance groups $\Gamma_g$ of the McKay--Thompson series (\ref{eqn:intro:moon:Four_Exp_MTg})?

In their remarkable paper \cite{ConNorMM}, Conway and Norton collected an overwhelming number of coincidences and, in particular, formulated the {\em moonshine conjecture}:
\begin{quote}
The series $\MT_g(\zz)$ is the normalized hauptmodul of a genus zero group $\Gamma_g$ lying between $\GamN$ and its normalizer in $\PSL_2(\RR)$, for each $g\in \MM$.
\end{quote}
The conjecture of McKay and Thompson was proven in \cite{FLM} by the explicit construction of a vertex operator algebra $\vn$ invariant under the Monster group. It was also shown in \cite{FLM} that the McKay--Thompson series $\MT_g(\zz)$ satisfy the moonshine conjecture of Conway and Norton for all elements $g$ in a subgroup of $\MM$ arising as an involution centralizer. Finally, a complete proof of the moonshine conjecture was given by Borcherds in \cite{BorMM}. However, in spite of all the subsequent spectacular developments related to the Monster, the consensus of the experts (cf. e.g. \cite{ConMcKSebDiscGpsM}), is that ``the real nature of moonshine is still remote''.

\subsection{Rademacher sums}\label{sec:intro:radsum}

In this paper we shed new light on the properties of the McKay--Thompson series associated to the Monster, including the mysterious genus zero phenomena. It turns out that the McKay--Thompson series admit presentations as sums over the coset space $B(\ZZ)\backslash \Gamma_g$. The prototype for this is the following expression for the modular invariant $J(\zz)$ which was originally derived by Rademacher in \cite{Rad_FuncEqnModInv}.
\begin{gather}\label{eqn:intro:radsum:Rad_Sum_J}
     J(\zz)+12
     =
     \ee(-{\zz})+
     \lim_{K\to \infty}
     \sum_{\substack{0<c<K\\-K^2<d<K^2\\(c,d)=1}}
     \ee\left(-\frac{a{\zz}+b}{c{\zz}+d}\right)
     -
     \ee\left(-\frac{a}{c}\right)
\end{gather}
The integers $a$ and $b$ occurring in each summand of (\ref{eqn:intro:radsum:Rad_Sum_J}) are chosen so that $ad-bc=1$. This expression (\ref{eqn:intro:radsum:Rad_Sum_J}) is remarkably simple, but the convergence is rather subtle. (See \cite{Kno_RadonJPoinSerNonPosWtsEichCohom} for a nice exposition.) The subtraction of the constant $\ee(-a/c)$ in (\ref{eqn:intro:radsum:Rad_Sum_J}) ensures the existence of the limit in what is otherwise a highly divergent series. The Rademacher sum (\ref{eqn:intro:radsum:Rad_Sum_J}) has been generalized to various discrete subgroups of $\PSL_2(\RR)$ in a series of papers by Knopp (cf. \cite{Kno_ConstMdlrFnsI}, \cite{Kno_ConstMdlrFnsII}, \cite{Kno_ConstAutFrmsSuppSeries}, \cite{Kno_AbIntsMdlrFns}).

Given $\Gamma<\PSL_2(\RR)$ we write $\Gamma_{\infty}$ for the subgroup of $\Gamma$ that fixes $\infty$. If $\Gamma$ is commensurable with the modular group $\PSL_2(\ZZ)$ and has the property that $\Gamma_{\infty}=B(\ZZ)$, then we may naturally associate a Rademacher sum to $\Gamma$ by setting
\begin{gather}\label{eqn:intro:radsum:RS_Gamma}
     \RS{\Gamma}{}{-1}(\zz)
     =
     \ee(-\zz)+
     \lim_{K\to \infty}
     \sum_{\gamma\in ({B(\ZZ)}\backslash\Gamma)_{\leq K}^{\times}}
     \ee(-\gamma\cdot\zz)-\ee(-\gamma\cdot\infty),
\end{gather}
where the sum here is taken over the {\em rectangle}
\begin{gather}\label{eqn:intro:radsum:Rad_Square}
     (B(\ZZ)\backslash\Gamma)_{\leq K}^{\times}=
     \left\{\gamma=B(\ZZ)
     \left[
       \begin{array}{cc}
         * & * \\
         c & d \\
       \end{array}
     \right]
     \in{B(\ZZ)}\backslash\Gamma
     \mid
      0<c\leq K,\,-K^2\leq d\leq K^2
     \right\}
\end{gather}
(cf. \S\ref{sec:radsum:constr}). In (\ref{eqn:intro:radsum:Rad_Square}) we write $\sqnom{a\;b}{c\;d}$ for the image in $\PSL_2(\RR)$ of a matrix $\binom{a\;b}{c\;d}$ in $\SL_2(\RR)$, and in (\ref{eqn:intro:radsum:RS_Gamma}) we write $\gamma\cdot\zz$ for $(a\zz+b)/(c\zz+d)$ and $\gamma\cdot\infty$ for $a/c$ when $\gamma=\sqnom{a\;b}{c\;d}$.

In general the series $\RS{\Gamma}{}{-1}(\zz)$ defines an abelian integral for $\Gamma$, meaning that we have
\begin{gather}\label{eqn:intro:radsum:RS_ab_int}
     \RS{\Gamma}{}{-1}(\gamma\cdot\zz)
     =\RS{\Gamma}{}{-1}(\zz)+\omega(\gamma)
\end{gather}
for some function $\omega:\Gamma\to \CC$. We show in \S\ref{sec:moon:genus} that the Rademacher sum $\RS{\Gamma}{}{-1}(\zz)$ is $\Gamma$-invariant if and only if $\Gamma$ has genus zero. Even more than this, we show (also in \S\ref{sec:moon:genus}) that when $\Gamma$ has genus zero the function
\begin{gather}\label{eqn:intro:radsum:RS_minus_const/2}
     \RS{\Gamma}{}{-1}(\zz)-\frac{1}{2}\fc{\Gamma}{}(-1,0)
\end{gather}
is the normalized hauptmodul for $\Gamma$ (cf. \S\ref{sec:conven:groups}),  for a certain constant $\fc{\Gamma}{}(-1,0)$. This constant $\fc{\Gamma}{}(-1,0)$ turns out to be the {\em Rademacher constant associated to $\Gamma$} as defined in \cite{Nor_MoreMoons} in the case that $\Gamma$ has genus zero (cf. \S\ref{sec:struapp:const}).

Rademacher's proof of the validity of the presentation (\ref{eqn:intro:radsum:Rad_Sum_J}) relies upon explicit formulas for the Fourier coefficients of $J(\zz)$. These formulas may be given in terms of Kloosterman sums and Bessel functions, and are recovered from the expression
\begin{gather}\label{eqn:intro:radsum:Coeff_Fns_Wt_0}
     \fc{\Gamma}{}(m,n)
     =
          \sum_{
          \gamma\in
               {B(\ZZ)}\backslash\Gamma/{B(\ZZ)}
          }
          \Kl_{\gamma}(m,n)
          \Bf_{\gamma}(m,n)
\end{gather}
upon taking $\Gamma=\PSL_2(\ZZ)$ and $m=-1$, and allowing $n$ to range over $\ZZ$. Writing $c(\gamma)$ for $|c|$ when $\gamma=\sqnom{a\;b}{c\;d}$, the functions $\Kl_{\gamma}(m,n)$ and $\Bf_{\gamma}(m,n)$ are defined as follows (cf. \S\ref{sec:radsum:coeff}), for $\gamma\in\Gamma$.
\begin{gather}
     \Kl_{\gamma}(m,n)\label{eqn:intro_Kloos}
     =
     \ee(m\gamma\cdot\infty)\ee(-n\gamma^{-1}\cdot\infty)\\
     \Bf_{\gamma}(m,n)\label{eqn:intro_Bessel}
     =\tpi
     \Res_{\xi=0}
     \ee(m c(\gamma)^{-2}\xi^{-1})\ee(n\xi){\rm d}\xi
\end{gather}
Note that $\Kl_{\gamma}(m,n)$ and $\Bf_{\gamma}(m,n)$ are well-defined complex numbers only when $c(\gamma)\neq 0$, and in this case
\begin{gather}
     \Kl_{\gamma}(m,n)=
     \ee\left(\frac{ma+nd}{c}\right),\quad
     %\Bf_{\gamma}(m,n)=\sum_{k\geq 0}(-{\fps})^{k+1}c^{-2k-2}\frac{m^{k+1}}{(k+1)!}\frac{n^k}{k!}
     \Bf_{\gamma}(m,n)=\sum_{k\geq 0}\left(\frac{\tpi}{c}\right)^{2k+2}\frac{m^{k+1}}{(k+1)!}\frac{n^k}{k!}
     ,
\end{gather}
for $\gamma=\sqnom{a\;b}{c\;d}$ and $c>0$. The expression (\ref{eqn:intro:radsum:Coeff_Fns_Wt_0}) with $\Gamma=\PSL_2(\ZZ)$ and $m=-1$ is, up to elementary transformations, the formula for the $n$-th coefficient of the elliptic modular invariant given originally by Rademacher in \cite{Rad_FouCoeffMdlrInv}. More generally, the formula (\ref{eqn:intro:radsum:Coeff_Fns_Wt_0}) recovers the Fourier coefficients (other than the constant term) of the Rademacher sum $\RS{\Gamma}{}{-1}(\zz)$, for $\Gamma$ an arbitrary group commensurable with $\PSL_2(\ZZ)$ and satisfying $\Gamma_{\infty}=B(\ZZ)$. The following formula (cf. Theorem \ref{thm:radsum:conver:Relate_RS_FR}) encodes the relationship precisely.
\begin{gather}\label{eqn:intro:radsum:FourExp_RS}
     \RS{\Gamma}{}{-1}(\zz)
     =\ee(-\zz)
     +\frac{1}{2}\fc{\Gamma}{}(-1,0)
     +\sum_{n>0}
     \fc{\Gamma}{}(-1,n)\ee(n\zz)
\end{gather}

\subsection{Conjugate Rademacher sums}\label{sec:intro:conradsum}

It is striking that the constant term in the Fourier expansion (\ref{eqn:intro:radsum:FourExp_RS}) of the Rademacher sum $\RS{\Gamma}{}{-1}(\zz)$ is $\fc{\Gamma}{}(-1,0)/2$ and not $\fc{\Gamma}{}(-1,0)$, given that the coefficient of $\ee(n\zz)$ in (\ref{eqn:intro:radsum:FourExp_RS}) is exactly $\fc{\Gamma}{}(-1,n)$ for $n>0$. One may also observe that the formula (\ref{eqn:intro:radsum:Coeff_Fns_Wt_0}) defining the values $\fc{\Gamma}{}(-1,n)$ makes sense for arbitrary integers $n\in\ZZ$, and so it is natural to ask what r\^ole is played by the $\fc{\Gamma}{}(-1,n)$ for $n<0$? An answer to this question is obtained when we consider the {\em conjugate Rademacher sum associated to $\Gamma$}, denoted $\CS{\Gamma}{}{-1}(\zz)$, and defined by setting
\begin{gather}\label{eqn:intro:conradsum:CS_Gamma}
     \CS{\Gamma}{}{-1}(\zz)
     =
     \ee(-\bar{\zz})+
     \lim_{K\to \infty}
     \sum_{\gamma\in ({B(\ZZ)}\backslash\Gamma)_{\leq K}^{\times}}
     \ee(-\gamma\cdot\bar{\zz})-\ee(-\gamma\cdot\infty)
\end{gather}
for $\zz\in \HH$. The conjugate Rademacher sum $\CS{\Gamma}{}{-1}(\zz)$ defines an anti-holomorphic function on $\HH$. At first glance it appears that we should recover the classical Rademacher sum $\RS{\Gamma}{}{-1}(\zz)$ after substituting $\bar{\zz}$ for $\zz$ in $\CS{\Gamma}{}{-1}(\zz)$, since the expressions defining $\RS{\Gamma}{}{-1}(\zz)$ and $\CS{\Gamma}{}{-1}(\bar{\zz})$ would appear to coincide, but in fact the delicate limit defining the Rademacher sums behaves very differently depending on which half plane the variable $\zz$ lies in, and the difference $\RS{\Gamma}{}{-1}(\zz)-\CS{\Gamma}{}{-1}(\bar{\zz})$ can be rather far from vanishing. As an illustration of this, we show in \S\ref{sec:radsum:conver} that the Fourier expansion of the conjugate Rademacher sum $\CS{\Gamma}{}{-1}(\zz)$ in terms of the anti-holomorphic exponential $\ee(\bar{\zz})$ recovers the values $\fc{\Gamma}{}(-1,n)$ for $n<0$ as Fourier coefficients. More precisely, we establish the following counterpart to (\ref{eqn:intro:radsum:FourExp_RS}) (cf. Theorem \ref{thm:radsum:conver:Relate_CS_FR}). 
\begin{gather}\label{eqn:intro:conradsum:FourExp_CS}
     \CS{\Gamma}{}{-1}(\zz)
     =\ee(-\bar{\zz})
     -\frac{1}{2}\fc{\Gamma}{}(-1,0)
     -\sum_{n<0}
     \fc{\Gamma}{}(-1,n)\ee(n\bar{\zz})
\end{gather}
We demonstrate in \S\ref{sec:radsum:var} that the holomorphic function $\CS{\Gamma}{}{-1}(\bar{\zz})$ is also an abelian integral (cf. Theorem \ref{thm:radsum:var:Var_RS_atp<1}). More than this, the difference
\begin{gather}\label{eqn:intro:conradsum:Diff_RS_CS}
     \RS{\Gamma}{}{-1}(\zz)
     -
     \CS{\Gamma}{}{-1}({\zz})
\end{gather}
is a $\Gamma$-invariant harmonic function on $\HH$ (cf. Theorem \ref{thm:radsum:var:Invar_RS-CS_atp=0}).

The $\Gamma$-invariance of the function (\ref{eqn:intro:conradsum:Diff_RS_CS}) entails a remarkable formula for $\CS{\Gamma}{}{-1}(\zz)$ in the case that $\Gamma$ has genus zero. For since the holomorphic Rademacher sum $\RS{\Gamma}{}{-1}(\zz)$ is $\Gamma$-invariant in the genus zero case, the anti-holomorphic function $\CS{\Gamma}{}{-1}({\zz})$ must also be $\Gamma$-invariant. Given our knowledge (\ref{eqn:intro:conradsum:FourExp_CS}) of the Fourier expansion of $\CS{\Gamma}{}{-1}({\zz})$ it follows that $\CS{\Gamma}{}{-1}({\zz})$ must be identically constant, and furthermore, this constant must be $-\fc{\Gamma}{}(-1,0)/2$. In the case that $\Gamma=\PSL_2(\ZZ)$ we have $\fc{\Gamma}{}(-1,0)=24$, and we thus obtain the remarkable identity
\begin{gather}\label{eqn:intro:conradsum:Rad_Sum_-12}
     \ee(-\bar{\zz})+
     \lim_{K\to \infty}
     \sum_{\substack{0<c<K\\-K^2<d<K^2\\(c,d)=1}}
     \ee\left(-\frac{a\bar{\zz}+b}{c\bar{\zz}+d}\right)
     -
     \ee\left(-\frac{a}{c}\right)
     =-12,
\end{gather}
where in each summand $a$ and $b$ are integers chosen so that $ad-bc=1$. This implies that the function $\RS{\Gamma}{}{-1}(\zz)-\CS{\Gamma}{}{-1}({\zz})$ is a (non-normalized) hauptmodul for $\Gamma$ when $\Gamma$ has genus zero, and in particular, 
\begin{gather}
\fc{\Gamma}{}(-1,n)=\delta_{-1,n}
\end{gather}
for $n<0$. We recover the normalized hauptmodul for $\Gamma$ (cf. \S\ref{sec:conven:groups}) by considering the sum $\RS{\Gamma}{}{-1}(\zz)+\CS{\Gamma}{}{-1}({\zz})$.

\subsection{Solid tori}\label{sec:intro:moduli}

It was observed in \cite{ConNorMM} that each group attached to the Monster via monstrous moonshine may be described as a group of {\em $n\|h$-type}, meaning a discrete group $\Gamma<\PSL_2(\RR)$ of the form $\Gamma=\Gamma_0(n\| h)+S$ for some $n,h\in \ZZp$ with $h|(n,24)$, and $S$ a subgroup of the group of exact divisors of $n/h$. For such $n$, $h$, and $S$, the group $\Gamma_0(n\| h)+S$ contains and normalizes $\Gamma_0(nh)$, and has the property that $(\Gamma_0(n\| h)+S)_{\infty}=B(\ZZ)$. (We recall the precise definition of $\Gamma_0(n\| h)+S$ in \S\ref{sec:conven:groups}.)

We shed new light on the groups of $n\|h$-type by demonstrating in \S\ref{sec:moon:moduli} that they solve a natural family of moduli problems for elliptic curves equipped with certain kinds of extra structure. As we will explain presently, a common feature of all the members of this family is a choice of filling of the elliptic curve in question; that is, a solid torus whose boundary is the given elliptic curve. Hence, we interpret the $n\|h$-type groups as defining moduli of (decorated) solid tori.

The fact that each group $\Gamma$ of $n\|h$-type satisfies $\Gamma_{\infty}=B(\ZZ)$ suggests the importance of the quotient $B(\ZZ)\backslash\HH$. Regarding $\HH$ as a moduli space for triples $(E,\gamma,\gamma')$ where $E$ is an elliptic curve and $(\gamma,\gamma')$ is an oriented basis for the first homology group $H_1(E,\ZZ)$ of $E$, we see that $B(\ZZ)\backslash\HH$ parameterizes pairs $(E,\gamma)$ where $E$ is an elliptic curve and $\gamma$ is a primitive element of $H_1(E,\ZZ)$. To each such pair $(E,\gamma)$ is naturally associated an infinite volume hyperbolic $3$-manifold with boundary an elliptic curve, for the cycle $\gamma$ specifies a way to fill in the surface defined by $E$, thus yielding a solid torus. According to a theorem of Sullivan (cf. \cite{Sul_ErgThyArbDscGrpHypMtns}, \cite{Mcm_RmnSrfGeo3Mfd}) every complete, smooth, infinite volume hyperbolic $3$-manifold with boundary an elliptic curve arises in this fashion, so we consider pairs $(E,C)$ where $E$ is an elliptic curve over $\CC$ and $C$ is an oriented subgroup of $E$ isomorphic to $S^1$. We call such a pair a {\em solid torus}; according to our discussion the quotient space $B(\ZZ)\backslash \HH$ naturally parameterizes solid tori.

For $(E,C)$ a solid torus and $n\in \ZZp$ we write $C[n]$ for the group of $n$-division points of $C$, and we define an {\em $n$-compatible isogeny of solid tori} $(E',C')\to (E,C)$ to be an isogeny $E'\to E$ of elliptic curves that maps $C'[n]$ to a subgroup of $C[n]$, and we define a {\em fully compatible isogeny of solid tori} to be an isogeny $E'\to E$ of elliptic curves that restricts to an orientation preserving map $C'\to C$ on the underlying primitive cycles. (Note that the $n$ in an $n$-compatible isogeny of solid tori does not restrict the degree of the underlying isogeny of elliptic curves.) For $(E,C)$ a solid torus, the canonical map $E\to E/C[n]$ may be viewed as defining a fully compatible isogeny of solid tori $(E,C)\to (E,C)/C[n]$, where $(E,C)/C[n]$ is a shorthand for the solid torus whose underlying elliptic curve is $E/C[n]$, and whose primitive cycle is the image of $C$ under the natural map $E\to E/C[n]$. It may happen that an $n$-compatible isogeny induces invertible $n/e$-compatible isogenies (i.e. {\em $n/e$-compatible isomorphisms}) $(E',C')/C'[e]\to (E,C)$ and $(E',C')\to (E,C)/C[e]$ for some divisor $e$ of $n$. In this case we say that $(E',C')$ and $(E,C)$ are {\em $n+e$-related}. By considering the existence or otherwise of such morphisms of solid tori we arrive at moduli interpretations for all the $n\|h$-type groups. For example, given a subgroup $S$ of the group of exact divisors of $n$, say solid tori $(E,C)$ and $(E',C')$ are {\em $n+S$-related} if $(E,C)$ and $(E',C')$ are {$n+e$-related} for some $e\in S$. Then the $n+S$-relation is an equivalence relation on solid tori, and the quotient $(\Gamma_0(n)+S)\backslash\HH$ is in natural correspondence with $n+S$-equivalence classes of solid tori.

In this way we demonstrate that the groups of $n\|h$-type may be explained as being the groups that solve a natural family of moduli problems for decorated solid tori with conformal structure on the boundary.

\subsection{Moonshine via Rademacher sums}\label{sec:intro:modularity}

The fact that the function (\ref{eqn:intro:radsum:RS_minus_const/2}) recovers the normalized hauptmodul of the group $\Gamma$ when $\Gamma$ has genus zero is a strong indication that the Rademacher sum may be useful for the purpose of understanding the functions of monstrous moonshine. A group theoretic characterization of the functions arising as McKay--Thompson series was given in \cite{ConMcKSebDiscGpsM}. There the authors showed that a holomorphic function on $\HH$ coincides with $\MT_g(\zz)$ for some $g\in \MM$ if and only if it is the normalized hauptmodul for a group $\Gamma$ satisfying each of the following properties.
\begin{enumerate}
\item\label{item:intro:modularity:Cond_Genus_0}
The Riemann surface $\Gamma\backslash\HH$ has genus zero.
\item\label{item:intro:modularity:Cond_n|h_group}
The group $\Gamma$ is of $n\|h$-type.
\item\label{item:intro:modularity:Cond_Exp2}
The group $\Gamma$ is an extension of $\Gamma_0(nh)$ by a group of exponent $2$.
\item\label{item:intro:modularity:Cond_Cusp}
Each cusp of $\Gamma$ can be mapped to the infinite cusp by an element $\sigma\in\PSL_2(\RR)$ with the property that $(\sigma\Gamma\sigma^{-1})_{\infty}=B(\ZZ)$ and the intersection $\Gamma\cap\sigma\Gamma\sigma^{-1}$ contains $\Gamma_0(nh)$.
\end{enumerate}
As we saw in \S\ref{sec:intro:radsum} the variance of the Rademacher sum $\RS{\Gamma}{}{-1}(\zz)$ with respect to the action of $\Gamma$ detects whether or not $\Gamma$ has genus zero, and we have discussed in \S\ref{sec:intro:moduli} the significance of the groups appearing in condition \ref{item:intro:modularity:Cond_n|h_group}. We show in \S\S\ref{sec:moon:exp},\ref{sec:moon:cusp} that the conditions \ref{item:intro:modularity:Cond_Exp2} and \ref{item:intro:modularity:Cond_Cusp} also admit natural reformulations in terms of the Rademacher sums.

Our reformulations come into view when we generalize the construction (\ref{eqn:intro:radsum:RS_Gamma}) so as to associate a Rademacher sum $\RS{\Gamma,\cp|\cq}{}{-1}(\zz)$ to each triple $(\Gamma,\cp,\cq)$ where $\Gamma$ is a group commensurable with $\PSL_2(\ZZ)$, and $\cp$ and $\cq$ are cusps of $\Gamma$. (This construction is presented in \S\ref{sec:radsum}.) %The function $\RS{\Gamma,\cp|\cq}{}{-1}(\zz)$ is to be regarded as giving the {\em expansion at $\cq$} of the {\em Rademacher sum $\RS{\Gamma,\cp}{}{-1}(\zz)$ associated to $\Gamma$ at $\cp$}, although this point of view is accurate only up to certain constant functions. 
We write $\RS{\Gamma,\cp}{}{-1}(\zz)$ as a shorthand for $\RS{\Gamma,\cp|\cq}{}{-1}(\zz)$ when $\cq$ is the {\em infinite cusp} $\Gamma\cdot\infty$. It develops in \S\ref{sec:radsum:var} that $\RS{\Gamma,\cp}{}{-1}(\zz)$ is an abelian integral for $\Gamma$ with a simple pole at $\cp$, and no other singularities, and the expansion of $\RS{\Gamma,\cp|\cq}{}{-1}(\zz)$ at the infinite cusp coincides with that of $\RS{\Gamma,\cp}{}{-1}(\zz)$ at $\cq$, up to (addition by) a certain constant function (cf. Theorem \ref{thm:radsum:var:Var_RS_atp<1}). Considering all the functions $\RS{\Gamma,\cp|\cq}{}{-1}(\zz)$ for varying cusps $\cp$ and $\cq$ we are able to access more subtle properties of the curve ${\sf X}_{\Gamma}=\Gamma\backslash\HH\cup\Gamma\backslash\hat{\QQ}$ (cf. \S\ref{sec:conven:groups}) beyond its genus. We find in \S\ref{sec:moon:exp} that condition \ref{item:intro:modularity:Cond_Exp2} translates into a certain condition of symmetry (cf. Proposition \ref{prop:moon:exp:cusp_commute_implies_scalg_coset_sq}) in the Rademacher sums associated to the Hecke congruence group $\Gamma_0(nh)$, where $n$ and $h$ are as in condition \ref{item:intro:modularity:Cond_n|h_group}, and we find in \S\ref{sec:moon:cusp} that condition \ref{item:intro:modularity:Cond_Cusp} translates into the condition (cf. Proposition \ref{prop:moon:cusp:scaling_coset_cont_Delta_iff_TSpq_inv}) that the function $\RS{\Gamma,\cp|\cq}{}{-1}$ be invariant for this subgroup $\Gamma_0(nh)$.

Our moduli interpretation of the $n\|h$-type groups, together with the invariance and symmetry conditions on Rademacher sums just described, facilitate a reformulation of the conditions of \cite{ConMcKSebDiscGpsM}. We thus arrive at a new characterization (cf. Theorem \ref{thm:moon:monster:char_thm}) of the functions of monstrous moonshine in terms of Rademacher sums, and moduli of solid tori.

\subsection{Modified Rademacher sums}\label{sec:intro:modradsum}

In order to understand better the behavior of the Rademacher sums $\RS{\Gamma}{}{-1}(\zz)$, and, in particular, the subtraction of the constant $\ee(-\gamma\cdot\infty)$ in (\ref{eqn:intro:radsum:RS_Gamma}), we consider (also in \S\ref{sec:radsum}) a generalization of the Rademacher sum adapted to the problem of constructing modular forms of arbitrary even integer weight. Our approach to Rademacher sums of negative weight is inspired by the work of Niebur in \cite{Nie_ConstAutInts}.

For $\atp,m\in \ZZ$ we define the {\em Rademacher sum of weight $2\atp$ and order $m$ associated to $\Gamma$} by setting 
\begin{gather}\label{eqn:intro:modradsum:RS_Gamma_atp}
     \RS{\Gamma}{\atp}{m}(\zz)
     =
     \ee(m\zz)+
     \lim_{K\to \infty}
     \sum_{\gamma\in ({B(\ZZ)}\backslash\Gamma)_{\leq K}^{\times}}
     \ee(m\gamma\cdot\zz)
     \Rreg{\atp}(m,\gamma,\zz)
     {(c\zz+d)^{-2\atp}},
\end{gather}
where $c,d\in\RR$ are chosen (for each summand) so that $\gamma$ is the image of a matrix $\binom{a\;b}{c\;d}\in \SL_2(\RR)$ in $\PSL_2(\RR)$, and $\Rreg{\atp}(m,\gamma,\zz)$ is the {\em Rademacher regularization factor of weight $2\atp$}, satisfying
\begin{gather}\label{eqn:intro:modradsum:Defn_Rreg}
     \Rreg{\atp}(m,\gamma,\zz)
     =
     1-
     \ee(m\gamma\cdot\infty-m\gamma\cdot\zz)
     \ee(m\gamma\cdot\zz-m\gamma\cdot\infty)_{<1-2\atp},
\end{gather}
where $\ee(\zz)_{<K}$ denotes the {\em partial exponential function} $\ee(\zz)_{<K}=\sum_{0\leq k<K}(\tpi\zz)^{k}/k!$ (cf. \S\ref{sec:radsum:constr}). Observe that we recover the Rademacher sum $\RS{\Gamma}{}{-1}(\zz)$ of (\ref{eqn:intro:radsum:RS_Gamma}) when $\atp=0$. Also, we have $\Rreg{\atp}(m,\gamma,\zz)=1$ when $\atp$ is positive, so that the Rademacher sum $\RS{\Gamma}{\atp}{m}(\zz)$ is a holomorphic Poincar\'e series (cf. \S\ref{sec:radsum:constr}) when $\atp>0$ (which is absolutely and locally uniformly convergent for $\atp>1$).

In \S\ref{sec:radsum:constr} we furnish generalizations of the functions $\Kl_{\gamma}(m,n)$ and $\Bf_{\gamma}(m,n)$ for arbitrary $\atp\in \ZZ$, obtaining a generalization of the formula (\ref{eqn:intro:radsum:Coeff_Fns_Wt_0}), and leading to the following analogue (cf. Theorem \ref{thm:radsum:conver:Relate_RS_FR}) of (\ref{eqn:intro:radsum:FourExp_RS}).
\begin{gather}\label{eqn:intro:modradsum:FourExp_RS_atp}
     \RS{\Gamma}{\atp}{m}(\zz)
     =\ee(m\zz)
     +\frac{1}{2}\fc{\Gamma}{\atp}(m,0)
     +\sum_{n>0}
     \fc{\Gamma}{\atp}(m,n)\ee(n\zz)
\end{gather}
The {\em conjugate Rademacher sum of weight $2\atp$ and order $m$ associated to $\Gamma$} is defined in direct analogy with the definition (\ref{eqn:intro:conradsum:CS_Gamma}) of the conjugate Rademacher sum $\CS{\Gamma}{}{-1}(\zz)$, which we presently recognize as the conjugate Rademacher sum of weight $0$ and order $1$ associated to $\Gamma$.
\begin{gather}\label{eqn:intro:modradsum:CS_Gamma_atp}
     \CS{\Gamma}{\atp}{m}(\zz)
     =
     \ee(m\bar{\zz})+
     \lim_{K\to \infty}
     \sum_{\gamma\in ({B(\ZZ)}\backslash\Gamma)_{\leq K}^{\times}}
     \ee(m\gamma\cdot\bar{\zz})
     \Rreg{\atp}(m,\gamma,\bar{\zz})
     {(c\bar{\zz}+d)^{-2\atp}}.
\end{gather}
We have the following analogue of (\ref{eqn:intro:conradsum:FourExp_CS}), and counterpart to (\ref{eqn:intro:modradsum:FourExp_RS_atp}), which leads naturally to analogues of (\ref{eqn:intro:conradsum:Diff_RS_CS}) for arbitrary $\atp\in \ZZ$ (cf. Theorem \ref{thm:radsum:conver:Relate_CS_FR}).
\begin{gather}\label{eqn:intro:modradsum:FourExp_CS_atp}
     \CS{\Gamma}{\atp}{m}(\zz)
     =\ee(m\bar{\zz})
     -\frac{1}{2}\fc{\Gamma}{\atp}(m,0)
     -\sum_{n<0}
     \fc{\Gamma}{\atp}(m,n)\ee(n\bar{\zz})
\end{gather}

The variance of the functions $\RS{\Gamma}{\atp}{m}(\zz)$ with respect to the natural (weight $2\atp$) action of the group $\Gamma$ was described by Niebur in \cite{Nie_ConstAutInts} for the case that $\Gamma$ has a single cusp. We verify in \S\ref{sec:modradsum:var} that for $\Gamma$ an arbitrary group commensurable with the modular group $\PSL_2(\ZZ)$, the function $\RS{\Gamma}{\atp}{m}(\zz)+\fc{\Gamma}{\atp}(m,0)/2$ defines an {\em automorphic integral of weight $2\atp$ for $\Gamma$} (cf. \S\ref{sec:conven:autfrm}) when $m<0$, meaning that $\RS{\Gamma}{\atp}{m}(\zz)$ is holomorphic on $\HH$, possibly having poles at cusps of $\Gamma$, and we have
\begin{gather}\label{eqn:intro:modradsum:RS_aut_int}
     \left(
     \RS{\Gamma}{\atp}{m}(\gamma\cdot\zz)
     +\frac{1}{2}\fc{\Gamma}{\atp}(m,0)
     \right)
     (c\zz+d)^{-2\atp}
     =\RS{\Gamma}{\atp}{m}(\zz)
     +\frac{1}{2}\fc{\Gamma}{\atp}(m,0)
     +\omega(\gamma)(\zz)
\end{gather}
for some function $\omega:\Gamma\to \CC[\zz]$, where the polynomial $\omega(\gamma)(\zz)$ has degree at most $-2\atp$ in $\zz$ (cf. Theorem \ref{thm:modradsum:var:QS_is_aut_int}, Proposition \ref{prop:modradsum:conver:Relate_QS_s=1_RS}). An identification of the function $\omega$ may be given in terms of a certain canonically defined map $I_{\atp}(\Gamma)\to S_{1-\atp}(\Gamma)$, where $I_{\atp}(\Gamma)$ denotes the space of automorphic integrals of weight $2\atp$ for $\Gamma$ (cf. \S\ref{sec:conven:autfrm}), and $S_{\atp}(\Gamma)$ denotes the space of cusp forms of weight $2\atp$ for $\Gamma$. We show (cf. Theorem \ref{thm:modradsum:var:MIS_exact}) that the sequence
\begin{gather}\label{eqn:intro:modradsum:MIS_Exact_Sequence}
     0
     \to
     M_{\atp}(\Gamma)
     \to
     I_{\atp}(\Gamma)
     \to
     S_{1-\atp}(\Gamma)
     \to
     0
\end{gather}
is exact, where $M_{\atp}(\Gamma)$ denotes the space of modular forms of weight $2\atp$ for $\Gamma$ possibly having poles at cusps (cf. \S\ref{sec:conven:autfrm}), and the map $M_{\atp}(\Gamma)\to I_{\atp}(\Gamma)$ is the natural inclusion.

Apart from the delicate convergence of the limit defining the Rademacher sum $\RS{\Gamma}{\atp}{m}(\zz)$, the most curious feature of the functions $\RS{\Gamma}{\atp}{m}(\zz)$ is that they transform naturally with respect to the group $\Gamma$ only after the addition of the constant $\fc{\Gamma}{\atp}(m,0)/2$, which is typically non-zero. We overcome this feature of the classical Rademacher sum by introducing in \S\ref{sec:modradsum} the {\em continued Rademacher sum of weight $2\atp$ and order $m$ associated to $\Gamma$}, defined for $\atp\leq 0$ and $m<0$ by setting
\begin{gather}\label{eqn:intro:modradsum:Defn_RS_cont}
     \TS{\Gamma}{\atp}{m}(\zz,\spp)
     =
     \ee(m\zz)
     +
     \sum_{\gamma\in ({B(\ZZ)}\backslash\Gamma)^{\times}}
     \ee(m\gamma\cdot\zz)
     \Treg{\atp}(m,\gamma,\zz,\spp)
     (c\zz+d)^{-2\atp}
\end{gather}
where $\Treg{\atp}(m,\gamma,\zz,\spp)$ is a generalization of the Rademacher regularization factor $\Rreg{\atp}(m,\gamma,\zz)$ of (\ref{eqn:intro:modradsum:Defn_Rreg}) satisfying $\Treg{\atp}(m,\gamma,\zz,1)=\Rreg{\atp}(m,\gamma,\zz)$.

The right hand side of (\ref{eqn:intro:modradsum:Defn_RS_cont}) converges absolutely and locally uniformly in $\zz$ and $\spp$ for $\zz\in \HH$ and $\Re(\spp)>1$. We define $\TS{\Gamma}{\atp}{m}(\zz)$ by taking the limit as $\spp$ tends to $1$ in $\TS{\Gamma}{\atp}{m}(\zz,\spp)$. We find that
\begin{gather}\label{eqn:intro:modradsum:Relate_TS_RS}
     \TS{\Gamma}{\atp}{m}(\zz)
     =
     \lim_{\spp\to 1^+}
     \TS{\Gamma}{\atp}{m}(\zz,\spp)
     =\RS{\Gamma}{\atp}{m}(\zz)-\frac{1}{2}\fc{\Gamma}{\atp}(m,0)
\end{gather}
(cf. Proposition \ref{prop:modradsum:conver:Relate_TS_FR_s=1}), so that the Fourier expansion of $\TS{\Gamma}{\atp}{m}(\zz)$ has vanishing constant term. Then, in the case that $\Gamma$ has genus zero, the function $\TS{\Gamma}{}{-1}(\zz)=\TS{\Gamma}{0}{1}(\zz)$ is precisely the normalized hauptmodul of $\Gamma$. We call $\TS{\Gamma}{\atp}{m}(\zz)$ the {\em normalized Rademacher sum of weight $2\atp$ and order $m$ associated to $\Gamma$}. We may regard the disappearance of the constant term in (\ref{eqn:intro:modradsum:Relate_TS_RS}) as a consequence of the noncommutativity of the limits $\spp\to 1$ and $K\to \infty$ when applied to (the expression obtained by replacing $({B(\ZZ)}\backslash\Gamma)^{\times}$ with $({B(\ZZ)}\backslash\Gamma)^{\times}_{\leq K}$ in) the right hand side of (\ref{eqn:intro:modradsum:Defn_RS_cont}).

To recover functions with natural modular properties at non-positive weights we modify the continued Rademacher sum by subtracting a renormalized value at $\zz=0$.
\begin{gather}
     \QS{\Gamma}{\atp}{m}(\zz,\spp)
     =\TS{\Gamma}{\atp}{m}(\zz,\spp)
     -\TSa{\Gamma}{\atp}{m}(\spp)
\end{gather}
The function $\TSa{\Gamma}{\atp}{m}(\spp)$ is defined by an expression analogous to (\ref{eqn:intro:modradsum:Defn_RS_cont}) which also converges absolutely and locally uniformly in $\spp$ for $\Re(\spp)>1$, and would vanish at $\spp=1$ if it were not for the noncommutativity of limits as $\spp\to 1$ and $K\to \infty$. We define the {\em modified Rademacher sum of weight $2\atp$ and order $m$ associated to $\Gamma$}, denoted $\QS{\Gamma}{\atp}{m}(\zz)$, by taking the limit as $\spp$ tends to $1$ in $\QS{\Gamma}{\atp}{m}(\zz,\spp)$, and as a counterpart to (\ref{eqn:intro:modradsum:Relate_TS_RS}) we have
\begin{gather}\label{eqn:intro:modradsum:Relate_QS_RS}
     \QS{\Gamma}{\atp}{m}(\zz)
     =
     \lim_{\spp\to 1^+}
     \QS{\Gamma}{\atp}{m}(\zz,\spp)
     =
     \RS{\Gamma}{\atp}{m}(\zz)
     +\frac{1}{2}\fc{\Gamma}{\atp}(m,0)
\end{gather}
(cf. Proposition \ref{prop:modradsum:conver:Relate_QS_s=1_RS}), so that, in light of (\ref{eqn:intro:modradsum:RS_aut_int}), the modified Rademacher sum $\QS{\Gamma}{\atp}{m}(\zz)$ is an automorphic integral of weight $2\atp$ for $\Gamma$.

The definition of the modified Rademacher sum $\QS{\Gamma}{\atp}{m}(\zz)$ is inspired by Hurwitz's relation for the Hurwitz zeta function, and our continuation procedure may be regarded as identifying the factor $1/2$, appearing in the constant term of the Fourier expansion (\ref{eqn:intro:modradsum:FourExp_RS_atp}) of the classical Rademacher sum of non-positive weight, with $-1$ times the value of the Riemann zeta function $\zeta(\spp)$ at $\spp=0$ (cf. \S\ref{sec:modradsum:dirser}).

In order to obtain and analyze spanning sets for spaces of automorphic integrals we attach modified Rademacher sums $\QS{\Gamma,\cp|\cq}{\atp}{m}(\zz)$ to each triple $(\Gamma,\cp,\cq)$ where $\Gamma$ is a group commensurable with the modular group $\PSL_2(\ZZ)$ and $\cp$ and $\cq$ are cusps of $\Gamma$ (cf. \S\ref{sec:modradsum:constr}). The function $\QS{\Gamma}{\atp}{m}(\zz)$ is then recovered upon taking both $\cp$ and $\cq$ to be the {\em infinite cusp} $\Gamma\cdot\infty$. If we write $\QS{\Gamma,\cp}{\atp}{m}(\zz)$ as a shorthand for $\QS{\Gamma,\cp|\cq}{\atp}{m}(\zz)$ when $\cq=\Gamma\cdot\infty$, then, as is shown in \S\ref{sec:modradsum:var}, the modified Rademacher sums $\QS{\Gamma,\cp}{\atp}{m}(\zz)$, for varying cusps $\cp\in\Gamma\backslash\hat{\QQ}$ and positive integers $m\in\ZZp$, constitute a basis for the space $I_{\atp}(\Gamma)$ when $\atp<0$ (cf. Theorem \ref{thm:modradsum:var:Basis_QS_p_m}). In the case that $\atp=0$ the modified Rademacher sums $\QS{\Gamma,\cp}{\atp}{m}(\zz)$ span a subspace of $I_{0}(\Gamma)$ of codimension $1$, and a full basis is obtained by including a constant function. The function $\QS{\Gamma,\cp}{\atp}{m}(\zz)$ defines an automorphic integral with a single pole at the cusp $\cp$, and the function $\QS{\Gamma,\cp|\cq}{\atp}{m}(\zz)$ encodes the Fourier expansion of $\QS{\Gamma,\cp}{\atp}{m}(\zz)$ at the cusp $\cq$, and no correction of constant terms is necessary for the validity of these statements.

\subsection{Hecke operators}\label{sec:intro:heckeops}
 
The original result of Rademacher \cite{Rad_FuncEqnModInv} identifies the first Rademacher sum $\RS{\Gamma}{}{-1}(\zz)$, for $\Gamma=\PSL_2(\ZZ)$, with the function $J(\zz)+12$, where $J(\zz)$ is the normalized hauptmodul for the modular group. In light of this it is natural to try to identify more general Rademacher sums in a similar fashion. In this section we will take $\atp=0$, but we will allow $m$ to be an arbitrary non-zero integer. The identity
\begin{gather}
     \RS{\Gamma}{}{m}(\zz)
     =
     \overline{\CS{\Gamma}{}{-m}(\zz)}
\end{gather}
follows immediately from the definitions of the classical and conjugate Rademacher sums, demonstrating that the conjugate Rademacher sums of negative order may be recovered from the classical Rademacher sums with positive order. On the other hand, one can show, as we did for the special case that $m=-1$, that $\CS{\Gamma}{}{m}(\zz)$ is the constant function with constant value $-\fc{\Gamma}{}(m,0)/2$ in case $m<0$ and $\Gamma$ has genus zero. When $\Gamma=\PSL_2(\ZZ)$ we have $\fc{\Gamma}{}(m,0)=24\sigma(-m,1)$ for $m<0$, where $\sigma(n,1)$ is the sum of the divisors of $n$, so we have the following generalization of (\ref{eqn:intro:conradsum:Rad_Sum_-12}) for $m<0$.
\begin{gather}\label{eqn:intro:conradsum:Rad_Sum_m}
     \ee(m\bar{\zz})+
     \lim_{K\to \infty}
     \sum_{\substack{0<c<K\\-K^2<d<K^2\\(c,d)=1}}
     \ee\left(m\frac{a\bar{\zz}+b}{c\bar{\zz}+d}\right)
     -
     \ee\left(m\frac{a}{c}\right)
     =-12\sigma(-m,1)
\end{gather}

To obtain an expression for the classical Rademacher sums $\RS{\Gamma}{}{m}(\zz)$ with negative $m$ we can analyze again its Fourier coefficients $\fc{\Gamma}{}(m,n)$. It follows from our results in \S\ref{sec:moon:genus} that 
\begin{gather}
\fc{\Gamma}{}(m,n)=\delta_{m,n}
\end{gather}
for $m,n<0$ in case $\Gamma$ has genus zero. The resulting function $\TS{\Gamma}{}{m}(\zz)=\RS{\Gamma}{}{m}(\zz)-\fc{\Gamma}{}(m,0)/2$, satisfying the identity
\begin{gather}\label{eqn:intro:heckeops:Defn_GD_Gamma_m}
     \TS{\Gamma}{}{m}(\zz)
     +\frac{1}{2}\fc{\Gamma}{}(m,0)
     =
     \ee(m\zz)+
     \lim_{K\to \infty}
     \sum_{\gamma\in ({B(\ZZ)}\backslash\Gamma)_{\leq K}^{\times}}
     \ee(m\gamma\cdot\zz)-\ee(m\gamma\cdot\infty),
\end{gather}
is therefore holomorphic on the upper half plane, invariant for the action of $\Gamma$, and of the form $\vq^{-|m|}+o(1)$ for $\vq=\ee(\zz)$. Consequently each normalized Rademacher sum $\MT^{(-m)}_{\Gamma}(\zz)$, for $m>0$, is expressible as a degree $m$ polynomial in $\MT^{(-1)}_{\Gamma}(\zz)=\MT_{\Gamma}(\zz)$. This polynomial is none other than the so-called {\em $m$-th Faber polynomial for $\MT_{\Gamma}(\zz)$}.

Proceeding from another direction, we can obtain the functions $\MT^{(-m)}_{\Gamma}(\zz)$ from the normalized hauptmodul $\MT_{\Gamma}(\zz)$ by applying Hecke operators. In \S\ref{sec:struapp:hops} we study, in particular, the case that $\Gamma$ is the modular group $\PSL_2(\ZZ)$. Recall that for $n\in\ZZp$ the action of the Hecke operator $\HO(n)$ on a modular function $f(\zz)$ may be given by setting
\begin{gather}\label{eqn:intro:heckeops:HeckeOpDefn}
     (\HO(n)f)(\zz)=\frac{1}{n}
          \sum_{\substack{ad=n\\0\leq b<d}}
          f\left(\frac{a\zz+b}{d}\right).
\end{gather}
From this description one may deduce that the difference $\MT_{\Gamma}^{(-m)}(\zz)-m(\HO(m)\MT_{\Gamma})(\zz)$ is holomorphic in $\HH$ and vanishes as $\zz\to \ii\infty$. We thus obtain the remarkable expression
\begin{gather}\label{eqn:intro:heckeops:HeckeOpId}
     \MT_{\Gamma}^{(-m)}(\zz)=m(\HO(m)\MT_{\Gamma})(\zz)
\end{gather} 
for the higher order Rademacher sums which generalizes the original result of Rademacher. Combining the presentation (\ref{eqn:intro:radsum:Rad_Sum_J}) of $\MT_{\Gamma}(\zz)$ as a Rademacher sum with the definition (\ref{eqn:intro:heckeops:HeckeOpDefn}) of the Hecke operator $\HO(m)$ we obtain the expression
\begin{gather}\label{eqn:intro:heckeops:TmJ_RS}
     m(\HO(m)\MT_{\Gamma})(\zz)
     +\frac{1}{2}\fc{\Gamma}{}(-m,0)
     =\ee(-m\zz)+
     \lim_{K\to \infty}
     \sum_{\gamma\in ({B(\ZZ)}\backslash M(m))_{\leq K}^{\times}}
     \ee(-\gamma\cdot\zz)-\ee(-\gamma\cdot\infty),
\end{gather}
where $M(m)$ denotes the (image in $\PGL_2^+(\QQ)$ of the) set of $2\times 2$ matrices with integral entries and determinant $m$. Comparison of (\ref{eqn:intro:heckeops:Defn_GD_Gamma_m}) with (\ref{eqn:intro:heckeops:TmJ_RS}), in view of the identity (\ref{eqn:intro:heckeops:HeckeOpId}), suggests an equality of some, if not all, exponential terms in both sums, and in fact there is an injective map
\begin{gather}\label{eqn:intro_CosetHeckeCorresp}
     \begin{split}
          {B(\ZZ)}\backslash M(1)&\hookrightarrow
               {B(\ZZ)}\backslash M(m)\\
               \gamma&\mapsto \tilde{\gamma},
     \end{split}
\end{gather}
since $M(1)=\Gamma$ when $\Gamma=\PSL_2(\ZZ)$, with the property that $\ee(-m\gamma\cdot\zz)=\ee(-\tilde{\gamma}\cdot\zz)$. In the case that $m$ is prime, the remaining terms assemble into a fractional power Rademacher sum
\begin{gather}\label{eqn:intro_RadSumFracPow}
     \ee\left(-\frac{1}{m}\zz\right)+
     \lim_{K\to \infty}
     \sum_{\gamma\in (B(m\ZZ)\backslash\Gamma)_{\leq K}^{\times}}
     \ee\left(-\frac{1}{m}\gamma\cdot\zz\right)
     -
     \ee\left(-\frac{1}{m}\gamma\cdot\infty\right),
\end{gather}
where $B(m\ZZ)$ is the {\em $m$-fold translation group}, generated by $\zz\mapsto \zz+m$. Thus we obtain another injective correspondence
\begin{gather}\label{eqn:intro_CosetHeckeCorresp2}
     \begin{split}
          {B(m\ZZ)}\backslash M(1)&\hookrightarrow
               {B(\ZZ)}\backslash M(m)\\
               \gamma&\mapsto \tilde{\gamma}
     \end{split}
\end{gather}
with the property that $\ee(-\gamma\cdot\zz/m)=\ee(-\tilde{\gamma}\cdot\zz)$. The identity (\ref{eqn:intro:heckeops:TmJ_RS}) implies the vanishing of the fractional power Rademacher sum (\ref{eqn:intro_RadSumFracPow}). This can also be proven directly, and we give a general vanishing result for fractional Rademacher sums in \S\ref{sec:struapp:frac}. More generally we obtain fractional power sums of the form
\begin{gather}
     \ee\left(-\frac{m}{l^2}\zz\right)+
     \lim_{K\to \infty}
     \sum_{\gamma\in (B(l\ZZ)\backslash \Gamma)_{\leq K}^{\times}}
     \ee\left(-\frac{m}{l^2}\gamma\cdot\zz\right)
     -\ee\left(-\frac{m}{l^2}\gamma\cdot\infty\right),
\end{gather}
for each exact divisor $l$ of $m$, and each of these fractional power sums vanishes except for the sum corresponding to $l=1$. Thus our analysis covers generalizations of the Rademacher sums to fractional orders in addition to the higher integral orders. All the results discussed in this section admit generalizations to arbitrary groups commensurable with the modular group but we restrict ourselves to the modular group in this article.

\subsection{Monstrous Lie algebras}\label{sec:intro:mlas}

The presentation (\ref{eqn:intro:heckeops:HeckeOpId}) of the functions $\MT_{\Gamma}^{(-m)}(\zz)$ by means of Hecke operators (we restrict to the case $\Gamma=\PSL_2(\ZZ)$ in \S\ref{sec:intro:heckeops}) immediately implies that their Fourier coefficients are positive integers. This fact points to the existence of further algebraic structures beyond the moonshine module vertex operator algebra $\vn$. In his proof of the moonshine conjectures \cite{BorMM} Borcherds introduced the Monster Lie algebra $\mla$, which admits a presentation as a bi-graded generalized Kac--Moody algebra. The key to Borcherds' method is the denominator identity for $\mla$. This in turn can be interpreted as a BGG-type resolution of the trivial $\mla$-module. Let $\mlav$ denote the Verma module with highest weight $0$ for $\mla$. Then $\mlav$ is the first term in this BGG-type resolution of the trivial module, and the bi-graded dimension $\gdim\mlav$ of $\mlav$ is obtained by computing the coefficients of $\vp$ and $\vq$ in the expression
\begin{gather}\label{eqn:intro:mlas:VermaDim}
     \gdim \mlav =
     \exp\left(\sum_{m>0}(\HO(m)\MT_{\Gamma})(\zz)p^m\right)
\end{gather}
where $q=\ee(\zz)$. Applying the operator $F\mapsto p\partial_p \log F$ to (\ref{eqn:intro:mlas:VermaDim}) we obtain a generating function for the higher (absolute) order normalized Rademacher sums $\TS{\Gamma}{}{-m}(\zz)$, by (\ref{eqn:intro:heckeops:HeckeOpId}). It is well-known (cf. \cite{Mac_SymFnsHallPolys}) that this operator relates the complete homogeneous symmetric functions to the power-sum symmetric functions. Thus the bi-graded dimension of the Verma module for the Monster Lie algebra may be viewed as the {\em complete Rademacher sum} (of weight $0$) associated to $\Gamma$.

In \S\ref{sec:gravity:mlas} we consider a family of generalized Kac--Moody algebras $\{\mla_g\}$, indexed by elements $g$ in the Monster group $\MM$. Just as the Monster Lie algebra may be constructed from the moonshine module $\vn$, the monstrous Lie algebras $\mla_g$, studied by Carnahan in \cite{Car_Phd}, may be constructed from the $g$-twisted $\vn$-modules $\vn_g$, for $g\in \MM$. The graded dimensions of the spaces $\vn_g$ are given by functions $\GD_g(\zz)=\GD_{\Gamma_g}(\zz)$, related to the McKay--Thompson series $\MT_g(\zz)$ via the involution $\zz\mapsto-1/\zz$.
\begin{gather}
     \GD_g(\zz)=\MT_g(-1/\zz)
\end{gather}
One should note that Borcherds set the precedent here, introducing a family $\{\mla_g'\}$ of monstrous Lie superalgebras in \cite{BorMM}. These algebras $\mla_g'$ are more directly related to the functions $\MT_g(\zz)$, rather than the $\GD_g(\zz)$, and for our purposes the algebras $\mla_g$ appear to be more convenient.

The monstrous Lie algebras $\mla_g$ are constructed in \cite{Car_Phd} using the semi-infinite cohomology version of the no-ghost theorem. We identify the bi-graded subspaces of the algebra $\mla_g$ in terms of the twisted modules $\vn_{h}$, with $h\in\lab g\rab$. One can show that the $\mla_g$ are generalized Kac--Moody algebras, and one can also deduce formulas for the bi-graded dimensions of their Verma modules $\mlav_g$, generalizing (\ref{eqn:intro:mlas:VermaDim}).
\begin{gather}\label{eqn:intro:mlas:VermagDim}
     \gdim \mlav_g =
     \exp\left(\sum_{m>0}(\HO(m)\MT_g)(\zz)p^m\right)
\end{gather}
Carnahan \cite{Car_Phd} also found a remarkable denominator identity which yields an alternative expression for the graded dimension of $\mlav_g$; namely,
\begin{gather}\label{eqn:intro:gdim:mlas:mlav_TJ}
     \gdim\mlav_g
     =\frac{1}{\vp(\MT_g(\ww)-\GD_g(\zz))}
\end{gather}
where $p=\ee(\ww)$. This identity (\ref{eqn:intro:gdim:mlas:mlav_TJ}) demonstrates that the bi-graded dimension can be viewed as a meromorphic function on $\HH\times\HH$ with poles at $\ww\in\Gamma_g\cdot(-1/\zz)$.

We also show in \S\ref{sec:gravity:mlas} that the bi-graded dimension of $\mlav_g$ is bounded below, coefficient-wise, by the generating function
\begin{gather}\label{eqn:intro:mlas:GenFnTwRadSums}
     \sum_{m>0}\GD_g^{(-m)}(\zz)\vp^m
\end{gather}
where $\GD_g^{(-m)}(\zz)=\MT_g^{(-m)}(-1/\zz)$ when $g$ is of Fricke type (cf. Proposition \ref{prop:conseq:mlas:gdim_mlavg_Ineq}). This suggests that the Verma modules $\mlav_g$ may contain naturally defined subspaces whose bi-graded dimensions coincide with the expressions (\ref{eqn:intro:mlas:GenFnTwRadSums}).

\subsection{Chiral gravity}

Our results on Rademacher sums, at first glance, add as much to the mystery of monstrous moonshine as they reveal. We have formulated a characterization of the groups of monstrous moonshine in terms of Rademacher sums and solid tori. Having done so we face the new question of where the Rademacher sums themselves appear in relation to the moonshine module vertex operator algebra $\vn$? It appears now that the relation to physics should again be useful, as it was in the construction of $\vn$ \cite{FLM}, and the proof of the moonshine conjecture \cite{BorMM}.

In \S\ref{sec:gravity:conj1} we explain how the Rademacher sums associated to elements of the Monster might be regarded as giving strong evidence for the existence of another construction of the moonshine module $\vn$, and its twisted sectors $\vn_g$, in which all the features of monstrous moonshine, including the genus zero property, become transparent. In fact the original Rademacher sum (\ref{eqn:intro:radsum:Rad_Sum_J}) does appear as a saddle point approximation to the partition function of the simplest chiral three dimensional quantum gravity. (See \cite{Wit_3DGravRev}, \cite{MalWit_QGravPartFns3D}, \cite{Man_AdS3PFnsRecon}, \cite{ManMoo_ModFryTail}, \cite{LiSonStr_ChGrav3D} and \cite{MalSonStr_ChiGravLogGravExtCFT} for the development of this idea in the physics literature.) Our new identity (\ref{eqn:intro:conradsum:Rad_Sum_-12}) for the conjugate Rademacher sum must reflect the chiral nature of this three dimensional quantum gravity.

We expect that our analytic continuation of the Rademacher sums, and the explanation of the appearance of the constant term, will also develop from a rigorous analysis of the saddle point approximation using a zeta function regularization of the chiral gravity partition function, as was done for various models of two dimensional CFT. Our analysis of the Rademacher sums corresponding to the discrete groups of monstrous moonshine, and of moduli spaces of decorated solid tori points to a description of all the spaces $\vn_g$, for $g\in \MM$, via a family of $g$-twisted versions of the simplest chiral three dimensional quantum gravity. We expect that our reformulation of monstrous moonshine in terms of Rademacher sums and moduli of solid tori will eventually be understood from the properties of this remarkable chiral three dimensional quantum gravity.

In \S\ref{sec:gravity:conj2} we also elucidate the relationship between the higher order Rademacher sums, and the key object of Borcherds' proof of the moonshine conjecture; viz., the Monster Lie algebra, and its twisted counterparts. This leads us to the stringy quantization (cf. \cite{DijMooVerVer_EllGenSndQntStrgs}) of $\vn$, and we generalize this to all twisted sectors $\vn_g$, for $g\in \MM$. We prove in Theorem \ref{thm:conseq:conj2:Isom_Sec_Quant_Sym_mlag} that the resulting second quantized spaces have the same bi-graded dimensions as the Verma modules $\mlav_g$ (cf. (\ref{eqn:intro:mlas:VermagDim})) and thus may be viewed as stringy realizations of the representations of the generalized Kac--Moody algebras $\mla_g$. It is an interesting problem to interpret the spaces $\mlav_g$, and their subspaces with graded dimension given by the higher order normalized Rademacher sums (\ref{eqn:intro:mlas:GenFnTwRadSums}), in terms of the second quantized chiral three dimensional quantum gravities. In particular, the simple description of the singularities of the partition function given by (\ref{eqn:intro:gdim:mlas:mlav_TJ}) should have a natural interpretation in terms of quantum gravity.

Thus our present results on the Rademacher sums, and the new conjectures about their origin, clearly indicate that the real nature of moonshine might not be as remote anymore, and more than this, the full examination of its structure might give new insight into the fundamental problem of modern physics.

%------------------------------------------------------------------%
\section{Conventions}\label{sec:conven}
%------------------------------------------------------------------%

We write $\ZZ^+$ for the set of positive integers, and the notations $\QQp$ and $\RR^+$ are to be interpreted similarly. We write $\NN$ for the set of non-negative integers. 
For $R$ a ring without zero divisors, we write $R^{\times}$ for the multiplicative monoid of non-zero elements in $R$.

\subsection{Functions}\label{sec:conven:fns}

For $z,s\in \CC$ with $\zz\neq 0$, we write $z^{s}$ as a shorthand
for $\exp(s\log(z))$, where $\log$ denotes the {principal branch of
the logarithm}, so that
\begin{gather}
     -\pi<\Im(\log(z))\leq \pi,
\end{gather}
and we write $z^{(s)}$ as a shorthand for $z^{s}/\Gamma(s+1)$,
where $\Gamma(\spp)$ denotes the {Gamma function}. %Then for $k$ a
%non-negative integer, $z^{(k)}$ is the {divided power} $z^k/k!$.
We adopt the convention of setting
\begin{gather}
     \ee(z)=\exp(\tpi z)=\sum_{k\geq 0} (\tpi z)^{(k)}
\end{gather}
for $z\in \CC$. We write $\Phi(a,b,\zz)$ for the analytic function
on $\CC^3$ defined by setting
\begin{gather}\label{eqn:conven:fns:Defn_Phi}
     \Phi(a,b,\zz)
          =
          \frac{_1F_1(a;b;2\pi\ii\zz)}{\Gamma(b)}
          =
          \sum_{k\geq 0}
          \frac{\Gamma(k+a)}{\Gamma(a)\Gamma(k+b)}(2\pi\ii\zz)^{(k)}
\end{gather}
where $_1F_1(a;b;x)$ denotes the {confluent hypergeometric
function} (cf. \cite[\S13]{AbrSte_Handbook}). {Kummer's transformations} for the confluent
hypergeometric function yield the following identity for $\Phi$.
\begin{gather}\label{eqn:conven:fns:Kummer_Xform}
     \Phi(a,b,\zz)=\ee(\zz)\Phi(b-a,b,-\zz)
\end{gather}
The exponential function $\ee(\zz)$ is a solution to the
differential equation
\begin{gather}
     (\zz\partial_{\zz}+1-\spp)(\partial_{\zz}-\tpi)u(\zz)=0,
\end{gather}
which has a regular singular point at $0$ and an irregular
singularity at $\infty$. Another solution is furnished by the
function $\zz\mapsto\Phi(1,1+\spp,\zz)(\tpi\zz)^{\spp}$, which we
denote also by $\ee(\zz,\spp)$.
\begin{gather}\label{eqn:conven:fns:Genzd_Exp}
     \ee(\zz,\spp)
     =
     \Phi(1,1+\spp,\zz)(\tpi\zz)^{\spp}
     =
     \sum_{k\geq 0}
     (\tpi\zz)^{(k+\spp)}
\end{gather}
Observe that we have $\ee(\zz,n)=\ee(\zz)-\ee(\zz)_{<n}$ for $n\in
\ZZ$, where $\ee(z)_{< K}$ denotes the {\em partial exponential}
$\ee(z)_{< K}=\sum_{{ 0\leq k< K}} (\tpi z)^{(k)}$. In
particular, $\ee(\zz,n)=\ee(\zz)$ when $n\leq 0$.

\subsection{Divisors}\label{sec:conven:div}

For $n\in\ZZp$ and $\spp\in \CC$ we write $\sigma(n,\spp)$ for the
{\em divisor function}
\begin{gather}\label{eqn:conven:fns:Div_Fn}
     \sigma(n,\spp)=\sum_{d|n}d^{\spp}.
\end{gather}

For $a,b\in \ZZ$ we write $(a,b)$ for the greatest common divisor of
$a$ and $b$, and we generalize this notation by writing $((a,b)$ for
the greatest positive integer $c$ say, such that $c^2|a$ and $c|b$,
and by writing $(a^{\infty},b)$ for the largest divisor of $b$ that
divides some power of $a$.
\begin{gather}\label{eqn:conven:div:genzd_gcd}
     ((a,b)={\rm max}
     \left\{
     c\in\ZZp\mid
     c^2|a\text{ and }c|b
     \right\}\\
     (a^{\infty},b)=\lim_{n\to\infty}(a^n,b)
\end{gather}
For $d,n\in \ZZp$ we write $d\|n$ in the case that $d|n$ and
$(d,n/d)=1$, and call such a $d$ an {\em exact divisor of $n$}. We
write $\Ex(n)$ for the set of exact divisors of $n$. Then $\Ex(n)$
becomes a group of exponent $2$ when equipped with the product
$(e,f)\mapsto ef/(e,f)^2$. Indeed, if we write $\Pi(n)$ for the set
of primes that divide $n$ then there is a naturally defined
isomorphism $\Ex(n)\cong (\ZZ/2)^{\Pi(n)}$ which associates to an
exact divisor $e\in \Ex(n)$ the indicator function
$\chi_e:\Pi(n)\to\ZZ/2$, satisfying $\chi_e(p)=1$ if $p|e$ and
$\chi_e(p)=0$ otherwise.

Let $\Pi$ denote the set of all primes. Then the infinite product
$(\ZZ/2)^{\Pi}$ comes equipped with a naturally defined projection
$(\ZZ/2)^{\Pi}\to(\ZZ/2)^{\Pi(n)}$, and a naturally defined section
for it $(\ZZ/2)^{\Pi(n)}\to(\ZZ/2)^{\Pi}$, for each $n\in \ZZp$.
Consequently we have naturally defined morphisms of groups $(\ZZ/2)^{\Pi}\to\Ex(n)$ and
$\Ex(n)\to (\ZZ/2)^{\Pi}$ for each $n$, and for any $m,n\in\ZZp$ we
have the natural map $\Ex(m)\to\Ex(n)$ furnished by the composition
$\Ex(m)\to(\ZZ/2)^{\Pi}\to\Ex(n)$. If $m$ divides $n$ then the map $\Ex(m)\to \Ex(n)$ is
injective, the map $\Ex(n)\to\Ex(m)$ is surjective, and the composition $\Ex(m)\to\Ex(n)\to\Ex(m)$ is the identity on $\Ex(m)$.

\subsection{Isometries}\label{sec:conven:isomhyp}

Let us write $G(\RR)$ for the simple real Lie group $\PSL_2(\RR)$.
The group $G(\RR)$ acts naturally, from the left, on the upper-half
plane $\HH=\{\zz\in\CC\mid\Im(\zz)>0\}$, and is just the group of
orientation preserving isometries of $\HH$ when we equip it with the
{\em hyperbolic measure} ${\rm d}\mu(\zz)={{\rm d}x{\rm d}y}/{y^2}$,
for $\zz=x+\ii y$.

It is convenient to enlist matrices in $\SL_2(\RR)$ for the purpose
of specifying elements of $G(\RR)$. We write $\sqnom{a\;b}{c\;d}$
for the image of a matrix $\binom{a\;b}{c\;d}\in \SL_2(\RR)$ in
$G(\RR)$. Analogously, we write $A\mapsto [A]$ for the canonical map
$\SL_2(\RR)\to G(\RR)$. The action of $G(\RR)$ on $\HH$ is now
described explicitly by the formula
\begin{gather}
     \gamma\cdot\zz
     =
     \frac{a\zz+b}{c\zz+d}
\end{gather}
for $\gamma=\sqnom{a\;b}{c\;d}$. The matrix
$\ogi=\binom{-1\;0}{0\;\;1}\in\GL_2(\RR)$ induces an
outer-automorphism of $G(\RR)$, which we call {\em conjugation on
$G(\RR)$}, and which we denote $\gamma\mapsto \bar{\gamma}$.
Explicitly, we have
\begin{gather}\label{eqn:conven:isomhyp:Conj}
     \bar{\gamma}
     =
     \left[
       \begin{array}{cc}
         -a & b \\
         c & -d \\
       \end{array}
     \right]
     \Longleftarrow
     \gamma=
     \left[
       \begin{array}{cc}
         a & b \\
         c & d \\
       \end{array}
     \right].
\end{gather}

For each matrix $A\in \GL_2^+(\QQ)$ there is a unique $\mu\in\RRp$
for which $A'=\mu A$ belongs to $\SL_2(\RR)$. The assignment
$\GL_2^+(\QQ)\to G(\RR)$ given by $A\mapsto [A']$ then factors
through $\PGL_2^+(\QQ)$, and the resulting map $\PGL_2^+(\QQ)\to
G(\RR)$ is an embedding of groups. We write $G(\QQ)$ for the copy of
$\PGL_2^+(\QQ)$ in $G(\RR)$ obtained in this way. We write $G(\ZZ)$
for the {\em modular group} $\PSL_2(\ZZ)$.

Extending the notation introduced above, we write $[A]$ and
$\sqnom{a\;b}{c\;d}$ for the image in $G(\QQ)<G(\RR)$ of a matrix
$A=\binom{a\;b}{c\;d}$ in $\GL_2^+(\QQ)$. Given $\mu\in \QQp$ we
write $[\mu]$ as a shorthand for $[A]$ when $A$ is the diagonal
matrix $\binom{\mu\;0}{0\;1}$. Then $[\mu]\cdot\zz=\mu\zz$ for
$\zz\in \HH$.
\begin{gather}\label{eqn:conven:isomhyp:Defn_[alpha]}
     [\mu]
     =
     \left[
       \begin{array}{cc}
         \mu & 0 \\
         0 & 1 \\
       \end{array}
     \right]
\end{gather}

Let $B(\RR)$ denote the Borel subgroup of $G(\RR)$ consisting of
images of upper-triangular matrices in $\SL_2(\RR)$, so that
$B(\RR)$ is just the subgroup of $G(\RR)$ that fixes the
distinguished point $\infty$ on the {\em pointed real projective
line} $\hat{\RR}=\RR\cup\{\infty\}$. Let $B(\QQ)$ and $B(\ZZ)$
denote the intersections $B(\RR)\cap G(\QQ)$ and $B(\RR)\cap
G(\ZZ)$, respectively. Then $B(\ZZ)$ is the group generated by the
{\em modular translation} $T=\sqnom{1\;1}{0\;1}$. Set
$\hat{\QQ}=\QQ\cup\{\infty\}\subset\hat{\RR}$. Then $G(\QQ)$ acts
naturally on $\hat{\QQ}$, and $B(\QQ)$ is just the subgroup of
$G(\QQ)$ that fixes the point $\infty$. For $\alpha\in \QQ$ we write
$T^{\alpha}$ for the element $\sqnom{1\;\alpha}{0\;1}\in B(\QQ)$.
\begin{gather}\label{eqn:conven:isomhyp:Defn_T^alpha}
     T^{\alpha}=
     \left[
     \begin{array}{cc}
     1 & \alpha \\
     0 & 1 \\
     \end{array}
     \right]
\end{gather}
Then the elements of $B(\QQ)$ of the form $T^{\alpha}$ for
$\alpha\in \QQ$ constitute the {\em unipotent subgroup of $B(\QQ)$},
which we denote $B_u(\QQ)$. Given $\alpha\in \QQp$ we write
$B(\alpha\ZZ)$ for the subgroup of $B_u(\QQ)$ generated by
$T^{\alpha}$.
\begin{gather}\label{eqn:conven:isomhyp:Defn_B(alphaZZ)}
     B(\alpha\ZZ)
     =
     \left\{
     \left[
     \begin{array}{cc}
     1 & \alpha n \\
     0 & 1 \\
     \end{array}
     \right]
     \mid
     n\in\ZZ
     \right\}
\end{gather}
The elements of $B(\QQ)$ of the form $[\mu]$ for $\mu\in \QQp$
constitute the {\em diagonal subgroup of $B(\QQ)$}, which we denote
$B_d(\QQ)$. The group $B(\QQ)$ is naturally isomorphic to the
semidirect product $B_u(\QQ)\rtimes B_d(\QQ)$. In particular,
$B_u(\QQ)$ is a normal subgroup of $B(\QQ)$ and for any $\chi\in
B(\QQ)$ we have $\chi=T^{\alpha}[\mu]$ for some uniquely determined
$\alpha\in \QQ$ and $\mu\in \QQp$.

An element $\gamma\in G(\QQ)$ will have many preimages in
$\GL_2^+(\QQ)$, but among these there will be exactly two that have
integral entries with no common divisor. Either of these two
matrices will be called a {\em preferred representative} for
$\gamma$. We define functions $\Pdet:G(\QQ)\to \ZZp$ and $c:
G(\QQ)\to \NN$ and $d: G(\QQ)\to\NN$ by setting
\begin{gather}\label{eqn:conven:isomhyp:Defn_Pdet}
     \Pdet(\gamma)
     =
     ad-bc,\quad
     c(\gamma)
     =
     |c|,\quad
     d(\gamma)
     =|d|,
\end{gather}
in case $\binom{a\;b}{c\;d}\in GL_2^+(\QQ)$ is a preferred
representative for $\gamma$. We call $\Pdet$ the {\em projective
determinant}. The group $G(\ZZ)$ is exactly the preimage of $1$ with
respect to the projective determinant. The projective determinant is
not multiplicative, but we have
$\Pdet(\gamma\sigma)=\Pdet(\sigma)=\Pdet(\sigma\gamma)$ for all
$\sigma\in G(\QQ)$ in case $\Pdet(\gamma)=1$, so for any $n\in
\ZZp$ the preimage of $n$ under $\Pdet$ admits commuting left and right actions of $G(\ZZ)$.

Observe that $c(T\gamma)=c(\gamma T)=c(\gamma)$ and
$d(T\gamma)=d(\gamma)$. Given $X\subset G(\QQ)$ we set
$X_{\infty}=X\cap B(\QQ)$ and $X^{\times}=X-X_{\infty}$. Then
$X_{\infty}=\{\chi\in X\mid c(\chi)=0\}$. We define
$X^{\times\times}$ to be the subset of $X^{\times}$ consisting of
$\chi\in X$ such that both $c(\chi)$ and $d(\chi)$ are non-zero.
\begin{gather}
     X_{\infty}\label{eqn:conven:isomhyp:Defn_X_infty}
     =
     \left\{\chi\in X\mid c(\chi)=0\right\},
     \\
     X^{\times}\label{eqn:conven:isomhyp:Defn_X_times}
     =
     \left\{\chi\in X\mid c(\chi)\neq 0\right\},
     \quad
     X^{\times\times}
     =
     \left\{\chi\in X\mid c(\chi)d(\chi)\neq 0\right\}
\end{gather}

For $\gamma\in G(\QQ)$ with $\binom{a\;b}{c\;d}$ a preferred
representative, we have
\begin{gather}\label{eqn:conven:isomhyp_DiffLeftMults}
     \gamma\cdot\zz-\gamma\cdot\zz'=
     \frac{\Pdet(\gamma)}{(c\zz+d)(c\zz'+d)}
     (\zz-\zz')
\end{gather}
for $\zz,\zz'\in\HH$, so that the derivative of the function
$\zz\mapsto \gamma\cdot\zz$, which we will denote
$\jac(\gamma,\zz)$, is given by
\begin{gather}\label{eqn:conven:isomhyp_jacleftmult}
     \jac(\gamma,\zz)=\frac{\Pdet(\gamma)}{(c\zz+d)^2}.
\end{gather}
Note that the assignment $\jac:G(\QQ)\to \mc{O}(\HH)$, which
associates the function $\zz\mapsto \jac(\gamma,\zz)$ to an element
$\gamma\in G(\QQ)$, descends naturally to the coset space
$B(\ZZ)\backslash G(\QQ)$, and even further to $B_u(\QQ)\backslash
G(\QQ)$.

From (\ref{eqn:conven:isomhyp_DiffLeftMults}) and
(\ref{eqn:conven:isomhyp_jacleftmult}) we see that
$|\gamma\cdot\zz-\gamma\cdot\zz'|=|\zz-\zz'|$ just in the case that
$\zz$ and $\zz'$ belong to the set $\{\zz\in\HH\mid
|\jac(\gamma,\zz)|=1\}$ which we call the {\em isometric locus of
$\gamma$}. The isometric locus of $\gamma$ is a (Euclidean)
semicircle just when $c(\gamma)>0$, in which case its center is
$\gamma^{-1}\cdot\infty$, and its radius is $\sqrt{\rads(\gamma)}$,
where
\begin{gather}\label{eqn:conven:isomhyp:Scaling_Factor}
     \rads(\gamma)=\frac{\Pdet(\gamma)}{c(\gamma)^2}
     =\frac{ad-bc}{c^2}
\end{gather}
in case $\binom{a\;b}{c\;d}$ is a preferred
representative for $\gamma$. We call $\rads(\gamma)$ the {\em scaling factor} associated to $\gamma$.

\subsection{Cosets}\label{sec:conven:cosets}

Given a subset $X\subset G(\QQ)$, we will write $\lBZ X\rBZ $ for
the set consisting of double cosets of the form $B(\ZZ)\chi B(\ZZ)$,
for $\chi\in X$. We will write $\lBZ \chi\rBZ $ as a shorthand for
the particular double coset $B(\ZZ)\chi B(\ZZ)$.
\begin{gather}
     \lBZ\chi\rBZ=B(\ZZ)\chi B(\ZZ),\quad
     \lBZ X\rBZ =
     \left\{\lBZ \chi\rBZ
          \in B(\ZZ)\backslash G(\QQ)/B(\ZZ)\mid\chi\in X\right\}
\end{gather}
Similarly, we will write $\lBZ X\rBZh $ for the set of right cosets
of the form $B(\ZZ)\chi$, for $\chi\in X$, and $\chi\mapsto \lBZ
\chi\rBZh $ will denote the natural map $X\to \lBZ X\rBZh $, and the
notations $\lBZh X\rBZ $ and $\chi\mapsto \lBZh \chi\rBZ $ will have
the analogous meanings. Observe that $B(\ZZ)$ is stable under the
operations of inversion and conjugation (cf.
(\ref{eqn:conven:isomhyp:Conj})), and thus these operations descend
naturally to the coset spaces $\lBZ G(\QQ)\rBZh$ and $\lBZh
G(\QQ)\rBZ$ and $\lBZ G(\QQ)\rBZ$.

Recall from \S\ref{sec:conven:isomhyp} that the assignment
$\chi\mapsto \jac(\chi,\zz)$ descends to a well-defined association
of holomorphic functions on $\HH$ to right cosets of $B(\ZZ)$ in
$G(\QQ)$. Accordingly, we may safely set
$\jac(\lBZ\chi\rBZh,\zz)=\jac(\chi,\zz)$, when given some
$\lBZ\chi\rBZh\in \lBZ G(\QQ)\rBZh$.
\begin{gather}
     \begin{split}
          \jac(\cdot\,,\zz):\lBZ G(\QQ)\rBZh&\to\mc{O}(\HH)\\
          \lBZ\chi\rBZh&\mapsto \jac(\lBZ\chi\rBZh,\zz)
     \end{split}
\end{gather}
Given $U\subset\lBZ G(\QQ)\rBZh$ or $S\subset\lBZ G(\QQ)\rBZ$ we set
$U_{\infty}=U\cap\lBZ B(\QQ)\rBZh$ and $U^{\times}=U-U_{\infty}$,
and similarly, $S_{\infty}=S\cap\lBZ B(\QQ)\rBZ$ and
$S^{\times}=S-S_{\infty}$. In particular, $\lBZ G(\QQ)\rBZ^{\times}$
denotes the set of double cosets $\lBZ \chi\rBZ $, for $\chi\in
G(\QQ)$, for which $c(\chi)\neq 0$.

Observe that the isometric loci (cf. \S\ref{sec:conven:isomhyp})
associated to $\chi,\chi'\in G(\QQ)$ coincide when
$\lBZ\chi\rBZh=\lBZ\chi'\rBZh$. Thus we may regard these loci as
naturally attached to right cosets of $B(\ZZ)$ in $G(\QQ)$, and the
locus attached to $\lBZ\chi\rBZh$ is a Euclidean semicircle just
when $\lBZ\chi\rBZh\in\lBZ G(\QQ)\rBZh^{\times}$. Also, the loci
associated to cosets $\lBZ\chi\rBZh,\lBZ\chi'\rBZh\in\lBZ
G(\QQ)\rBZh^{\times}$ have the same radii when
$\lBZ\chi\rBZ=\lBZ\chi'\rBZ$; that is to say, the function $\rads$
of (\ref{eqn:conven:isomhyp:Scaling_Factor}) descends to a
well-defined map $\lBZ G(\QQ)\rBZ^{\times}\to \QQp$. It is worth
noting that the function $\chi\mapsto\rads(\chi)$ actually satisfies
an even stronger invariance condition.
\begin{lem}\label{lem:conven:cosets_RadsInv}
The assignment $\lBZ\chi\rBZ\mapsto \rads\lBZ\chi\rBZ$ descends
naturally to a well-defined function on the double coset space
$B_u(\QQ)\backslash G(\QQ)/ B_u(\QQ)$. That is, we have $\rads\lBZ
T^{\alpha}\chi\rBZ=\rads\lBZ\chi T^{\alpha}\rBZ=\rads\lBZ\chi\rBZ$
for any $\alpha\in\QQ$ and $\chi\in G(\QQ)^{\times}$.
\end{lem}
The function $\rads$ is also invariant under inversion and
conjugation, so that we have
$\rads\lBZ\chi\rBZ=\rads\lBZ\bar{\chi}\rBZ=\rads\lBZ\chi^{-1}\rBZ$
for any $\chi\in G(\QQ)^{\times}$. For $\chi\in G(\QQ)^{\times}$ the
identity (\ref{eqn:conven:isomhyp_DiffLeftMults}) may be rewritten
\begin{gather}
     \chi\cdot\zz-\chi\cdot\zz'
     =\frac{\rads\lBZ\chi\rBZ}
     {(\zz-\chi^{-1}\cdot\infty)(\zz'-\chi^{-1}\cdot\infty)}\
     (\zz-\zz'),
\end{gather}
and we may consider the limit as $\zz'\to\infty$, which yields a
useful expression for $\chi\cdot\zz-\chi\cdot\infty$; viz.,
\begin{gather}\label{eqn:conven:cosets_UsefulChiDotTau}
     \chi\cdot\zz-\chi\cdot\infty
     =-\frac{\rads\lBZ\chi\rBZ}{\zz-\chi^{-1}\cdot\infty}.
\end{gather}
Comparing with (\ref{eqn:conven:isomhyp_jacleftmult}) and
(\ref{eqn:conven:cosets_UsefulChiDotTau}) we have
\begin{gather}\label{eqn:conven:isomhyp:jac_rads_over_zchi}
     \jac(\lBZ\chi\rBZh,\zz)
     =\frac{\rads\lBZ\chi\rBZ}
     {(\zz-\lBZh\chi^{-1}\rBZ\cdot\infty)^2}
\end{gather}
for $\lBZ\chi\rBZh\in\lBZ G(\QQ)\rBZh^{\times}$ (and
$\jac(\lBZ\chi\rBZh,\zz)=1$ otherwise).

The functions $c,d:G(\QQ)\to\NN$, of \S\ref{sec:conven:isomhyp},
descend to well-defined functions on the coset space $\lBZ
G(\QQ)\rBZh=B(\ZZ)\backslash G(\QQ)$. The function $c$ descends
further, to be well-defined on the double coset space $\lBZ
G(\QQ)\rBZ =B(\ZZ)\backslash G(\QQ)/B(\ZZ)$. For $K\in \RRp$ and $U$
a subset of $\lBZ G(\QQ)\rBZh$, define $U_{\leq K}$ to be the subset
of $U$ consisting of cosets $\lBZ \chi\rBZh \in U$ for which
$c(\chi)\leq K$ and $d(\chi)\leq K^2$. Analogously, for $S$ a subset
of $\lBZ G(\QQ)\rBZ$, define $S_{\leq K}$ to be the subset of $S$
consisting of double cosets $\lBZ \chi\rBZ \in S$ for which
$c(\chi)\leq K$.
\begin{gather}\label{eqn:conven:cosets:Defn U_sub_K}
     U_{\leq K}
     =\left\{\lBZ \chi\rBZh \in U\mid
          c(\chi)\leq K,\,d(\chi)\leq K^2\right\}\\\label{eqn:conven:cosets:Defn S_sub_K}
     S_{\leq K}
     =\left\{\lBZ \chi\rBZ \in S\mid
          c(\chi)\leq K\right\}
\end{gather}

\subsection{Groups}\label{sec:conven:groups}

The group $G(\QQ)$ is the commensurator of $G(\ZZ)$. Consequently,
any group $\Gamma<G(\RR)$ that is commensurable with the modular
group $G(\ZZ)$ is automatically a subgroup of $G(\QQ)$, and is an
example of a {Fuchsian group of the first kind}. For such a group
$\Gamma$ there is a natural way to equip the orbit space
$\Gamma\backslash \HH$ with the structure of a Riemann surface; we
will denote this object by ${\sf Y}_{\Gamma}$. Any group $\Gamma$
commensurable with $G(\ZZ)$ has parabolic elements, and any fixed
point of any parabolic element of $\Gamma$ lies on the (pointed)
rational projective line $\hat{\QQ}\subset\hat{\RR}$. The orbit
space ${\sf P}_{\Gamma}=\Gamma\backslash\hat{\QQ}$ is the set of
{\em cusps of $\Gamma$}. For $\Gamma$ commensurable with $G(\ZZ)$,
the Riemann surface ${\sf Y}_{\Gamma}$ is not compact, but can be
compactified in a natural way by adjoining a single point for each
cusp of $\Gamma$. We will write ${\sf X}_{\Gamma}$ for the
corresponding compact Riemann surface obtained by the adjunction of
the cusps of $\Gamma$.
\begin{gather}\label{eqn:conven:groups:Defn_X_Gamma_Y_Gamma_P_Gamma}
     {\sf X}_{\Gamma}
     ={\sf Y}_{\Gamma}\cup{\sf P}_{\Gamma}
     =\Gamma\backslash\HH\cup\Gamma\backslash\hat{\QQ}
\end{gather}
We will say that $\Gamma$ is a {\em group of genus zero} in the case
that ${\sf X}_{\Gamma}$ has genus zero as a Riemann surface. We set
$\genus(\Gamma)=\genus({\sf X}_{\Gamma})$ for $\Gamma<G(\RR)$
commensurable with $G(\ZZ)$.

We write $\hat{\CC}$ for the {\em (pointed) complex projective line}
$\CC\cup\{\infty\}$, which we regard as a compact Riemann surface in
the usual way. If $\Gamma$ is a group for which the compact Riemann
surface ${\sf X}_{\Gamma}$ has genus zero, then there is an
isomorphism $\phi:{\sf X}_{\Gamma}\xrightarrow{\sim}\hat{\CC}$ that
witnesses this fact, and the set of such isomorphisms admits a simply transitive action by $G(\CC)$. Thus we may assume that $\phi$ maps the point of ${\sf
X}_{\Gamma}$ corresponding to the {\em infinite cusp}
$\Gamma\cdot\infty$ to the distinguished point $\infty\in
\hat{\CC}$. Such an isomorphism $\phi$ determines a
$\Gamma$-invariant holomorphic function, $f$ say, on $\HH$, which
admits an expression $f(\zz)=\sum_{n\geq -1}c(n/h)\ee(n\zz/h)$ with
$c(-1/h)\neq 0$ for some positive integer $h$ called the {\em width of $\Gamma$ at infinity}. After multiplying by $1/c(-1/h)$ we may assume that
$c(-1/h)=1$. A $\Gamma$-invariant function of the form
$\ee(-\zz/h)+\sum_{n\geq 0}c(n/h)\ee(n\zz/h)$ which is holomorphic on
$\HH$ and induces an isomorphism ${\sf X}_{\Gamma}\xrightarrow{\sim}
\hat{\CC}$ will be called a {\em hauptmodul} for $\Gamma$. In case
$c(0)=0$ we say that $f$ is a {\em normalized hauptmodul} for
$\Gamma$.

The {\em Hecke congruence groups}, denoted $\Gamma_0(n)$ for $n\in
\ZZp$, play a special r\^ole in our analysis.
\begin{gather}
     \Gamma_0(n)
     =\left\{
     \left[
       \begin{array}{cc}
         a & b \\
         cn & d \\
       \end{array}
     \right]\mid
     a,b,c,d\in \ZZ,\,ad-bcn=1
     \right\}
\end{gather}
According to \cite{ConNorMM} the normalizer $N(\Gamma_0(n))$ of
$\Gamma_0(n)$ in $G(\RR)$ is commensurable with $G(\ZZ)$ and admits
the description
\begin{gather}\label{eqn:conven:groups:Normalizer_Gamma0(n)}
     N(\Gamma_0(n))
     =\left\{
     \left[
       \begin{array}{cc}
         ae & b/h \\
         cn/h & de \\
       \end{array}
     \right]\mid
     a,b,c,d\in \ZZ,\,e\in \ZZp,\,e\|n/h,\,ade-bcn/eh^2=1
     \right\}
\end{gather}
where $h=((n,24)$ is the largest divisor of $24$ such that $h^2$ divides
$n$ (cf. \S\ref{sec:conven:div}). The expression (\ref{eqn:conven:groups:Normalizer_Gamma0(n)})
tells us, in particular, that $N(\Gamma_0(n))_{\infty}=\lab
T^{1/h}\rab$ for $h=((n,24)$.
\begin{prop}\label{prop:conven:groups:Normalizer_Gamma0(n)}
Let $n\in \ZZp$. Then $N(\Gamma_0(n))$ acts transitively on
$\hat{\QQ}$. We have $N(\Gamma_0(n))_{\infty}=\Gamma_0(n)_{\infty}$
if and only if $n$ is not divisible by $4$ or $9$.
\end{prop}
An important family of groups, each one commensurable with the
modular group $G(\ZZ)$, and each one containing and normalizing some
$\Gamma_0(n)$, was introduced in \cite{ConNorMM}; these are the
groups of {\em $n\|h$-type}, and we now recall their definition. For $n,h\in \ZZp$ with $h|n$, and for $S$ a
subgroup of $\Ex(n/h)$ (cf. \S\ref{sec:conven:div}), we define a group $\Gamma_0(n|h)+S$ by
setting
\begin{gather}\label{eqn:conven:groups:Gamma0(n|h)+S}
     \Gamma_0(n|h)+S
     =\left\{
     \left[
       \begin{array}{cc}
         ae & b/h \\
         cn & de \\
       \end{array}
     \right]
     \mid
     a,b,c,d\in \ZZ,\,
     e\in S,\,
     ade-bcn/eh=1
     \right\}.
\end{gather}
Evidently the group $\Gamma_0(n|h)+S$ contains $\Gamma_0(nh)$. In
case $h|24$ it also normalizes $\Gamma_0(nh)$. Assume then that
$h|24$. We say that $\Gamma_0(n|h)+S$ is of {\em Fricke type} if the
element $\sqnom{0\;-1}{nh\;\;0}$ belongs to it.

The group $\Gamma_0(n\|h)+S$ is defined (cf. \cite{ConNorMM}) as the
subgroup of $\Gamma_0(n|h)+S$ arising as the kernel of a certain
morphism $\Gamma_0(n|h)+S\to \ZZ/h$ which factors through the
canonical map $\Gamma_0(n|h)+S\to (\Gamma_0(n|h)+S)/\Gamma_0(nh)$.
In order to describe it let $S'$ denote the image of $S$ in $\Ex(nh)$ under the natural injection $\Ex(n/h)\to\Ex(nh)$ (cf. \S\ref{sec:conven:div}). Then as generators for the quotient $(\Gamma_0(n|h)+S)/\Gamma_0(nh)$
we may take the cosets $X$, $Y$, and $W_{e'}$ for $e'\in S'$, given
by
\begin{gather}
     X
     =
     \left[
       \begin{array}{cc}
         1 & 1/h \\
         0 & 1 \\
       \end{array}
     \right]
     \Gamma_0(nh),\quad
     Y
     =
     \left[
       \begin{array}{cc}
         1 & 0 \\
         n & 1 \\
       \end{array}
     \right]
     \Gamma_0(nh),\quad
     W_{e'}
     =
     \left[
       \begin{array}{cc}
         ae' & b \\
         cnh & de' \\
       \end{array}
     \right]
     \Gamma_0(nh),
\end{gather}
where the $a$, $b$, $c$ and $d$ in the definition of $W_{e'}$ are
arbitrary integers for which $ade'-bcnh/e'=1$, and we may define $\Gamma_0(n\|h)+S$ to be the
kernel of the composition
\begin{gather}
  \Gamma_0(n|h)+S
  \to
  (\Gamma_0(n|h)+S)/\Gamma_0(nh)
  \xrightarrow{\lambda}
  \ZZ/h
\end{gather}
where the map $\lambda$ is determined by the requirements that
$\lambda(W_{e'})=0$ for all $e'\in S'$, $\lambda(X)=1$, and
$\lambda(Y)$ is $1$ or $-1$ according as $\Gamma_0(n|h)+S$ is of
Fricke type or not. 

The group $\Gamma_0(n\|h)+S$ is also described in \cite{Fer_Genus0prob}, where it is denoted $(1/h)\Gamma_0(n|h)+e_1,e_2,\ldots$ for $S=\{1,e_1,e_2,\ldots\}$, although we should remark that certain variations on $\lambda$ are considered there (in the case that $h$ is even), so that the class of groups considered in \cite{Fer_Genus0prob} properly contains the groups of $n\|h$-type.

\subsection{Scalings}\label{sec:conven:scaling}

Let $\Gamma$ be a group commensurable with $G(\ZZ)$ and let $\cp\in
\cP_{\Gamma}$ (cf.
(\ref{eqn:conven:groups:Defn_X_Gamma_Y_Gamma_P_Gamma})) be a cusp of
$\Gamma$. An element $\cpr\in \cp\subset \hat{\QQ}$ will be called a
{\em representative for $\cp$}.
\begin{lem}\label{lem:conven:scaling:scaling_elts_exist}
Let $\Gamma$ be a group commensurable with $G(\ZZ)$ and let
$\cpr\in\hat{\QQ}$. Then there exists an element $\sigma_{\cpr}\in
G(\QQ)$ such that
\begin{gather}\label{eqn:conven:scaling:scaling_elt_conds}
     \cpr=\sigma_{\cpr}\cdot\infty,\quad
     (\sigma_{\cpr}^{-1}\Gamma\sigma_{\cpr})_{\infty}=B(\ZZ).
\end{gather}
\end{lem}
\begin{proof}
In case $p=\infty$ the group $\Gamma_{\infty}$ (cf.
(\ref{eqn:conven:isomhyp:Defn_X_infty})) of elements in $\Gamma$
that fix $\infty$ is infinite cyclic and generated by some element
$\gamma_{\infty}\in B(\QQ)$. We claim that in fact $\gamma_{\infty}$
lies in $B_u(\QQ)$, and thus equals $T^{\alpha}$ (cf.
(\ref{eqn:conven:isomhyp:Defn_T^alpha})) for some $\alpha\in \QQp$.
For otherwise $\Pdet(\gamma_{\infty})=e$ for some $e>1$ (cf.
(\ref{eqn:conven:isomhyp:Defn_Pdet})), so that
$\Pdet(\gamma_{\infty}^n)=e^n$ for $n\in \ZZ$, since the restriction
$\Pdet:B(\QQ)\to \ZZp$ is multiplicative. Then each power of
$\gamma_{\infty}$ lies in a distinct coset of the intersection
$\Gamma\cap G(\ZZ)$ in $\Gamma$ since $\Pdet$ is invariant under
multiplication by elements of $G(\ZZ)$. This contradicts the
hypothesis that $\Gamma$ is commensurable with $G(\ZZ)$, so we
conclude that $\Gamma_{\infty}=\lab T^{\alpha}\rab$ for some
$\alpha\in \QQp$. Then we may take $\sigma_{\infty}=[\alpha]$ (cf.
(\ref{eqn:conven:isomhyp:Defn_[alpha]})), for upon calculating
$[\mu]T^{\alpha}[1/\mu]=T^{\mu\alpha}$ we find that
\begin{gather}
     ([1/\alpha]\Gamma[\alpha])_{\infty}
     =
     [1/\alpha]\Gamma_{\infty}[\alpha]
     =[1/\alpha]\lab T^{\alpha}\rab[\alpha]
     =\lab T\rab,
\end{gather}
and the group $\lab T\rab$ is just $B(\ZZ)$, so the element
$\sigma_{\infty}=[\alpha]$ satisfies the two conditions
(\ref{eqn:conven:scaling:scaling_elt_conds}).

In case $p=a/c$ for coprime integers $a,c\in \ZZ$ with $c\neq 0$, we
may choose $b,d\in\ZZ$ such that $ad-bc=1$ and set
$\sigma=\sqnom{a\;b}{c\;d}$. Then $p=\sigma\cdot\infty$ and $\sigma$
satisfies the first condition in
(\ref{eqn:conven:scaling:scaling_elt_conds}). Now we rerun the
argument of the paragraph above with $\sigma^{-1}\Gamma\sigma$ in
place of $\Gamma$ to find that
$(\sigma^{-1}\Gamma\sigma)_{\infty}=\lab T^{\alpha}\rab$ for some
$\alpha\in \QQp$, and $\sigma_p=\sigma[\alpha]$ satisfies the two
conditions (\ref{eqn:conven:scaling:scaling_elt_conds}). This
completes the proof of the lemma.
\end{proof}
An element $\sigma_{\cpr}\in G(\QQ)$ satisfying the two conditions
(\ref{eqn:conven:scaling:scaling_elt_conds}) of Lemma
\ref{lem:conven:scaling:scaling_elts_exist} will be called a {\em
scaling element for $\Gamma$ at the cusp representative $\cpr$}. It
is useful to have a replacement for the notion of scaling element
that is independent of a choice of cusp representative.
\begin{lem}\label{lem:conven:scaling:scaling_cosets_exist}
Let $\Gamma$ be a group commensurable with $G(\ZZ)$ and let
$\cp\in\cP_{\Gamma}$ be a cusp of $\Gamma$. Then there exists a
coset $\Sigma_{\cp}\in \Gamma\backslash G(\QQ)$ such that
\begin{gather}\label{eqn:conven:scaling:scaling_coset_conds}
     \cp=\Sigma_{\cp}\cdot\infty,\quad
     (\Sigma_{\cp}^{-1}\Sigma_{\cp})_{\infty}=B(\ZZ).
\end{gather}
\end{lem}
\begin{proof}
Let $\cpr\in \cp\subset\hat{\QQ}$ be a representative for $\cp$ and
let $\sigma_{\cpr}\in G(\QQ)$ satisfy the conditions
(\ref{eqn:conven:scaling:scaling_elt_conds}) of Lemma
\ref{lem:conven:scaling:scaling_elts_exist}. Then the coset
$\Sigma_{\cp}=\Gamma\sigma_{\cpr}$ satisfies the conditions
(\ref{eqn:conven:scaling:scaling_coset_conds}).
\end{proof}
A coset $\Sigma_{\cp}\in \Gamma\backslash G(\QQ)$ satisfying the two
conditions (\ref{eqn:conven:scaling:scaling_coset_conds}) of Lemma
\ref{lem:conven:scaling:scaling_cosets_exist} will be called a {\em
scaling coset for $\Gamma$ at the cusp $\cp$}. Any element
$\sigma\in \Sigma_{\cp}$ is a scaling element for $\Gamma$ at the
representative $\sigma\cdot\infty$ for the cusp $\cp$, and if
$\sigma_{\cpr}$ is a scaling element for $\Gamma$ at the cusp
representative $\cpr$ then the coset $\Gamma\sigma_{\cpr}$ is a
scaling coset for $\Gamma$ at the cusp $\cp=\Gamma\cdot\cpr$
represented by $\cpr$.

Scaling cosets are determined only up to right multiplication by
elements of $B_u(\QQ)$. Indeed, both the conditions defining the
notion of scaling coset are invariant under the replacement of
$\Sigma_{\cp}$ by $\Sigma_{\cp} T^{\alpha}$ for some $\alpha\in\QQ$
(cf. (\ref{eqn:conven:isomhyp:Defn_T^alpha})). A set
$\{\Sigma_{\cp}\mid \cp\in \cP_{\Gamma}\}\subset \Gamma\backslash
G(\QQ)$ such that $\Sigma_{\cp}$ is a scaling coset for $\Gamma$ at
$\cp$ for each $\cp\in \cP_{\Gamma}$ will be called a {\em system of
scaling cosets for $\Gamma$}.

The problem of constructing scaling cosets for a group $\Gamma$ may
be viewed in the following way. The group $G(\QQ)$ acts transitively
on $\hat{\QQ}$, and the subgroup $B(\QQ)$ is just the stabilizer of
the distinguished point $\infty\in\hat{\QQ}$. Thus the map
$\Gamma\backslash G(\QQ)\to\cP_{\Gamma}$ given by
$\Gamma\sigma\mapsto \Gamma\sigma\cdot\infty$ induces an isomorphism
$\Gamma\backslash G(\QQ)/B(\QQ)\cong\cP_{\Gamma}$ which we may
regard as identifying the sets $\Gamma\backslash G(\QQ)/B(\QQ)$ and
$\cP_{\Gamma}$. Now $B(\QQ)$ is naturally isomorphic to the
semidirect product $B_u(\QQ)\rtimes B_d(\QQ)$ (cf.
\S\ref{sec:conven:groups}), so the natural map $\Gamma\backslash
G(\QQ)\to \Gamma\backslash G(\QQ)/B(\QQ)$ factors through
$\Gamma\backslash G(\QQ)/B_u(\QQ)$, yielding a naturally defined
sequence
\begin{gather}\label{eqn:conven:scaling:Gamma_G_Bu_to_P_Gamma}
     \Gamma\backslash G(\QQ)
     \to
     \Gamma\backslash G(\QQ)/B_u(\QQ)
     \to
     \Gamma\backslash G(\QQ)/B(\QQ)
     \cong
     \cP_{\Gamma}
\end{gather}
where the fibres of the second map are torsors for the diagonal
group $B_d(\QQ)\simeq \QQp$. Given a cusp $\cp\in \cP_{\Gamma}$, any
preimage of $\cp$ in $\Gamma\backslash G(\QQ)$ under the composition
(\ref{eqn:conven:scaling:Gamma_G_Bu_to_P_Gamma}) is a coset of
$\Gamma$ satisfying the first condition defining a scaling coset for
$\Gamma$ at $\cp$. In order to satisfy also the second condition we
should multiply this coset by $[\mu]$ (cf.
(\ref{eqn:conven:isomhyp:Defn_[alpha]})) for some (uniquely defined)
$\mu\in \QQp$. Thus the two conditions defining scaling cosets
define a section of the $B_d(\QQ)$-bundle $\Gamma\backslash
G(\QQ)/B_u(\QQ)\to \cP_{\Gamma}$. We denote this map $\cp\mapsto
\gt{S}_{\cp}$. We may naturally identify $\gt{S}_{\cp}$ with the set
of scaling cosets for $\Gamma$ at $\cp$, for these are precisely the
preimages of $\gt{S}_{\cp}$ in $\Gamma\backslash G(\QQ)$ under the
first map of (\ref{eqn:conven:scaling:Gamma_G_Bu_to_P_Gamma}).
\begin{gather}\label{eqn:conven:scaling:cusp_to_scaling_cosets}
     \begin{split}
     \cP_{\Gamma}
     &\to
     \Gamma\backslash G(\QQ)/B_u(\QQ)\\
     \cp
     &\mapsto
     \gt{S}_{\cp}
     =
     \left\{
     \Sigma_{\cp}T^{\alpha}\mid\alpha\in\QQ
     \right\}
     \end{split}
\end{gather}

One says that $\Gamma$ has {\em width one at infinity} in the case
that $\Gamma_{\infty}=B(\ZZ)$. Observe that $\Gamma$ has width one
at infinity if and only if $\Gamma$ is a scaling coset for itself at
the infinite cusp $\Gamma\cdot\infty$; that is, if and only if
$\Gamma\in\gt{S}_{\Gamma\cdot\infty}$. If $\Gamma$ does not have
width one at infinity then there is a unique $\mu\in \QQp$ with the
property that $\Gamma[\mu]$ is a scaling coset for $\Gamma$ at the
infinite cusp $\Gamma\cdot\infty$, and then the group
$\Gamma^{[\mu]}=[1/\mu]\Gamma[\mu]$ is a group with width one at
infinity.

When engaged in the task of computing Fourier coefficients of
modular forms for a group $\Gamma$ say, one frequently has use for
double coset spaces of the form
$\Gamma_{\cpr}\backslash\Gamma/\Gamma_{\cqr}$, for some
$\cpr,\cqr\in \hat{\QQ}$, where $\Gamma_{\cpr}$ denotes the
stabilizer in $\Gamma$ of $\cpr$. If $\sigma_{\cpr}$ and
$\sigma_{\cqr}$ are scaling elements for $\Gamma$ at $\cpr$ and
$\cqr$, respectively, then we have
$\sigma_{\cpr}^{-1}\Gamma_{\cpr}\sigma_{\cpr}=B(\ZZ)$, and similarly
with $\cqr$ in place of $\cpr$, so the set of {\em translates}
\begin{gather}
     \sigma_{\cpr}^{-1}
     (\Gamma_{\cpr}\backslash\Gamma/\Gamma_{\cqr})
     \sigma_{\cqr}
     =
     \{
     \sigma_{\cpr}^{-1}\Gamma_{\cpr}\gamma\Gamma_{\cqr}\sigma_{\cqr}
     \mid \gamma\in\Gamma
     \}
\end{gather}
is in fact a set of double cosets of $B(\ZZ)$. In the notation of
\S\ref{sec:conven:cosets} we have
\begin{gather}
     \sigma_{\cpr}^{-1}
     (\Gamma_{\cpr}\backslash\Gamma/\Gamma_{\cqr})
     \sigma_{\cqr}
     =\lBZ\sigma_{\cpr}^{-1}\Gamma\sigma_{\cqr}\rBZ
     =\left\{
     \lBZ\sigma_{\cpr}^{-1}\gamma\sigma_{\cqr}\rBZ \mid \gamma\in\Gamma
     \right\}.
\end{gather}
Set $\cp=\Gamma\cdot\cpr$ and $\cq=\Gamma\cdot\cqr$, and set
$\Sigma_{\cp}=\Gamma\sigma_{\cpr}$ and
$\Sigma_{\cq}=\Gamma\sigma_{\cqr}$, so that $\Sigma_{\cp}$ and
$\Sigma_{\cq}$ are scaling cosets for $\Gamma$ at $\cp$ and $\cq$,
respectively. Then we have
$\sigma_{\cpr}^{-1}\Gamma\sigma_{\cqr}=\Sigma_{\cp}^{-1}\Sigma_{\cq}$,
so that these translates depend only on cusps, and not on cusp
representatives. We see from this discussion that for any pair of
cusps $\cp,\cq\in \cP_{\Gamma}$, with scaling cosets $\Sigma_{\cp}$
and $\Sigma_{\cq}$, respectively, the set
$\Sigma_{\cp}^{-1}\Sigma_{\cq}$ is a union of double cosets of
$B(\ZZ)$. The assignment
$(\cp,\cq)\mapsto\Sigma_{\cp}^{-1}\Sigma_{\cq}$ is sensitive to the
choice of scaling cosets $\Sigma_{\cp}$ and $\Sigma_{\cq}$, to the
extent that a different choice will replace
$\Sigma_{\cp}^{-1}\Sigma_{\cq}$ with a set of the form
$T^{\alpha}\Sigma_{\cp}^{-1}\Sigma_{\cq}T^{\beta}$ for some
$\alpha,\beta\in \QQ$.

In general there may be no canonical choice of scaling coset for a
particular group at a particular cusp, but there are situations in
which some choices might be preferred over others. For example, it
is natural to take $\Sigma_{\Gamma\cdot\infty}=\Gamma$ in case
$\Gamma$ has width one at infinity, for we then have
$\Sigma_{\cp}^{-1}\Sigma_{\cq}=\Sigma_{\cp}^{-1}$ when
$\cq=\Gamma\cdot\infty$, and
$\Sigma_{\cp}^{-1}\Sigma_{\cq}=\Sigma_{\cq}$ when
$\cp=\Gamma\cdot\infty$, and $\Sigma_{\cp}^{-1}\Sigma_{\cq}=\Gamma$
when $\cp=\cq=\Gamma\cdot\infty$. We shall always take the
scaling coset $\Sigma_{\Gamma\cdot\infty}$, for $\Gamma$ at the
infinite cusp, to be of the form $\Gamma[\mu]$ for $\mu\in \QQp$.
This value $\mu$ is uniquely determined.

Observe that if $\Sigma_{\cp}$ and $\Sigma_{\cq}$ are scaling cosets
for $\Gamma$ at $\cp$ and $\cq$, respectively, then the intersection
$\Sigma_{\cp}^{-1}\Sigma_{\cq}\cap
B(\QQ)=(\Sigma_{\cp}^{-1}\Sigma_{\cq})_{\infty}$ can be non-empty
only in the case that $\cp=\cq$. For $\cp$ and $\cq$ cusps for $\Gamma$ we define $\delta_{\Gamma|\cp,\cq}$ to be $1$ or $0$ according as $\cp=\cq$ or not.

Given a system $\{\Sigma_{\cp}\mid\cp\in\cP_{\Gamma}\}$ of scaling
cosets for $\Gamma$, we write $\Gamma^{\cp}$ as a shorthand for the
group obtained as the conjugate of $\Gamma$ by $\Sigma_{\cp}$.
\begin{gather}\label{eqn:conven:scaling:Defn_Conj_Gamma_by_sigma_cq}
     \Gamma^{\cp}=\Sigma_{\cp}^{-1}\Sigma_{\cp}
\end{gather}
The notation (\ref{eqn:conven:scaling:Defn_Conj_Gamma_by_sigma_cq})
suppresses the dependence on the choice of scaling coset: a
different choice will replace $\Gamma^{\cp}$ with a group of the
form $T^{-\alpha}\Gamma^{\cp}T^{\alpha}$ for some $\alpha\in \QQ$.

\subsection{Integrals}\label{sec:conven:autfrm}

Write $\mc{O}(\HH)$ for the ring of holomorphic functions on the
upper half plane $\HH$.

For $\atp\in \ZZ$, we define the {\em weight $2\atp$ (right) action
of $G(\QQ)$ on $\mc{O}(\HH)$}, to be denoted $(f,\chi)\mapsto
f\sop{\atp}\chi$, by setting
\begin{gather}\label{eqn:conven:autfrm:Defn_slash_op}
          \left(
          f\sop{\atp}\chi
          \right)
          (\zz)
          =f(\chi\cdot\zz)
          \jac(\chi,\zz)^{\atp}
\end{gather}
for $f\in \mc{O}(\HH)$ and $\chi\in G(\QQ)$. For $\Gamma$ a group
commensurable with $G(\ZZ)$ (and hence a subgroup of $G(\QQ)$), we
call $f\in \mc{O}(\HH)$ an {\em unrestricted modular form of weight
$2\atp$ for $\Gamma$} in case it is a fixed point for the weight
$2\atp$ action of $\Gamma$. Suppose $f$ is an unrestricted modular
form of weight $2\atp$ for $\Gamma$. Then for $X\in \Gamma\backslash
G(\QQ)$ a right coset of $\Gamma$ in $G(\QQ)$ we may define a
function $(f\cop{\Gamma}{\atp}X)(\zz)$ by setting
\begin{gather}\label{eqn:conven:autfrm:Defn_cop}
     f\cop{\Gamma}{\atp}X
     =
     f\sop{\atp}\chi
\end{gather}
where $\chi$ is any representative for the coset
$X\in\Gamma\backslash G(\QQ)$. Let $\cp\in \cP_{\Gamma}$ be a cusp
of $\Gamma$ and let $\Sigma_{\cp}\in\Gamma\backslash G(\QQ)$ a
scaling coset for $\Gamma$ at $\cp$. Then for $f$ an unrestricted
modular form of weight $2\atp$ for $\Gamma$ we define
$f_{|\cp}\in\mc{O}(\HH)$ by setting
$f_{|\cp}=f\cop{\Gamma}{\atp}\Sigma_{\cp}$. Then
$f_{|\cp}(\zz+1)=f_{|\cp}(\zz)$ for all $\zz\in \HH$, by the
defining properties of $\Sigma_{\cp}$ (cf.
\S\ref{sec:conven:scaling}), so we have
\begin{gather}\label{eqn:conven:autfrm_FourExpp}
     f_{|\cp}(\zz)=\sum_{n\in\ZZ} c_{|\cp}(n)\ee(n\zz)
\end{gather}
for some $c_{|\cp}(n)\in \ZZ$. We call the right hand side of
(\ref{eqn:conven:autfrm_FourExpp}) the {\em Fourier expansion of $f$
at $\cp$ with respect to $\Sigma_{\cp}$}, and we say that $f$ is
{\em meromorphic at $\cp$} if the right hand side of
(\ref{eqn:conven:autfrm_FourExpp}) is a meromorphic function of
$\vq=\ee(\zz)$ in a neighborhood of $\vq=0$. A different choice of
scaling coset will replace $c_{|\cp}(n)$ with
$\ee(n\alpha)c_{|\cp}(n)$ for some $\alpha\in \QQ$, so the notion of
being meromorphic at $\cp$ is independent of the choice of scaling
coset at $\cp$.

We define $M_{\atp}(\Gamma)$ to be the vector space consisting of
unrestricted modular forms of weight $2\atp$ for $\Gamma$ that are
meromorphic at the cusps of $\Gamma$. In a slight departure from
standard convention we call $M_{\atp}(\Gamma)$ the space of {\em
modular forms of weight $2\atp$ for $\Gamma$}. We define
$S_{\atp}(\Gamma)$ to be the vector space consisting of modular
forms of weight $2\atp$ for $\Gamma$ that vanish at the cusps of
$\Gamma$. We call $S_{\atp}(\Gamma)$ the space of {\em cusp forms of
weight $2\atp$ for $\Gamma$}. The vector space $S_{\atp}(\Gamma)$ is
the zero vector space when $\atp\leq 0$.

Let $\atp\in \ZZ$ such that $\atp\leq 0$, and let $\ww\in
\HH\cup\hat{\QQ}$. Given a holomorphic function $g(\zz)$ on $\HH$ with sufficiently rapid decay as the imaginary part of $\zz$ tends to $\infty$ we define a function $\IO{\ww}{\atp}g$ on $\CC$ by setting
$\IO{\infty}{\atp}g$ to be the zero function and by setting
\begin{gather}\label{eqn:conven:autfrm:Defn_IO_w_atp}
     (\IO{\ww}{\atp}g)(\zz)
     =\tpi
     \int_{\ww}^{\infty}
     g(\xi)(\tpi(\xi-\zz))^{(-2\atp)}
     {\rm d}\xi
\end{gather}
for $\ww\neq\infty$, where the integral is taken over the vertical
line $\{\ww+\ii t\mid t\in\RRp\}$. In our applications $g$ will either be a cusp form (for some $\Gamma<\PSL_2(\RR)$ commensurable with $\PSL_2(\ZZ)$) or will admit a finite power series expansion of the form $g(\zz)=\sum_{n=1}^N c(n)\ee(n\zz)$. In these cases the convergence of the integral defining $\IO{\ww}{\atp}g$ is clear.  

We define a closely related
operator $g\mapsto \JO{\ww}{\atp}g$ by setting
\begin{gather}\label{eqn:conven:autfrm:Defn_JO_w_atp}
     (\JO{\ww}{\atp}g)(\zz)
     =\overline{(\IO{\ww}{\atp}g)(\bar{\zz})}.
\end{gather}
Observe that when $\atp\leq 0$, the functions
$(\IO{\ww}{\atp}g)(\zz)$ and $(\JO{\ww}{\atp}g)(\zz)$ are
polynomials of degree at most $-2\atp$ in $\zz$	. For
$\chi\in G(\QQ)$ we have
\begin{gather}\label{eqn:conven:autfrm:Var_IO_w_atp}
     \left.
     (\IO{\ww}{\atp}g)
     \right\sop{\atp}\chi
     =
     \IO{\ww\cdot\chi}{\atp}
     \left(g\sop{1-\atp}\chi\right)
     -
     \IO{\infty\cdot\chi}{\atp}
     \left(g\sop{1-\atp}\chi\right)
\end{gather}
where $\ww\cdot\chi=\chi^{-1}\cdot\ww$ denotes the natural right
action of $G(\QQ)$ on $\HH\cup\hat{\QQ}$. In particular, if $g\in
M_{1-\atp}(\Gamma)$ for some group $\Gamma$ then we have
$\left.(\IO{\ww}{\atp}g)\right\sop{\atp}\gamma=\IO{\ww\cdot\gamma}{\atp}g-\IO{\infty\cdot\gamma}{\atp}g$
for $\gamma\in\Gamma$.
\begin{lem}\label{lem:conven:autfrm:Faithfulness_I_p_atp}
Let $\Gamma$ be a group commensurable with $G(\ZZ)$, let $\atp\in
\ZZ$ such that $\atp\leq 0$, and let $g\in S_{1-\atp}(\Gamma)$.
Then we have $I_{\cpr}^{\atp}g=0$ for all $\cpr\in \hat{\QQ}$ if and
only if $g=0$.
\end{lem}
\begin{proof}
Certainly $\IO{\cpr}{\atp}g=0$ for all $\cpr\in\hat{\QQ}$ if $g=0$,
so let $g\in S_{1-\atp}(\Gamma)$ such that $\IO{\cpr}{\atp}g=0$ for
all $\cpr\in \hat{\QQ}$. Following the proof of Lemma 3.2 in
\cite{Nie_ConstAutInts} we define a function $f(\zz)$ by setting
\begin{gather}
     f(\zz)
     =
     (\IO{\zz}{\atp}g)(\zz)
     =
     \tpi
     \int_{\zz}^{\infty}
     g(\xi)(\tpi(\xi-\zz))^{(-2\atp)}
     {\rm d}\xi.
\end{gather}
The identity (\ref{eqn:conven:autfrm:Var_IO_w_atp}) implies
$f\sop{\atp}\gamma=f-\IO{\infty\cdot\gamma}{\atp}g$ for $\gamma\in
\Gamma$, so that $f$ is an unrestricted modular form of weight
$2\atp$ for $\Gamma$. We should consider the behavior of $f$ at the
cusps of $\Gamma$. Let $\cp\in \cP_{\Gamma}$, and let $\Sigma_{\cp}$
be a scaling coset for $\Gamma$ at $\cp$. Then we have
\begin{gather}
     g_{|\cp}(\xi)
     =\sum_{n>0}
     b_{|\cp}(n)\ee(n\xi)
\end{gather}
for some $b_{|\cp}(n)\in \CC$, where
$g_{|\cp}=g\cop{\Gamma}{1-\atp}\Sigma_{\cp}$ (cf.
(\ref{eqn:conven:autfrm:Defn_cop})). In order to analyze
$f_{|\cp}=f\cop{\Gamma}{\atp}\Sigma_{\cp}$ we choose $\sigma\in
\Sigma_{\cp}$ and compute
\begin{gather}
     f\sop{\atp}\sigma
     =
     \IO{\zz}{\atp}(g_{|\cp})
     -
     \IO{\infty\cdot\sigma}{\atp}(g_{|\cp})
\end{gather}
by (\ref{eqn:conven:autfrm:Var_IO_w_atp}). Now the term
$\IO{\infty\cdot\sigma}{\atp}(g_{|\cp})$ vanishes, since we have
\begin{gather}
     \begin{split}
     (\IO{\infty\cdot\sigma}{\atp}(g_{|\cp}))(\zz)
     &=\tpi^{1-2\atp}
     \int_{\sigma^{-1}\cdot\infty}^{\infty}
     g(\sigma\cdot\xi)
     \jac(\sigma,\xi)^{1-\atp}
     (\xi-\zz)^{(-2\atp)}
     {\rm d}\xi\\
     &=\tpi^{1-2\atp}
     \int_{\sigma^{-1}\cdot\infty}^{\infty}
     g(\sigma\cdot\xi)
     (\sigma\cdot\xi-\sigma\cdot\zz)^{(-2\atp)}
     \jac(\sigma,\zz)^{\atp}
     {\rm d}(\sigma\cdot\xi)\\
     &=
     -(\IO{\sigma\cdot\infty}{\atp}g)
     \jac(\sigma,\zz)^{\atp},
     \end{split}
\end{gather}
which vanishes by our hypothesis on $g$. So we have the Fourier
expansion
\begin{gather}
     f_{|\cp}(\zz)
     =
     \sum_{n>0}
     c_{|\cp}(n)
     \ee(n\zz)
\end{gather}
for $f$ at $\cp$, where $c_{|\cp}(n)$ is determined by
$b_{|\cp}(n)$ according to the formula
\begin{gather}
     c_{|\cp}(n)
     =
     b_{|\cp}(n)
     \tpi
     (\IO{0}{\atp}g)(0)
     =-n^{2\atp-1}b_{\cp}(n).
\end{gather}
We conclude that $f$ is a cusp form of weight $\atp$. Since
$\atp\leq 0$ we have $S_{\atp}(\Gamma)=0$, and so
$c_{|\cp}(n)=0$ for all $n\in\ZZ$ and all $\cp\in\cP_{\Gamma}$, and
so $b_{|\cp}(n)=0$ for all $n\in\ZZ$ and all $\cp\in \cP_{\Gamma}$.
That is to say, $g$ vanishes identically, which is what we required
to show.
\end{proof}
Let $\atp\in \ZZ$ such that $\atp\leq 0$. Following Niebur (cf.
\cite{Nie_ConstAutInts}) we say that a holomorphic function $f\in
\mc{O}(\HH)$ is an {\em unrestricted automorphic integral of weight
$2\atp$ for $\Gamma$} in case there exists a cusp form $g\in
S_{1-\atp}(\Gamma)$ such that
\begin{gather}\label{eqn:conven:autfrm:AutInt_Xform}
     f\sop{\atp}\gamma
     =
     f
     -
     \JO{\infty\cdot\gamma}{\atp}g
\end{gather}
for each $\gamma\in \Gamma$. According to Lemma
\ref{lem:conven:autfrm:Faithfulness_I_p_atp} the cusp form $g\in
S_{1-\atp}(\Gamma)$ satisfying
(\ref{eqn:conven:autfrm:AutInt_Xform}) is uniquely determined. We
call it the {\em cusp form associated to the unrestricted
automorphic integral $f$}. The operator $f\mapsto
f\cop{\Gamma}{\atp}X$, for $X$ a right coset of $\Gamma$, may be
extended to unrestricted automorphic integrals as follows. Suppose
$f$ is an unrestricted automorphic integral of weight $2\atp$ for
$\Gamma$ with associated cusp form $g$. Then for $X\in
\Gamma\backslash G(\QQ)$ we may define $f\cop{\Gamma}{\atp}X\in
\mc{O}(\HH)$ by setting
\begin{gather}\label{eqn:conven:autfrm:Defn_cop_on_ai}
     f\cop{\Gamma}{\atp}X
     =
     (f-\JO{\chi\cdot\infty}{\atp}g)\sop{\atp}\chi
\end{gather}
where $\chi$ is any representative for the coset
$X\in\Gamma\backslash G(\QQ)$. The transformation properties
(\ref{eqn:conven:autfrm:Var_IO_w_atp}) and
(\ref{eqn:conven:autfrm:AutInt_Xform}) confirm that the function
$f\cop{\Gamma}{\atp}X$ is independent of the choice of coset
representative $\chi$. Let $\Sigma_{\cp}$ be a scaling coset for
$\Gamma$ at a cusp $\cp\in\cP_{\Gamma}$, and define $f_{|\cp}$ by
setting
\begin{gather}
     f_{|\cp}
     =
     f\cop{\Gamma}{\atp}\Sigma_{\cp}.
\end{gather}
By the defining properties of scaling cosets we find that
$f_{|\cp}(\zz+1)=f_{|\cp}(\zz)$ for all $\zz\in \HH$, and so we have
\begin{gather}\label{eqn:conven:autfrm:AutInt_FourExp_p}
     f_{|\cp}(\zz)=\sum_{n\in\ZZ} c_{|\cp}(n)\ee(n\zz)
\end{gather}
for some $c_{|\cp}(n)\in \CC$. We call
(\ref{eqn:conven:autfrm:AutInt_FourExp_p}) the {\em Fourier
expansion of $f$ at $\cp$ with respect to $\Sigma_{\cp}$}, and we
say that $f$ is {\em meromorphic at $\cp$} if the right hand side of
(\ref{eqn:conven:autfrm:AutInt_FourExp_p}) is a meromorphic function
of $\vq=\ee(\zz)$ in a neighborhood of $\vq=0$.

We define $I_{\atp}(\Gamma)$ to be the vector space consisting of
unrestricted automorphic integrals of weight $2\atp$ for $\Gamma$
that are meromorphic at every cusp of $\Gamma$. We call
$I_{\atp}(\Gamma)$ the space of {\em automorphic integrals of weight
$2\atp$ for $\Gamma$}.

The space $M_{\atp}(\Gamma)$ of modular forms of weight $2\atp$ for
$\Gamma$ is a subspace of $I_{\atp}(\Gamma)$ by definition. Lemma
\ref{lem:conven:autfrm:Faithfulness_I_p_atp} may be regarded as
stating that $M_{\atp}(\Gamma)$ is the kernel of the map
$I_{\atp}(\Gamma)\to S_{1-\atp}(\Gamma)$ which sends an automorphic
integral $f$ to its associated cusp form $g$ (cf.
(\ref{eqn:conven:autfrm:AutInt_Xform})). We will see in
\S\ref{sec:modradsum:var} that the map $I_{\atp}(\Gamma)\to
S_{1-\atp}(\Gamma)$ is surjective, so that we have an exact sequence
of vector spaces
\begin{gather}\label{eqn:conven:autfrm:Ex_Seq_MIS}
     0\to
     M_{\atp}(\Gamma)
     \to
     I_{\atp}(\Gamma)
     \to
     S_{1-\atp}(\Gamma)
     \to
     0.
\end{gather}
The sequence (\ref{eqn:conven:autfrm:Ex_Seq_MIS}) is trivially exact
in case $\atp>0$, for then $S_{1-\atp}(\Gamma)=0$, so that an
automorphic integral of weight $2\atp$ is automatically a modular
form. Indeed, and more generally, if $\atp\in\ZZ$ is chosen so that
$S_{1-\atp}(\Gamma)=0$, then the spaces $M_{\atp}(\Gamma)$ and
$I_{\atp}(\Gamma)$ coincide by the definition (cf.
(\ref{eqn:conven:autfrm:AutInt_Xform})) of an automorphic integral.

%------------------------------------------------------------------%
\section{Rademacher sums}\label{sec:radsum}
%------------------------------------------------------------------%

In this section we associate Rademacher sums, and conjugate Rademacher sums (cf. \S\ref{sec:intro:conradsum}), of arbitrary even integer weight to triples $(\Gamma,\cp,\cq)$ where $\Gamma$ is a group commensurable with the modular group and $\cp$ and $\cq$ are cusps for $\Gamma$. We derive explicit expressions for the Fourier expansions of these functions, and study how they transform under the action of $\Gamma$.

\subsection{Construction}\label{sec:radsum:constr}

Given $\atp,m\in\ZZ$ and $\lBZ\chi\rBZh\in \lBZ G(\QQ)\rBZh$, define
a function $\zz\mapsto \PS{\lBZ\chi\rBZh}{\atp}{m}(\zz)$ on $\HH$ by
setting
\begin{gather}\label{eqn:radsum:constr:Defn_PS_chi}
     \PS{\lBZ\chi\rBZh}{\atp}{m}(\zz)
     =
     \ee(m\lBZ\chi\rBZh\cdot\zz)
     \jac(\lBZ\chi\rBZh,\zz)^{\atp}.
\end{gather}
Given $U\subset\lBZ G(\QQ)\rBZh$, define the {\em holomorphic
Poincar\'e series of weight $2\atp$ and order $m$ associated to
$U$}, to be denoted $\PS{U}{\atp}{m}(\zz)$, by setting
\begin{gather}\label{eqn:radsum:constr:Defn_PS_U_m}
     \PS{U}{\atp}{m}(\zz)
     =\lim_{K\to\infty}
     \sum_{\lBZ\chi\rBZh\in U_{\leq K}}
     \PS{\lBZ\chi\rBZh}{\atp}{m}(\zz)
\end{gather}
where $U_{\leq K}$ is as in (\ref{eqn:conven:cosets:Defn U_sub_K}). Of course, the limit defining $\PS{U}{\atp}{m}(\zz)$ may or may not converge depending on the choice of $\kappa$, $m$, and $U$. If $\atp>1$ and $m\geq 0$ then for $U=\lBZ\Gamma\rBZh$ say, for $\Gamma$ a group commensurable with $G(\ZZ)$, the limiting sum in (\ref{eqn:radsum:constr:Defn_PS_U_m}) converges absolutely and uniformly on compact subsets of $\HH$ (cf. Lemma \ref{lem:radsum:var:RS_U_slash_sigma_vs_RS_Usigma}) and the symbols $\lim_{K\to \infty}$ and $_{\leq K}$ are unnecessary. If $\atp\leq 0$ then even with the limit the expression on the right hand side of (\ref{eqn:radsum:constr:Defn_PS_U_m}) fails to converge, and the Poincar\'e series $\PS{U}{\atp}{m}(\zz)$ is not defined.

As was originally demonstrated by Rademacher in
\cite{Rad_FuncEqnModInv}, the non-convergence of the Poincar\'e
series $\PS{U}{\atp}{m}(\zz)$ when $\atp=0$ can be circumvented by
replacing $\PS{\lBZ\chi\rBZh}{0}{m}(\zz)$ with
$\PS{\lBZ\chi\rBZh}{0}{m}(\zz)-\PSa{\lBZ\chi\rBZh}{0}{m}$ in the
right-hand side of (\ref{eqn:radsum:constr:Defn_PS_U_m}) where
\begin{gather}
     \PS{\lBZ\chi\rBZh}{0}{m}(\zz)
     -\PSa{\lBZ\chi\rBZh}{0}{m}
     =
     \ee(m\lBZ\chi\rBZh\cdot\zz)
     -\ee(m\lBZ\chi\rBZh\cdot\infty)
\end{gather}
in case $\lBZ\chi\rBZh\in \lBZ G(\QQ)\rBZh^{\times}$, and
$\PSa{\lBZ\chi\rBZh}{0}{m}=0$ otherwise. Rademacher showed for $m=-1$
(cf. loc. cit.) that the resulting expression is a (conditionally)
convergent series which recovers a function invariant for the
modular group $G(\ZZ)$ in the case that $U=\lBZ G(\ZZ)\rBZh$.
Generalizations of Rademacher's construction were given by Knopp
(cf. \cite{Kno_ConstAutFrmsSuppSeries}) and Niebur (cf.
\cite{Nie_ConstAutInts}) so as to obtain conditionally convergent
Poincar\'e series for arbitrary real non-positive weights, for
various subgroups of $G(\RR)=\PSL_2(\RR)$.

With the constructions of Rademacher, Knopp and Niebur in mind, we
introduce the {\em Rademacher component function of weight $2\atp$
and order $m$ associated to $\lBZ\chi\rBZh$}, denoted $\zz\mapsto
\RS{\lBZ\chi\rBZh}{\atp}{m}(\zz)$ and defined for
$\atp,m\in\ZZ$ and $\lBZ\chi\rBZh\in \lBZ G(\QQ)\rBZh$, by setting
\begin{gather}\label{eqn:radsum:constr:Defn_RS_chi}
     \RS{\lBZ\chi\rBZh}{\atp}{m}(\zz)
     =
     \ee(m\lBZ\chi\rBZh\cdot\zz)
     \Rreg{\atp}(m,\lBZ\chi\rBZh,\zz)
     \jac(\lBZ\chi\rBZh,\zz)^{\atp},
\end{gather}
where $\Rreg{\atp}(m,\lBZ\chi\rBZh,\zz)$ is the {\em Rademacher
regularization factor of weight $2\atp$}, given by
\begin{gather}\label{eqn:radsum:constr:Defn_Rad_reg}
     \Rreg{\atp}(m,\lBZ\chi\rBZh,\zz)
     =
     \Phi(1-2\atp,2-2\atp,
     m\lBZ\chi\rBZh\cdot\infty
     -m\lBZ\chi\rBZh\cdot\zz)
     (
     \tpi(
          m\lBZ\chi\rBZh\cdot\zz
          -m\lBZ\chi\rBZh\cdot\infty
          )
          )^{1-2\atp}
\end{gather}
in case $\chi\in G(\QQ)^{\times}$, and
$\Rreg{\atp}(m,\lBZ\chi\rBZh,\zz)=1$ otherwise (cf.
(\ref{eqn:conven:fns:Defn_Phi})). The function $\RS{\lBZ\chi\rBZh}{\atp}{m}(\zz)$ is entire in case $\chi\in B(\QQ)$, and is otherwise holomorphic away from $z=\chi^{-1}\cdot\infty$. For $U\subset \lBZ G(\QQ)\rBZh$ we
define the {\em classical Rademacher sum of weight $2\atp$ and order
$m$ associated to $U$}, to be denoted $\RS{U}{\atp}{m}(\zz)$, by
setting
\begin{gather}\label{eqn:radsum:constr:Defn_RS_U_m}
     \RS{U}{\atp}{m}(\zz)
     =\lim_{K\to\infty}
     \sum_{\lBZ\chi\rBZh\in U_{\leq K}}
     \RS{\lBZ\chi\rBZh}{\atp}{m}(\zz).
\end{gather}
Again, the expression defining $\RS{U}{\atp}{m}(\zz)$ will not converge for all choices of $\kappa$, $m$ and $U$. Employing the {Kummer transformation}
(\ref{eqn:conven:fns:Kummer_Xform}) we may rewrite the Rademacher
component function $\RS{\lBZ\chi\rBZh}{\atp}{m}(\zz)$ as
\begin{gather}\label{eqn:radsum:constr:RS_chi_using_gen_exp}
     \RS{\lBZ\chi\rBZh}{\atp}{m}(\zz)
     =
     \ee(m\lBZ\chi\rBZh\cdot\infty)
     \ee(
     m\lBZ\chi\rBZh\cdot\zz
     -
     m\lBZ\chi\rBZh\cdot\infty
     ,1-2\atp)
     \jac(\lBZ\chi\rBZh,\zz)^{\atp}
\end{gather}
in case $\lBZ\chi\rBZh\in \lBZ G(\QQ)\rBZh^{\times}$, where
$\ee(\zz,\spp)$ is the generalized exponential function of
(\ref{eqn:conven:fns:Genzd_Exp}). Since we restrict attention to the
case that $\atp$ is an integer (cf. \cite{Nie_ConstAutInts} for
non-integral weights), the generalized exponential in
(\ref{eqn:radsum:constr:RS_chi_using_gen_exp}) may in turn be
written in terms of the partial exponential function $\ee(\zz)_{<
K}$ (cf. \S\ref{sec:conven:fns}). We thus obtain the identity
\begin{gather}\label{eqn:radsum:constr:Rreg_Part_Exp}
     \ee(m\lBZ\chi\rBZh\cdot\zz)
     \Rreg{\atp}(m,\lBZ\chi\rBZh,\zz)
     =
     \ee(m\lBZ\chi\rBZh\cdot\zz)
     -\ee(m\lBZ\chi\rBZh\cdot\infty)
     \ee(
     m\lBZ\chi\rBZh\cdot\zz
     -m\lBZ\chi\rBZh\cdot\infty
     )_{<1-2\atp}
\end{gather}
when $\chi\in G(\QQ)^{\times}$. In particular, we have
$\Rreg{\atp}(m,\lBZ\chi\rBZh,\zz)=1$ when $\atp$ is positive, so
that the Rademacher sum $\RS{U}{\atp}{m}(\zz)$ is exactly (albeit formally) the
Poincar\'e series $\PS{U}{\atp}{m}(\zz)$ when $\atp>0$.

The case that $\atp=0$ is special, for the corresponding Poincar\'e
series, if it were to converge, would define a
$\Gamma$-invariant function on $\HH$ for any group $\Gamma$
satisfying $U\Gamma=U$. With $\atp=0$ we have
$\Rreg{0}(m,\lBZ\chi\rBZh,\zz)=1-\ee(m\lBZ\chi\rBZh\cdot\infty-m\lBZ\chi\rBZh\cdot\zz)$
when $\chi\in G(\QQ)^{\times}$, and hence
\begin{gather}\label{eqn:radsum:constr:reg_s=0}
     \RS{\lBZ\chi\rBZh}{0}{m}(\zz)
     =
     \PS{\lBZ\chi\rBZh}{0}{m}(\zz)
     -\PSa{\lBZ\chi\rBZh}{0}{m}
\end{gather}
for all $\lBZ\chi\rBZh\in \lBZ G(\QQ)\rBZh^{\times}$. In other
words, the Rademacher sum $\RS{U}{0}{m}(\zz)$ is given by
\begin{gather}\label{eqn:radsum:constr:RS_atp=0}
     \begin{split}
     \RS{U}{0}{m}(\zz)
     &=\lim_{K\to\infty}
     \sum_{\lBZ\chi\rBZh\in U_{\leq K}^{\times}}
     \PS{\lBZ\chi\rBZh}{0}{m}(\zz)
     -\PSa{\lBZ\chi\rBZh}{0}{m}
     \\
     &=\sum_{\lBZ\chi\rBZh\in U_{\infty}}
     \ee(m\lBZ\chi\rBZh\cdot\zz)
     +\lim_{K\to\infty}
     \sum_{\lBZ\chi\rBZh\in U_{\leq K}^{\times}}
     \ee(m\lBZ\chi\rBZh\cdot\zz)-
     \ee(m\lBZ\chi\rBZh\cdot\infty),
     \end{split}
\end{gather}
with the latter expression holding at least when there are only
finitely many cosets of $B(\ZZ)$ in $U_{\infty}$. This expression
(\ref{eqn:radsum:constr:RS_atp=0}), with $m=-1$ and $U=\lBZ
G(\ZZ)\rBZh$, is the series originally considered by Rademacher in
\cite{Rad_FuncEqnModInv}.

It is interesting to consider the function obtained by substituting
$\bar{\zz}$ for $\zz$ in the right hand side of
(\ref{eqn:radsum:constr:Defn_RS_U_m}). We define the {\em conjugate
Rademacher sum of weight $2\atp$ and order $m$ associated to $U$},
to be denoted $\CS{U}{\atp}{m}(\zz)$, by setting
\begin{gather}\label{eqn:radsum:constr:Defn_CS_U_m}
     \CS{U}{\atp}{m}(\zz)
     =\lim_{K\to\infty}
     \sum_{\lBZ\chi\rBZh\in U_{\leq K}}
     \CS{\lBZ\chi\rBZh}{\atp}{m}({\zz}),
\end{gather}
where
$\CS{\lBZ\chi\rBZh}{\atp}{m}({\zz})=\RS{\lBZ\chi\rBZh}{\atp}{m}(\bar{\zz})$
by definition. Evidently, the conjugate Rademacher sum
$\CS{U}{\atp}{m}(\zz)$ is an anti-holomorphic function on $\HH$
whenever $U$, $\atp$ and $m$ are such that the limit in (\ref{eqn:radsum:constr:Defn_CS_U_m}) converges locally uniformly in $\zz$. At first glance it appears
that we should recover the classical Rademacher sum
$\RS{U}{\atp}{m}(\zz)$ from the conjugate Rademacher sum according
to the identity $\RS{U}{\atp}{m}(\zz)=\CS{U}{\atp}{m}(\bar{\zz})$,
but as we shall see in \S\ref{sec:radsum:var} the relationship
between the holomorphic functions $\RS{U}{\atp}{m}(\zz)$ and
$\CS{U}{\atp}{m}(\bar{\zz})$ is generally more interesting that
this.

We typically take $U$ to be of the form
$U=\lBZ\Sigma_{\cp}^{-1}\Sigma_{\cq}\rBZh$ where
$\{\Sigma_{\cp}\mid\cp\in \cP_{\Gamma}\}$ is a system of scaling
cosets (cf. \S\ref{sec:conven:scaling}) for some group $\Gamma$
commensurable with the modular group, and $\cp,\cq\in\cP_{\Gamma}$
are cusps of $\Gamma$. In this case we write
$\RS{\Gamma,\cp|\cq}{\atp}{m}(\zz)$ for $\RS{U}{\atp}{m}(\zz)$, and
similarly for the holomorphic Poincar\'e series, suppressing the
choice of scaling cosets from notation. A change in the choice of
$\Sigma_{\cp}$ and $\Sigma_{\cq}$ replaces
$\RS{\Gamma,\cp|\cq}{\atp}{m}(\zz)$ with a function of the form
$\ee(\alpha)\RS{\Gamma,\cp|\cq}{\atp}{m}(\zz+\beta)$ for some
$\alpha,\beta\in \QQ$, and similarly for the holomorphic Poincar\'e
series $\PS{\Gamma,\cp|\cq}{\atp}{m}(\zz)$.

In the case that $\cp$ or $\cq$ is the {\em infinite cusp}
$\Gamma\cdot\infty$ we omit it from notation, writing
$\RS{\Gamma,\cp}{\atp}{m}(\zz)$ for
$\RS{\Gamma,\cp|\Gamma\cdot\infty}{\atp}{m}(\zz)$, and
$\RS{\Gamma|\cq}{\atp}{m}(\zz)$ for
$\RS{\Gamma,\Gamma\cdot\infty|\cq}{\atp}{m}(\zz)$, and similarly for
the holomorphic Poincar\'e series. The functions
$\RS{\Gamma,\cp}{\atp}{m}(\zz)$ are, in a sense, the most important,
for we shall see in \S\ref{sec:radsum:var} that
$\RS{\Gamma,\cp}{\atp}{m}(\zz)$ is, up to addition by a certain constant
function, an automorphic integral of weight $2\atp$ for $\Gamma$
with a single pole at the cusp $\cp$ in case $\Gamma$ has width one
at infinity (cf. \S\ref{sec:conven:scaling}) and $\atp\leq 0$ and $m<0$.
We shall see also in \S\ref{sec:radsum:var} (cf. Theorems \ref{thm:radsum:var:Var_RS_atp>0} and \ref{thm:radsum:var:Var_RS_atp<1}) that the function
$\RS{\Gamma,\cp|\cq}{\atp}{m}(\zz)$, once corrected by adding a certain
constant function in case $\atp\leq 0$, is the expansion (of the correction by addition of a certain constant) of $\RS{\Gamma,\cp}{\atp}{m}(\zz)$ at the cusp $\cq$ in the sense of \S\ref{sec:conven:autfrm}.

To further emphasize the importance of the functions
$\RS{\Gamma,\cp}{\atp}{m}(\zz)$ we observe that every Rademacher
sum $\RS{\Gamma,\cp|\cq}{\atp}{m}(\zz)$ is of the form
$\RS{\Gamma',\cp'}{\atp}{m}(\zz)$ for some group $\Gamma'$ with
width one at infinity, and some cusp $\cp'$ of $\Gamma'$. For if
$\Sigma_{\cp}$ and $\Sigma_{\cq}$ are the chosen scaling sets for
$\Gamma$ at $\cp$ and $\cq$, respectively, then we have
$\RS{\Gamma,\cp|\cq}{\atp}{m}(\zz)=\RS{U}{\atp}{m}(\zz)$ for
$U=\lBZ\Sigma_{\cp}^{-1}\Sigma_{\cq}\rBZh$. Recall (cf.
(\ref{eqn:conven:scaling:Defn_Conj_Gamma_by_sigma_cq})) that
$\Gamma^{\cq}$ is a shorthand for the group
$\Sigma_{\cq}^{-1}\Sigma_{\cq}$, and this group $\Gamma^{\cq}$ has
width one at infinity by the defining properties of scaling cosets.
If we define $\cp^{\cq}=\Sigma_{\cq}^{-1}\cdot\cp\subset\hat{\QQ}$
then $\cp^{\cq}$ is a cusp of $\Gamma^{\cq}$, and
$\Sigma_{\cp^{\cq}}=\Sigma_{\cq}^{-1}\Sigma_{\cp}$ is a scaling
coset for $\Gamma^{\cq}$ at $\cp^{\cq}$. We now have
$\Sigma_{\cp}^{-1}\Sigma_{\cq}=\Sigma_{\cp^{\cq}}^{-1}$, so that
$U=\lBZ\Sigma_{\cp^{\cq}}^{-1}\rBZh$. Since $\Gamma^{\cq}$ has width
one at infinity we may take it to be a scaling coset for itself at
the infinite cusp, and we thus have
\begin{gather}\label{eqn:radsum:constr:RS_twocusps_to_onecusp}
     \RS{\Gamma,\cp|\cq}{\atp}{m}(\zz)
     =
     \RS{\Gamma^{\cq},\cp^{\cq}}{\atp}{m}(\zz)
\end{gather}
subject to a consistent choice of scaling cosets for $\Gamma$ and $\Gamma^{\cq}$. More explicitly, the equality in (\ref{eqn:radsum:constr:RS_twocusps_to_onecusp}) holds in case
$\Gamma^{\cq}=\Sigma_{\cq}^{-1}\Sigma_{\cq}$ and
$\cp^{\cq}=\Sigma_{\cq}^{-1}\cdot\cp$ and
$\Sigma_{\cp^{\cq}}=\Sigma_{\cq}^{-1}\Sigma_{\cp}$.

Our primary interest in this article is in the distinguished case
that $\atp=0$. In order to simplify notation, and maintain
consistency with the notation of \S\ref{sec:intro}, we write
$\RS{\Gamma,\cp|\cq}{}{m}(\zz)$ as a shorthand for
$\RS{\Gamma,\cp|\cq}{0}{m}(\zz)$.

We conclude this section with a result which expresses the
Rademacher component function $\RS{\lBZ\chi\rBZh}{\atp}{m}(\zz)$ in
terms of the functions $\PS{\lBZ\chi\rBZh}{\atp}{m}(\zz)$ of
(\ref{eqn:radsum:constr:Defn_PS_chi}) and the integral operator
$\JO{\ww}{\atp}$ of \S\ref{sec:conven:autfrm}.
\begin{lem}\label{lem:radsum:constr:Relate_RS_chi_m_to_IO_w_atp}
Let $\atp,m \in \ZZ$ such that $\atp\leq 0$ and $m<0$, and
let $\lBZ\chi\rBZh\in\lBZ G(\QQ)\rBZh$. Then we have
\begin{gather}\label{eqn:radsum:constr:Relate_RS_chi_m_to_IO_w_atp}
     \RS{\lBZ\chi\rBZh}{\atp}{m}(\zz)
     =
     \PS{\lBZ\chi\rBZh}{\atp}{m}(\zz)
     -
     m^{1-2\atp}
     \left(
     \JO{\infty\cdot\lBZ\chi\rBZh}{\atp}
     \PS{\lBZ\chi\rBZh}{1-\atp}{-m}
     \right)(\zz).
\end{gather}
\end{lem}
\begin{proof}
In case $\lBZ\chi\rBZh\in \lBZ G(\QQ)\rBZh_{\infty}$ the second term
in the right hand side of
(\ref{eqn:radsum:constr:Relate_RS_chi_m_to_IO_w_atp}) vanishes (cf.
(\ref{eqn:radsum:constr:Apply_JO_to_PS_chi_1-atp_-m})), and the
identity (\ref{eqn:radsum:constr:Relate_RS_chi_m_to_IO_w_atp}) then
holds by the definition of $\RS{\lBZ\chi\rBZh}{\atp}{m}(\zz)$, so we
assume henceforth that $\lBZ\chi\rBZh\in \lBZ G(\QQ)\rBZh^{\times}$.
By (\ref{eqn:radsum:constr:Rreg_Part_Exp}) we may write the
Rademacher component function $\RS{\lBZ\chi\rBZh}{\atp}{m}(\zz)$ as
\begin{gather}\label{eqn:radsum:constr:RS_chi_m_as_diff}
     \ee(m\lBZ\chi\rBZh\cdot\zz)
     \jac(\lBZ\chi\rBZh,\zz)^{\atp}
     -
     \ee(m\lBZ\chi\rBZh\cdot\infty)
     \ee(
     m\lBZ\chi\rBZh\cdot\zz
     -
     m\lBZ\chi\rBZh\cdot\infty
     )_{<1-2\atp}
     \jac(\lBZ\chi\rBZh,\zz)^{\atp}.
\end{gather}
The first term in (\ref{eqn:radsum:constr:RS_chi_m_as_diff}) is just
$\PS{\lBZ\chi\rBZh}{\atp}{m}(\zz)$, and a contour integral
calculation confirms that the second term in
(\ref{eqn:radsum:constr:RS_chi_m_as_diff}) is $m^{1-2\atp}$ times
the image of $\PS{\lBZ\chi\rBZh}{1-\atp}{-m}(\zz)$ under
$\JO{\infty\cdot\lBZ\chi\rBZh}{\atp}$. The details of this contour
integral calculation are as follows. We have
\begin{gather}\label{eqn:radsum:constr:Apply_JO_to_PS_chi_1-atp_-m}
     \overline{
     \left(
     \JO{\infty\cdot\lBZ\chi\rBZh}{\atp}
     \PS{\lBZ\chi\rBZh}{1-\atp}{-m}
     \right)
     (\zz)
     }
     =
     (\tpi)^{1-2\atp}
     \int_{\lBZh\chi^{-1}\rBZ\cdot\infty}^{\infty}
     \ee(-m\lBZ\chi\rBZh\cdot\xi)\jac(\lBZ\chi\rBZh,\xi)^{1-\atp}
     (\xi-\bar{\zz})^{(-2\atp)}
     {\rm d}\xi
\end{gather}
by the definition of $\JO{\ww}{\atp}$ (cf.
\S\ref{sec:conven:autfrm}). Choosing a representative $\chi$ for the
coset $\lBZ\chi\rBZh$ we compute
\begin{gather}
     \begin{split}\label{eqn:radsum:var:Compute_integral_in_JO_chi_PS_1-atp}
     &
     \int_{\lBZh\chi^{-1}\rBZ\cdot\infty}^{\infty}
     \ee(-m\lBZ\chi\rBZh\cdot\xi)\jac(\lBZ\chi\rBZh,\xi)^{1-\atp}
     (\xi-\bar{\zz})^{(-2\atp)}
     {\rm d}\xi
     \\
     &=
     \ee(-m\chi\cdot\infty)
     \int_{\infty}^{\chi\cdot\infty}
     \ee(
     m\chi\cdot\infty
     -m\chi\cdot\xi
     )
     (\chi\cdot\xi-\chi\cdot\bar{\zz})^{(-2\atp)}
     {\rm d}(\chi\cdot\xi)
     \jac(\chi,\bar{\zz})^{\atp}
     \\
     &
     =-
     \ee(-m\chi\cdot\infty)
     \int_{\chi\cdot\infty}^{\infty}
     \ee(m\chi\cdot\infty-m\xi)
     (\xi-\chi\cdot\bar{\zz})^{(-2\atp)}
     {\rm d}\xi
     \jac(\chi,\bar{\zz})^{\atp}
     \\
     &
     =-
     \ee(-m\chi\cdot\infty)
     \int_{0}^{\infty}
     \ee(-m\xi)
     (\xi+\chi\cdot\infty-\chi\cdot\bar{\zz})^{(-2\atp)}
     {\rm d}\xi
     \jac(\chi,\bar{\zz})^{\atp}
     \end{split}
\end{gather}
where the last integral in
(\ref{eqn:radsum:var:Compute_integral_in_JO_chi_PS_1-atp}) is taken
over the vertical line $\{\ii t\mid t\in \RRp\}$. Thus, making the
substitution $\tpi m \xi= t$, we find that
\begin{gather}
     \begin{split}
     &
     \int_{\lBZh\chi^{-1}\rBZ\cdot\infty}^{\infty}
     \ee(-m\lBZ\chi\rBZh\cdot\xi)\jac(\lBZ\chi\rBZh,\xi)^{1-\atp}
     (\xi-\bar{\zz})^{(-2\atp)}
     {\rm d}\xi
     \\
     &
     =
     -
     \ee(-m\chi\cdot\infty)
     (\tpi m)^{-1}
     \int_{0}^{\infty}
     e^{-t}
     ((\tpi m)^{-1} t+\chi\cdot\infty-\chi\cdot\bar{\zz})^{(-2\atp)}
     {\rm d}t
     \jac(\chi,\bar{\zz})^{\atp}
     \\
     &
     =
     -
     \ee(-m\chi\cdot\infty)
     \sum_{k=0}^{-2\atp}
     (\tpi m)^{2\atp-1}
     (\tpi(m\chi\cdot\infty-m\chi\cdot\bar{\zz}))^{(k)}
     \jac(\chi,\bar{\zz})^{\atp},
     \end{split}
\end{gather}
and from this we deduce that
\begin{gather}
     \begin{split}
     \left(
     \JO{\infty\cdot\lBZ\chi\rBZh}{\atp}
     \PS{\lBZ\chi\rBZh}{1-\atp}{-m}
     \right)
     (\zz)
     =
     m^{2\atp-1}
     \ee(m\chi\cdot\infty)
     \ee(m\chi\cdot{\zz}-m\chi\cdot\infty)_{<1-2\atp}
     \jac(\chi,{\zz})^{\atp}.
     \end{split}
\end{gather}
Upon comparison with (\ref{eqn:radsum:constr:RS_chi_m_as_diff}) we
obtain the required identity
(\ref{eqn:radsum:constr:Relate_RS_chi_m_to_IO_w_atp}). This
completes the proof.
\end{proof}

\subsection{Coefficients}\label{sec:radsum:coeff}

For $\lBZ\chi\rBZ\in \lBZ G(\QQ)\rBZ^{\times}$ and $m,n\in\ZZ$ we define a holomorphic function $\spp\mapsto\Kl_{\lBZ\chi\rBZ}^{\spp}(m,n)$ on $\CC$ by setting
\begin{gather}\label{eqn:radsum:coeff:Defn_Klmns}
     \Kl_{\lBZ\chi\rBZ}^{\spp}(m,n)
     =
     \ee(m\lBZ\chi\rBZ\cdot\infty)
     \ee(-n\lBZ\chi^{-1}\rBZ\cdot\infty)
     \rads\lBZ\chi\rBZ^{\spp}.
\end{gather}
Here $\rads\lBZ\chi\rBZ$ denotes the scaling factor of (\ref{eqn:conven:isomhyp:Scaling_Factor}) (cf. Lemma \ref{lem:conven:cosets_RadsInv}). If $\chi\in G(\QQ)^{\times}$ has a preferred representative $\binom{a\;b}{c\;d}$ with $c>0$, then we have
\begin{gather}
     \Kl_{\lBZ\chi\rBZ}^{\spp}(m,n)=
     \ee\left(\frac{ma+nd}{c}\right)
     \frac{(ad-bc)^{\spp}}{c^{2\spp}},
\end{gather}
so that sums of the $\Kl_{\lBZ\chi\rBZ}^{\spp}(m,n)$ over suitable subsets of $\lBZ G(\QQ)\rBZ^{\times}$ recover zeta functions of various kinds, as we shall see presently. With this in mind consider the series $\Kl_{S}^{\spp}(m,n)$ defined for $S$ a subset of $\lBZ G(\QQ)\rBZ^{\times}$ by setting
\begin{gather}\label{eqn:radsum:coeff:Defn_SKzeta}
     \Kl_{S}^{\spp}(m,n)
     =
     \lim_{K\to\infty}
     \sum_{\lBZ\chi\rBZ\in S_{\leq K}}\Kl_{\lBZ\chi\rBZ}^{\spp}(m,n)
\end{gather}
where $S_{\leq K}$ is as in (\ref{eqn:conven:cosets:Defn S_sub_K}). Write $\Kl_{\Gamma,\cp|\cq}^{\spp}(m,n)$ for $\Kl_{S}^{\spp}(m,n)$ in the case that $S=\lBZ \Sigma_{\cp}^{-1}\Sigma_{\cq}\rBZ^{\times}$ for $\Gamma$ a subgroup of $G(\QQ)$ that is commensurable with $G(\ZZ)$ and $\{\Sigma_{\cp}\mid \cp\in\cP_{\Gamma}\}$ is a system of scaling cosets for $\Gamma$. We will verify presently (cf. Proposition \ref{prop:radsum:coeff:SKzeta_conv}) that the Dirichlet series $\spp\mapsto \Kl_{\Gamma,\cp|\cq}^{\spp}(m,n)$ converges absolutely for $\Re(\spp)>1$ (so that the limit is unnecessary there) and therefore defines a holomorphic function in this region; these functions $\Kl_{\Gamma,\cp|\cq}^{\spp}(m,n)$ are among the {\em Selberg--Kloosterman zeta functions} introduced by Selberg in \cite{Sel_EstFouCoeffs} and shown there to admit meromorphic continuation to the entire complex plane. Selberg demonstrated further in loc. cit. that the eigenfunctions of the {hyperbolic Laplacian} $\Delta=(\zz-\bar{\zz})^2\partial_{\zz}\partial_{\bar{\zz}}$ on $\Gamma\backslash\HH$ determine the poles of $\Kl_{S}^{\spp}(m,n)$ in such a way that a pole at $\spp$ corresponds to an eigenfunction with eigenvalue $\lambda=\spp(1-\spp)$. The convergence of expressions very similar to (\ref{eqn:radsum:coeff:Defn_SKzeta}) in certain regions $\Re(\spp)\geq\sigma$ with $\sigma\leq 1$ was considered by Knopp in \cite{Kno_SmlPosPowTheta} and \cite{Kno_SmlPosWgt}\footnote{Please note that the latter of these references contains important corrections to the former.}.

\begin{prop}\label{prop:radsum:coeff:SKzeta_conv}
Let $\Gamma<G(\QQ)$ be a group commensurable with $G(\ZZ)$, let $\{\Sigma_{\cp}\mid\cp\in\cP_{\Gamma}\}$ be a system of scaling cosets for $\Gamma$ and set $S=\lBZ\Sigma_{\cp}^{-1}\Sigma_{\cq}\rBZ^{\times}$ for some $\cp,\cq\in\cP_{\Gamma}$. Then the series (\ref{eqn:radsum:coeff:Defn_SKzeta}) defining $\Kl_{S}^{\spp}(m,n)$ converges for $\spp=1$ when $m$ and $n$ are not both zero, and converges absolutely and locally uniformly in $\Re(s)>1$ for any $m,n\in\ZZ$.%, and the convergence is uniform in $m$ and $n$.
\end{prop}
\begin{proof}
For positive integers $r$ and $c$ let $S_c^r$ denote the subset of $S$ consisting of double cosets $\lBZ\chi\rBZ\in\lBZ G(\QQ)\rBZ^{\times}$ with $\Pdet(\chi)=r$ and $c(\chi)=c$, so that $\rads(\chi)=r/c^2$ (cf. (\ref{eqn:conven:isomhyp:Defn_Pdet})) for $\lBZ\chi\rBZ\in S_c^r$. Observe that there are at most $c$ elements in $S_c^r$ so we have 
\begin{gather}\label{eqn:radsum:coeff:Kl_triv_bound}
|\Kl_{S_c^r}^{\spp}(m,n)|
\leq 
\sum_{\lBZ\chi\rBZ\in S_c^r}
|\Kl_{\lBZ\chi\rBZ}^{\spp}(m,n)|
\leq
r^{\spp}c^{1-2\spp}
\end{gather}
for $\spp\in\CC$. Now $S=\bigcup_{r,c>0}S^r_c$ and there are only finitely many $r$ for which $S^r=\bigcup_{c>0}S^r_c$ is non-empty since $\Gamma$ is commensurable with $G(\ZZ)$, so (\ref{eqn:radsum:coeff:Kl_triv_bound}) and the identity $\Kl_{S}^{\spp}(m,n)=\sum_{c>0}\sum_{r>0}\Kl_{S_c^r}^{\spp}(m,n)$ implies that the series of absolute values obtained from (\ref{eqn:radsum:coeff:Defn_SKzeta}) is bounded above by a constant independent of $m$ and $n$ so long as $\Re(\spp)>1$. Indeed, we have
\begin{gather}\label{eqn:radsum:coeff:Kl_triv_est}
	|\Kl_S^{\spp}(m,n)|
	\leq \sum_{c>0}\sum_{r\in \Pdet(S)}r^{\sigma}c^{1-2\sigma}
	\leq \#\Pdet(S)r_{\rm max}^{\sigma}\frac{1}{2\sigma-2}
\end{gather}
where $\#\Pdet(S)$ is the number of values of $r$ for which $S^r$ is non-empty, $r_{\rm max}$ is the maximum $r$ for which $S^r$ is non-empty, and $\sigma=\Re(\spp)$. This proves the required convergence of (\ref{eqn:radsum:coeff:Defn_SKzeta}) in the region $\Re(\spp)>1$.

We now consider convergence at $\spp=1$. As mentioned above, Selberg demonstrated the meromorphic continuation of the Selberg--Kloosterman zeta function $\spp\to\Kl_{S}^{\spp}(m,n)$ and related its poles to the spectrum of the Laplacian on $\Gamma\backslash\HH$ in \cite{Sel_EstFouCoeffs}. It turns out that those poles in $\Re(\spp)>1/2$ are finite in number and are confined to the real line segment $(1/2,1)$ in the case that $m$ and $n$ are not both zero, and the absence of a pole at $\spp=1$ derives from the fact that the eigenfunctions for $\Delta$ with eigenvalue $0$ are harmonic, and thus constant. (See \cite{Iwa_SpecMetAutFrms}, and Chapter 9 especially, for a thorough development of these ideas.) In particular, if we write $Z_{m,n}(\spp)$ for Selberg's meromorphic continuation of $\Kl_{S}^{\spp}(m,n)$, then $Z_{m,n}(\spp)$ is holomorphic at $\spp$ when $\Re(\spp)=1$ and $m$ and $n$ are not both zero, and the only pole of $Z_{0,0}(\spp)$ on the line $\Re(\spp)=1$ is at $\spp=1$. On the other hand Goldfeld--Sarnak determined the estimate 
\begin{gather}\label{eqn:radsum:coeff:GSest}
	%|\Kl_{S}^{\spp}(m,n)|
	|Z_{m,n}(\spp)|=O\left(|m||n|\frac{|\spp|^{1/2}}{\Re(\spp)-1/2}\right)
\end{gather}
for $mn\neq 0$ as $\Im(\spp)\to\infty$ with $\Re(\spp)>1/2$ in \cite{GolSar_SumsKloosSums}. (Loc. cit. considers the case that $m,n>0$. See \cite{Pri_GnlzdGolSarEst} for the extension to $mn\neq 0$, and an improvement upon (\ref{eqn:radsum:coeff:GSest}).) Armed with this bound and the holomorphy of $Z_{m,n}(\spp)$ in $\Re(\spp)>1-a$ for some $a>0$ we can deduce the convergence of (\ref{eqn:radsum:coeff:Defn_SKzeta}) at $\spp=1$ in the following way, which is an adaptation of the argument presented in \S2 of \cite{Kno_SmlPosPowTheta}. %For simplicity we consider just the case that $\spp=1$, which is the case most relevant for this paper. 
Note that we can choose $a>0$ such that $Z_{m,n}(\spp)$ has no poles in $(1-a,1)$ for all $m,n\in\ZZ$ since the values of $\spp$ such that $\Re(\spp)>1/2$ and $\spp(1-\spp)$ is an eigenvalue of $\Delta$ are finite in number. For such $a$ the estimate (\ref{eqn:radsum:coeff:GSest}) holds uniformly in $\spp$ for $\Re(\spp)>1-a$.
 
Let $a>0$ such that the interval $(1-a,1)$ contains no poles of $Z_{m,n}(\spp)$ for any $m,n\in\ZZ$, let $T>0$ and set $C$ to be the positively oriented rectangular contour with corners $\pm a\pm \ii T$. Consider the contour integral
\begin{gather}
	I=\frac{1}{\tpi}\int_{C}Z_{m,n}\left(\frac{t}{2}+1\right)x^t\frac{{\rm d}t}{t}
\end{gather}
where $x>0$. By our choice of $a$ the integrand has a unique simple pole inside $C$ and the residue there is $Z_{m,n}(1)$, so $I=Z_{m,n}(1)$. On the other hand the integral over the right most vertical segment of $C$ contributes $\Sigma_T(x)$ where 
\begin{gather}
	\Sigma_T(x)=\frac{1}{\tpi}\int_{a-\ii T}^{a+\ii T}Z_{m,n}\left(\frac{t}{2}+1\right)x^t\frac{{\rm d}t}{t}.
\end{gather}
Observe that in the limit as $T\to\infty$ we obtain the function 
\begin{gather}\label{eqn:radsum:coeff:defn_Sigma(x)}
\Sigma(x)
=\lim_{T\to\infty}\Sigma_T(x)
=\sum_{r>0}\sum_{0<c<x\sqrt{r}}\Kl_{S_c^r}^1(m,n)
\end{gather}
(so long as $x$ does not coincide with $c/\sqrt{r}$ for some $c,r>0$ with $S_c^r$ non-empty) and the difference $|\Sigma(x)-\Sigma_T(x)|$ is $O(x^aT^{-1})$ where the implied constant is independent of $m$ and $n$. The estimate (\ref{eqn:radsum:coeff:GSest}) implies that the contribution to $I$ of the left most vertical segment of $C$ is $O(|m||n|x^{-a}T^{1/2})$, and we may show that the contributions of the two horizontal components of $C$ are $O(|m||n|x^aT^{1/4-1/4a})$ so long as $T<x^{4a}$ by applying the Phragm\'en--Lindel\"of theorem (as stated in Theorem 14 of \cite{HarRie_DirSeries}) to the estimates $|Z_{m,n}(t/2+1)|=O(|m||n||\Im(t)|^{1/2})$ for $\Re(t)=-a$ and $|Z_{m,n}(t/2+1)|=O(1)$ for $\Re(t)=a$ (where the former of these is another application of (\ref{eqn:radsum:coeff:GSest}) and the constant implied by the latter is independent of both $m$ and $n$ according to Proposition \ref{prop:radsum:coeff:SKzeta_conv}). Setting $T=x^{3a/2}$ we obtain 
\begin{gather}\label{eqn:radsum:coeff:Sigma_to_zeta}
	|\Sigma(x)-Z_{m,n}(1)|=O(|m||n|x^{-a/4})
\end{gather}
under the condition that $a<3/13$, and in the limit as $x\to \infty$ we arrive at the convergence of the sum (\ref{eqn:radsum:coeff:Defn_SKzeta}) at $\spp=1$ in the case that $mn\neq 0$. 

For the case that $mn=0$ but $m+n\neq 0$ we note that a similar estimate to (\ref{eqn:radsum:coeff:GSest})
\begin{gather}\label{eqn:radsum:coeff:Kubest}
	%|\Kl_{S}^{\spp}(m,n)|
	|Z_{m,n}(\spp)|=O\left({|\spp|^{1/2}}\right)
\end{gather}
holds uniformly in $\spp$ for $\Re(\spp)>1-a$, with the implied constant independent of $m$ and $n$, as a consequence of the boundedness of Fourier coefficients of non-analytic Eisenstein series in this region (cf. \cite{Kub_EisSer}). Now the argument of the preceding paragraph goes through with (\ref{eqn:radsum:coeff:Kubest}) used in place of (\ref{eqn:radsum:coeff:GSest}) and we obtain the desired convergence of $\Kl_S^{\spp}(m,n)$ at $\spp=1$ for all $m,n\in\ZZ$ such that $m$ and $n$ are not both zero. This completes the proof of the proposition.
\end{proof}

\begin{rmk}
Consider the case that $S=\lBZ G(\ZZ)\rBZ^{\times}$. Then the function $\Kl_{S}^{\spp}(m,n)$ is closely related to the Riemann zeta function $\zeta(s)$, for we have
\begin{gather}\label{eqn:radsum:coeff:Kloos_to_sigma_over_zeta}
     \Kl_{S}^{\spp}(m,0)
     =\sum_{c\in \ZZp}
     \sum_{a\in (\ZZ/c)^{*}}
     \frac{\ee(ma/c)}{c^{2\spp}}
     =\frac{\sigma(|m|,{1-2\spp})}{\zeta(2\spp)}
\end{gather}
for $\Re(\spp)>1$ in the case that $m< 0$, where $\sigma(n,\spp)$ is the divisor function of (\ref{eqn:conven:fns:Div_Fn}). In particular, the function $s\mapsto\Kl_{S}^{\spp/2}(1,0)$ is just the reciprocal of the Riemann zeta function when $S=\lBZ G(\ZZ)\rBZ^{\times}$. When $m=n=0$ we have $\Kl_{S}^{\spp}(0,0)=\sum_{c>0}\phi(c)c^{-2\spp}$ where $\phi$ is Euler's totient function so clearly the sum diverges at $\spp=1$ in this case. In general the residue of $\Kl_{S}^{\spp}(0,0)$ at $\spp=1$ is the reciprocal of the hyperbolic area of the quotient $\Gamma\backslash\HH$ when $S=\lBZ\Gamma\rBZ^{\times}$ for some group $\Gamma$. (Cf. \cite[\S6.4]{Iwa_SpecMetAutFrms}.)
\end{rmk}

The following result is evident from the definition of
$\Kl_{\lBZ\chi\rBZ}^{\spp}(m,n)$, and the definition of the
conjugation $\chi\mapsto \bar{\chi}$ (cf.
\S\ref{sec:conven:isomhyp}).
\begin{lem}\label{lem:radsum:coeff:Symmetries_Kl}
For any $\chi\in G(\QQ)^{\times}$, $m,n\in\ZZ$, and $\spp\in \CC$,
we have
\begin{gather}\label{eqn:conven:kloos_SymmsKl}
     \Kl_{\lBZ\chi\rBZ}^{\spp}(m,n)
     =\Kl_{\lBZ\chi^{-1}\rBZ}^{\spp}(n,m)
     =\Kl_{\lBZ\bar{\chi}\rBZ}^{\spp}(-m,-n),\\
     \Kl_{\lBZ\chi\rBZ}^{\spp}(m,n)\rads\lBZ\chi\rBZ^{1-\spp}
     =\Kl_{\lBZ\chi\rBZ}^{1-\spp}(m,n)\rads\lBZ\chi\rBZ^{s}.
\end{gather}
\end{lem}

For a continuous analogue of the function
$\Kl_{\lBZ\chi\rBZ}^{\spp}(m,n)$ we define
$\Bf_{\lBZ\chi\rBZ}^{\atp}(\ww,\zz)$, for $\chi\in G(\QQ)^{\times}$,
$\ww,\zz\in\CC$ and $\atp\in \ZZ$, by setting
\begin{gather}
     \Bf_{\lBZ\chi\rBZ}^{\atp}(\ww,\zz)
     =
     \tpi\Res_{\xi=0}
     {\ee\left(\frac{\ww\rads\lBZ\chi\rBZ}{\xi}\right)
     \ee(\zz \xi)}\,
     \frac{1}{\xi^{2\atp}}\,
     {\rm d}\xi.
\end{gather}

\begin{lem}\label{lem:radsum:coeff:Power_Series_Bf}
The function $\Bf_{\lBZ\chi\rBZ}^{\atp}(\ww,\zz)$ admits the
following series representations.
\begin{gather}
     \Bf_{\lBZ\chi\rBZ}^{\atp}(\ww,\zz)=\label{eqn:radsum:coeff:Pow_Ser_Bf_s_pos}
          %(-1)^{\atp}
          \sum_{k\geq 0}
          (\tpi)^{2k+2\atp}
          \rads\lBZ\chi\rBZ^k
          \ww^{(k)}\zz^{(k+2\atp-1)}\\
     \Bf_{\lBZ\chi\rBZ}^{\atp}(\ww,\zz)=\label{eqn:radsum:coeff:Pow_Ser_Bf_1-s_pos}
          %(-1)^{\atp}
          \sum_{k\geq 0}
          (\tpi)^{2k+2-2\atp}
          \rads\lBZ\chi\rBZ^{k+1-2\atp}
          \ww^{(k+1-2\atp)}\zz^{(k)}
\end{gather}
The first expression (\ref{eqn:radsum:coeff:Pow_Ser_Bf_s_pos}) holds
for $\atp>0$. The second expression
(\ref{eqn:radsum:coeff:Pow_Ser_Bf_1-s_pos}) holds for
$\atp\leq 0$.
\end{lem}
\begin{lem}\label{lem:radsum:coeff:Symmetries_Bf}
For any $\chi\in G(\QQ)^{\times}$ and $\ww,\zz\in\CC$ and
$\atp\in\ZZ$ we have
\begin{gather}\label{eqn:conven:bessel_SymmsBf}
     \Bf_{\lBZ\chi\rBZ}^{\atp}(\ww,\zz)\rads\lBZ\chi\rBZ^{\atp}
     =-\Bf_{\lBZ\chi\rBZ}^{1-\atp}(\zz,\ww)\rads\lBZ\chi\rBZ^{1-\atp}
     =-\Bf_{\lBZ\chi\rBZ}^{\atp}(-\ww,-\zz)\rads\lBZ\chi\rBZ^{\atp}.
\end{gather}
\end{lem}

The main construction of this section is the following. Given a
subset $S\subset\lBZ G(\QQ)\rBZ^{\times}$ and $\atp\in \ZZ$, we
define the {\em coefficient function}, denoted $(m,n)\mapsto
\fc{S}{\atp}(m,n)$, by setting
\begin{gather}\label{eqn:radsum:coeff:Defn_fcS}
     \fc{S}{\atp}(m,n)
     =
     \lim_{K\to \infty}
     \sum_{
          \lBZ\chi\rBZ\in S_{\leq K}
          }
     \Kl_{\lBZ\chi\rBZ}^{\atp}(m,n)
     \Bf_{\lBZ\chi\rBZ}^{\atp}(m,n)
\end{gather}
(cf. (\ref{eqn:conven:cosets:Defn S_sub_K})) whenever this limit converges. In applications we take $S$ to be of the form
$S=\lBZ\Sigma_{\cp}^{-1}\Sigma_{\cq}\rBZ^{\times}$ where
$\Sigma_{\cp}$ and $\Sigma_{\cq}$ are scaling cosets %(cf. \S\ref{sec:conven:scaling}) 
for a group $\Gamma$ at cusps $\cp,\cq\in \cP_{\Gamma}$ where $\Gamma$ is commensurable with $G(\ZZ)$. We will see presently that $\fc{S}{\atp}(m,n)$ converges for any such $S$. We obtain the following formulas by combining (\ref{eqn:radsum:coeff:Defn_Klmns}) and Lemma \ref{lem:radsum:coeff:Power_Series_Bf}.
\begin{lem}\label{lem:radsum:coeff:Power_Series_fcS}
The coefficient functions $\fc{S}{\atp}(m,n)$ admit the following
series representations.
\begin{gather}
     \fc{S}{\atp}(m,n)
     =\label{eqn:radsum:coeff:Pow_Ser_fcS_s_pos}
          %(-1)^{\atp}
               \lim_{K\to \infty}
     \sum_{
          \lBZ\chi\rBZ\in S_{\leq K}
          }
          %\sum_{\lBZ\chi\rBZ\in S}
          \sum_{k\geq 0}
          \ee(m\lBZ\chi\rBZ\cdot\infty)
          \ee(-n\lBZ\chi^{-1}\rBZ\cdot\infty)
          (-\fps\rads\lBZ\chi\rBZ)^{k+\atp}
          m^{(k)}n^{(k+2\atp-1)}
     \\
     \fc{S}{\atp}(m,n)
     =\label{eqn:radsum:coeff:Pow_Ser_fcS_1-s_pos}
          %(-1)^{\atp}
	     \lim_{K\to \infty}
     	\sum_{
          \lBZ\chi\rBZ\in S_{\leq K}
         		 }
          %\sum_{\lBZ\chi\rBZ\in S}
          \sum_{k\geq 0}
          \ee(m\lBZ\chi\rBZ\cdot\infty)
          \ee(-n\lBZ\chi^{-1}\rBZ\cdot\infty)
          (-\fps\rads\lBZ\chi\rBZ)^{k+1-\atp}
          m^{(k+1-2\atp)}n^{(k)}
\end{gather}
The expression (\ref{eqn:radsum:coeff:Pow_Ser_fcS_s_pos}) holds for
$\atp>0$. The expression
(\ref{eqn:radsum:coeff:Pow_Ser_fcS_1-s_pos}) holds for $\atp\leq 0$.
\end{lem}
\begin{prop}\label{prop:radsum:coeff:Func_Eqn_fcSmns}
For $S\subset \lBZ G(\QQ)\rBZ^{\times}$ and $m,n,\atp\in\ZZ$ we have
\begin{gather}
     \fc{S}{\atp}(m,n)
     =-\fc{S^{-1}}{1-\atp}(n,m)
     =-\fc{\bar{S}}{\atp}(-m,-n).
\end{gather}
\end{prop}
\begin{proof}
These identities follow directly from Lemma
\ref{lem:radsum:coeff:Power_Series_fcS}, and the observation that
the values $\rads\lBZ\chi\rBZ$, $\rads\lBZ\chi^{-1}\rBZ$ and
$\rads\lBZ\bar{\chi}\rBZ$ all coincide for any $\chi\in
G(\QQ)^{\times}$.
\end{proof}

\begin{prop}\label{prop:radsum:coeff:fcS_conv}
Let $\Gamma$ be a group commensurable with $G(\ZZ)$ and let
$\cp,\cq\in \cP_{\Gamma}$ be cusps of $\Gamma$. Then the expression (\ref{eqn:radsum:coeff:Defn_fcS}) defining the coefficient function $\fc{S}{\atp}(m,n)$ converges for $S=\lBZ\Sigma_{\cp}^{-1}\Sigma_{\cq}\rBZ^{\times}$ and any $\atp,m,n\in\ZZ$, and further, we have
\begin{gather}\label{eqn:radsum:coeff:fcS_est}
	|\fc{S}{\atp}(m,n)|=O\left(|m|^{1/2-\atp}|n|^{\atp-1/2}e^{C|mn|^{1/2}}\right)
\end{gather}
as $|n|\to \infty$ where $C$ depends only on $S$ and the implied constant depends only on $S$ and $\atp$.
\end{prop}
\begin{proof}
Observe that according to the $\atp\mapsto 1-\atp$ symmetry of Proposition \ref{prop:radsum:coeff:Func_Eqn_fcSmns} it suffices to consider the case that $\atp>0$. Observe also that for $\atp>0$ the coefficient function $\fc{S}{\atp}(m,n)$ vanishes for all $m$ when $n=0$, and when $m=0$ we have
\begin{gather}\label{eqn:radsum:coeff:fcS_conv_0}
	\fc{S}{\atp}(0,n)=(-4\pi^2)^{\atp}
	%\sum_{c>0}\sum_{r>0}\Kl_{S_c^r}^{\atp}(0,n)
	n^{(2\atp-1)}
	\Kl_{S}^{\atp}(0,n)
	%=O\left(\cdots\right)
\end{gather}	
according to (\ref{eqn:radsum:coeff:Pow_Ser_fcS_s_pos}), so the convergence in case $mn=0$ is either trivial or follows from the convergence of the Selberg--Kloosterman zeta function $\Kl_{S}^{\atp}(m,n)$ which is established in Proposition \ref{prop:radsum:coeff:SKzeta_conv}. Assume then that $mn\neq 0$. By (\ref{eqn:radsum:coeff:Pow_Ser_fcS_s_pos}) we have
\begin{gather}\label{eqn:radsum:coeff:fcS_conv_1}
	\fc{S}{\atp}(m,n)=(-4\pi^2)^{\atp}\sum_{c>0}\sum_{r>0}\sum_{k\geq 0}\Kl_{S_c^r}^{\atp+k}(m,n)(-4\pi^2m)^{(k)}n^{(k+2\atp-1)}
\end{gather}
where $S_c^r$ is defined as in the proof of Proposition \ref{prop:radsum:coeff:SKzeta_conv} and the sum over $r$ is finite. In case $\atp>1$ we can reorder the summation over $k$ in (\ref{eqn:radsum:coeff:fcS_conv_1}) past the other two since the sum 
\begin{gather}\label{eqn:radsum:coeff:fcS_conv_2}
	\fc{S}{\atp}(m,n)=
	(-4\pi^2)^{\atp}
	\sum_{k\geq 0}
	(-4\pi^2m)^{(k)}n^{(k+2\atp-1)}
	%\sum_{c>0}\sum_{r>0}
	\Kl_{S}^{\atp+k}(m,n)
\end{gather}
so obtained is absolutely convergent according to the estimate (\ref{eqn:radsum:coeff:Kl_triv_est}). In detail,
\begin{gather}\label{eqn:radsum:coeff:fcS_conv_3}
	\begin{split}
	|\fc{S}{\atp}(m,n)|
	&\leq 
	(4\pi^2)^{\atp}
	\sum_{k\geq 0}(4\pi^2|m|)^{(k)}|n|^{(k+2\atp-1)}
	\#\Pdet(S)
	r_{\rm max}^{\atp+k}
	\frac{1}{2\atp-2+2k}\\
	&\leq 
	\#\Pdet(S)
	(4\pi^2r_{\rm max})^{\atp}	
	\sum_{k\geq 0}(4\pi^2r_{\rm max}|m|)^{(k)}|n|^{(k+2\atp-1)}<\infty.
	\end{split}
\end{gather}
Now for $A$ and $B$ non-negative we have
\begin{gather}\label{eqn:radsum:coeff:fcS_conv_4}
%	\begin{split}
	\sum_{k\geq 0}A^{(k)}B^{(k+2\atp-1)}
%	&
%	=
%	\left(\frac{B}{A}\right)^{\atp-1/2}\sum_{k\geq 0}\frac{((AB)^{1/2})^{2k+2\atp-1}}{k!(k+2\atp-1)!}
%	\\
%	&
%	\leq
%	\left(\frac{B}{A}\right)^{\atp-1/2}\sum_{k\geq 0}\frac
%	{(2(AB)^{1/2})^{2k+2\atp-1}}{(2k+2\atp-1)!}
%	\\
%	&
	\leq
	A^{1/2-\atp}B^{\atp-1/2}%\left(\frac{B}{A}\right)^{\atp-1/2}
	\sum_{k\geq 0}(2(AB)^{1/2})^{(2k+2\atp-1)}
%	\exp({2(AB)^{1/2}})
%	\end{split}
\end{gather}
so that setting $A=4\pi^2r_{\rm max}|m|$ and $B=|n|$ in (\ref{eqn:radsum:coeff:fcS_conv_4}) we obtain the required estimate (\ref{eqn:radsum:coeff:fcS_est}) from (\ref{eqn:radsum:coeff:fcS_conv_3}) with $C=4\pi r_{\rm max}^{1/2}$.

It remains to establish the estimate (\ref{eqn:radsum:coeff:fcS_est}) in case $\atp=1$. For this we write
\begin{gather}\label{eqn:radsum:coeff:fcS_conv_5}
	\fc{S}{\atp}(m,n)
	=
	(-4\pi^2)^{\atp}
	n^{(2\atp-1)}
	\Kl_S^{\atp}(m,n) +\fc{S}{\atp}(m,n)_+
\end{gather}
where
\begin{gather}\label{eqn:radsum:coeff:fcS_conv_6}
	\fc{S}{\atp}(m,n)_+=
		(-4\pi^2)^{\atp}
	\sum_{k> 0}
	(-4\pi^2m)^{(k)}n^{(k+2\atp-1)}
	%\sum_{c>0}\sum_{r>0}
	\Kl_{S}^{\atp+k}(m,n)
\end{gather}
is an absolutely convergent sum (cf. (\ref{eqn:radsum:coeff:fcS_conv_2})) and we have 
\begin{gather}\label{eqn:radsum:coeff:fcS_conv_7}
	|\fc{S}{\atp}(m,n)_+|=O\left(|m|^{1/2-\atp}|n|^{\atp-1/2}e^{C|mn|^{1/2}}\right)
\end{gather}
with $C=4\pi r_{\rm max}^{1/2}$ according to the estimates (\ref{eqn:radsum:coeff:fcS_conv_3}) and (\ref{eqn:radsum:coeff:fcS_conv_4}). Now we see from (\ref{eqn:radsum:coeff:fcS_conv_5}) that the estimate (\ref{eqn:radsum:coeff:fcS_conv_7}) also holds for $\fc{S}{\atp}(m,n)$ since (\ref{eqn:radsum:coeff:GSest}) (or (\ref{eqn:radsum:coeff:Kubest}) in case $m=0$)  implies that $\Kl_S^{\atp}(m,n)$ has at most polynomial growth in $n$ as $|n|\to \infty$ when $\atp=1$. This completes the proof.
\end{proof}

We write $\fc{\Gamma,\cp|\cq}{\atp}(m,n)$
for $\fc{S}{\atp}(m,n)$ in the case that
$S=\lBZ\Sigma_{\cp}^{-1}\Sigma_{\cq}\rBZ^{\times}$. This notation
suppresses the dependence on the choice of scaling cosets; the
following lemma describes this dependence explicitly.
\begin{lem}
Let $\cp,\cq\in \cP_{\Gamma}$, let $\Sigma_{\cp}$ and
$\Sigma_{\cp}'$ be scaling cosets for $\Gamma$ at $\cp$ and let
$\Sigma_{\cq}$ and $\Sigma_{\cq}'$ be scaling cosets for $\Gamma$ at
$\cq$. Set $S=\lBZ\Sigma_{\cp}^{-1}\Sigma_{\cq}\rBZ^{\times}$ and
$S'=\lBZ(\Sigma_{\cp}')^{-1}\Sigma_{\cq}'\rBZ^{\times}$. Then we
have $\Sigma_{\cp}'=\Sigma_{\cp}T^{\alpha}$ and
$\Sigma_{\cq}'=\Sigma_{\cq}T^{\beta}$ and
$\fc{S'}{\atp}(m,n)=\ee(m\alpha+n\beta)\fc{S}{\atp}(m,n)$ for some
$\alpha,\beta\in \QQ$.
\end{lem}

With a fixed choice of subset $S\subset\lBZ G(\QQ)\rBZ^{\times}$ and
integer $\atp\in \ZZ$, we assemble the coefficient functions
$\fc{S}{\atp}(m,n)$ into a formal series
$\tilde{F}_S^{\atp}(\vp,\vq)$ by setting
\begin{gather}
     \tilde{F}_{S}^{\atp}(\vp,\vq)
     =\sum_{m,n\in\ZZ} \fc{S}{\atp}(m,n)\vp^m\vq^n.
\end{gather}
The formal Fourier coefficients of the formal series
$\tilde{F}_{S}^{\atp}(\vp,\vq)$ with respect to the variable $\vp$
are of particular importance. For fixed $S\subset \lBZ
G(\QQ)\rBZ^{\times}$ and $\atp\in \ZZ$ we define formal series
$\FS{S}{\atp}{m}(\vq)$ by requiring that
\begin{gather}
     \tilde{F}_{S}^{\atp}(\vp,\vq)
     =\sum_{m\in\ZZ}\FS{S}{\atp}{m}(\vq)\vp^m.
\end{gather}
We define the {\em singular} and {\em regular part of
$\FS{S}{\atp}{m}(\vq)$}, to be denoted $\FS{S}{\atp}{m}(\vq)_{\rm
sing}$ and $\FS{S}{\atp}{m}(\vq)_{\rm reg}$, respectively, by
setting
\begin{gather}
     \FS{S}{\atp}{m}(\vq)_{\rm sing}
     =\sum_{n\geq 0}\fc{S}{\atp}(m,-n-1)\vq^{-n-1},\quad
     \FS{S}{\atp}{m}(\vq)_{\rm reg}
     =\sum_{n\geq 0}\fc{S}{\atp}(m,n)\vq^n.
\end{gather}
We define $\FS{S}{\atp}{m}(\vq)_{\rm van}$ to be the series obtained
by removing the constant term from $\FS{S}{\atp}{m}(\vq)_{\rm reg}$.
\begin{gather}
     \FS{S}{\atp}{m}(\vq)_{\rm van}
     =\sum_{n>0}\fc{S}{\atp}(m,n)\vq^n
\end{gather}
Given the formal series $\tilde{F}_{S}^{\atp}(\vp,\vq)$ for some $S$
and $\atp$, one may ask if the coefficient series
$\FS{S}{\atp}{m}(\vq)$ define holomorphic functions on $\HH$ upon
the substitution of $\ee(\zz)$ for $\vq$. Supposing this to be the
case, we define $\FR{S}{\atp}{m}(\zz)$ by setting
\begin{gather}
     \FR{S}{\atp}{m}(\zz)
     =\FS{S}{\atp}{m}(\ee(\zz))
     =\sum_{n\in\ZZ}\fc{S}{\atp}(m,n)\ee(n\zz),
\end{gather}
and we define $\FR{S}{\atp}{m}(\zz)_{\rm sing}$,
$\FR{S}{\atp}{m}(\zz)_{\rm reg}$ and $\FR{S}{\atp}{m}(\zz)_{\rm
van}$ in the analogous way.

We write $\FR{\Gamma,\cp|\cq}{\atp}{m}(\zz)$ for
$\FR{S}{\atp}{m}(\zz)$ in the case that
$S=\lBZ\Sigma_{\cp}^{-1}\Sigma_{\cq}\rBZ^{\times}$ for $\Gamma$ a
group commensurable with $G(\ZZ)$ and for
$\Sigma_{\cp},\Sigma_{\cq}\in \Gamma\backslash G(\QQ)$ scaling
cosets for $\Gamma$ at cusps $\cp,\cq\in\cP_{\Gamma}$. We write $\FR{S}{}{m}(\zz)$ as a shorthand for $\FR{S}{0}{m}(\zz)$, and
interpret the notation $\FR{\Gamma,\cp|\cq}{}{m}(\zz)$, \&c.,
similarly.

The regular part of the formal series $\FS{S}{\atp}{m}(\vq)$ indeed
defines a holomorphic function on $\HH$, for suitable
$\atp,m\in\ZZ$, in the case that $S$ is of the form
$\lBZ\Sigma_{\cp}^{-1}\Sigma_{\cq}\rBZ^{\times}$.
\begin{prop}\label{prop:radsum:constr:Conver_FR_van}
Let $\Gamma$ be a group commensurable with $G(\ZZ)$, let $\cp,\cq\in \cP_{\Gamma}$ be cusps of $\Gamma$ and let $\atp,m\in \ZZ$. Then the series
\begin{gather}\label{eqn:radsum:constr:Conver_FR_van}
     \FR{\Gamma,\cp|\cq}{\atp}{m}(\zz)_{\rm van}
     =
     \sum_{n>0}
     \fc{\Gamma,\cp|\cq}{\atp}(m,n)
     \ee(n\zz)
\end{gather}
converges absolutely and locally uniformly for $\zz\in \HH$. In particular,
$\FR{\Gamma,\cp|\cq}{\atp}{m}(\zz)_{\rm van}$ is a holomorphic
function on $\HH$ for any $\atp,m\in\ZZ$.
\end{prop}
\begin{proof}
For $m\neq 0$ the claim follows directly from the estimate (\ref{eqn:radsum:coeff:fcS_est}). If $m=0$ and $\atp\leq 0$ then $\FR{\Gamma,\cp|\cq}{\atp}{m}(\zz)_{\rm van}$ vanishes identically according to (\ref{eqn:radsum:coeff:Pow_Ser_fcS_1-s_pos}). For $m=0$ and $\atp>0$ the required convergence of (\ref{eqn:radsum:constr:Conver_FR_van}) follows from (\ref{eqn:radsum:coeff:fcS_conv_0}) since $\Kl_{S}^{\atp}(0,n)$ has at most polynomial growth in $n$ according to the proof of Proposition \ref{prop:radsum:coeff:SKzeta_conv}.
\end{proof}
By a similar argument to that just given we
see that the singular part of the formal series
$\FS{S}{\atp}{m}(\vq)$ defines an anti-holomorphic function on $\HH$
when $S$ is of the form
$\lBZ\Sigma_{\cp}^{-1}\Sigma_{\cq}\rBZ^{\times}$.
\begin{prop}\label{prop:radsum:constr:Conver_FR_sing}
Let $\Gamma$ be a group commensurable with $G(\ZZ)$, let
$\cp,\cq\in \cP_{\Gamma}$ be cusps of $\Gamma$ and let $\atp,m\in \ZZ$.
%such that either $\atp>0$ and $m\geq 0$ or $\atp\leq 0$ and $m<0$. 
Then the series
\begin{gather}
     \FR{\Gamma,\cp|\cq}{\atp}{m}(\bar{\zz})_{\rm sing}
     =
     \sum_{n>0}
     \fc{\Gamma,\cp|\cq}{\atp}(m,-n)
     \ee(-n\bar{\zz})
\end{gather}
converges absolutely and locally uniformly for $\zz\in \HH$. In particular,
$\FR{\Gamma,\cp|\cq}{\atp}{m}(\bar{\zz})_{\rm sing}$ is an
anti-holomorphic function on $\HH$ for any $\atp,m\in\ZZ$.
\end{prop}
The notation $\FR{\Gamma,\cp|\cq}{\atp}{m}$ suppresses the
dependence upon the choice of scaling cosets $\Sigma_{\cp}$ and
$\Sigma_{\cq}$. The next result encodes this dependence precisely.
\begin{prop}
Let $\Gamma$ be a group commensurable with $G(\ZZ)$, and let
$\cp,\cq\in \cP_{\Gamma}$. Let $\Sigma_{\cp}$ and $\Sigma_{\cp}'$ be
scaling cosets for $\Gamma$ at $\cp$ and let $\Sigma_{\cq}$ and
$\Sigma_{\cq}'$ be scaling cosets for $\Gamma$ at $\cq$. Set
$S=\lBZ\Sigma_{\cp}^{-1}\Sigma_{\cq}\rBZ^{\times}$ and
$S'=\lBZ(\Sigma_{\cp}')^{-1}\Sigma_{\cq}'\rBZ^{\times}$. Then we
have $\Sigma_{\cp}'=\Sigma_{\cp}T^{\alpha}$ and
$\Sigma_{\cq}'=\Sigma_{\cq}T^{\beta}$ and
$\FR{S'}{\atp}{m}(\zz)=\ee(m\alpha)\FR{S}{\atp}{m}(\zz+\beta)$ for
some $\alpha,\beta\in \QQ$.
\end{prop}

\subsection{Convergence}\label{sec:radsum:conver}

Our main goal in this section is to establish the convergence, and
Fourier series expansion, of the Rademacher sum
$\RS{\Gamma,\cp|\cq}{\atp}{m}(\zz)$ (cf. \S\ref{sec:radsum:constr}),
attached to a group $\Gamma$ commensurable with $G(\ZZ)$, and cusps
$\cp$ and $\cq$ for $\Gamma$. We will require mostly minor
modifications of the arguments furnished by Niebur in
\cite{Nie_ConstAutInts}. An important technical tool in these
arguments is the identity
\begin{gather}\label{eqn:radsum:conver:Lipschitz_spp>1}
     \sum_{n>0}
     n^{(\spp-1)}
     \ee(n\zz)
     =
     \sum_{n\in\ZZ}
     (-\tpi)^{-\spp}
     (\zz+n)^{-\spp},
\end{gather}
valid for $\Re(\spp)>1$, in which both sums converge absolutely and uniformly in $z$ on compact subsets of $\HH$. A nice proof of (\ref{eqn:radsum:conver:Lipschitz_spp>1}) appears in \cite{KnoRob_RieFnlEqnLipSum}. This identity can be extended to $\spp=1$ at the expense of absolute convergence. 
\begin{lem}\label{lem:radsum:conver:Refinement_Lipschitz}
Let $\zz\in \HH$ and let $K\in\ZZp$ such that $K+1/2>\Re(\zz)$. Then we have
\begin{gather}\label{eqn:radsum:conver:Refinement_Lipschitz}
     \sum_{n>0}
     \ee(n\zz)
     =
     \sum_{\substack{n\in\ZZ\\-K\leq n\leq K}}
     (-\tpi)^{-1}
     (\zz+n)^{-1}
     +\lambda_K(\zz)
\end{gather}
where the error term $\lambda_K(\zz)$ is given by
\begin{gather}\label{eqn:radsum:conver:Refinement_Lipschitz_error}
     \lambda_K(\zz)
	=\frac{1}{2\pi}\int_{-\infty}^{\infty}
	\left(
	\frac{1}{\zz+\ii v-K-1/2}-\frac{1}{\zz+\ii v+K+1/2}
	\right)
	\frac{1}{1+e^{-2\pi v}}
	{\rm d}v.
\end{gather}
\end{lem}
\begin{proof}
The proof follows that of (the more general) Lemma 4.1 in \cite{Nie_ConstAutInts}. Define a function $f(\xi)$ by setting
$f(\xi)=(\tpi\xi)^{-1}(\ee(\xi-\zz)-1)^{-1}$. Then $f(\xi)$ has
a pole at $\xi=0$ and a pole also at $\xi=\zz+n$ for each $n\in \ZZ$. For the residue of $f(\xi){\rm d}\xi$ at $\xi=0$ we have
\begin{gather}
	\Res_{\xi=0}f(\xi){\rm d}\xi
	=
	\frac{1}{\tpi}\frac{1}{\ee(-\zz)-1}
	=
	\frac{1}{\tpi}\sum_{n>0}\ee(n\zz),
\end{gather}
which is $1/\tpi$ times the left hand side of (\ref{eqn:radsum:conver:Refinement_Lipschitz}), while the residue at $\xi=\zz+n$ is $(\tpi)^{-2}(\zz+n)^{-1}$. Let $K$ be a positive integer such that $K+1/2>\Re(\zz)$ and set $Q=K+1/2$ for convenience. Let $\gamma $ be the positively oriented boundary of the rectangle with corners $\Re(\zz)\pm Q\pm \ii L$ where $L>\Im(\zz)$. Then by the residue theorem we have
\begin{gather}	
	\int_{\gamma}f(\xi){\rm d}\xi
	=
	\sum_{n>0}\ee(n\zz)
	+\sum_{-K\leq n\leq K}(\tpi)^{-1}(\zz+n)^{-1}.
\end{gather}
On the other hand, the contributions of the horizontal edges of $\gamma$ to $\int_{\gamma}f(\xi){\rm d}\xi$ tend to $0$ as $L\to\infty$ since $f(\xi)$ decays exponentially as $\Im(\xi)\to\infty$. The residues are independent of $L$ for $L$ sufficiently large so taking the limit as $L\to\infty$ we obtain
\begin{gather}\label{eqn:radsum:conver:Refinement_Lipschitz_1}
	\sum_{n>0}\ee(n\zz)
	+\sum_{-K\leq n\leq K}(\tpi)^{-1}(\zz+n)^{-1}
	=
	\lambda_K(\zz)
%	\ii\int_{-\infty}^{\infty}
%	(f(x+Q+\ii v)-f(x-Q+\ii v)){\rm d}v
\end{gather}
where $\lambda_K(\zz)$ is the limit as $L\to\infty$ of the contributions of the vertical edges of $\gamma$. Setting $x=\Re(\zz)$ we have
\begin{gather}
%\begin{split}
	\lambda_K(\zz)
	%&
	=
	\ii\int_{-\infty}^{\infty}
	(f(x+Q+\ii v)-f(x-Q+\ii v)){\rm d}v.%\\
%\end{split}
\end{gather}
After replacing $v$ with $v+\Im(\zz)$ and observing that $\ee(Q)=\ee(-Q)=-1$ since $K$ is an integer we obtain the expression
\begin{gather}
	\lambda_K(\zz)
	=\frac{1}{2\pi}\int_{-\infty}^{\infty}
	\left(
	\frac{1}{\zz+\ii v-Q}-\frac{1}{\zz+\ii v+Q}
	\right)
	\frac{1}{1+\ee(\ii v)}
	{\rm d}v
\end{gather}
and combining this with (\ref{eqn:radsum:conver:Refinement_Lipschitz_1}) recovers the required identity.
\end{proof}

Setting $Q=K+1/2$ and substituting $Qv$ for $v$ in (\ref{eqn:radsum:conver:Refinement_Lipschitz_error}) we find that
\begin{gather}
\begin{split}
	\lambda_K(\zz)
	&=
	\frac{1}{2\pi}\int_{-\infty}^{\infty}
	\left(
	\frac{1}{\zz/Q+\ii v-1}-\frac{1}{\zz/Q+\ii v+1}
	\right)
	\frac{1}{1+\ee(\ii Q v)}
	{\rm d}v\\
	&=
	\frac{1}{\pi}\int_{-\infty}^{\infty}
	\frac{1}{((\zz/Q+\ii v)^2-1)}
	\frac{1}{1+e^{-2\pi Q v}}
	{\rm d}v
\end{split}
\end{gather}
which demonstrates that $\lambda_K(\zz)$ tends to $-1/\pi$ times $\int_0^{\infty}(1+v^2)^{-1}{\rm d}v=\pi/2$ as $K\to\infty$. We thus obtain
\begin{gather}\label{eqn:radsum:conver:Lipschitz_spp=1}
     \sum_{n>0}
     \ee(n\zz)
     =
     -
     \frac{1}{2}
     +
     \lim_{K\to \infty}
     \sum_{%\substack{n\in\ZZ\\
     -K<n<K}%}
     (-\tpi)^{-1}
     (\zz+n)^{-1}.
\end{gather}
The formulas (\ref{eqn:radsum:conver:Lipschitz_spp>1}) and
(\ref{eqn:radsum:conver:Lipschitz_spp=1}) are collectively known as
the {\em Lipschitz summation formula}.

\begin{thm}\label{thm:radsum:conver:Relate_RS_FR}
Let $\Gamma$ be a group commensurable with $G(\ZZ)$, and let
$\cp,\cq\in \cP_{\Gamma}$ be cusps of $\Gamma$. Let $\atp,m\in \ZZ$
such that $\atp\leq 0$ and $m<0$. Then the limit
defining the classical Rademacher sum
$\RS{\Gamma,\cp|\cq}{\atp}{m}(\zz)$ converges locally uniformly to a holomorphic
function on $\HH$, and we have
\begin{gather}\label{eqn:radsum:conver:Relate_RS_FR}
     \RS{\Gamma,\cp|\cq}{\atp}{m}(\zz)
     =
     \delta_{\Gamma,\cp|\cq}\ee(m\zz)
     +
     \frac{1}{2}
     \fc{\Gamma,\cp|\cq}{\atp}(m,0)
     +
     \FR{\Gamma,\cp|\cq}{\atp}{m}(\zz)_{\rm van}.
\end{gather}
\end{thm}
\begin{proof}
Set $U=\lBZ\Sigma_{\cp}^{-1}\Sigma_{\cq}\rBZh$ and $S=\lBZ\Sigma_{\cp}^{-1}\Sigma_{\cq}\rBZ^{\times}$. For $\atp<0$ the result follows from Lemmas 4.2 and 4.3 of \cite{Nie_ConstAutInts}. Indeed, using Lemma
\ref{lem:radsum:constr:Relate_RS_chi_m_to_IO_w_atp} we may identify
the Rademacher component function $\RS{\lBZ\chi\rBZh}{\atp}{m}(\zz)$
with the function $s_{\lBZ\chi\rBZh}(\zz)+t_{\lBZ\chi\rBZh}(\zz)$
defined on page 376 of \cite{Nie_ConstAutInts}. Then the expression
\begin{gather}\label{eqn:radsum:conver:RS_lim_dbl_sum_to_sum_lim_sum}
     %\RS{\Gamma,\cp|\cq}{\atp}{m}(\zz)
     %=
     %\RS{U}{\atp}{m}(\zz)
     %=
     \sum_{c\in\ZZp}
     \lim_{K\to \infty}
     \sum_{\substack{
          \lBZ\chi\rBZh\in U_{\leq K}\\
          c(\chi)=c
          }}
     \RS{\lBZ\chi\rBZh}{\atp}{m}(\zz)
\end{gather}
is, up to addition by a certain constant function, the definition of {Rademacher sum}
used by Niebur in loc. cit. The Lipschitz summation formula(s) (\ref{eqn:radsum:conver:Lipschitz_spp>1}) and (\ref{eqn:radsum:conver:Lipschitz_spp=1}) are the main tools used in equating (\ref{eqn:radsum:conver:RS_lim_dbl_sum_to_sum_lim_sum}) with the right hand side of (\ref{eqn:radsum:conver:Relate_RS_FR}).

In case $\atp=0$ we require to reinforce the argument of \cite{Nie_ConstAutInts} with a non-trivial estimate for the Selberg--Kloosterman zeta function $\Kl_S^{\spp}(m,n)$ at $\spp=1$; we will apply the result (\ref{eqn:radsum:coeff:GSest}) of Goldfeld--Sarnark to this purpose. Assume then that $\atp=0$ and observe that we have
\begin{gather}
\begin{split}
	&\RS{\Gamma,\cp|\cq}{\atp}{m}(\zz)
	=
	\delta_{\Gamma,\cp|\cq}
	\ee(m\zz)
	+
	\lim_{K\to \infty}
	\sum_{\lBZ\chi\rBZh\in U_{\leq K}^{\times}}
	\ee(m\lBZ\chi\rBZh\cdot\infty)
	\ee(m\lBZ\chi\rBZh\cdot\zz-m\lBZ\chi\rBZh\cdot\infty,1)\\
	&=
	\delta_{\Gamma,\cp|\cq}
	\ee(m\zz)
	+
	\lim_{K\to \infty}
	\sum_{\lBZ\chi\rBZh\in U_{\leq K}^{\times}}
	\ee(m\lBZ\chi\rBZh\cdot\infty)
	\sum_{k\geq 0}
	(-\tpi\rads\lBZ\chi\rBZ)^{k+1}m^{(k+1)}(\zz-\lBZh\chi^{-1}\rBZ\cdot\infty)^{-k-1}
\end{split}	
\end{gather}
according to (\ref{eqn:radsum:constr:RS_chi_using_gen_exp}), the definition (\ref{eqn:conven:fns:Genzd_Exp}) of the generalized exponential $\ee(\zz,\spp)$, and the identity (\ref{eqn:conven:cosets_UsefulChiDotTau}). We write $\RS{\Gamma,\cp|\cq}{\atp}{m}(\zz)=\ee(m\zz)+R_0+R_+$ where
\begin{gather}
	\label{eqn:radsum:conver:Defn_R0}
	R_0=	
	\lim_{K\to \infty}
	\sum_{\lBZ\chi\rBZh\in U_{\leq K}^{\times}}
	\ee(m\lBZ\chi\rBZh\cdot\infty)
	(-\tpi\rads\lBZ\chi\rBZ)m(\zz-\lBZh\chi^{-1}\rBZ\cdot\infty)^{-1},\\
	\label{eqn:radsum:conver:Defn_Rp}
	R_+=
	\sum_{\lBZ\chi\rBZh\in U^{\times}}
	\ee(m\lBZ\chi\rBZh\cdot\infty)
	\sum_{k> 0}
	(-\tpi\rads\lBZ\chi\rBZ)^{k+1}m^{(k+1)}(\zz-\lBZh\chi^{-1}\rBZ\cdot\infty)^{-k-1},
\end{gather}
and the second sum $R_+$ is absolutely convergent, locally uniformly for $\zz\in\HH$. After choosing a representative $\chi$ for each double coset $\lBZ\chi\rBZ$ in $S$ we may rewrite $R_+$ as
\begin{gather}
	R_+=\sum_{\lBZ\chi\rBZ\in S}
	\ee(m\lBZ\chi\rBZ\cdot\infty)
	\sum_{k>0}
	(-\tpi\rads\lBZ\chi\rBZ)^{k+1}m^{(k+1)}
	\sum_{n\in\ZZ}
	(\zz-\lBZh\chi^{-1}\rBZ\cdot\infty+n)^{-k-1},
\end{gather}
and applying the Lipschitz summation formula (\ref{eqn:radsum:conver:Lipschitz_spp>1}) to this we obtain
\begin{gather}\label{eqn:radsum:conver:Rconv_1-atp_pos_R+_01}
	\begin{split}
	R_+
	&=\sum_{n>0}%\Kl_S^0(m,n)
	\sum_{\lBZ\chi\rBZ\in S}\ee(m\lBZ\chi\rBZ\cdot\infty)\ee(-n\lBZ\chi^{-1}\rBZ\cdot\infty)
	\sum_{k>0}(-4\pi^2\rads\lBZ\chi\rBZ)^{k+1}m^{(k+1)}n^{(k)}\ee(n\zz)\\
	&=\sum_{n>0}
	\fc{S}{\atp}(m,n)_+\ee(n\zz)
	\end{split}
\end{gather}
where $\fc{S}{\atp}(m,n)_+$ satisfies $\fc{S}{\atp}(m,n)=(-4\pi^2)m\Kl_S^1(m,n)+\fc{S}{\atp}(m,n)_+$ according to (\ref{eqn:radsum:coeff:Pow_Ser_fcS_1-s_pos}). For $K$ a positive integer set 
\begin{gather}
	R_0(K)=(-\tpi)m\sum_{U_{\leq K}^{\times}}\ee(m\lBZ\chi\rBZh\cdot\infty)\rads\lBZ\chi\rBZ(\zz-\lBZh\chi^{-1}\rBZ\cdot\infty)^{-1}
\end{gather}
so that $R_0=\lim_{K\to \infty}R_0(K)$. Then choosing as above a representative $\chi$ for each $\lBZ\chi\rBZ\in S$ we have
\begin{gather}\label{eqn:radsum:conver:Rconv_1-atp_pos_R0_01}
	\begin{split}
	R_0(K)
	&=(-\tpi)m
	%\sum_{0<c<K}
	\sum_{\lBZ\chi\rBZ\in S_{\leq K}}\ee(m\lBZ\chi\rBZ\cdot\infty)\rads\lBZ\chi\rBZ
	\sum_{\substack{n\in\ZZ\\-K^2\leq d(\chi)+nc(\chi)\leq K^2}}(\zz-\lBZh\chi^{-1}\rBZ\cdot\infty+n)^{-1}\\
	&=(-\tpi)m
	\sum_{0<c<K}
	\sum_{\lBZ\chi\rBZ\in S_c}
	\ee(m\lBZ\chi\rBZ\cdot\infty)\rads\lBZ\chi\rBZ
	\left(\sum_{%\substack{n\in\ZZ\\ 
			|n|\leq K^2/c%}
			}(\zz-\lBZh\chi^{-1}\rBZ\cdot\infty+n)^{-1}+O(c/K^2)\right)
	\end{split}
\end{gather}
where the term $O(c/K^2)$ accounts for the discrepency between the two summations over $n$. Applying the Lipschitz summation formula (\ref{eqn:radsum:conver:Refinement_Lipschitz_1}) to the second line of (\ref{eqn:radsum:conver:Rconv_1-atp_pos_R0_01}) we obtain
\begin{gather}\label{eqn:radsum:conver:Rconv_1-atp_pos_R0_02}
	R_0(K)=(-4\pi^2)m
	\sum_{0<c<K}
	\sum_{\lBZ\chi\rBZ\in S_{c}}\ee(m\lBZ\chi\rBZ\cdot\infty)\rads\lBZ\chi\rBZ
	\left(\frac{1}{2}+\sum_{n>0}\ee(-n\lBZ\chi^{-1}\rBZ\cdot\infty)\ee(n\zz)+O(c/K^2)\right)
\end{gather}
since the error term $\lambda_X(\zz)$ appearing in (\ref{eqn:radsum:conver:Refinement_Lipschitz_1}) satisfies $\lambda_X(\zz)+1/2=O(1/X)$. We conclude from this that 
\begin{gather}\label{eqn:radsum:conver:Rconv_1-atp_pos_R0_03}
	R_0
	%=\lim_{K\to \infty}R_0(K)
	=
	\lim_{K\to \infty}
	\sum_{\lBZ\chi\rBZ\in S_{\leq K}}\ee(m\lBZ\chi\rBZ\cdot\infty)	
	(-\fps)\rads\lBZ\chi\rBZ
	m
	\left(\frac{1}{2}+\sum_{n>0}\ee(-n\lBZ\chi^{-1}\rBZ\cdot\infty)\ee(n\zz)\right).
\end{gather}
We claim now that $R_0=\fc{S}{\atp}(m,0)/2+\sum_{n>0}\fc{S}{\atp}(m,n)_0\ee(n\zz)$ where $\fc{S}{\atp}(m,n)_0=\fc{S}{\atp}(m,n)-\fc{S}{\atp}(m,n)_+=(-4\pi^2)m\Kl_S^1(m,n)$. To see this observe that for fixed positive $K$ we have
\begin{gather}\label{eqn:radsum:conver:Rconv_1-atp_pos_R0_04}
	\sum_{\lBZ\chi\rBZ\in S_{\leq K}}\ee(m\lBZ\chi\rBZ\cdot\infty)	
	\rads\lBZ\chi\rBZ
	\sum_{n>0}\ee(-n\lBZ\chi^{-1}\rBZ\cdot\infty)\ee(n\zz)
	=
	\sum_{n>0}
	\Kl_{S_{\leq K}}^1(m,n)
%	\sum_{0<c\leq K}
%	\Kl_{S_c}^1(m,n)
%	\sum_{\lBZ\chi\rBZ\in S_{\leq K}}
%	\ee(m\lBZ\chi\rBZ\cdot\infty)	
%	\ee(-n\lBZ\chi^{-1}\rBZ\cdot\infty)
%	\rads\lBZ\chi\rBZ
	\ee(n\zz)
\end{gather}
since the left hand side converges absolutely. Now $\Kl_{S_{\leq K}}^1(m,n)=\Kl_S^1(m,n)+O(n)$ as $K\to \infty$ according to (\ref{eqn:radsum:coeff:Sigma_to_zeta}) since the $\Sigma(x)$ there is just $\Kl_{S_{\leq x}}^1(m,n)$ (cf. (\ref{eqn:radsum:coeff:defn_Sigma(x)})). Then (\ref{eqn:radsum:coeff:GSest}) implies $\Kl_{S_{\leq K}}^1(m,n)=O(n)$ uniformly in $K$ and $n$ and we conclude that the limit as $K\to\infty$ of the right hand side of (\ref{eqn:radsum:conver:Rconv_1-atp_pos_R0_04}) is $\sum_{n>0}\Kl_S^1(m,n)\ee(n\zz)$, and thus 
\begin{gather}\label{eqn:radsum:conver:Rconv_1-atp_pos_R0_05}
	R_0=\fc{S}{\atp}(m,0)/2+
	\sum_{n>0}
	\fc{S}{\atp}(m,n)_0
	%(-\fps)m\Kl_S^1(m,n)
	\ee(n\zz)
\end{gather} 
as was claimed. Taking (\ref{eqn:radsum:conver:Rconv_1-atp_pos_R0_05}) together with (\ref{eqn:radsum:conver:Rconv_1-atp_pos_R+_01}) we arrive at the required identity (\ref{eqn:radsum:conver:Relate_RS_FR}) in case $\atp=0$. The proof is complete.
\end{proof}
The next result identifies the Fourier series expansion of the
conjugate Rademacher sum $\CS{\Gamma,\cp|\cq}{\atp}{m}(\zz)$. As
such it is a natural counterpart to Theorem
\ref{thm:radsum:conver:Relate_RS_FR}. The method of proof is
directly analogous to that of Theorem \ref{thm:radsum:conver:Relate_RS_FR}.
\begin{thm}\label{thm:radsum:conver:Relate_CS_FR}
Let $\Gamma$ be a group commensurable with $G(\ZZ)$, and let
$\cp,\cq\in \cP_{\Gamma}$ be cusps of $\Gamma$. Let $\atp,m\in \ZZ$
such that $\atp\leq 0$ and $m<0$. Then the limit
defining the conjugate Rademacher sum
$\CS{\Gamma,\cp|\cq}{\atp}{m}(\zz)$ converges locally uniformly to an anti-holomorphic
function on $\HH$, and we have
\begin{gather}
     \CS{\Gamma,\cp|\cq}{\atp}{m}(\zz)
     =
     \delta_{\Gamma,\cp|\cq}\ee(m\bar{\zz})
     -
     \frac{1}{2}
     \fc{\Gamma,\cp|\cq}{\atp}(m,0)
     -
     \FR{\Gamma,\cp|\cq}{\atp}{m}(\bar{\zz})_{\rm sing}.
\end{gather}
\end{thm}
Theorems \ref{thm:radsum:conver:Relate_RS_FR} and
\ref{thm:radsum:conver:Relate_CS_FR} furnish the Fourier expansions
of the classical and conjugate Rademacher sums of non-positive
weight. Directly analogous methods can be used to determine
expressions for the Fourier coefficients of the Rademacher sums of
positive weight, which are, after all, just the holomorphic
Poincar\'e series (cf. \S\ref{sec:radsum:constr}). Indeed, there are
fewer technical difficulties in the case of positive weight, since
the sum appearing in the definition
(\ref{eqn:radsum:constr:Defn_PS_U_m}) of the Poincar\'e series
$\PS{\Gamma,\cp|\cq}{\atp}{m}$ is absolutely and locally uniformly convergent (at least
when $\atp>1$). For completeness we state an analogue of Theorem
\ref{thm:radsum:conver:Relate_RS_FR} for $\atp>0$. The result
is standard (cf. \cite{Iwa_SpecMetAutFrms}).
\begin{thm}\label{thm:radsum:conver:Relate_PS_FR_atp>0}
Let $\Gamma$ be a group commensurable with $G(\ZZ)$, and let
$\cp,\cq\in \cP_{\Gamma}$ be cusps of $\Gamma$. Let $\atp,m\in\ZZ$ such that $\atp>0$ and $m\geq 0$. Then for $\zz\in \HH$ we
have
\begin{gather}\label{eqn:radsum:conver:Relate_PS_FR_atp>0}
     \RS{\Gamma,\cp|\cq}{\atp}{m}(\zz)
     =
     \delta_{\Gamma,\cp|\cq}\ee(m\zz)
     +
     \FR{\Gamma,\cp|\cq}{\atp}{m}(\zz)_{\rm van}.
\end{gather}
\end{thm}

Recall from \S\ref{sec:radsum:constr} that we write
$\RS{\Gamma,\cp|\cq}{}{m}(\zz)$ for
$\RS{\Gamma,\cp|\cq}{0}{m}(\zz)$ and interpret
$\CS{\Gamma,\cp|\cq}{}{m}(\zz)$ similarly. To conclude this section we
consider the difference
$\RS{\Gamma,\cp|\cq}{}{m}(\zz)-\CS{\Gamma,\cp|\cq}{}{m}(\zz)$.
Combining Theorems \ref{thm:radsum:conver:Relate_RS_FR} and
\ref{thm:radsum:conver:Relate_CS_FR} we obtain the following
identification of the Fourier expansion of this function in terms of
the exponential function $\ee(\zz)$ and its conjugate
$\ee(-\bar{\zz})$.
\begin{thm}\label{thm:radsum:var:Relate_RS-CS_FR_atp=0}
Let $\Gamma$ be a group commensurable with $G(\ZZ)$, let $\cp,\cq\in
\cP_{\Gamma}$ be cusps of $\Gamma$, and let $m<0$. Then we
have
\begin{gather}\label{eqn:radsum:var:Relate_RS-CS_FR_atp=0}
     \begin{split}
     \RS{\Gamma,\cp|\cq}{}{m}(\zz)
     -
     \CS{\Gamma,\cp|\cq}{}{m}(\zz)
     &=
     \delta_{\Gamma,\cp|\cq}
     (\ee(m\zz)-\ee(m\bar{\zz}))\\
     &\qquad+
     \FR{\Gamma,\cp|\cq}{}{m}(\bar{\zz})_{\rm sing}
     +
     \fc{\Gamma,\cp|\cq}{}(m,0)
     +
     \FR{\Gamma,\cp|\cq}{}{m}(\zz)_{\rm van}
     \end{split}
\end{gather}
for $\zz\in \HH$.
\end{thm}

\subsection{Variance}\label{sec:radsum:var}

Recall from \S\ref{sec:radsum:constr} that we write
$\RS{\Gamma,\cp}{\atp}{m}(\zz)$ for
$\RS{\Gamma,\cp|\cq}{\atp}{m}(\zz)$ in case $\cq$ is the infinite
cusp $\Gamma\cdot\infty$. Suppose that $\Gamma$ is commensurable
with the modular group and has width one at infinity (cf.
\S\ref{sec:conven:scaling}). We will establish in this section that
the Rademacher sum $\RS{\Gamma,\cp}{\atp}{m}(\zz)$, once corrected
by the addition of a certain constant function, is an automorphic integral of weight
$2\atp$ for $\Gamma$ whenever $\atp=1$ and $m>0$ or $\atp>1$ and $m\geq 0$ or $\atp\leq 0$ and $m<0$. This is quite general since, according to the discussion of
\S\ref{sec:radsum:constr}, any Rademacher sum
$\RS{\Gamma,\cp|\cq}{\atp}{m}(\zz)$ can be expressed in the form
$\RS{\Gamma',\cp'}{\atp}{m}(\zz)$ for some group $\Gamma'$ having
width one at infinity (cf.
(\ref{eqn:radsum:constr:RS_twocusps_to_onecusp})).

Let us begin by considering the weight $2\atp$ action of $G(\QQ)$ on
the Rademacher component function $\RS{\lBZ\chi\rBZh}{\atp}{m}(\zz)$
of (\ref{eqn:radsum:constr:Defn_RS_chi}).
\begin{lem}\label{lem:radsum:var:Apply_sigma_RS_chi}
Let $\lBZ\chi\rBZh\in \lBZ G(\QQ)\rBZh^{\times}$ and $\sigma\in
G(\QQ)$, and let $\atp,m\in\ZZ$. In case $\atp>0$ and $m\geq 0$ we have
\begin{gather}\label{eqn:radsum:var:Apply_sigma_RS_chi_atp>0}
     \left(
     \left.
     \RS{\lBZ\chi\rBZh}{\atp}{m}
     \right\sop{\atp}
     \sigma
     \right)
     (\zz)
     =
     \RS{\lBZ\chi\sigma\rBZh}{\atp}{m}(\zz),
\end{gather}
and in case $\atp\leq 0$ and $m<0$ we have
\begin{gather}\label{eqn:radsum:var:Apply_sigma_RS_chi_atp<1}
     \left(
     \left.
     \RS{\lBZ\chi\rBZh}{\atp}{m}
     \right\sop{\atp}
     \sigma
     \right)
     (\zz)
     =
     \RS{\lBZ\chi\sigma\rBZh}{\atp}{m}(\zz)
     +
     m^{1-2\atp}
     \left(
     \JO{\infty\cdot\lBZ\sigma\rBZh}{\atp}
     \PS{\lBZ\chi\sigma\rBZh}{1-\atp}{-m}
     \right)
     (\zz).
\end{gather}
\end{lem}
\begin{proof}
Note that we have
\begin{gather}\label{eqn:radsum:var:Apply_sigma_PS_chi}
     \left(
     \PS{\lBZ\chi\rBZh}{\atp}{m}\sop{\atp}\sigma
     \right)(\zz)
     =
     \PS{\lBZ\chi\sigma\rBZh}{\atp}{m}(\zz)
\end{gather}
for all $\atp,m\in\ZZ$, where
$\PS{\lBZ\chi\rBZh}{\atp}{m}(\zz)=\ee(m\lBZ\chi\rBZh\cdot\zz)\jac(\lBZ\chi\rBZh,\zz)^{\atp}$
(cf. (\ref{eqn:radsum:constr:Defn_PS_chi})), so the identity
(\ref{eqn:radsum:var:Apply_sigma_RS_chi_atp>0}) follows from
(\ref{eqn:radsum:var:Apply_sigma_PS_chi}) and the fact that
$\RS{\lBZ\chi\rBZh}{\atp}{m}(\zz)=\PS{\lBZ\chi\rBZh}{\atp}{m}(\zz)$
in case $\atp>0$, by the definition of
$\RS{\lBZ\chi\rBZh}{\atp}{m}(\zz)$ (cf. \S\ref{sec:radsum:constr}).
For (\ref{eqn:radsum:var:Apply_sigma_RS_chi_atp<1}) we suppose
$\atp\leq 0$ and $m<0$ and use Lemma
\ref{lem:radsum:constr:Relate_RS_chi_m_to_IO_w_atp} to write the
Rademacher component function $\RS{\lBZ\chi\rBZh}{\atp}{m}(\zz)$ in
terms of the functions $\PS{\lBZ\chi\rBZh}{\atp}{m}(\zz)$ and the
integral operator $\JO{\ww}{\atp}$. Then, employing
(\ref{eqn:radsum:var:Apply_sigma_PS_chi}) and the identity
(\ref{eqn:conven:autfrm:Var_IO_w_atp}), we compute
\begin{gather}\label{eqn:radsum:var:Apply_sigma_to_RS_atp_m}
     \begin{split}
     \left(
     \left.
     \RS{\lBZ\chi\rBZh}{\atp}{m}
     \right\sop{\atp}
     \sigma
     \right)
     (\zz)
     &=
     \left(
     \left.
     \PS{\lBZ\chi\rBZh}{\atp}{m}
     \right\sop{\atp}\sigma
     \right)
     (\zz)
     -
     m^{1-2\atp}
     \left(
     \left.\left(
     \JO{\infty\cdot\lBZ\chi\rBZh}{\atp}
     \PS{\lBZ\chi\rBZh}{1-\atp}{-m}
     \right)\right\sop{\atp}\sigma
     \right)
     (\zz)
     \\
     &=
     \PS{\lBZ\chi\sigma\rBZh}{\atp}{m}
     (\zz)
     -
     m^{1-2\atp}
     \left(
     \JO{\infty\cdot\lBZ\chi\sigma\rBZh}{\atp}
     \PS{\lBZ\chi\sigma\rBZh}{1-\atp}{-m}
     \right)
     (\zz)
     +
     m^{1-2\atp}
     \left(
     \JO{\infty\cdot\lBZ\sigma\rBZh}{\atp}
     \PS{\lBZ\chi\sigma\rBZh}{1-\atp}{-m}
     \right)
     (\zz)
     .
     \end{split}
\end{gather}
Applying Lemma \ref{lem:radsum:constr:Relate_RS_chi_m_to_IO_w_atp}
once again we recognize the first two terms on the right hand side
(of the second line) of
(\ref{eqn:radsum:var:Apply_sigma_to_RS_atp_m}) to be
$\RS{\lBZ\chi\sigma\rBZh}{\atp}{m}(\zz)$. We thus obtain the
required identity (\ref{eqn:radsum:var:Apply_sigma_RS_chi_atp<1}).
\end{proof}
As for technical tools, in addition to Lemma
\ref{lem:radsum:conver:Refinement_Lipschitz}, we also use the
following result which shows that there is some flexibility in the
convergence of the limit (\ref{eqn:radsum:constr:Defn_RS_U_m})
defining the classical Rademacher sum.
\begin{lem}\label{lem:radsum:var:RS_U_slash_sigma_vs_RS_Usigma}
Let $\Gamma$ be a group commensurable with $G(\ZZ)$, let $\cp\in
\cP_{\Gamma}$ be a cusp of $\Gamma$ and let $\sigma\in G(\QQ)$. Let
$\atp,m\in\ZZ$ such that either $\atp=1$ and $m> 0$ or $\atp>1$ and $m\geq 0$ or $\atp\leq 0$ and $m<0$. Then we have
\begin{gather}\label{eqn:radsum:var:U_K_sigma_vs_U_sigma_K}
     \lim_{K\to \infty}
     \sum_{\lBZ\chi\rBZh\in(U_{\leq K})\sigma}
     \RS{\lBZ\chi\rBZh}{\atp}{m}(\zz)
     =
     \lim_{K\to \infty}
     \sum_{\lBZ\chi\rBZh\in(U\sigma)_{\leq K}}
     \RS{\lBZ\chi\rBZh}{\atp}{m}(\zz)
\end{gather}
for $U=\lBZ\Sigma_{\cp}^{-1}\rBZh$ and $\zz\in\HH$.
\end{lem}
\begin{proof}
For $\atp>1$ and $m\geq 0$ both sums in
(\ref{eqn:radsum:var:U_K_sigma_vs_U_sigma_K}) are absolutely
convergent, and the identity certainly holds. In case $\atp=1$ and
$m> 0$, or $\atp\leq 0$ and $m<0$, the identity
(\ref{eqn:radsum:var:U_K_sigma_vs_U_sigma_K}) is established via the
argument of Lemma 4.3 in \cite{Nie_ConstAutInts}. Note that this argument fails in the case that $\atp=1$ and $m=0$, and indeed, the identity (\ref{eqn:radsum:var:U_K_sigma_vs_U_sigma_K}) does not hold in general when $\atp=1$ and $m=0$.
\end{proof}

\begin{thm}\label{thm:radsum:var:Var_RS_atp>0}
Let $\Gamma$ be a group commensurable with $G(\ZZ)$ that has width
one at infinity and let $\cp\in \cP_{\Gamma}$ be a cusp of $\Gamma$.
Let $\atp,m\in\ZZ$ such that $\atp>1$ and $m\geq 0$ or $\atp=1$ and $m>0$. Then
the function $\RS{\Gamma,\cp}{\atp}{m}(\zz)$ is a modular form of
weight $2\atp$ for $\Gamma$, and is a cusp form in case $m>0$. Further, for $\cq\in \cP_{\Gamma}$ another cusp of $\Gamma$ the Fourier expansion of the function $\RS{\Gamma,\cp|\cq}{\atp}{m}(\zz)$ is the expansion of $\RS{\Gamma,\cp}{\atp}{m}(\zz)$ at the cusp $\cq$.
\end{thm}
\begin{proof}
Let $\atp$ and $m$ be as in the statement of the proposition. Since
$\Gamma$ has width one at infinity we may take $\Gamma$ to be the
scaling coset for $\Gamma$ at the infinite cusp. Let $\Sigma_{\cp}$
be a scaling coset for $\Gamma$ at $\cp$ and set
$U=\lBZ\Sigma_{\cp}^{-1}\rBZh$. Then in the case that $\atp>1$ we
have
\begin{gather}\label{eqn:radsum:var:Expr_PS_U_atp>1}
     \RS{\Gamma,\cp}{\atp}{m}(\zz)
     =
     \sum_{\lBZ\chi\rBZh\in U}
     \RS{\lBZ\chi\rBZh}{\atp}{m}(\zz)
     =
     \sum_{\lBZ\chi\rBZh\in U}
     \ee(m\lBZ\chi\rBZh\cdot\zz)
     \jac(\lBZ\chi\rBZh,\zz)^{\atp},
\end{gather}
with the sum(s) in (\ref{eqn:radsum:var:Expr_PS_U_atp>1}) converging
absolutely for $\zz\in \HH$. Let $\sigma\in G(\QQ)$. Then we have
\begin{gather}\label{eqn:radsum:var:PS_Gamma_p_sigma}
     \left(
     \left.
     \RS{\Gamma,\cp}{\atp}{m}
     \right\sop{\atp}\sigma
     \right)(\zz)
     =
     \sum_{\lBZ\chi\rBZh\in U}
     \left(
     \left.
     \RS{\lBZ\chi\rBZh}{\atp}{m}
     \right\sop{\atp}\sigma
     \right)(\zz)
     =
     \sum_{\lBZ\chi\rBZh\in U}
     \RS{\lBZ\chi\sigma\rBZh}{\atp}{m}(\zz)
     =\sum_{\lBZ\chi\rBZh\in U\sigma}
     \RS{\lBZ\chi\rBZh}{\atp}{m}(\zz)
\end{gather}
by Lemma \ref{lem:radsum:var:Apply_sigma_RS_chi}, with all sums
absolutely convergent, locally uniformly in $\zz$. Taking $\sigma\in \Gamma$ we see that
$\RS{\Gamma,\cp}{\atp}{m}(\zz)$ is an unrestricted modular form of
weight $2\atp$ for $\Gamma$ since
$\Sigma_{\cp}^{-1}\sigma=\Sigma_{\cp}^{-1}$ for $\sigma\in \Gamma$.
Taking $\cq\in \cP_{\Gamma}$ and $\sigma\in \Sigma_{\cq}$ we see
that Theorem \ref{thm:radsum:conver:Relate_PS_FR_atp>0} gives the
Fourier expansion of $\RS{\Gamma,\cp}{\atp}{m}(\zz)$ at $\cq$, in
the sense of \S\ref{sec:conven:autfrm}, since
$\Sigma_{\cp}^{-1}\sigma=\Sigma_{\cp}^{-1}\Sigma_{\cq}$ for any
$\sigma\in\Sigma_{\cq}$. We conclude that the (a priori
unrestricted) modular form $\RS{\Gamma,\cp}{\atp}{m}(\zz)$ is
vanishing at every cusp $\cq\in\cP_{\Gamma}$, except possibly for
$\cq=\cp$, and vanishes even there just when $m>0$, and thus the
claim is verified for $\atp>1$.

Consider now the case that $\atp=1$. We have
\begin{gather}\label{eqn:radsum:var:Expr_PS_U_atp=1}
     \RS{\Gamma,\cp}{1}{m}(\zz)
     =
     \lim_{K\to \infty}
     \sum_{\lBZ\chi\rBZh\in U_{\leq K}}
     \RS{\lBZ\chi\rBZh}{1}{m}(\zz)
          =
     \lim_{K\to \infty}
     \sum_{\lBZ\chi\rBZh\in U_{\leq K}}
     \ee(m\lBZ\chi\rBZh\cdot\zz)
     \jac(\lBZ\chi\rBZh,\zz),
\end{gather}
and it is necessary to consider (some kind of) a limit since the convergence is not absolute in this case. Let $\sigma\in G(\QQ)$
and consider the weight $2$ action of $\sigma$ on
$\RS{\Gamma,\cp}{1}{m}(\zz)$.
\begin{gather}
     \left(
     \left.
     \RS{\Gamma,\cp}{1}{m}
     \right|_1\sigma
     \right)
     (\zz)
     =
     \lim_{K\to \infty}
     \sum_{\lBZ\chi\rBZh\in U_{\leq K}}
     \left(
     \left.
     \RS{\lBZ\chi\rBZh}{1}{m}
     \right|_1\sigma
     \right)
     (\zz)
     =
     \lim_{K\to \infty}
     \sum_{\lBZ\chi\rBZh\in U_{\leq K}}
     \RS{\lBZ\chi\sigma\rBZh}{1}{m}(\zz).
\end{gather}
Taking $\sigma\in \Gamma$ and applying Lemma
\ref{lem:radsum:var:RS_U_slash_sigma_vs_RS_Usigma} we see that
$\RS{\Gamma,\cp}{1}{m}(\zz)$ is a(n unrestricted) modular form for
$\Gamma$. Applying Lemma
\ref{lem:radsum:var:RS_U_slash_sigma_vs_RS_Usigma} in the case that
$\sigma\in\Sigma_{\cq}$ for some cusp $\cq\in \cP_{\Gamma}$, we see
that the Fourier expansion of the function
$\RS{\Gamma,\cp|\cq}{1}{m}(\zz)$ is the expansion of
$\RS{\Gamma,\cp}{1}{m}(\zz)$ at $\cq$ in the sense of
\S\ref{sec:conven:autfrm}, so Theorem
\ref{thm:radsum:conver:Relate_PS_FR_atp>0} applies also when
$\atp=1$, and thus $\RS{\Gamma,\cp}{1}{m}(\zz)$ is a modular form of
weight $2$ for $\Gamma$ and even a cusp form since $m>0$.
This completes the proof.
\end{proof}

We now present an analogue of Theorem
\ref{thm:radsum:var:Var_RS_atp>0} for the case that $\atp\leq 0$.
\begin{thm}\label{thm:radsum:var:Var_RS_atp<1}
Let $\Gamma$ be a group commensurable with $G(\ZZ)$ that has width
one at infinity and let $\cp\in \cP_{\Gamma}$ be a cusp of $\Gamma$.
Let $\atp,m\in \ZZ$ such that $\atp\leq 0$ and $m<0$.
Then the function
$\RS{\Gamma,\cp}{\atp}{m}(\zz)+\fc{\Gamma,\cp}{\atp}(m,0)/2$ is an
automorphic integral of weight $2\atp$ for $\Gamma$, and for $\cq\in \cP_{\Gamma}$ another cusp of $\Gamma$ the Fourier expansion of the function $\RS{\Gamma,\cp|\cq}{\atp}{m}(\zz)+\fc{\Gamma,\cp|\cq}{\atp}(m,0)/2$ is the expansion of $\RS{\Gamma,\cp}{\atp}{m}(\zz)+\fc{\Gamma,\cp}{\atp}(m,0)/2$ at the cusp $\cq$.
\end{thm}
\begin{proof}
Let $\atp$ and $m$ be as in the statement of the theorem. Since
$\Gamma$ has width one at infinity we take $\Gamma$ to be the
scaling coset for $\Gamma$ at the infinite cusp. Let $\Sigma_{\cp}$
be a scaling coset for $\Gamma$ at $\cp$ and set
$U=\lBZ\Sigma_{\cp}^{-1}\rBZh$. Let $\sigma\in G(\QQ)$. Using Lemma
\ref{lem:radsum:var:Apply_sigma_RS_chi} we compute
\begin{gather}
     \begin{split}
     \left(
     \left.
     \RS{U}{\atp}{m}
     \right\sop{\atp}
     \sigma
     \right)
     (\zz)
     &=
     \lim_{K\to \infty}
     \sum_{\lBZ\chi\rBZh\in U_{\leq K}}
     \RS{\lBZ\chi\sigma\rBZh}{\atp}{m}
     (\zz)
     +
     m^{1-2\atp}
     \left(
     \JO{\infty\cdot\lBZ\sigma\rBZh}{\atp}
     \PS{\lBZ\chi\sigma\rBZh}{1-\atp}{-m}
     \right)(\zz)
     \\
     &=
     \lim_{K\to \infty}
     \sum_{\lBZ\chi\rBZh\in (U_{\leq K})\sigma}
     \RS{\lBZ\chi\rBZh}{\atp}{m}(\zz)
     +
     m^{1-2\atp}
     \left(
     \JO{\infty\cdot\lBZ\sigma\rBZh}{\atp}
     \PS{\lBZ\chi\rBZh}{1-\atp}{-m}
     \right)(\zz)
     \end{split}
\end{gather}
for the weight $2\atp$ action of $\sigma$ on
$\RS{U}{\atp}{m}(\zz)=\RS{\Gamma,\cp}{\atp}{m}(\zz)$. Now Lemma
\ref{lem:radsum:var:RS_U_slash_sigma_vs_RS_Usigma} shows that
\begin{gather}\label{eqn:radsum:var:RS_U_slash_sigma}
     \left(
     \left.
     \RS{U}{\atp}{m}
     \right\sop{\atp}\sigma
     \right)
     (\zz)
     =
     \RS{U\sigma}{\atp}{m}
     (\zz)
     +
     m^{1-2\atp}
     \lim_{K\to \infty}
     \sum_{\lBZ\chi\rBZh\in (U_{\leq K})\sigma}
     \left(
     \JO{\infty\cdot\lBZ\sigma\rBZh}{\atp}
     \PS{\lBZ\chi\rBZh}{1-\atp}{-m}
     \right)(\zz),
\end{gather}
suggesting that $\RS{U}{\atp}{m}(\zz)$ might be an automorphic
integral of weight $2\atp$ for $\Gamma$ with associated cusp form
$-m^{1-2\atp}\PS{U}{1-\atp}{-m}(\zz)$ (cf.
\S\ref{sec:conven:autfrm}). This is not accurate, however, since the
limit in (\ref{eqn:radsum:var:RS_U_slash_sigma}) does not in general
commute with the integral operator
$\JO{\infty\cdot\lBZ\sigma\rBZh}{\atp}$. Rather, by the argument of
Lemma 4.4. in \cite{Nie_ConstAutInts}, we have the identity
\begin{gather}\label{eqn:radsum:var:Interchg_lim_JO_w}
     \begin{split}
     &
     \lim_{K\to \infty}
     \left(
     \JO{\infty\cdot\lBZ\sigma\rBZh}{\atp}
     \left(
     \sum_{\lBZ\chi\rBZh\in (U_{\leq K})\sigma}
     \PS{\lBZ\chi\rBZh}{1-\atp}{-m}
     \right)
     \right)(\zz)
     -
     \left(
     \JO{\infty\cdot\lBZ\sigma\rBZh}{\atp}
     \left(
     \lim_{K\to \infty}
     \sum_{\lBZ\chi\rBZh\in (U_{\leq K})\sigma}
     \PS{\lBZ\chi\rBZh}{1-\atp}{-m}
     \right)
     \right)(\zz)\\
     &\qquad=
     -
     \frac{m^{2\atp-1}}{2}
     \left(
     \fc{S}{\atp}(m,0)\jac(\sigma,\zz)^{\atp}
     -
     \fc{S'}{\atp}(m,0)
     \right)\\
     \end{split}
\end{gather}
where $S=\{\lBZ\chi\rBZ\mid \lBZ\chi\rBZh\in U^{\times}\}$ and
$S'=\{\lBZ\chi\rBZ\mid \lBZ\chi\rBZh\in (U\sigma)^{\times}\}$. Lemma \ref{lem:radsum:conver:Refinement_Lipschitz} plays a crucial r\^ole
in the verification of (\ref{eqn:radsum:var:Interchg_lim_JO_w}).
Applying (\ref{eqn:radsum:var:Interchg_lim_JO_w}) to
(\ref{eqn:radsum:var:RS_U_slash_sigma}) we find that
\begin{gather}\label{eqn:radsum:var:RS+const/2_slash_sigma}
     \left.
     \left(
     \RS{U}{\atp}{m}
     +\frac{1}{2}
     \fc{S}{\atp}(m,0)
     \right)
     \right\sop{\atp}\sigma
     =
     \RS{U'}{\atp}{m}
     +\frac{1}{2}
     \fc{S'}{\atp}(m,0)
     +
     m^{1-2\atp}
     \JO{\infty\cdot\lBZ\sigma\rBZh}{\atp}
     \PS{U'}{1-\atp}{-m}
\end{gather}
for $U'=U\sigma$. Taking $\sigma\in \Gamma$ and applying the
identity (\ref{eqn:modradsum:conver:Relate_QS_RS}) we obtain
\begin{gather}
     \left.
     \left(
     \RS{\Gamma,\cp}{\atp}{m}
     +\frac{1}{2}
     \fc{\Gamma,\cp}{\atp}(m,0)
     \right)
     \right\sop{\atp}\sigma
     =
     \RS{\Gamma,\cp}{\atp}{m}
     +\frac{1}{2}
     \fc{\Gamma,\cp}{\atp}(m,0)
     +
     m^{1-2\atp}
     \JO{\infty\cdot\lBZ\sigma\rBZh}{\atp}
     \PS{\Gamma,\cp}{1-\atp}{-m}
\end{gather}
from (\ref{eqn:radsum:var:RS+const/2_slash_sigma}), demonstrating
that the function
$\RS{\Gamma,\cp}{\atp}{m}(\zz)+\fc{\Gamma,\cp}{\atp}(m,0)/2$ is an
unrestricted automorphic integral of weight $2\atp$ for $\Gamma$
with associated cusp form
$-m^{1-2\atp}\PS{\Gamma,\cp}{1-\atp}{-m}(\zz)$. Letting $\cq\in
\cP_{\Gamma}$ and applying
(\ref{eqn:radsum:var:RS+const/2_slash_sigma}) with $\sigma\in
\Sigma_{\cq}$ we obtain
\begin{gather}
     \begin{split}
     \RS{\Gamma,\cp|\cq}{\atp}{m}
     +
     \frac{1}{2}
     \fc{\Gamma,\cp|\cq}{\atp}(m,0)
     &
     =
     \left.
     \left(
     \RS{\Gamma,\cp}{\atp}{m}
     +
     \frac{1}{2}
     \fc{\Gamma,\cp}{\atp}(m,0)
     \right)
     \right\sop{\atp}\sigma
     -
     m^{1-2\atp}
     \JO{\infty\cdot\lBZ\sigma\rBZh}{\atp}
     \PS{\Gamma,\cp|\cq}{1-\atp}{-m}
     \\
     &
     =
     \left.
     \left(
     \RS{\Gamma,\cp}{\atp}{m}
     +
     \frac{1}{2}
     \fc{\Gamma,\cp}{\atp}(m,0)
     +
     m^{1-2\atp}
     \JO{\sigma\cdot\infty}{\atp}
     \PS{\Gamma,\cp}{1-\atp}{-m}
     \right)
     \right\sop{\atp}
     \sigma,
     \end{split}
\end{gather}
demonstrating that the Fourier expansion of the function
$\RS{\Gamma,\cp|\cq}{\atp}{m}(\zz)+\fc{\Gamma,\cp|\cq}{\atp}(m,0)/2$
is the expansion at $\cq$ of the automorphic integral
$\RS{\Gamma,\cp}{\atp}{m}(\zz)+\fc{\Gamma,\cp}{\atp}(m,0)/2$ in the
sense of \S\ref{sec:conven:autfrm}. Now Theorem
\ref{thm:radsum:conver:Relate_RS_FR} confirms that
$\RS{\Gamma,\cp}{\atp}{m}(\zz)+\fc{\Gamma,\cp}{\atp}(m,0)/2$
is meromorphic at the cusps of $\Gamma$, and is thus an automorphic
integral of weight $2\atp$ for $\Gamma$. This completes the proof of
the proposition.
\end{proof}

We now specialize to the case that $\atp=0$. Our final goal in this
section is to establish the $\Gamma$-invariance of the harmonic
function $\RS{\Gamma,\cp}{}{m}(\zz)-\CS{\Gamma,\cp}{}{m}(\zz)$ for
$m<0$ in case $\Gamma$ has width one at infinity. For this we
utilize the spectral theory of the {hyperbolic Laplacian}
$\Delta=(\zz-\bar{\zz})^2\partial_{\zz}\partial_{\bar{\zz}}$.

Consider the {\em Whittaker functions} $\ev(\zz,\spp)$ and
$\ew(\zz,\spp)$, defined for $\zz,\spp\in\CC$ with $\zz\notin\RR$ by setting
\begin{gather}
     \ev(\zz,\spp)\label{eqn:radsum:var:Defn_ev}
     =
     \ee(\zz)
     {\Gamma(\spp)}
     \Phi(\spp,2\spp,\zz-\bar{\zz})
     (\tpi
     (\bar{\zz}-{\zz})
     )^{\spp},\\
     \ew(\zz,\spp)\label{eqn:radsum:var:Defn_ew}
     =\frac{\sin(\pi\spp)}{\sin(2\pi\spp)}
     \left(
     \ev(\zz,\spp)-\ev(\zz,1-\spp)
     \right),
\end{gather}
in case $\Im(\zz)>0$ where the function $\Phi$ is defined in (\ref{eqn:conven:fns:Defn_Phi}). The values $\ev(\zz,\spp)$ for $\Im(\zz)<0$ are obtained by imposing the symmetry $\ev(\zz,\spp)=\ev(\bar{\zz},\spp)$ and similarly for $\ew(\zz,\spp)$. Then $\ev(\zz,\spp)$ and
$\ew(\zz,\spp)$ are eigenfunctions for $\Delta$ with eigenvalue
$\spp(1-\spp)$, and the identity
$\ev(\zz,1)=\ee(\bar{\zz})-\ee(\zz)$ hints at a connection with the
right hand side of (\ref{eqn:radsum:var:Relate_RS-CS_FR_atp=0}).
Define the {\em free space Green's function} $G(\ww,\zz,\spp)$, for $\ww,\zz\in\HH$ with $\ww\neq\zz$ and $\spp\in\CC$ such that $\Re(\spp)>1$, by
setting
\begin{gather}\label{eqn:radsum:var:Defn_FreeSpaceGreensFn}
     G(\ww,\zz,\spp)
     =
     \frac{\Gamma(\spp)^2}{\Gamma(2\spp)}
     \,_2F_1\left(\spp,\spp;2\spp;{h}\right)
     {h^{\spp}}
     =
     \sum_{k\geq 0}
     \frac{\Gamma(k+\spp)^2}{\Gamma(k+2\spp)\Gamma(k+1)}h^{k+\spp},
\end{gather}
where $_2F_1(a,b;c;x)$ denotes the {Gauss hypergeometric function} and $h=h(\ww,\zz)$ is given as follows and depends only on the
hyperbolic distance between $\ww$ and $\zz$.
\begin{gather}
     h(\ww,\zz)
     =\frac{(\ww-\bar{\ww})(\zz-\bar{\zz})}
     {(\ww-\zz)(\bar{\ww}-\bar{\zz})}
     =\frac{4\Im(\ww)\Im(\zz)}{|\ww-\zz|^2}
\end{gather}
Then the function $\zz\mapsto G(\ww,\zz,\spp)$ satisfies $\Delta
G(\ww,\zz,\spp)=\spp(1-\spp)G(\ww,\zz,\spp)$ and
the integral operator with kernel $G(\ww,\zz,\spp)$ furnishes a
right-inverse to the differential operator $f\mapsto(\Delta
-\lambda)f$ when $\lambda=\spp(1-\spp)$.

Let $\Gamma$ be a group commensurable with $G(\ZZ)$, let $\cp$ and
$\cq$ be cusps of $\Gamma$, and let $\Sigma_{\cp}$ and
$\Sigma_{\cq}$ be scaling cosets for $\Gamma$ at $\cp$ and $\cq$,
respectively. Define $G_{\Gamma,\cp|\cq}(\ww,\zz,\spp)$ for $\ww,\zz\in\HH$ such that $\Gamma\cdot\ww\neq \Gamma\cdot\zz$ by setting
\begin{gather}\label{eqn:radsum:var:Defn_G_Gamma}
     G_{\Gamma,\cp|\cq}(\ww,\zz,\spp)
     =\sum_{\chi\in\Sigma_{\cp}^{-1}\Sigma_{\cq}}
     G(\ww,\chi\cdot\zz,\spp).
\end{gather}
The series (\ref{eqn:radsum:var:Defn_G_Gamma}) converges absolutely and locally uniformly
when $\Re(\spp)>1$ and the
functions $\ww\mapsto G_{\Gamma,\cp|\cq}(\ww,\zz,\spp)$ and
$\zz\mapsto G_{\Gamma,\cp|\cq}(\ww,\zz,\spp)$ are invariant for the
actions of $\Gamma^{\cp}$ and $\Gamma^{\cq}$, respectively.
According to \cite{Hej_TrFormPSL2R} (see also
\cite{Iwa_SpecMetAutFrms}) we have the Fourier--Whittaker series
expansion
\begin{gather}\label{eqn:radsum:var:SerExp_G_Gamma}
     \begin{split}
     &G_{\Gamma,\cp|\cq}(\ww,\zz,\spp)
     =\\
     &
     \sum_{m\in\ZZp}
     \GS{\Gamma,\cp|\cq}{}{m}(\zz,\spp)
     \ew(-m\bar{\ww},\spp)
     +
     \GS{\Gamma,\cp|\cq}{}{0}(\zz,\spp)
     \Im(\ww)^{1-\spp}
     +
     \sum_{m\in\ZZp}
     \GS{\Gamma,\cp|\cq}{}{-m}(\zz,\spp)
     \ew(m{\ww},\spp),
     \end{split}
\end{gather}
converging absolutely and locally uniformly when $\Re(\spp)>1$ so long as
$\Im(\ww)>\Im(\zz)$ and $\Im(\zz)\Im(\ww)>\rads_{\Gamma,\cp|\cq}$,
for a certain constant $\rads_{\Gamma,\cp|\cq}$ depending only on
$\Gamma$ and $\cp$ and $\cq$, where the coefficient functions
$\GS{\Gamma,\cp|\cq}{}{m}(\zz,\spp)$ are $\Gamma^{\cq}$-invariant,
and themselves admit the following series expansions of Poincar\'e
type, converging absolutely for $\Re(\spp)>1$ and locally uniformly for $\Im(\zz)>0$.
\begin{gather}
     \GS{\Gamma,\cp|\cq}{}{m}(\zz,\spp)
     =\label{eqn:radsum:var:Defn_GS_-m}
     \frac{1}{m}
     \sum_{\lBZ\chi\rBZh\in\lBZ\Sigma_{\cp}^{-1}\Sigma_{\cq}\rBZh}
     \ev(m\lBZ\chi\rBZh\cdot\zz,\spp)\\
     \GS{\Gamma,\cp|\cq}{}{0}(\zz,\spp)
     =\label{eqn:radsum:var:Defn_GS_0}
     \frac{4\pi}{2\spp-1}
     \sum_{\lBZ\chi\rBZh\in\lBZ\Sigma_{\cp}^{-1}\Sigma_{\cq}\rBZh}
     \Im(\lBZ\chi\rBZh\cdot\zz)^{\spp}\\
     \GS{\Gamma,\cp|\cq}{}{-m}(\zz,\spp)
     =\label{eqn:radsum:var:Defn_GS_m}
     \frac{1}{m}
     \sum_{\lBZ\chi\rBZh\in\lBZ\Sigma_{\cp}^{-1}\Sigma_{\cq}\rBZh}
     \ev(-m\lBZ\chi\rBZh\cdot\bar{\zz},\spp)
\end{gather}
Let $m>0$. Then we have the following analogue of
(\ref{eqn:radsum:var:SerExp_G_Gamma}) for the Green's function
coefficient $\GS{\Gamma,\cp|\cq}{}{-m}(\zz,\spp)$ which is again
reminiscent of the right hand side of
(\ref{eqn:radsum:var:Relate_RS-CS_FR_atp=0}).
\begin{gather}\label{eqn:radsum:var:SerExp_GS_m}
     \begin{split}
     &
     \GS{\Gamma,\cp|\cq}{}{-m}(\zz,\spp)
     =
     \delta_{\Gamma,\cp|\cq}
     \frac{1}{m}
     \ev(-m\bar{\zz},\spp)+\\
     &
     \sum_{n>0}
     \gc{\Gamma,\cp|\cq}{}(-m,-n,\spp)
     \ew(-n\bar{\zz},\spp)
     +
     \gc{\Gamma,\cp|\cq}{}(-m,0,\spp)
     \Im(\zz)^{1-\spp}
     +
     \sum_{n>0}
     \gc{\Gamma,\cp|\cq}{}(-m,n,\spp)
     \ew(n{\zz},\spp)
     \end{split}
\end{gather}
The coefficients in
(\ref{eqn:radsum:var:SerExp_GS_m}) may be defined via the following
formulas, in which we assume $m,n>0$.
\begin{gather}
     \begin{split}
     &
     \gc{\Gamma,\cp|\cq}{}(-m,-n,\spp)=\\
     &
     \frac{1}{m^{1-\spp}n^{\spp}}
     \lim_{K\to \infty}
     \sum_{\lBZ\chi\rBZ\in S_{\leq K}}
          %\sum_{\lBZ\chi\rBZ\in S}^{\lim}
          \sum_{k\geq 0}
          \ee(-m\lBZ\chi\rBZ\cdot\infty)
          \ee(-n\lBZ\chi^{-1}\rBZ\cdot\infty)
          (\fps\rads\lBZ\chi\rBZ)^{k+\spp}
          m^{(k)}n^{(k+2\spp-1)}
     \end{split}\\
     \gc{\Gamma,\cp|\cq}{}(-m,0,\spp)=
     \frac{1}{m^{1-\spp}}
     \frac{4\pi}{2\spp-1}
     \lim_{K\to \infty}
     \sum_{\lBZ\chi\rBZ\in S_{\leq K}}
          %\sum_{\lBZ\chi\rBZ\in S}^{\lim}
          \sum_{k\geq 0}
          \ee(-m\lBZ\chi\rBZ\cdot\infty)
          \rads\lBZ\chi\rBZ^{\spp}
     \frac{\pi^{\spp}}{\Gamma(\spp)}
     \\
     \begin{split}
     &
     \gc{\Gamma,\cp|\cq}{}(-m,n,\spp)=
     \\
     &
     \frac{1}{m^{1-\spp}n^{\spp}}
     \lim_{K\to \infty}
     \sum_{\lBZ\chi\rBZ\in S_{\leq K}}
          %\sum_{\lBZ\chi\rBZ\in S}^{\lim}
          \sum_{k\geq 0}
          \ee(-m\lBZ\chi\rBZ\cdot\infty)
          \ee(n\lBZ\chi^{-1}\rBZ\cdot\infty)
          (-1)^k
          (\fps\rads\lBZ\chi\rBZ)^{k+\spp}
          m^{(k)}n^{(k+2\spp-1)}
     \end{split}
\end{gather}
Comparing with Lemma (\ref{lem:radsum:coeff:Power_Series_fcS}) we
obtain a direct relationship between the coefficient functions
$\gc{\Gamma,\cp|\cq}{}(m,n,\spp)$ arising from the Green's function,
and the coefficient functions $\fc{\Gamma,\cp|\cq}{\atp}(m,n)$ of
\S\ref{sec:radsum:coeff} which arise from the Rademacher sums.
\begin{prop}\label{prop:radsum:var:Relate_fcs_gcs}
Let $\Gamma$ be a group commensurable with $G(\ZZ)$, let
$\cp,\cq\in\cP_{\Gamma}$, and let $\atp>0$ and $m<0$. Then we have the
following identities for $n>0$.
\begin{gather}
     \fc{\Gamma,\cp|\cq}{1-\atp}(m,-n)
     =
     %(-1)^{1-\atp}m^{\atp}n^{1-\atp}
     -m^{\atp}n^{1-\atp}
     \gc{\Gamma,\cp|\cq}{}(m,-n,\atp)\\
     \fc{\Gamma,\cp|\cq}{1-\atp}(m,0)
     =
     %(-1)^{1-\atp}m^{\atp}
     -m^{\atp}
     \gc{\Gamma,\cp|\cq}{}(m,0,\atp)\\
     \fc{\Gamma,\cp|\cq}{1-\atp}(m,n)
     =
     %(-1)^{1-\atp}m^{\atp}n^{1-\atp}
     -m^{\atp}n^{1-\atp}
     \gc{\Gamma,\cp|\cq}{}(m,n,\atp)
\end{gather}
\end{prop}
Taking $\atp=1$ in Proposition \ref{prop:radsum:var:Relate_fcs_gcs}
we obtain the identity
$\fc{\Gamma,\cp|\cq}{}(m,n)=-m\gc{\Gamma,\cp|\cq}{}(m,n,1)$ for $m<0$, which,
by Theorem \ref{thm:radsum:var:Relate_RS-CS_FR_atp=0}, the
series expansion (\ref{eqn:radsum:var:SerExp_GS_m}) and the
identity $\ew(\zz,1)=\ee(\zz)$, may be reformulated as
\begin{gather}\label{eqn:radsum:var:Relate_RS_CS_GS}
     \RS{\Gamma,\cp|\cq}{}{m}(\zz)
     -
     \CS{\Gamma,\cp|\cq}{}{m}(\zz)
     =-m\GS{\Gamma,\cp|\cq}{}{m}(\zz,1)
\end{gather}
for $m<0$. The $\Gamma^{\cq}$-invariance of the difference
$\RS{\Gamma,\cp|\cq}{}{m}(\zz)-\CS{\Gamma,\cp|\cq}{}{m}(\zz)$ of
classical and conjugate Rademacher sums of weight $0$ now follows
from the $\Gamma^{\cq}$-invariance of the Green's function
coefficient $\GS{\Gamma,\cp|\cq}{}{m}(\zz,\spp)$ at $\spp=1$ (cf.
\cite{Nie_ClassNonAnalyticAutFuns}). In particular, we have the
following result.
\begin{thm}\label{thm:radsum:var:Invar_RS-CS_atp=0}
Let $\Gamma$ be a group commensurable with $G(\ZZ)$ that has width
one at infinity, let $\cp\in \cP_{\Gamma}$ be a cusp of $\Gamma$,
and let $m<0$. Then the harmonic function
$\RS{\Gamma,\cp}{}{m}(\zz)-\CS{\Gamma,\cp}{}{m}(\zz)$ is
$\Gamma$-invariant.
\end{thm}

%------------------------------------------------------------------%
\section{Modified Rademacher sums}\label{sec:modradsum}
%------------------------------------------------------------------%

We have seen in \S\ref{sec:radsum:var} that the classical Rademacher
sum $\RS{\Gamma,\cp}{\atp}{m}(\zz)$ of \S\ref{sec:radsum} is an
automorphic integral for $\Gamma$ only after the addition of a
particular constant function, which does not generally vanish. In
this section we introduce a modification of Rademacher's
construction via which the correct constant term appears naturally.
Our approach employs an analytic continuation of the component
functions defining the classical Rademacher sums, and entails the
assignment of a Dirichlet series to each triple $(\Gamma,\cp,\cq)$
where $\Gamma$ is a group commensurable with the modular group and
$\cp,\cq\in \cP_{\Gamma}$ are cusps for $\Gamma$.

\subsection{Construction}\label{sec:modradsum:constr}

For $\atp,m\in\ZZ$ such that $\atp\leq 0$ and $m<0$, and for
$\lBZ\chi\rBZh\in \lBZ G(\QQ)\rBZh^{\times}$ we define the {\em
continued Rademacher component function}, denoted $(\zz,\spp)\mapsto
\TS{\lBZ\chi\rBZh}{\atp}{m}(\zz,\spp)$, by setting
\begin{gather}\label{eqn:modradsum:constr:Defn_TS_chi}
     \TS{\lBZ\chi\rBZh}{\atp}{m}(\zz,\spp)
     =
     \ee(m\lBZ\chi\rBZh\cdot\zz)
     \Treg{\atp}(m,\lBZ\chi\rBZh,\zz,\spp)
     \jac(\lBZ\chi\rBZh,\zz)^{\atp}
\end{gather}
for $\zz\in \HH$ and $\spp\in \CC$ such that $\Re(\spp)>1$ where $\Treg{\atp}(m,\lBZ\chi\rBZh,\zz,\spp)$ is the {\em continued Rademacher regularization factor of weight $2\atp$} which is in
turn given by
\begin{gather}\label{eqn:modradsum:constr:Defn_Treg}
     \Treg{\atp}(m,\lBZ\chi\rBZh,\zz,\spp)
     =
     \Phi(\spp-2\atp,1+\spp-2\atp,
     m\lBZ\chi\rBZh\cdot\infty
     -m\lBZ\chi\rBZh\cdot\zz
     )
     (\tpi
     (
     m\lBZ\chi\rBZh\cdot\zz
     -m\lBZ\chi\rBZh\cdot\infty
     )
     )^{\spp-2\atp}
\end{gather}
in case $\lBZ\chi\rBZh\in \lBZ G(\QQ)\rBZh^{\times}$ and
$\Treg{\atp}(m,\lBZ\chi\rBZh,\zz,\spp)=1$ otherwise. We also define
a function $\spp\mapsto \TSa{\lBZ\chi\rBZh}{\atp}{m}(\spp)$ by
setting
\begin{gather}\label{eqn:modradsum:constr:Defn_TSa_chi}
     \begin{split}
     \TSa{\lBZ\chi\rBZh}{\atp}{m}(\spp)
     =&
     \ee(m\lBZ\chi\rBZh\cdot\infty)
     \left(
     \tpi
     (
     m\lBZ\chi\rBZh\cdot 0
     -m\lBZ\chi\rBZh\cdot\infty
     )
     \right)^{(\spp-2\atp)}
     {\jac(\lBZ\chi\rBZh,0)^{\atp}}\\
     &-\ee(\spp/2)
     \ee(m\lBZ\chi\rBZh\cdot\infty)
     \left(
     \tpi
     (
     m\lBZ\chi\rBZh\cdot \infty
     -m\lBZ\chi\rBZh\cdot 0
     )
     \right)^{(\spp-2\atp)}
     {\jac(\lBZ\chi\rBZh,0)^{\atp}}
     \end{split}
\end{gather}
in case $\lBZ\chi\rBZh\in \lBZ G(\QQ)\rBZh^{\times\times}$, and by
setting $\TSa{\lBZ\chi\rBZh}{\atp}{m}(\spp)=0$ otherwise, and we
define the {\em modified continued Rademacher component function},
denoted $(\zz,\spp)\mapsto \QS{\lBZ\chi\rBZh}{\atp}{m}(\zz,\spp)$,
by subtracting $\TSa{\lBZ\chi\rBZh}{\atp}{m}(\zz,\spp)$ from
$\TS{\lBZ\chi\rBZh}{\atp}{m}(\zz,\spp)$.
\begin{gather}\label{eqn:modradsum:constr:Defn_QS_chi}
     \QS{\lBZ\chi\rBZh}{\atp}{m}(\zz,\spp)
     =
     \TS{\lBZ\chi\rBZh}{\atp}{m}(\zz,\spp)
     -
     \TSa{\lBZ\chi\rBZh}{\atp}{m}(\zz,\spp)
\end{gather}
Now for $U\subset\lBZ G(\QQ)\rBZh$, and for $\atp,m\in\ZZ$ such that
$\atp\leq 0$ and $m<0$, we define the {\em continued
Rademacher sum} $\TS{U}{\atp}{m}(\zz,\spp)$, and the {\em modified
continued Rademacher sum} $\QS{U}{\atp}{m}(\zz,\spp)$, by setting
\begin{gather}
     \TS{U}{\atp}{m}(\zz,\spp)
     \label{eqn:modradsum:constr:Defn_TS_U_m}
     =
     \sum_{\lBZ\chi\rBZh\in U}
     \TS{\lBZ\chi\rBZh}{\atp}{m}(\zz,\spp),\\
     \QS{U}{\atp}{m}(\zz,\spp)
     \label{eqn:modradsum:constr:Defn_QS_U_m}
     =
     \sum_{\lBZ\chi\rBZh\in U}
     \QS{\lBZ\chi\rBZh}{\atp}{m}(\zz,\spp),
\end{gather}
and we define $\TS{U}{\atp}{m}(\zz)$ and $\QS{U}{\atp}{m}(\zz)$ by
taking the limit as $\spp$ tends to $1$ through the region
$\Re(\spp)>1$ in $\TS{U}{\atp}{m}(\zz,\spp)$ and
$\QS{U}{\atp}{m}(\zz,\spp)$, respectively, so long as these limits exist.
\begin{gather}
     \TS{U}{\atp}{m}(\zz)\label{eqn:modradsum:constr:Defn_TS_s=1}
     =
     \lim_{s\to 1^+}\TS{U}{\atp}{m}(\zz,\spp)\\
     \QS{U}{\atp}{m}(\zz)\label{eqn:modradsum:constr:Defn_QS_s=1}
     =
     \lim_{s\to 1^+}\QS{U}{\atp}{m}(\zz,\spp)
\end{gather}
The function $\QS{U}{\atp}{m}(\zz)$ is the most important from the
point of view of automorphy for subgroups of $G(\QQ)$, while the
function $\TS{U}{\atp}{m}(\zz)$ plays a special r\^ole in the
critical case that $\atp=0$. We call $\QS{U}{\atp}{m}(\zz)$ the {\em
modified Rademacher sum of weight $2\atp$ and order $m$ associated
to $U$}, and we call $\TS{U}{\atp}{m}(\zz)$ the {\em normalized
Rademacher sum of weight $2\atp$ and order $m$ associated to $U$}.

Observe that we recover the Rademacher component function
$\RS{\lBZ\chi\rBZh}{\atp}{m}(\zz)$ by taking $s=1$ in either the
continued Rademacher component function
$\TS{\lBZ\chi\rBZh}{\atp}{m}(\zz,\spp)$ or the modified continued
Rademacher component function
$\QS{\lBZ\chi\rBZh}{\atp}{m}(\zz,\spp)$.
\begin{lem}\label{lem:modradsum:constr:Relate_QS_chi_s=1_RS_chi}
Let $\atp,m\in\ZZ$ such that $\atp\leq 0$ and $m<0$, and let
$\lBZ\chi\rBZh\in \lBZ G(\QQ)\rBZh$. Then we have
$\TS{\lBZ\chi\rBZh}{\atp}{m}(\zz,1)=\RS{\lBZ\chi\rBZh}{\atp}{m}(\zz)$
and $\TSa{\lBZ\chi\rBZh}{\atp}{m}(1)=0$, and hence
$\QS{\lBZ\chi\rBZh}{\atp}{m}(\zz,1)=\RS{\lBZ\chi\rBZh}{\atp}{m}(\zz)$.
\end{lem}
Despite the result of Lemma
\ref{lem:modradsum:constr:Relate_QS_chi_s=1_RS_chi}, it is generally
not the case that the modified Rademacher sum $\QS{U}{\atp}{m}(\zz)$
and the classical Rademacher sum $\RS{U}{\atp}{m}(\zz)$ coincide
(cf. Proposition \ref{prop:modradsum:conver:Relate_QS_s=1_RS}).

We typically take $U$ to be of the form
$U=\lBZ\Sigma_{\cp}^{-1}\Sigma_{\cq}\rBZh$ where
$\{\Sigma_{\cp}\mid\cp\in \cP_{\Gamma}\}$ is a system of scaling
cosets (cf. \S\ref{sec:conven:scaling}) for some group $\Gamma$
commensurable with the modular group, and $\cp,\cq\in\cP_{\Gamma}$
are cusps of $\Gamma$. In this case we write
$\QS{\Gamma,\cp|\cq}{\atp}{m}(\zz)$ for $\QS{U}{\atp}{m}(\zz)$ and
$\TS{\Gamma,\cp|\cq}{\atp}{m}(\zz)$ for $\TS{U}{\atp}{m}(\zz)$,
suppressing the choice of scaling cosets from notation. A change in
the choice of scaling cosets $\Sigma_{\cp}$ and $\Sigma_{\cq}$
replaces $\QS{\Gamma,\cp|\cq}{\atp}{m}(\zz)$ with a function of the
form $\ee(\alpha)\QS{\Gamma,\cp|\cq}{\atp}{m}(\zz+\beta)$ for some
$\alpha,\beta\in \QQ$, and similarly for
$\TS{\Gamma,\cp|\cq}{\atp}{m}(\zz)$.

In the case that $\cp$ or $\cq$ is the {\em infinite cusp}
$\Gamma\cdot\infty$ we omit it from notation, writing
$\QS{\Gamma,\cp}{\atp}{m}(\zz)$ for
$\QS{\Gamma,\cp|\Gamma\cdot\infty}{\atp}{m}(\zz)$, and
$\QS{\Gamma|\cq}{\atp}{m}(\zz)$ for
$\QS{\Gamma,\Gamma\cdot\infty|\cq}{\atp}{m}(\zz)$, and similarly for
the functions $\TS{\Gamma,\cp|\cq}{\atp}{m}(\zz)$. The functions
$\QS{\Gamma,\cp}{\atp}{m}(\zz)$ and $\TS{\Gamma,\cp}{\atp}{m}(\zz)$
are the most important, for we shall see in
\S\ref{sec:modradsum:var} that $\QS{\Gamma,\cp}{\atp}{m}(\zz)$ is an
automorphic integral of weight $2\atp$ for $\Gamma$ with a single
pole at the cusp $\cp$ in case $\Gamma$ has width one at infinity,
and we will see in \S\ref{sec:modradsum:conver} that
$\QS{\Gamma,\cp}{\atp}{m}(\zz)$ and $\TS{\Gamma,\cp}{\atp}{m}(\zz)$
differ only by a constant function. We call
$\QS{\Gamma,\cp}{\atp}{m}(\zz)$ the {\em modified Rademacher sum of
weight $2\atp$ and order $m$ associated to $\Gamma$ at the cusp
$\cp$}, and we call $\TS{\Gamma,\cp}{\atp}{m}(\zz)$ the {\em
normalized Rademacher sum of weight $2\atp$ and order $m$ associated
to $\Gamma$ at the cusp $\cp$}. We shall see also in
\S\ref{sec:modradsum:var} that the Fourier expansion of the function
$\QS{\Gamma,\cp|\cq}{\atp}{m}(\zz)$ is the expansion of $\QS{\Gamma,\cp}{\atp}{m}(\zz)$ at the cusp $\cq$, in the sense of \S\ref{sec:conven:autfrm}.

Just as for the classical Rademacher sums (cf.
\S\ref{sec:radsum:constr}) we have the result that every modified
Rademacher sum $\QS{\Gamma,\cp|\cq}{\atp}{m}(\zz)$ may be expressed
in the form $\QS{\Gamma',\cp'}{\atp}{m}(\zz)$ for some group
$\Gamma'$ with width one at infinity, and some cusp $\cp'$ of
$\Gamma'$, and similarly for the normalized Rademacher sums
$\TS{\Gamma,\cp|\cq}{\atp}{m}(\zz)$. Precisely, we have
\begin{gather}
     \QS{\Gamma,\cp|\cq}{\atp}{m}(\zz)
     =\label{eqn:modradsum:constr:QS_twocusps_to_onecusp}
     \QS{\Gamma^{\cq},\cp^{\cq}}{\atp}{m}(\zz)\\
     \TS{\Gamma,\cp|\cq}{\atp}{m}(\zz)
     =\label{eqn:modradsum:constr:TS_twocusps_to_onecusp}
     \TS{\Gamma^{\cq},\cp^{\cq}}{\atp}{m}(\zz)
\end{gather}
subject to the understanding that the data defining the right hand
sides of (\ref{eqn:modradsum:constr:QS_twocusps_to_onecusp}) and
(\ref{eqn:modradsum:constr:TS_twocusps_to_onecusp}) is related to
the data defining the left hand sides of
(\ref{eqn:modradsum:constr:QS_twocusps_to_onecusp}) and
(\ref{eqn:modradsum:constr:TS_twocusps_to_onecusp}) by
$\Gamma^{\cq}=\Sigma_{\cq}^{-1}\Sigma_{\cq}$ and
$\cp^{\cq}=\Sigma_{\cq}^{-1}\cdot\cp$ and
$\Sigma_{\cp^{\cq}}=\Sigma_{\cq}^{-1}\Sigma_{\cp}$.

Our primary interest in this article is in the distinguished case
that $\atp=0$. In order to simplify notation, and maintain
consistency with the notation of \S\ref{sec:intro}, we write
$\QS{\Gamma,\cp|\cq}{}{m}(\zz)$ as a shorthand for
$\QS{\Gamma,\cp|\cq}{0}{m}(\zz)$, and we write
$\TS{\Gamma,\cp|\cq}{}{m}(\zz)$ as a shorthand for
$\TS{\Gamma,\cp|\cq}{0}{m}(\zz)$.

\subsection{Coefficients}\label{sec:modradsum:coeff}

In order to recover explicit expressions of the Fourier coefficients
of the modified Rademacher sums we employ the following
generalizations of $\Bf_{\lBZ\chi\rBZ}^{\atp}(m,n)$ and
$\fc{S}{\atp}(m,n)$ for $\atp,m,n\in \ZZ$ such that $\atp\leq 0$ and $m<0<n$. For $\atp\leq 0$ and $m<0<n$ we define the entire function
$\spp\mapsto\Bf_{\lBZ\chi\rBZ}^{\atp}(m,n,\spp)$ by setting
\begin{gather}\label{eqn:radsum:coeff:Pow_Ser_Bf_1-s_pos_cont}
     \Bf_{\lBZ\chi\rBZ}^{\atp}(m,n,\spp)
     =
     (-1)^{\atp}
     \sum_{k\geq 0}
     (\fps)^{k-\atp+\spp}
     \rads\lBZ\chi\rBZ^{k-2\atp+\spp}
     (-m)^{(k-2\atp+\spp)}
     n^{(k+\spp-1)},
\end{gather}
so that
$\Bf_{\lBZ\chi\rBZ}^{\atp}(m,n,1)=\Bf_{\lBZ\chi\rBZ}^{\atp}(m,n)$.
We then define the {\em continued coefficient function} $\spp\mapsto
\fc{S}{\atp}(m,n,\spp)$ for $\atp\leq 0$ and $m<0<n$ by setting
\begin{gather}
     \fc{S}{\atp}(m,n,\spp)
     =
     \sum_{
     \lBZ\chi\rBZ\in S
     }
     \Kl_{\lBZ\chi\rBZ}^{\atp}(m,n)
     \Bf_{\lBZ\chi\rBZ}^{\atp}(m,n,\spp).
\end{gather}
\begin{lem}\label{lem:radsum:coeff:Power_Series_fcS_cont}
The functions $\spp\mapsto\fc{S}{\atp}(m,n,\spp)$ admit the
following series representation.
\begin{gather}
     \begin{split}
     &\fc{S}{\atp}(m,n,\spp)\\
     &=\label{eqn:radsum:coeff:Pow_Ser_fcS_1-s_pos_cont}
          (-1)^{\atp}
     \lim_{K\to \infty}
     \sum_{\lBZ\chi\rBZ\in S_{\leq K}}
          %\sum_{\lBZ\chi\rBZ\in S}^{\lim}
          \sum_{k\geq 0}
          \ee(m\lBZ\chi\rBZ\cdot\infty)
          \ee(-n\lBZ\chi^{-1}\rBZ\cdot\infty)
          (\fps\rads\lBZ\chi\rBZ)^{k-\atp+\spp}
          (-m)^{(k-2\atp+\spp)}n^{(k+\spp-1)}
     \end{split}
\end{gather}
We have $\fc{S}{\atp}(m,n,1)=\fc{S}{\atp}(m,n)$.
\end{lem}
We usually take $S$ to be of the form
$S=\lBZ\Sigma_{\cp}^{-1}\Sigma_{\cq}\rBZ^{\times}$ where
$\Sigma_{\cp}$ and $\Sigma_{\cq}$ are scaling cosets (cf.
\S\ref{sec:conven:scaling}) for a group $\Gamma$ commensurable with
$G(\ZZ)$ at cusps $\cp,\cq\in \cP_{\Gamma}$. We write
$\fc{\Gamma,\cp|\cq}{\atp}(m,n,\spp)$ for $\fc{S}{\atp}(m,n,\spp)$
in the case that $S=\lBZ\Sigma_{\cp}^{-1}\Sigma_{\cq}\rBZ^{\times}$.

Utilizing the continued coefficient functions $\spp\mapsto
\fc{S}{\atp}(m,n,\spp)$ we generalize the function
$\FR{S}{\atp}{m}(\zz)_{\rm van}$ in case $\atp\leq 0$ and $m<0$ by
setting
\begin{gather}
     \FR{S}{\atp}{m}(\zz,\spp)_{\rm van}
     =\sum_{n>0}\fc{S}{\atp}(m,n,\spp)\ee(n\zz).
\end{gather}
We write $\FR{\Gamma,\cp|\cq}{\atp}{m}(\zz,\spp)_{\rm van}$ for
$\FR{S}{\atp}{m}(\zz,\spp)_{\rm van}$ in the case that
$S=\lBZ\Sigma_{\cp}^{-1}\Sigma_{\cq}\rBZ^{\times}$ for $\Gamma$ a
group commensurable with $G(\ZZ)$ and for
$\Sigma_{\cp},\Sigma_{\cq}\in \Gamma\backslash G(\QQ)$ scaling
cosets for $\Gamma$ at cusps $\cp,\cq\in\cP_{\Gamma}$. We write
$\FR{S}{}{m}(\zz,\spp)_{\rm van}$ as a shorthand for
$\FR{S}{0}{m}(\zz,\spp)_{\rm van}$, and interpret the notation
$\FR{\Gamma,\cp|\cq}{}{m}(\zz,\spp)_{\rm van}$ similarly.

We have the following analogue of Proposition
\ref{prop:radsum:constr:Conver_FR_van} for the functions
$\FR{S}{\atp}{m}(\zz,\spp)_{\rm van}$ in case $S$ is of the form
$\lBZ\Sigma_{\cp}^{-1}\Sigma_{\cq}\rBZ^{\times}$.
\begin{prop}\label{prop:modradsum:constr:Conver_FRs_van}
Let $\Gamma$ be a group commensurable with $G(\ZZ)$, and let
$\cp,\cq\in \cP_{\Gamma}$ be cusps of $\Gamma$. Let $\atp,m\in \ZZ$
such that $\atp\leq 0$ and $m<0$. Then the series
\begin{gather}
     \FR{\Gamma,\cp|\cq}{\atp}{m}(\zz,\spp)_{\rm van}
     =
     \sum_{n>0}
     \fc{\Gamma,\cp|\cq}{\atp}(m,n,\spp)
     \ee(n\zz)
\end{gather}
converges absolutely and locally uniformly in $\zz$ and $\spp$ for $\zz\in \HH$ and $\Re(\spp)\geq 1$. In
particular, the assignment
$\zz\mapsto\FR{\Gamma,\cp|\cq}{\atp}{m}(\zz,\spp)_{\rm van}$ is a
holomorphic function on $\HH$ whenever $\atp\leq 0$ and $m<0$ and
$\Re(\spp)\geq 1$.
\end{prop}

\subsection{Dirichlet series}\label{sec:modradsum:dirser}

Given $\lBZ\chi\rBZ\in\lBZ G(\QQ)\rBZ^{\times}$, and $\atp,m\in\ZZ$
such that $\atp\leq 0$ and $m<0$, we define a Dirichlet
series $\DS{\lBZ\chi\rBZ}{\atp}{m}(\spp)$ by setting
\begin{gather}\label{eqn:modradsum:constr:Defn_DS_chi}
\begin{split}
     \DS{\lBZ\chi\rBZ}{\atp}{m}(\spp)
     &=
     (-1)^{1-\atp}
     \sum_{n\in \ZZ^{\times}}
     \ee(m\lBZ\chi\rBZ\cdot\infty)
     \ee(-n\lBZ\chi^{-1}\rBZ\cdot\infty)
     (\fps\rads\lBZ\chi\rBZ)^{1-\spp-\atp}
     |m|^{(1-\spp-2\atp)}
     n^{(-\spp)}\\
     &=
     (-1)^{1-\atp}
     \ee(m\lBZ\chi\rBZ\cdot\infty)
     (\fps\rads\lBZ\chi\rBZ)^{1-\spp-\atp}
     |m|^{(1-\spp-2\atp)}
     \sum_{n\in \ZZ^{\times}}
     \ee(-n\lBZ\chi^{-1}\rBZ\cdot\infty)
     n^{(-\spp)}.
\end{split}
\end{gather}
(Recall that $\ZZ^{\times}$ denotes the non-zero elements of $\ZZ$.) This series (\ref{eqn:modradsum:constr:Defn_DS_chi}) converges
absolutely and locally uniformly in the half plane $\Re(\spp)>1$, and admits a meromorphic
continuation to all of $\CC$. Given $S\subset \lBZ
G(\QQ)\rBZ^{\times}$ we define the {\em zeta function of weight
$2\atp$ and order $m$ associated to $S$}, to be denoted
$\DS{S}{\atp}{m}(\spp)$, by setting
\begin{gather}\label{eqn:modradsum:constr:Defn_DS}
     \DS{S}{\atp}{m}(\spp)
     =
     \sum_{\lBZ\chi\rBZ\in S}
     \DS{\lBZ\chi\rBZ}{\atp}{m}(\spp).
\end{gather}
As it is defined here $\DS{S}{\atp}{m}$ is only a formal sum but it will be shown to converge, and admit an analytic continuation, for various choices of $S$ in \S\ref{sec:modradsum:conver} (cf. Proposition \ref{prop:modradsum:conver:Eval_sum_QSa}, Theorem \ref{eqn:modradsum:conver:QSa_to_DS}). 
In order to describe the analytic continuation of the zeta function
$\spp\mapsto \DS{S}{\atp}{m}(\spp)$ explicitly, we formulate an
identity which expresses it in terms of the Hurwitz zeta function.
In preparation for this we define a function $\spp\mapsto
\DSt{\lBZ\chi\rBZ}{\atp}{m}(\spp)$, for each $\lBZ\chi\rBZ\in \lBZ
G(\QQ)\rBZ^{\times}$, by setting
\begin{gather}\label{eqn:modradsum:constr:Defn_DS_chi_p}
     \DSt{\lBZ\chi\rBZ}{\atp}{m}(\spp)
     =
     \ee(m\lBZ\chi\rBZ\cdot\infty)
     \rads\lBZ\chi\rBZ^{\spp-\atp}
     (-\tpi m)^{(\spp-2\atp)}
     (\ee(\spp/2)-\ee(-\spp/2))
     \zeta(1-\alpha_{\lBZ\chi\rBZ},\spp)
\end{gather}
where $\zeta(\alpha,\spp)$ is defined by
$\zeta(\alpha,\spp)=\sum_{n\geq 0}(n+\alpha)^{-\spp}$ for
$\Re(\alpha)>0$ and $\Re(\spp)>1$, and $\alpha_{\lBZ\chi\rBZ}\in
\QQ$ is chosen so that
\begin{gather}
     0\leq \alpha_{\lBZ\chi\rBZ}<1,\quad
     \alpha_{\lBZ\chi\rBZ}+\ZZ=-\chi^{-1}\cdot\infty+\ZZ.
\end{gather}
We then define $\DSt{S}{\atp}{m}(\spp)$ by setting
$\DSt{S}{\atp}{m}(\spp)=\sum_{\lBZ\chi\rBZ\in
S}\DSt{\lBZ\chi\rBZ}{\atp}{m}(\spp)$.
\begin{lem}\label{lem:modradsum:constr:DS_to_DSp}
We have $\DS{S}{\atp}{m}(1-\spp)=\DSt{S}{\atp}{m}(\spp)$.
\end{lem}
\begin{proof}
Let $\lBZ\chi\rBZ\in \lBZ G(\QQ)\rBZ^{\times}$ and choose
$\alpha_{\lBZ\chi\rBZ}\in\QQ$ such that $0\leq
\alpha_{\lBZ\chi\rBZ}<1$ and
$\alpha_{\lBZ\chi\rBZ}+\ZZ=-\chi^{-1}\cdot\infty+\ZZ$ for any
$\chi\in \lBZ\chi\rBZ$. We will show that
$\DS{\lBZ\chi\rBZ}{\atp}{m}(\spp)=\DSt{\lBZ\chi\rBZ}{\atp}{m}(1-\spp)$.
Observe that we have
\begin{gather}\label{eqn:modradsum:constr:Transform_DS}
     \sum_{n\in\ZZ^{\times}}
     \ee(-n\lBZ\chi^{-1}\rBZ\cdot\infty)
     n^{(-\spp)}
     =
     \frac{1}{\Gamma(1-\spp)}
     \left(
     F\left(-\lBZ\chi^{-1}\rBZ\cdot\infty,\spp\right)
     +\ee(-\spp/2)
     F\left(\lBZ\chi^{-1}\rBZ\cdot\infty,\spp\right)
     \right),
\end{gather}
where $F(\alpha,\spp)$ denotes the periodic zeta function, defined
by $F(\alpha,\spp)=\sum_{n>0}\ee(n\alpha)n^{-\spp}$ for
$\Re(\spp)>1$. The {\em Hurwitz relation}, which may be stated in
the form
\begin{gather}\label{eqn:modradsum:constr:Hur_Reln}
     F(\alpha,\spp)
     +\ee(-\spp/2)F(-\alpha,\spp)
     =
     \frac{(-\tpi)^{\spp}}{\Gamma(\spp)}
     \zeta(1-\alpha,1-\spp)
\end{gather}
(cf. \cite{KnoRob_RieFnlEqnLipSum}) implies the identity
\begin{gather}\label{eqn:modradsum:constr:DSn_to_HurZeta}
     \sum_{n\in\ZZ^{\times}}
     \ee(-n\lBZ\chi^{-1}\rBZ\cdot\infty)
     n^{(-\spp)}
     =
     \frac{(-\tpi)^{\spp}}{\Gamma(1-\spp)\Gamma(\spp)}
     \zeta(1-\alpha_{\lBZ\chi\rBZ},1-\spp),
\end{gather}
so the expression (\ref{eqn:modradsum:constr:Defn_DS_chi}) may be
reformulated as
\begin{gather}\label{eqn:modradsum:constr:Zeta_Fn_DS_transform}
     \DS{\lBZ\chi\rBZ}{\atp}{m}(\spp)
     =
     (-1)^{1-\atp}
     \ee(m\lBZ\chi\rBZ\cdot\infty)
     (\fps\rads\lBZ\chi\rBZ)^{1-\spp-\atp}
     |m|^{(1-\spp-2\atp)}
     \frac{(-\tpi)^{\spp}}{\Gamma(1-\spp)\Gamma(\spp)}
     \zeta(1-\alpha_{\lBZ\chi\rBZ},1-\spp).
\end{gather}
We obtain the coincidence of the right hand side of
(\ref{eqn:modradsum:constr:Zeta_Fn_DS_transform}) with
$\DSt{\lBZ\chi\rBZ}{\atp}{m}(1-\spp)$ by applying the functional
equation for the Gamma function, which may be expressed as
\begin{gather}
     (\ee(\spp/2)-\ee(-\spp/2))
     \Gamma(\spp)\Gamma(1-\spp)
     =\tpi
\end{gather}
This completes the proof of the claim.
\end{proof}

We write $\DS{\Gamma,\cp|\cq}{\atp}{m}(\spp)$ for
$\DS{S}{\atp}{m}(\spp)$ in the case that
$S=\lBZ\Sigma_{\cp}^{-1}\Sigma_{\cq}\rBZ^{\times}$ for $\Gamma$ a
group commensurable with $G(\ZZ)$ and for
$\Sigma_{\cp},\Sigma_{\cq}\in \Gamma\backslash G(\QQ)$ scaling
cosets for $\Gamma$ at cusps $\cp,\cq\in\cP_{\Gamma}$.

\subsection{Convergence}\label{sec:modradsum:conver}

Our main objective in this section is to furnish explicit
expressions for the Fourier expansions of the functions
$\zz\mapsto\TS{\Gamma,\cp|\cq}{\atp}{m}(\zz,\spp)$ and
$\zz\mapsto\QS{\Gamma,\cp|\cq}{\atp}{m}(\zz,\spp)$, defined by the
continued Rademacher sums and modified continued Rademacher sums,
respectively. This will yield the Fourier expansions of the modified
Rademacher sums $\QS{\Gamma,\cp|\cq}{\atp}{m}(\zz)$ and the
normalized Rademacher sums $\TS{\Gamma,\cp|\cq}{\atp}{m}(\zz)$. In
preparation for the derivation of these expansions we verify the
absolute and locally uniform convergence of the series defining
$\TS{\Gamma,\cp|\cq}{\atp}{m}(\zz,\spp)$ and
$\QS{\Gamma,\cp|\cq}{\atp}{m}(\zz,\spp)$ in case $\Re(\spp)>1$.
\begin{prop}\label{prop:modradsum:conver:Abs_Conv_sum_TS_chi}
Let $\Gamma$ be a group commensurable with $G(\ZZ)$ and let
$\cp,\cq\in \cP_{\Gamma}$ be cusps of $\Gamma$. Let $\atp,m\in\ZZ$
such that $\atp\leq 0$ and $m<0$. Then the continued
Rademacher sum
\begin{gather}\label{eqn:modradsum:conver:Abs_Conv_sum_TS_chi}
     \TS{\Gamma,\cp|\cq}{\atp}{m}(\zz,\spp)
     =
     \sum_{\lBZ\chi\rBZh\in \lBZ\Sigma_{\cp}^{-1}\Sigma_{\cq}\rBZh}
     \TS{\lBZ\chi\rBZh}{\atp}{m}(\zz,\spp)
\end{gather}
is absolutely convergent, locally uniformly in $\zz$ and $\spp$, for $\zz\in \HH$ and $\Re(\spp)>1$.
\end{prop}
\begin{proof}
We begin by observing that we have
\begin{gather}\label{eqn:modradsum:conver:Mod_Rad_Sum_Term_GExp}
     \TS{\lBZ\chi\rBZh}{\atp}{m}(\zz,\spp)
     =
     \ee(m\lBZ\chi\rBZh\cdot\infty)
     \ee
     (
     m\lBZ\chi\rBZh\cdot\zz
	-
     m\lBZ\chi\rBZh\cdot\infty
     ,\spp-2\atp
     )
     \jac(\lBZ\chi\rBZh,\zz)^{\atp}
\end{gather}
in case $\lBZ\chi\rBZh\in \lBZ G(\QQ)\rBZh^{\times}$, and
$\TS{\lBZ\chi\rBZh}{\atp}{m}(\zz,\spp)=\ee(m\lBZ\chi\rBZh\cdot\zz)\jac(\lBZ\chi\rBZh,\zz)^{\atp}$
otherwise, where $\ee(\zz,\spp)$ denotes the generalized exponential
function of (\ref{eqn:conven:fns:Genzd_Exp}). This identity
(\ref{eqn:modradsum:conver:Mod_Rad_Sum_Term_GExp}) follows from an
application of the {\em Kummer transformation}
$\Phi(a,b,\zz)=\ee(\zz)\Phi(b-a,b,-\zz)$ to the expression
(\ref{eqn:modradsum:constr:Defn_TS_chi}) defining
$\TS{\lBZ\chi\rBZh}{\atp}{m}(\zz,\spp)$. The absolute and locally uniform convergence of
(\ref{eqn:modradsum:conver:Abs_Conv_sum_TS_chi}) can be established by implementing methods present in the proofs of Proposition \ref{prop:radsum:coeff:fcS_conv} and Theorem \ref{thm:radsum:conver:Relate_RS_FR}. Namely, the right hand side of (\ref{eqn:modradsum:conver:Abs_Conv_sum_TS_chi}) is very similar to the quantity $R_+$ defined in (\ref{eqn:radsum:conver:Defn_Rp}) and we can rewrite it so as to arrive at an analogue of (\ref{eqn:radsum:conver:Rconv_1-atp_pos_R+_01}). Then the coefficients of $\ee(n\zz)$ in the resulting sum can be estimated just as the values $\fc{S}{\atp}(m,n)$ are in (\ref{eqn:radsum:coeff:fcS_conv_3}) for $\atp>1$.
\end{proof}

\begin{prop}\label{prop:modradsum:conver:Abs_Conv_sum_TSa_chi}
Let $\Gamma$ be a group commensurable with $G(\ZZ)$ and let
$\cp,\cq\in \cP_{\Gamma}$ be cusps of $\Gamma$. Let $\atp,m\in \ZZ$
such that $\atp\leq 0$ and $m<0$. Then the sum
\begin{gather}
     \sum_{\lBZ\chi\rBZh\in \lBZ\Sigma_{\cp}^{-1}\Sigma_{\cq}\rBZh}
     \TSa{\lBZ\chi\rBZh}{\atp}{m}(\spp)
\end{gather}
is absolutely convergent, locally uniformly in $\spp$, for $\Re(\spp)>1$.
\end{prop}

Propositions \ref{prop:modradsum:conver:Abs_Conv_sum_TS_chi} and
\ref{prop:modradsum:conver:Abs_Conv_sum_TSa_chi} imply the absolute
and locally uniform convergence of the expression defining the modified continued
Rademacher sum $\QS{\Gamma,\cp|\cq}{\atp}{m}(\zz,\spp)$ for $\zz\in
\HH$ and $\Re(\spp)>1$.
\begin{prop}\label{prop:modradsum:conver:Abs_Conv_QS}
Let $\Gamma$ be a group commensurable with $G(\ZZ)$, let $\cp,\cq\in
\cP_{\Gamma}$ be cusps of $\Gamma$, and let $\Sigma_{\cp}$ and
$\Sigma_{\cq}$ be scaling cosets for $\Gamma$ at $\cp$ and $\cq$,
respectively. Let $\atp\in \ZZ$ such that $\atp\leq 0$, and let
$m\in\ZZp$. Then the modified Rademacher sum
\begin{gather}
     \QS{\Gamma,\cp|\cq}{\atp}{m}(\zz,\spp)
     =
     \sum_{\lBZ\chi\rBZh\in \lBZ\Sigma_{\cp}^{-1}\Sigma_{\cq}\rBZh}
     \TS{\lBZ\chi\rBZh}{\atp}{m}(\zz,\spp)
     -
     \TSa{\lBZ\chi\rBZh}{\atp}{m}(\spp)
\end{gather}
converges absolutely, and locally uniformly in $\zz$ and $\spp$, for $\zz\in \HH$ and $\Re(\spp)>1$.
\end{prop}

We next seek to relate the modified Rademacher sums
$\QS{\Gamma,\cp|\cq}{\atp}{m}(\zz,\spp)$ to the functions
$\FR{\Gamma,\cp|\cq}{\atp}{m}(\zz,\spp)$ and
$\DS{\Gamma,\cp|\cq}{\atp}{m}(\spp)$ of \S\ref{sec:modradsum:coeff}
and \S\ref{sec:modradsum:dirser}, respectively.
\begin{prop}\label{prop:modradsum:conver:Eval_sum_QSa}
Let $\Gamma$ be a group commensurable with $G(\ZZ)$, let $\cp,\cq\in
\cP_{\Gamma}$ be cusps of $\Gamma$, with scaling matrices
$\Sigma_{\cp}$ and $\Sigma_{\cq}$, respectively. Let $\atp\in \ZZ$
such that $\atp\leq 0$ and let $m\in \ZZp$. Then we have
\begin{gather}\label{eqn:modradsum:conver:QSa_to_DS}
     \sum_{\lBZ\chi\rBZh\in \lBZ\Sigma_{\cp}^{-1}\Sigma_{\cq}\rBZh}
     \TSa{\lBZ\chi\rBZh}{\atp}{m}(\spp)
     =
     -\DS{\Gamma,\cp|\cq}{\atp}{m}(1-\spp)
\end{gather}
for $\Re(\spp)>1$ where the convergence is absolute and locally uniform.
\end{prop}
\begin{proof}
Set $U=\lBZ\Sigma_{\cp}^{-1}\Sigma_{\cq}\rBZh$ and
$S=\lBZ\Sigma_{\cp}^{-1}\Sigma_{\cq}\rBZ^{\times}$. Recall that
$\TSa{\lBZ\chi\rBZh}{\atp}{m}(\spp)=0$ when $\lBZ\chi\rBZh\notin
\lBZ G(\QQ)\rBZh^{\times\times}$, and otherwise
\begin{gather}
     \begin{split}
     \TSa{\lBZ\chi\rBZh}{\atp}{m}(\spp)
     =&
     \ee(m\lBZ\chi\rBZh\cdot\infty)
     \left(
     \tpi
     (
     m\lBZ\chi\rBZh\cdot 0
     -
     m\lBZ\chi\rBZh\cdot\infty
     )
     \right)^{(\spp-2\atp)}
     {\jac(\lBZ\chi\rBZh,0)^{\atp}}\\
     &-\ee(\spp/2)
     \ee(m\lBZ\chi\rBZh\cdot\infty)
     \left(
     \tpi
     (
     m\lBZ\chi\rBZh\cdot \infty
     -
     m\lBZ\chi\rBZh\cdot 0
     )
     \right)^{(\spp-2\atp)}
     {\jac(\lBZ\chi\rBZh,0)^{\atp}},
     \end{split}
\end{gather}
by the definition (cf. \S\ref{sec:modradsum:constr}) of
$\TSa{\lBZ\chi\rBZh}{\atp}{m}(\spp)$. We apply the identities
$\chi\cdot\infty-\chi\cdot 0=\rads(\chi)/(-\chi^{-1}\cdot\infty)$
and $\jac(\lBZ\chi\rBZh,0)=\rads(\chi)/(-\chi^{-1}\cdot\infty)^2$
(cf. \S\ref{sec:conven:cosets}) so as to write
\begin{gather}
     \TSa{\lBZ\chi\rBZh}{\atp}{m}(\spp)
     =
     \ee(m\lBZ\chi\rBZ\cdot\infty)
     \rads\lBZ\chi\rBZ^{\spp-\atp}
     (-\tpi m)^{(\spp-2\atp)}
     (
     (-\lBZh\chi^{-1}\rBZ\cdot\infty)^{-\spp}
     -\ee(\spp/2)
     (\lBZh\chi^{-1}\rBZ\cdot\infty)^{-\spp}
     ).
\end{gather}
Let $\lBZ\chi\rBZ\in \lBZ G(\QQ)\rBZ^{\times}$. Then we have
\begin{gather}
     \sum_{\lBZ\chi\rBZh\in \lBZ\chi\rBZ}
     \TSa{\lBZ\chi\rBZh}{\atp}{m}(\spp)
     =
     \ee(m\lBZ\chi\rBZ\cdot\infty)
     \rads\lBZ\chi\rBZ^{\spp-\atp}
     (-\tpi m)^{(\spp-2\atp)}
     {\sum_{n\in \ZZ}}^*
     (
     (n+\alpha_{\lBZ\chi\rBZ})^{-\spp}
     -\ee(\spp/2)
     (n-\alpha_{\lBZ\chi\rBZ})^{-\spp}
     )
\end{gather}
where $\alpha_{\lBZ\chi\rBZ}\in \QQ$ is chosen so that $0\leq
\alpha_{\lBZ\chi\rBZ}<1$ and
$\alpha_{\lBZ\chi\rBZ}+\ZZ=-\chi^{-1}\cdot\infty+\ZZ$, and the
superscript in the summation $\sum_{n\in\ZZ}^*$ indicates to omit
the term corresponding to $n=0$ in case $\alpha_{\lBZ\chi\rBZ}=0$.
We compute
\begin{gather}
     {\sum_{n\in \ZZ}}^*
     (
     (n+\alpha_{\lBZ\chi\rBZ})^{-\spp}
     -\ee(\spp/2)
     (n-\alpha_{\lBZ\chi\rBZ})^{-\spp}
     )
     =
     (\ee(-\spp/2)-\ee(\spp/2))
     \zeta(1-\alpha_{\lBZ\chi\rBZ},\spp)
\end{gather}
where $\zeta(\alpha,\spp)$ denotes the Hurwitz zeta function (cf.
\S\ref{sec:modradsum:dirser}), so we have
\begin{gather}\label{eqn:modradsum:conver:QSa_Hur_Zeta}
     \sum_{\lBZ\chi\rBZh\in U^{\times\times}}
     \TSa{\lBZ\chi\rBZh}{\atp}{m}(\spp)
     =
     \sum_{\lBZ\chi\rBZ\in S^{\times}}
     \ee(m\lBZ\chi\rBZ\cdot\infty)
     \rads\lBZ\chi\rBZ^{\spp-\atp}
     (-\tpi m)^{(\spp-2\atp)}
     (\ee(-\spp/2)-\ee(\spp/2))
     \zeta(1-\alpha_{\lBZ\chi\rBZ},\spp).
\end{gather}
Comparing with the expression
(\ref{eqn:modradsum:constr:Defn_DS_chi_p}) defining
$\DS{\lBZ\chi\rBZ}{\atp}{m}(\spp)'$ we see that the right hand side
of (\ref{eqn:modradsum:conver:QSa_Hur_Zeta}) coincides with
$-\DS{S^{\times}}{\atp}{m}(\spp)'$. Now Lemma
\ref{lem:modradsum:constr:DS_to_DSp} yields the required identity
(\ref{eqn:modradsum:conver:QSa_to_DS}). This completes the proof.
\end{proof}

\begin{thm}\label{thm:modradsum:conver:Relate_QS_DS_FR}
Let $\Gamma$ be a group commensurable with $G(\ZZ)$ and let
$\cp,\cq\in \cP_{\Gamma}$ be cusps of $\Gamma$. Let $\atp,m\in \ZZ$
such that $\atp\leq 0$ and $m<0$. Then we have
\begin{gather}\label{eqn:modradsum:conver:Relate_QS_DS_FR}
     \QS{\Gamma,\cp|\cq}{\atp}{m}(\zz,\spp)
     =
     \delta_{\Gamma,\cp|\cq}\ee(-m\zz)
     +
     \DS{\Gamma,\cp|\cq}{\atp}{m}(1-\spp)
     +
     \FR{\Gamma,\cp|\cq}{\atp}{m}(\zz,\spp)_{\rm van}
\end{gather}
for $\zz\in \HH$ and $\Re(\spp)>1$.
\end{thm}
\begin{proof}
Let $m<0<\atp$, let $\spp\in\CC$ with $\Re(\spp)>1$, set
$U=\lBZ\Sigma_{\cp}^{-1}\Sigma_{\cq}\rBZh$ and set
$S=\lBZ\Sigma_{\cp}^{-1}\Sigma_{\cq}\rBZ^{\times}$. By the defining
properties of scaling cosets (cf. \S\ref{sec:conven:scaling}) and by
the definition (\ref{eqn:modradsum:constr:Defn_TS_chi}) of the
continued Rademacher component function
$\TS{\lBZ\chi\rBZh}{1-\atp}{m}(\zz,\spp)$ we have
\begin{gather}
     \sum_{\lBZ\chi\rBZh\in U}
     \TS{\lBZ\chi\rBZh}{1-\atp}{m}(\zz,\spp)
     =\delta_{\Gamma,\cp|\cq}
     \ee(m\zz)
     +
     \sum_{\lBZ\chi\rBZh\in U^{\times}}
     \TS{\lBZ\chi\rBZh}{1-\atp}{m}(\zz,\spp)
\end{gather}
so by Proposition \ref{prop:modradsum:conver:Eval_sum_QSa}, and the
definition of the modified continued Rademacher component function
$\QS{\lBZ\chi\rBZh}{1-\atp}{m}(\zz,\spp)$, it suffices for us to
show that
\begin{gather}\label{eqn:modradsum:conver:TScont_to_FRcont}
     \sum_{\lBZ\chi\rBZh\in
     U^{\times}}
     \TS{\lBZ\chi\rBZh}{1-\atp}{m}(\zz,\spp)
     =\FR{S}{1-\atp}{m}(\zz,\spp)_{\rm van}.
\end{gather}
We will verify the equality of
(\ref{eqn:modradsum:conver:TScont_to_FRcont}) by using the Lipschitz
summation formula (\ref{eqn:radsum:conver:Lipschitz_spp>1}) to
transform the expression on the right hand side of
(\ref{eqn:modradsum:conver:TScont_to_FRcont}) into that on the left.
This is essentially the approach employed originally by Rademacher
in \cite{Rad_FuncEqnModInv}, and the reverse of that employed by
Niebur in Lemma 4.2 of \cite{Nie_ConstAutInts}, except that we can
avoid the technical difficulties of \cite{Rad_FuncEqnModInv} and
\cite{Nie_ConstAutInts}, such as the need to employ the identity
(\ref{eqn:radsum:conver:Lipschitz_spp=1}), by working only with
the versions of the Lipschitz summation formula (viz., $\Re(\spp)>1$) in which both sides of the identity are absolutely and locally uniformly convergent series.

We begin by inspecting Lemma
\ref{lem:radsum:coeff:Power_Series_fcS_cont} to find that
\begin{gather}
     \begin{split}
     &(-1)^{1-\atp}
     \fc{S}{1-\atp}(m,n,\spp)\\
     &=
     \sum_{\lBZ\chi\rBZ\in S}
     \sum_{k\geq 0}
     \ee(m\lBZ\chi\rBZ\cdot\infty)
     \ee(-n\lBZ\chi^{-1}\rBZ\cdot\infty)
     (4\pi^2\rads\lBZ\chi\rBZ)^{k+\atp+\spp-1}
     |m|^{(k+2\atp+\spp-2)}
     n^{(k+\spp-1)},
     \end{split}
\end{gather}
As a shorthand let us set $F_{\rm
van}=\FR{S}{1-\atp}{m}(\zz,\spp)_{\rm van}$. Then we have the
following expression for $F_{\rm van}$ as a triple sum
\begin{gather}\label{eqn:modradsum:conver:Fvan_nchik}
     \begin{split}
     &(-1)^{1-\atp}
     F_{\rm van}
     =\\
     &\sum_{n>0}
     \sum_{\lBZ\chi\rBZ\in S}^{\lim}
     \sum_{k\geq 0}
          \ee(m\lBZ\chi\rBZ\cdot\infty)
          \ee(-n\lBZ\chi^{-1}\rBZ\cdot\infty)
          (4\pi^2 \rads\lBZ\chi\rBZ)^{k+\atp+\spp-1}
          |m|^{(k+2\atp+\spp-2)}n^{(k+\spp-1)}\ee(n\zz)
     \end{split}
\end{gather}
where $\sum_{\lBZ\chi\rBZ\in S}^{\lim}$ is a shorthand for $\lim_{K\to \infty}\sum_{\lBZ\chi\rBZ\in S_{\leq K}}$. We now move the summation over $n$ past the other two summations,
simultaneously pulling the terms $\ee(m\lBZ\chi\rBZ\cdot\infty)$
outside the summation over $k$, and combining the exponents
involving $n$. We thus obtain
\begin{gather}\label{eqn:modradsum:conver:Fvan_chikn}
     \begin{split}
     &(-1)^{1-\atp}F_{\rm van}
     =\\
     &\sum_{\lBZ\chi\rBZ\in S}^{\lim}
          \ee(m\lBZ\chi\rBZ\cdot\infty)
     \sum_{k\geq 0}
          (4\pi^2 \rads\lBZ\chi\rBZ)^{k+\atp+\spp-1}
          |m|^{(k+2\atp+\spp-2)}
               \sum_{n>0}
               n^{(k+\spp-1)}
      \ee(n(\zz-\lBZ\chi^{-1}\rBZ\cdot\infty)).
      \end{split}
\end{gather}
Applying the Lipschitz summation formula
(\ref{eqn:radsum:conver:Lipschitz_spp>1}) to each summation over $n$
in (\ref{eqn:modradsum:conver:Fvan_chikn}) we find that
\begin{gather}
     \begin{split}
     &(-1)^{1-\atp}F_{\rm van}=\\
    &\sum_{\lBZ\chi\rBZ\in S}^{\lim}
          \ee(m\lBZ\chi\rBZ\cdot\infty)
     \sum_{k\geq 0}
          (\fps\rads\lBZ\chi\rBZ)^{\atp-1}
          |m|^{(k+2\atp+\spp-2)}
     \sum_{n\in\ZZ}
     (\tpi\rads\lBZ\chi\rBZ)^{k+\spp}
     (\zz+n-\lBZh\chi^{-1}\rBZ\cdot\infty)^{-k-\spp}
     \end{split}
\end{gather}
where $\lBZh \chi^{-1}\rBZ $ is any (left) coset of $B(\ZZ)$ in the
double coset $\lBZ\chi^{-1}\rBZ$, for each $\lBZ\chi\rBZ\in S$.
Using the identity (\ref{eqn:conven:isomhyp:jac_rads_over_zchi}) we
write $F_{\rm van}$ as $F_{\rm van}=\sum_{\lBZ\chi\rBZ\in
S}^{\lim}\ee(m\lBZ\chi\rBZ\cdot\infty)F_{{\rm van},\lBZ\chi\rBZ}$
where
\begin{gather}
     F_{{\rm van},\lBZ\chi\rBZ}
     =
     \sum_{k\geq 0}
     \sum_{n\in\ZZ}
     (-\tpi\rads\lBZ\chi\rBZ m)^{k+2\atp+\spp-2}
     (\zz+n-\lBZh\chi^{-1}\rBZ\cdot\infty)^{-k-2\atp-\spp+2}
     \jac(\lBZ\chi\rBZh,\zz+n)^{1-\atp}.
\end{gather}
Recalling the definition of the generalized exponential function
$\ee(\zz,\spp)$ from (\ref{eqn:conven:fns:Genzd_Exp}) we see that
\begin{gather}\label{eqn:modradsum:conver:Fvan_sum_over_U}
     F_{\rm van}
     =
     \sum_{\lBZ\chi\rBZh\in U^{\times}}
          \ee(m\lBZ\chi\rBZ\cdot\infty)
               \ee
               \left(
                    \frac{-\rads\lBZ\chi\rBZ m}
                    {\zz-\lBZh \chi^{-1}\rBZ \cdot\infty}
                    ,
                    \spp-2(1-\atp)
               \right)
     \jac(\lBZ\chi\rBZh,\zz)^{1-\atp}.
\end{gather}
According to (\ref{eqn:conven:cosets_UsefulChiDotTau}) we have
$\rads\lBZ\chi\rBZ/(\zz-\chi^{-1}\cdot\infty)=\chi\cdot\infty-\chi\cdot\zz$.
Using this identity together with the Kummer transformation
$\Phi(a,b,\zz)=\ee(\zz)\Phi(b-a,b,-\zz)$ (cf.
(\ref{eqn:conven:fns:Defn_Phi})) we readily find that
\begin{gather}
     \ee(m\lBZ\chi\rBZ\cdot\infty)
     \ee\left(
     \frac{-\rads\lBZ\chi\rBZ m}
     {\zz-\lBZh \chi^{-1}\rBZ \cdot\infty}
     ,
     \spp-2(1-\atp)
     \right)
     =
     \ee(m\lBZ\chi\rBZ\cdot\zz)
     \Treg{1-\atp}(m,\lBZ\chi\rBZh,\zz,\spp),
\end{gather}
so that we have
\begin{gather}
     F_{\rm van}
     =
     \sum_{\lBZ\chi\rBZh\in U^{\times}}
     \TS{\lBZ\chi\rBZh}{1-\atp}{m}(\zz,\spp).
\end{gather}
This establishes the required identity
(\ref{eqn:modradsum:conver:TScont_to_FRcont}).
\end{proof}
In the course of proving Theorem
\ref{thm:modradsum:conver:Relate_QS_DS_FR} we have also established
the Fourier expansion of the continued Rademacher sums
$\TS{\Gamma,\cp|\cq}{\atp}{m}(\zz,\spp)$. We record the result as
follows.
\begin{thm}\label{thm:modradsum:conver:Relate_TS_FR}
Let $\Gamma$ be a group commensurable with $G(\ZZ)$ and let
$\cp,\cq\in \cP_{\Gamma}$ be cusps of $\Gamma$. Let $\atp,m\in \ZZ$ such that $\atp\leq 0$ and $m<0$. Then we have
\begin{gather}\label{eqn:modradsum:conver:Relate_TS_FR}
     \TS{\Gamma,\cp|\cq}{\atp}{m}(\zz,\spp)
     =
     \delta_{\Gamma,\cp|\cq}\ee(m\zz)
     +
     \FR{\Gamma,\cp|\cq}{\atp}{m}(\zz,\spp)_{\rm van}
\end{gather}
for $\zz\in \HH$ and $\Re(\spp)>1$.
\end{thm}

Theorem \ref{thm:modradsum:conver:Relate_QS_DS_FR} will facilitate
the identification of the Fourier expansion of the modified
Rademacher sum $\QS{\Gamma,\cp|\cq}{\atp}{m}(\zz)$ obtained by
sending $\spp$ to $1$ in $\QS{\Gamma,\cp|\cq}{\atp}{m}(\zz,\spp)$.
Indeed, we will see presently (cf. Proposition
\ref{prop:modradsum:conver:Relate_QS_FR_s=1}) that the series
$\FS{\Gamma,\cp|\cq}{\atp}{m}(\vq)_{\rm reg}$ encodes the regular
part of the fourier series expansion of
$\QS{\Gamma,\cp|\cq}{\atp}{m}(\zz)$ precisely, in the case that
$\atp\leq 0$. In preparation for this we state the following
result which identifies the value of the zeta function
$\DS{\Gamma,\cp|\cq}{\atp}{m}(\spp)$ at $\spp=0$.
\begin{prop}\label{prop:modradsum:conver:DS_s=0}
Let $S\subset\lBZ G(\QQ)\rBZ^{\times}$, let $\atp,m\in\ZZ$ such that $\atp\leq 0$ and $m<0$. Then we have
\begin{gather}
     \DS{S}{\atp}{m}(0)
     =\fc{S}{\atp}(m,0).
\end{gather}
\end{prop}
\begin{proof}
According to the proof of Lemma \ref{lem:modradsum:constr:DS_to_DSp}
the function $\Gamma(\spp)\DS{\lBZ\chi\rBZ}{\atp}{m}(\spp)$
coincides with
\begin{gather}\label{eqn:modradsum:constr:Zeta_Fn_Per_Riem_Zeta_Fn}
     (-1)^{1-\atp}
     \ee(m\lBZ\chi\rBZ\cdot\infty)
     (\fps\rads\lBZ\chi\rBZ)^{1-\spp-\atp}
     |m|^{(1-\spp-2\atp)}
     (F(\alpha_{\lBZ\chi\rBZ},\spp)
     +
     \ee(-\spp/2)F(-\alpha_{\lBZ\chi\rBZ},\spp))
\end{gather}
where $0\leq \alpha_{\lBZ\chi\rBZ}<1$ and
$\alpha_{\lBZ\chi\rBZ}+\ZZ=-\chi^{-1}\cdot\infty+\ZZ$. We have
$F(\alpha,0)+F(-\alpha,0)=-1$ for $\alpha\in \RR$ (cf.
\cite{KnoRob_RieFnlEqnLipSum}), so taking $\spp=0$ in
(\ref{eqn:modradsum:constr:Zeta_Fn_Per_Riem_Zeta_Fn}) we obtain
\begin{gather}
     \DS{\lBZ\chi\rBZ}{\atp}{m}(0)
     =
     %(-1)^{\atp}
     \ee(m\lBZ\chi\rBZ\cdot\infty)
     (-\fps\rads\lBZ\chi\rBZ)^{1-\atp}
     m^{(1-2\atp)}.
\end{gather}
Upon inspection of (\ref{eqn:radsum:coeff:Pow_Ser_fcS_1-s_pos}) we
see that
\begin{gather}
     \fc{S}{\atp}(m,0)=%(-1)^{\atp}
     \sum_{\lBZ\chi\rBZ\in S}
     \ee(m\lBZ\chi\rBZ\cdot\infty)
     (-\fps\rads\lBZ\chi\rBZ)^{1-\atp}
     m^{(1-2\atp)}
\end{gather}
for $\atp\leq 0$, and the claim now follows since
$\DS{S}{\atp}{m}(\spp)=\sum_{\lBZ\chi\rBZ\in
S}\DS{\lBZ\chi\rBZ}{\atp}{m}(\spp)$ by definition.
\end{proof}

\begin{prop}\label{prop:modradsum:conver:Relate_QS_FR_s=1}
Let $\Gamma$ be a group commensurable with $G(\ZZ)$, and let
$\cp,\cq\in \cP_{\Gamma}$ be cusps of $\Gamma$. Let $\atp,m\in \ZZ$
such that $\atp\leq 0$ and $m<0$. Then we have
\begin{gather}\label{eqn:modradsum:conver:Relate_QS_FR_s=1}
     \QS{\Gamma,\cp|\cq}{\atp}{m}(\zz)
     =
     \delta_{\Gamma,\cp|\cq}\ee(m\zz)
     +
     \FR{\Gamma,\cp|\cq}{\atp}{m}(\zz)_{\rm reg}
\end{gather}
for $\zz\in \HH$.
\end{prop}
\begin{proof}
We take the limit as $\spp$ tends to $1$ in the right hand side of
the identity (\ref{eqn:modradsum:conver:Relate_QS_DS_FR}). We have
$\FR{\Gamma,\cp|\cq}{\atp}{m}(\zz,1)_{\rm
van}=\FR{\Gamma,\cp|\cq}{\atp}{m}(\zz)_{\rm van}$ by definition (cf.
\S\ref{sec:modradsum:coeff}), and we have
$\DS{\Gamma,\cp|\cq}{\atp}{m}(0)=\fc{\Gamma,\cp|\cq}{\atp}(m,0)$ by
Proposition \ref{prop:modradsum:conver:DS_s=0}. The required
identity (\ref{eqn:modradsum:conver:Relate_QS_FR_s=1}) now follows
from the fact that $\FR{\Gamma,\cp|\cq}{\atp}{m}(\zz)_{\rm
reg}=\fc{\Gamma,\cp|\cq}{\atp}(m,0)+\FR{\Gamma,\cp|\cq}{\atp}{m}(\zz)_{\rm
van}$, also by definition (cf. \S\ref{sec:radsum:coeff}).
\end{proof}
The proof of Proposition
\ref{prop:modradsum:conver:Relate_QS_FR_s=1} implies a result
analogous to (\ref{eqn:modradsum:conver:Relate_QS_FR_s=1}) for the
normalized Rademacher sum $\TS{\Gamma,\cp|\cq}{\atp}{m}(\zz)$.
\begin{prop}\label{prop:modradsum:conver:Relate_TS_FR_s=1}
Let $\Gamma$ be a group commensurable with $G(\ZZ)$ and let
$\cp,\cq\in \cP_{\Gamma}$ be cusps of $\Gamma$. Let $\atp,m\in \ZZ$ such that $\atp\leq 0$ and $m<0$.
Then we have
\begin{gather}\label{eqn:modradsum:conver:Relate_TS_FR_s=1}
     \TS{\Gamma,\cp|\cq}{\atp}{m}(\zz)
     =
     \delta_{\Gamma,\cp|\cq}\ee(m\zz)
     +
     \FR{\Gamma,\cp|\cq}{\atp}{m}(\zz)_{\rm van}
\end{gather}
for $\zz\in \HH$.
\end{prop}
We may compare the functions $\QS{\Gamma,\cp|\cq}{\atp}{m}(\zz)$ and
$\RS{\Gamma,\cp|\cq}{\atp}{m}(\zz)$ defined by the modified and
classical Rademacher sums, respectively. Inspecting Theorem
\ref{thm:radsum:conver:Relate_RS_FR} and Proposition
\ref{prop:modradsum:conver:Relate_QS_FR_s=1} we obtain the precise
relationship, which we record in the following proposition.
\begin{prop}\label{prop:modradsum:conver:Relate_QS_s=1_RS}
Let $\Gamma$ be a group commensurable with $G(\ZZ)$ and let
$\cp,\cq\in \cP_{\Gamma}$ be cusps of $\Gamma$. Let $\atp,m\in\ZZ$
such that $\atp\leq 0$ and $m<0$. Then we have
\begin{gather}\label{eqn:modradsum:conver:Relate_QS_RS}
     \QS{\Gamma,\cp|\cq}{\atp}{m}(\zz)
     =
     \RS{\Gamma,\cp|\cq}{\atp}{m}(\zz)
     +\frac{1}{2}
     \fc{\Gamma,\cp|\cq}{\atp}(m,0).
\end{gather}
\end{prop}
According to Propositions
\ref{prop:modradsum:conver:Relate_QS_FR_s=1} and
\ref{prop:modradsum:conver:Relate_TS_FR_s=1} the modified Rademacher
sum $\QS{\Gamma,\cp|\cq}{\atp}{m}(\zz)$ and the normalized
Rademacher sum $\TS{\Gamma,\cp|\cq}{\atp}{m}(\zz)$ also differ only
by a constant function.
\begin{prop}\label{prop:modradsum:conver:Relate_QS_TS_s=1}
Let $\Gamma$ be a group commensurable with $G(\ZZ)$ and let
$\cp,\cq\in \cP_{\Gamma}$ be cusps of $\Gamma$. Let $\atp,m\in \ZZ$
such that $\atp\leq 0$ and $m<0$. Then we have
\begin{gather}\label{eqn:modradsum:conver:Relate_QS_TS_s=1}
     \QS{\Gamma,\cp|\cq}{\atp}{m}(\zz)
     =
     \TS{\Gamma,\cp|\cq}{\atp}{m}(\zz)
     +\fc{\Gamma,\cp|\cq}{\atp}(m,0).
\end{gather}
\end{prop}

\subsection{Variance}\label{sec:modradsum:var}

In this section we examine how the modified Rademacher sum
$\QS{\Gamma,\cp}{\atp}{m}(\zz)$ transforms under the weight $2\atp$
action of $\Gamma$. Combining Theorem
\ref{thm:radsum:var:Var_RS_atp<1} and Proposition
\ref{prop:modradsum:conver:Relate_QS_s=1_RS} we obtain the result
that the modified Rademacher sum $\QS{\Gamma,\cp}{\atp}{m}(\zz)$ is
an automorphic integral of weight $2\atp$ for $\Gamma$ in case
$\Gamma$ has width one at infinity.
\begin{thm}\label{thm:modradsum:var:QS_is_aut_int}
Let $\Gamma$ be a group commensurable with $G(\ZZ)$ that has width
one at infinity and let $\cp\in \cP_{\Gamma}$ be a cusp of $\Gamma$.
Let $\atp,m\in\ZZ$ such that $\atp\leq 0$ and $m<0$.
Then the modified Rademacher sum $\QS{\Gamma,\cp}{\atp}{m}(\zz)$ is
an automorphic integral of weight $2\atp$ for $\Gamma$, and for $\cq\in\cP_{\Gamma}$ another cusp of $\Gamma$ the Fourier expansion of the function $\QS{\Gamma,\cp|\cq}{\atp}{m}(\zz)$ is the expansion of $\QS{\Gamma,\cp}{\atp}{m}(\zz)$ at the cusp $\cq$.
\end{thm}
The proof of Theorem \ref{thm:radsum:var:Var_RS_atp<1} yields the
following explicit description of the associated cusp form map
$I_{\atp}(\Gamma)\to S_{1-\atp}(\Gamma)$ on the subspace of
$I_{\atp}(\Gamma)$ spanned by the modified Rademacher sums
$\QS{\Gamma,\cp}{\atp}{m}(\zz)$.
\begin{prop}\label{prop:modradsum:var:acf_on_QS}
Let $\Gamma$ be a group commensurable with $G(\ZZ)$ that has width
one at infinity and let $\cp\in \cP_{\Gamma}$ be a cusp of $\Gamma$.
Let $\atp,m\in\ZZ$ such that $\atp\leq 0<m$.
Then
\begin{gather}
     m^{\atp}
     \QS{\Gamma,\cp}{\atp}{-m}
     \mapsto
     m^{1-\atp}
     \PS{\Gamma,\cp}{1-\atp}{m}
     \mapsto
     0,
\end{gather}
under the maps $I_{\atp}(\Gamma)\mapsto S_{1-\atp}(\Gamma)$ of
\S\ref{sec:conven:autfrm}.
\end{prop}
We next seek to establish the utility of Proposition
\ref{prop:modradsum:var:acf_on_QS} by identifying the subspace of
$I_{\atp}(\Gamma)$ that is spanned by the modified Rademacher sums
$\QS{\Gamma,\cp}{\atp}{m}(\zz)$.
\begin{thm}\label{thm:modradsum:var:Basis_QS_p_m}
Let $\Gamma$ be a group commensurable with $G(\ZZ)$ that has width
one at infinity, and let $\atp\in\ZZ$ with $\atp\leq 0$. Then
the set $\{\QS{\Gamma,\cp}{\atp}{-m}(\zz)\mid \cp\in\cP_{\Gamma},
m\in \ZZp\}$ is a basis for the space of automorphic integrals of
weight $2\atp$ for $\Gamma$ in case $\atp<0$. When $\atp=0$ the set
$\{\QS{\Gamma,\cp}{}{-m}(\zz)\mid \cp\in\cP_{\Gamma}, m\in \ZZp\}$
spans a subspace of $I_{0}(\Gamma)$ of codimension $1$.
\end{thm}
\begin{proof}
For $m>0$ the function $\QS{\Gamma,\cp}{\atp}{-m}(\zz)$ has principal part
$\ee(-m\zz)$ at the cusp $\cp$, and vanishes at all the other cusps,
so the collection $\{\QS{\Gamma,\cp}{\atp}{-m}(\zz)\mid
\cp\in\cP_{\Gamma}, m\in \ZZp\}$ is linearly independent. To show
that it furnishes a basis we follow Niebur's proof of Theorem 3.3 in
\cite{Nie_ConstAutInts}, employing Petersson's generalized
Riemann--Roch Theorem to show that the dimension of the subspace of
$I_{\atp}(\Gamma)$ containing automorphic integrals with poles
of degree not more than $m$ say at $\cp$ is bounded above by $m$ less
the dimension of the space $S_{1-\atp}(\Gamma)$, for $m$
sufficiently large.
\end{proof}
We write $I_0'(\Gamma)$ for the subspace of $I_0(\Gamma)$ spanned by
the modified Rademacher sums $\QS{\Gamma,\cp}{}{-m}(\zz)$ of weight
$0$ for varying $\cp\in \cP_{\Gamma}$ and $m\in \ZZp$.
\begin{thm}\label{thm:modradsum:var:MIS_exact}
Let $\Gamma$ be a group commensurable with $G(\ZZ)$ that has width
one at infinity. Then for any $\atp\in \ZZ$ the sequence
\begin{gather}\label{eqn:modradsum:var:MIS_exact}
     0
     \to
     M_{\atp}(\Gamma)
     \to
     I_{\atp}(\Gamma)
     \to
     S_{1-\atp}(\Gamma)
     \to
     0
\end{gather}
is exact.
\end{thm}
\begin{proof}
The second map of (\ref{eqn:modradsum:var:MIS_exact}) is an
inclusion by definition, and we observed already in
\S\ref{sec:conven:autfrm} that $M_{\atp}(\Gamma)$ is the kernel of
the associated cusp form map $I_{\atp}(\Gamma)\to
S_{1-\atp}(\Gamma)$ (cf. Lemma
\ref{lem:conven:autfrm:Faithfulness_I_p_atp}), so we require to show
only that the second to last map of
(\ref{eqn:modradsum:var:MIS_exact}) is surjective. For this we
recall the fact (cf. \cite{Iwa_TopsAutFrms}) that the holomorphic
Poincar\'e series $\PS{\Gamma,\cp}{1-\atp}{m}(\zz)$, with varying
$m\in \ZZp$, span the space $S_{1-\atp}(\Gamma)$ of cusp forms of
weight $2-2\atp$ for $\Gamma$ whenever $\atp\leq 0$. Indeed, one
need not vary the cusp $\cp$ in order to obtain a spanning set.
Given an arbitrary cusp form $g\in S_{1-\atp}(\Gamma)$ we may then
write
\begin{gather}
     g
     =
     \sum_{n>0}
     a_{n}\PS{\Gamma}{1-\atp}{n}
\end{gather}
for some $a_{n}\in \CC$, with only finitely many $a_{n}$ non-zero.
Employing Proposition \ref{prop:modradsum:var:acf_on_QS} we see that
the automorphic integral
\begin{gather}
     f
     =
     \sum_{n>0}
     n^{2\atp-1}a_{n}
     \QS{\Gamma}{\atp}{-n}
\end{gather}
is mapped to $g$ by the associated cusp form map. We conclude that
the associated cusp form map is surjective. This completes the
proof.
\end{proof}
Theorem \ref{thm:modradsum:var:Basis_QS_p_m} and
Proposition \ref{prop:modradsum:var:acf_on_QS} together describe the
associated cusp form map $I_{\atp}(\Gamma)\to S_{1-\atp}(\Gamma)$
explicitly, and thus we obtain a powerful criterion for determining
when an automorphic integral $f\in I_{\atp}(\Gamma)$ lies in the
subspace $M_{\atp}(\Gamma)$ of modular forms. For example, we have
the following corollary, indicating exactly when the modified
Rademacher sum $\QS{\Gamma,\cp}{\atp}{m}(\zz)$ is a modular form for
$\Gamma$ in the sense of \S\ref{sec:conven:autfrm}.
\begin{cor}\label{cor:modradsum:var:Relate_QS_inv_to_PS_van}
Let $\Gamma$ be a group commensurable with $G(\ZZ)$ that has width
one at infinity and let $\cp\in \cP_{\Gamma}$ be a cusp of $\Gamma$.
Let $\atp,m\in\ZZ$ such that $\atp\leq 0$ and $m<0$.
Then the modified Rademacher sum $\QS{\Gamma,\cp}{\atp}{m}(\zz)$ is
a modular form of weight $2\atp$ for $\Gamma$ if and only if the
cusp form $\PS{\Gamma,\cp}{1-\atp}{-m}(\zz)$ vanishes identically.
\end{cor}

%------------------------------------------------------------------%
\section{Structural applications}\label{sec:struapp}
%------------------------------------------------------------------%

We have seen already in \S\ref{sec:modradsum} various applications
of the modified Rademacher sums, and the continuation procedure
introduced in \S\ref{sec:modradsum:constr}, such as the basis
theorem (Theorem \ref{thm:modradsum:var:Basis_QS_p_m}) for
automorphic integrals, and the explicit description (cf. Proposition
\ref{prop:modradsum:var:acf_on_QS}) of the associated cusp form map
$I_{\atp}(\Gamma)\to S_{1-\atp}(\Gamma)$. In this section we explore
further consequences of the modified Rademacher sum construction for
the structure of the spaces of automorphic integrals associated to
groups commensurable with the modular group. Applications of the
modified Rademacher sums to monstrous moonshine and quantum gravity
will be developed in \S\S\ref{sec:moon},\ref{sec:gravity}.

\subsection{Constants}\label{sec:struapp:const}

Perhaps the most striking application of the continuation procedure
of \S\ref{sec:modradsum:constr} is the correction of the constant
terms appearing in the classical Rademacher sums. For negative
weights, the constant term in the Fourier expansion of an
automorphic integral $f\in I_{\atp}(\Gamma)$ is determined by the
automorphy condition (\ref{eqn:conven:autfrm:AutInt_Xform}). At
weight $0$ constant functions are themselves automorphic integrals,
so the classical Rademacher sums $\RS{\Gamma,\cp}{}{m}(\zz)$ of
weight $0$ do not fail to be automorphic, and it is natural then to
wonder about the significance of the constant term
$\fc{\Gamma,\cp}{}(m,0)$ appearing in the Fourier expansion of the
modified Rademacher sum $\QS{\Gamma,\cp}{}{m}(\zz)$ of weight $0$
associated to a group $\Gamma$ at a cusp $\cp\in \cP_{\Gamma}$.

Let $\commod$ denote the set of subgroups of $G(\QQ)$ that are
commensurable with $G(\ZZ)$. Recall from \S\ref{sec:conven:autfrm}
that for $\Gamma\in \commod$ we write $M_0(\Gamma)$ for the space of
holomorphic functions on $\HH$ that are invariant for the natural
action of $\Gamma$ and are meromorphic at the cusps of $\Gamma$. Let
$\gt{M}_0$ be the union of the spaces $M_0(\Gamma)$ for $\Gamma\in
\commod$.
\begin{gather}
     \gt{M}_0=\bigcup_{\Gamma\in\commod}M_0(\Gamma)
\end{gather}
Since $\commod$ is closed under intersections, the set $\gt{M}_0$ is
in fact a subalgebra of the space $\mc{O}(\HH)$ of holomorphic
functions on $\HH$. Since $G(\QQ)$ is the commensurator of $G(\ZZ)$,
the algebra $\gt{M}_0$ is $G(\QQ)$-invariant. The constant functions
evidently constitute a $G(\QQ)$-invariant subspace $\CC
1\subset\gt{M}_0$. We may enquire as to the $G(\QQ)$-module
structure of the quotient $\gt{M}_0/\CC 1$. According to
\cite{Nor_MoreMoons} this quotient is irreducible, and the natural
map $\gt{M}_0\to\gt{M}_0/\CC 1$ admits a section.
\begin{thm}[\cite{Nor_MoreMoons}]\label{thm:struapp:const:G(Q)_Decomp}
There exists a unique $G(\QQ)$-submodule $\gt{M}_0'\subset \gt{M}_0$
with the property that the natural inclusions $\CC
1,\gt{M}_0'\subset\gt{M}_0$ induce an isomorphism
\begin{gather}
     \gt{M}_0\cong \CC 1\oplus \gt{M}_0'
\end{gather}
of $G(\QQ)$-modules.
\end{thm}
According to Theorem \ref{thm:struapp:const:G(Q)_Decomp}, for each
$f\in \gt{M}_0$ there is a unique $c(f)\in \CC$ with the property
that $f+c(f)\in \gt{M}_0'$. The value $c(f)$ is called the {\em
Rademacher constant of $f$}. The assignment $f\mapsto c(f)$ defines
a linear function on $\gt{M}_0$ with kernel $\gt{M}_0'$ which we
call the {\em Rademacher constant function}. We suggest that the
Rademacher constant function may be generalized to automorphic integrals of weight $0$ in the following way. Let
$\gt{I}_0$ be the union of the spaces $I_0(\Gamma)$ for $\Gamma\in
\commod$.
\begin{gather}
     \gt{I}_0=\bigcup_{\Gamma\in\commod}I_0(\Gamma)
\end{gather}
Since $\commod$ is closed under intersections, the set $\gt{I}_0$ is
also a subalgebra of $\mc{O}(\HH)$, and since $G(\QQ)$ is the commensurator of
$G(\ZZ)$, the algebra $\gt{I}_0$ is also stable under the action of
$G(\QQ)$. The constant functions again constitute a
$G(\QQ)$-invariant subspace $\CC 1\subset\gt{I}_0$. We conjecture
that the natural analogue of Theorem
\ref{thm:struapp:const:G(Q)_Decomp} holds for $\gt{I}_0$. Precisely,
we conjecture that there is a unique $G(\QQ)$-submodule
$\gt{I}_0'\subset \gt{I}_0$ with the property that the natural
inclusions induce an isomorphism
\begin{gather}
     \gt{I}_0\cong \CC 1\oplus \gt{I}_0'
\end{gather}
of $G(\QQ)$-modules. Further, we conjecture that $\gt{I}_0'$ is just
the space spanned by the modified Rademacher sums
$\QS{\Gamma,\cp}{\atp}{-m}(\zz)$ for $\Gamma\in \commod$ and $\cp\in
\cP_{\Gamma}$ and $m\in \ZZp$.

\subsection{Inner products}\label{sec:struapp:iprods}

Let $\Gamma$ be a group commensurable with $G(\ZZ)$ and let $\atp\in
\ZZp$. Then the space $S_{\atp}(\Gamma)$ of cusp forms of weight
$2\atp$ for $\Gamma$ becomes a Hilbert space when equipped with the
{\em Petersson inner product}, defined by setting
\begin{gather}
     \left\langle f,g\right\rangle
     =
     \int_{\gt{F}_{\Gamma}}
     f(\zz)\overline{g(\zz)}
     \Im(\zz)^{2\atp}
     {\rm d}\mu(\zz)
\end{gather}
for $f,g\in S_{\atp}(\Gamma)$, where $\gt{F}_{\Gamma}$ is a
fundamental domain for $\Gamma$. Let us define a {\em normalized
Petersson inner product of weight $2\atp$ for $\Gamma$}, to be
denoted $\lab\cdot\,,\cdot\rab_{\Gamma}^{\atp}$, by setting
\begin{gather}
     \left\langle f,g\right\rangle_{\Gamma}^{\atp}
     =
     {(4\pi)^{(2\atp-2)}}
     \int_{\gt{F}_{\Gamma}}
     f(\zz)\overline{g(\zz)}
     \Im(\zz)^{2\atp}
     {\rm d}\mu(\zz)
\end{gather}
for $f,g\in S_{\atp}(\Gamma)$. It is well known (and follows from
Theorem \ref{thm:radsum:var:Var_RS_atp>0}), that for $\atp,m\in
\ZZp$ and $\cp$ a cusp for $\Gamma$, the holomorphic Poincar\'e
series $\PS{\Gamma,\cp}{\atp}{m}(\zz)$ is a cusp form of weight
$2\atp$ for $\Gamma$ in case $\Gamma$ has width one at infinity. A
standard calculation (cf. \cite{Iwa_TopsAutFrms}) shows that we have
\begin{gather}\label{eqn:struapp:iprods:Defn_Norm_Petersson_IP}
     \left\langle
     \PS{\Gamma,\cp}{\atp}{m},
     \PS{\Gamma,\cq}{\atp}{n}
     \right\rangle_{\Gamma}^{\atp}
     =
     n^{2-2\atp}
     (
     \delta_{\Gamma,\cp|\cq}
     \delta_{m,n}
     +\fc{\Gamma,\cp|\cq}{\atp}(m,n)
     )
\end{gather}
for $\cp,\cq\in \cP_{\Gamma}$ and $m,n\in \ZZp$. According to
Theorem \ref{thm:radsum:var:Var_RS_atp>0} the right hand side of
(\ref{eqn:struapp:iprods:Defn_Norm_Petersson_IP}) is, up to the
scalar factor $n^{2-2\atp}$, the coefficient of $\ee(n\zz)$ in the
expansion $\PS{\Gamma,\cp|\cq}{\atp}{m}(\zz)$ of the Poincar\'e
series $\PS{\Gamma,\cp}{\atp}{m}(\zz)$ at $\cq$. Now the modified
Rademacher sums $\QS{\Gamma,\cp}{\atp}{m}(\zz)$, for
$-\atp,-m\in\ZZp$, span the space $I_{\atp}(\Gamma)$ of automorphic
integrals for weight $2\atp$ for $\Gamma$, according to Theorem
\ref{thm:modradsum:var:Basis_QS_p_m}, and thus serve an analogous
r\^ole for the spaces $I_{\atp}(\Gamma)$ as the Poincar\'e series
$\PS{\Gamma,\cp}{\atp}{m}(\zz)$ do for the spaces
$S_{\atp}(\Gamma)$. It is natural then to use the right hand side of
(\ref{eqn:struapp:iprods:Defn_Norm_Petersson_IP}) to extend the
normalized Petersson inner product to negative weights. In light of
the fact, established in Proposition
\ref{prop:modradsum:conver:Relate_QS_FR_s=1}, that the coefficient
of $\ee(-n\zz)$ in the expansion $\QS{\Gamma,\cp|\cq}{\atp}{-m}(\zz)$
of the modified Rademacher sum $\QS{\Gamma,\cp}{\atp}{-m}(\zz)$ at
$\cq$ is $\delta_{\Gamma,\cp|\cq}\delta_{m,n}$ for $m,n\in\ZZp$, we
define an inner product $\lab\cdot\,,\cdot\rab_{\Gamma}^{\atp}$ on
$I_{\atp}(\Gamma)$ by setting
\begin{gather}\label{eqn:struapp:iprods:Defn_Norm_IP_atp<0}
     \left\langle
     \QS{\Gamma,\cp}{\atp}{-m},
     \QS{\Gamma,\cq}{\atp}{-n}
     \right\rangle_{\Gamma}^{\atp}
     =
     n^{2-2\atp}
     \delta_{\Gamma,\cp|\cq}
     \delta_{m,n}
\end{gather}
for $\cp,\cq\in \cP_{\Gamma}$ and $m,n\in \ZZp$, when $\atp<0$. In case $\atp=0$ the modified Rademacher sums
$\QS{\Gamma,\cp}{\atp}{-m}(\zz)$ span a subspace $I_0'(\Gamma)$ of
$I_0(\Gamma)$ of codimension $1$ (cf. Theorem
\ref{thm:modradsum:var:Basis_QS_p_m} and \S\ref{sec:struapp:const}) and we may use the formula (\ref{eqn:struapp:iprods:Defn_Norm_IP_atp<0})
to define an inner product on this subspace
$I_0'(\Gamma)$.
\begin{gather}\label{eqn:struapp:iprods:Defn_Norm_IP_atp=0}
     \left\langle
     \QS{\Gamma,\cp}{}{-m},
     \QS{\Gamma,\cq}{}{-n}
     \right\rangle_{\Gamma}^{0}
     =
     n^{2}
     \delta_{\Gamma,\cp|\cq}
     \delta_{m,n}
\end{gather}
A complement to $I_0'(\Gamma)$ in $I_0(\Gamma)$ is
spanned by the constant functions. We extend the inner product
$\lab\cdot\,,\cdot\rab_{\Gamma}^{0}$ to all of $I_0(\Gamma)$ by
adopting the convention that
\begin{gather}
     \left\langle
     \QS{\Gamma,\cp}{}{-m},
     f
     \right\rangle_{\Gamma}^{0}
     =0
\end{gather}
for all $\cp\in \cP_{\Gamma}$ and $m\in\ZZp$ whenever $f$ is
identically constant.

\subsection{Branching}\label{sec:struapp:branch}

Let $\Delta$ and $\Gamma$ be groups commensurable with $G(\ZZ)$ and
suppose that $\Delta$ is a subgroup of $\Gamma$. Then an automorphic
integral for $\Gamma$ is also an automorphic integral for $\Delta$.
By Theorem \ref{thm:modradsum:var:Basis_QS_p_m} the modified
Rademacher sums span the spaces of automorphic integrals, so it is
natural to consider the problem of expressing the modified
Rademacher sums $\QS{\Gamma,\cp}{\atp}{m}$ associated to $\Gamma$ in
terms of the modified Rademacher sums $\QS{\Delta,\co}{\atp}{n}$
associated to $\Delta$.

Suppose then that $\Delta$ and $\Gamma$ are groups commensurable
with $G(\ZZ)$ and $\Delta$ is a subgroup of $\Gamma$. Let $\cp$ be a
cusp for $\Gamma$, let $\Sigma_{\cp}$ be a scaling coset for
$\Gamma$ at $\cp$, and choose a system
$\{\Sigma_{\co}\mid\co\in\cP_{\Delta}\}$ of scaling cosets for
$\Delta$. Suppose for now that $\atp$ is a negative integer. Then by
Theorem \ref{thm:modradsum:var:Basis_QS_p_m} we have
\begin{gather}\label{eqn:struapp:branch:Decomp_QS_Gamma_over_Delta}
     \QS{\Gamma,\cp}{\atp}{m}(\zz)
     =
     \sum_{\co\in \cP_{\Delta}}
     b_{\co}^n
     \QS{\Delta,\co}{\atp}{n}(\zz)
\end{gather}
for some $b_{\co}^n\in \CC$, and these $b_{\co}^n$ are non-zero for
only finitely many $n<0$. Furthermore, since the modified
Rademacher sum $\QS{\Gamma,\cp}{\atp}{m}(\zz)$ has no poles away
from $\cp$, the coefficient $b_{\co}^n$ can be non-zero only when
$\co$ lies in the preimage of $\cp$ under the natural map ${\sf
X}_{\Delta}\to {\sf X}_{\Gamma}$, which is to say, $b_{\co}^n$ is
non-zero only when $\co\in\Delta\backslash\cp$. It is natural then
to consider the map
\begin{gather}\label{eqn:struapp:branch:coset to_cusp}
     \begin{split}
     \Delta\backslash\Sigma_{\cp}
     &\to
     \Delta\backslash\cp\\
     \Delta\sigma
     &\mapsto
     \Delta\sigma\cdot\infty
     \end{split}
\end{gather}
sending right cosets of $\Delta$ in $\Sigma_{\cp}$ to cusps of
$\Delta$ contained in $\cp\subset\hat{\QQ}$. The map
(\ref{eqn:struapp:branch:coset to_cusp}) is always surjective. We
may ask under what circumstances it is also injective.
\begin{lem}\label{lem:struapp:branch:coset_to_cusp_inj_iff_cosets_scale}
Let $\Delta$ and $\Gamma$ be groups commensurable with $G(\ZZ)$ and
suppose that $\Delta$ is a subgroup of $\Gamma$. Let $\cp$ be a cusp
for $\Gamma$ and let $\Sigma_{\cp}$ be a scaling coset for $\Gamma$
at $\cp$. Then the map $\Delta\backslash\Sigma_{\cp}\to
{\Delta}\backslash\cp$ of (\ref{eqn:struapp:branch:coset to_cusp})
is injective if and only if $\Delta\sigma$ is a scaling coset for
$\Delta$ at $\Delta\sigma\cdot\infty$ for every $\Delta\sigma\in
\Delta\backslash\Sigma_{\cp}$.
\end{lem}
\begin{proof}
Suppose the map $\Delta\backslash\Sigma_{\cp}\to
\Delta\backslash\cp$ of (\ref{eqn:struapp:branch:coset to_cusp}) is
injective, and let $\sigma\in\Sigma_{\cp}$. The forward implication
of the lemma follows if we can show that $\Delta\sigma$ is a scaling
coset for $\Delta$ at $\Delta\sigma\cdot\infty$. For this it
suffices to show that $(\sigma^{-1}\Delta\sigma)_{\infty}=B(\ZZ)$.
Since $\sigma^{-1}\Delta{\sigma}$ is a subgroup of
$\sigma^{-1}\Gamma\sigma$, and
$(\sigma^{-1}\Gamma\sigma)_{\infty}=(\Sigma_{\cp}^{-1}\Sigma_{\cp})_{\infty}=B(\ZZ)$
by the defining properties of $\Sigma_{\cp}$, we have the inclusion
$(\sigma^{-1}\Delta\sigma)_{\infty}\subset B(\ZZ)$, so it suffices
to show that the translation $T$ belongs to
$\sigma^{-1}\Delta\sigma$. Now $\Delta\sigma T$ is also a coset of
$\Delta$ in $\Sigma_{\cp}$, since $\Gamma\sigma=\Sigma_{\cp}$ is a
union of left cosets of $B(\ZZ)$. Since $T$ fixes $\infty$ we have
$\Delta\sigma T\cdot\infty=\Delta\sigma\cdot\infty$, and thus the
cosets $\Delta\sigma T$ and $\Delta\sigma$ coincide by the assumed
injectivity of the map (\ref{eqn:struapp:branch:coset to_cusp}). It
follows that $T\in \sigma^{-1}\Delta\sigma$, so that $\Delta\sigma$
is indeed a scaling coset for $\Delta$.

For the reverse implication suppose that $\Delta\sigma$ is a scaling
coset for $\Delta$ for every $\sigma\in \Sigma_{\cp}$. Let
$\sigma,\sigma'\in \Sigma_{\cp}$ and suppose that
$\Delta\sigma\cdot\infty=\Delta\sigma'\cdot\infty$. Then
$\sigma^{-1}\delta\sigma'\in
(\Sigma_{\cp}^{-1}\Sigma_{\cp})_{\infty}=B(\ZZ)$ for some $\delta\in
\Delta$, so $\sigma^{-1}\delta\sigma'=T^n$ for some $n\in \ZZ$, and
this implies $\Delta\sigma'=\Delta\sigma T^n$. Now the assumption
that $\Delta\sigma$ is a scaling coset for $\Delta$ implies that
$\Delta\sigma$ is a union of left cosets of $B(\ZZ)$, so
$\Delta\sigma T^n=\Delta\sigma$. We conclude that
$\Delta\sigma'=\Delta\sigma$, so that the map
$\Delta\backslash\Sigma_{\cp}\to \Delta\backslash\cp$ is indeed
injective. This completes the proof of the lemma.
\end{proof}
According to the discussion of \S\ref{sec:conven:scaling} a coset
$\Delta\sigma\in \Delta\backslash\Sigma_{\cp}$ becomes a scaling
coset for $\Delta$ at $\Delta\sigma\cdot\infty$ only once we
multiply it on the right by $[\beta]$ (cf.
(\ref{eqn:conven:isomhyp:Defn_[alpha]})) for some $\beta\in \QQp$.
In fact, this $\beta$ is a positive integer, for we have
$(\sigma^{-1}\Delta\sigma)_{\infty}\subset(\Sigma_{\cp}^{-1}\Sigma_{\cp})_{\infty}=B(\ZZ)$,
so that $(\sigma^{-1}\Delta\sigma)_{\infty}$ is generated by $T^n$
for some $n\in\ZZp$. In order that $\Delta\sigma[\beta]$ be a
scaling coset for $\Delta$ at $\Delta\sigma\cdot\infty$ we should
have $([1/\beta]\sigma^{-1}\Delta\sigma[\beta])_{\infty}=B(\ZZ)$, so
the computation $[1/\beta]T^n[\beta]=T^{n/\beta}$ shows that
$\beta=n$. Evidently we may define a function
$\Delta\backslash\Sigma_{\cp}\to \ZZp$ by mapping the coset
$\Delta\sigma$ to the positive integer $n$ such that $T^{n}$
generates $(\sigma^{-1}\Delta\sigma)_{\infty}$. We next show that
this function $\Delta\backslash\Sigma_{\cp}\to \ZZp$ factors through
the map $\Delta\backslash\Sigma_{\cp}\to\Delta\backslash\cp$ of
(\ref{eqn:struapp:branch:coset to_cusp}).
\begin{lem}\label{lem:struapp:branch:coset_to_int_factors_thru_coset_to_cusp}
Let $\Delta$ and $\Gamma$ be groups commensurable with $G(\ZZ)$, and
suppose that $\Delta$ is a subgroup of $\Gamma$. Let $\cp\in
\cP_{\Gamma}$ be a cusp of $\Gamma$ and let $\Sigma_{\cp}$ be a
scaling coset for $\Gamma$ at $\cp$. Then for any
$\sigma_1,\sigma_2\in \Sigma_{\cp}$ the identity
$\Delta\sigma_1\cdot\infty=\Delta\sigma_2\cdot\infty$ implies
$(\sigma_1^{-1}\Delta\sigma_1)_{\infty}=(\sigma_2^{-1}\Delta\sigma_2)_{\infty}$.
\end{lem}
\begin{proof}
Set $\co=\Delta\sigma_1\cdot\infty=\Delta\sigma_2\cdot\infty$. Since
$(\sigma^{-1}\Gamma\sigma)_{\infty}=(\Sigma_{\cp}^{-1}\Sigma_{\cp})_{\infty}=B(\ZZ)$
for any $\sigma\in \Sigma_{\cp}$ the groups
$(\sigma_1^{-1}\Delta\sigma_1)_{\infty}$ and
$(\sigma_2^{-1}\Delta\sigma_2)_{\infty}$ are both contained in
$B(\ZZ)$. Let $n_1$ and $n_2$ be the positive integers such that
$(\sigma_1^{-1}\Delta\sigma_1)_{\infty}$ is generated by $T^{n_1}$
and $(\sigma_2^{-1}\Delta\sigma_2)_{\infty}$ is generated by
$T^{n_2}$. Then $\Delta\sigma_1[n_1]$ and $\Delta\sigma_2[n_2]$ are
both scaling cosets for $\Delta$ at $\co$, so
$\Delta\sigma_1[n_1]=\Delta\sigma_2[n_2]T^{\alpha}$ for some
$\alpha\in \QQ$. This implies that $[n_2]T^{\alpha}[1/n_1]$ is an
element of $(\Sigma_{\cp}^{-1}\Sigma_{\cp})_{\infty}=B(\ZZ)$, so
that $[n_2]T^{\alpha}[1/n_1]=T^k$ for some $k\in\ZZ$. Now we compute
\begin{gather}
     [n_2]T^{\alpha}[1/n_1]
     =
     \left[
       \begin{array}{cc}
         n_2 & 0 \\
         0 & 1 \\
       \end{array}
     \right]
     \left[
       \begin{array}{cc}
         1 & \alpha \\
         0 & 1 \\
       \end{array}
     \right]
     \left[
       \begin{array}{cc}
         1 & 0 \\
         0 & n_1 \\
       \end{array}
     \right]
     =
     \left[
       \begin{array}{cc}
         n_2 & n_2\alpha n_1 \\
         0 & n_1 \\
       \end{array}
     \right]
     =
     \left[
       \begin{array}{cc}
         1 & k \\
         0 & 1 \\
       \end{array}
     \right]
\end{gather}
and conclude that $n_1=n_2$, as we required to show.
\end{proof}
On the basis of Lemma
\ref{lem:struapp:branch:coset_to_int_factors_thru_coset_to_cusp} we
may define a function $\Delta\backslash\cp\to\ZZp$ for each cusp
$\cp$ of $\Gamma$ by first choosing a scaling coset $\Sigma_{\cp}$
for $\Gamma$ at $\cp$, and then sending the cusp $\co\in
\Delta\backslash\cp$ of $\Delta$ to the unique positive integer $n$
satisfying $(\sigma^{-1}\Delta\sigma)_{\infty}=\lab T^n\rab$ when
$\co=\Delta\sigma\cdot\infty$. Observe now that this map does not
depend upon the choice of scaling coset $\Sigma_{\cp}$, for if we
replace $\sigma\in \Sigma_{\cp}$ with $\sigma'=\sigma T^{\alpha}\in
\Sigma_{\cp}T^{\alpha}$ for some $\alpha\in\QQ$, then we have
\begin{gather}
     ((\sigma')^{-1}\Delta\sigma')_{\infty}
     =
     (T^{-\alpha}\sigma^{-1}\Delta\sigma T^{\alpha})_{\infty}
     =
     T^{-\alpha}(\sigma^{-1}\Delta\sigma)_{\infty}T^{\alpha}
     =
     (\sigma^{-1}\Delta\sigma)_{\infty}.
\end{gather}
In light of this we may define a map $w_{\Gamma}:\cP_{\Delta}\to
\ZZp$, which we call the {\em width function associated to
$\Gamma$}, by setting $w_{\Gamma}(\co)=n$ in case
$\co=\Delta\sigma\cdot\infty$ and $T^n$ generates
$(\sigma^{-1}\Delta\sigma)_{\infty}$ and $\Gamma\sigma$ is a scaling
coset for $\Gamma$ at the unique cusp of $\Gamma$ containing $\co$.
\begin{gather}\label{eqn:struapp:branch:Defn_width_function}
     \begin{split}
          w_{\Gamma}:\cP_{\Delta}&\to\ZZp\\
          \co=\Delta\sigma\cdot\infty,\,
          \Gamma\sigma\in\gt{S}_{\Gamma\co}
          &\Longrightarrow
          (\sigma^{-1}\Delta\sigma)_{\infty}
          =
          \lab T^{w_{\Gamma}(\co)}\rab
     \end{split}
\end{gather}
We call $w_{\Gamma}(\co)$ the {\em width of $\co$ with respect to
$\Gamma$}. The next result gives upper bounds for the width
functions.
\begin{lem}\label{lem:struapp:branch:sum_wGamma_is_index}
Let $\Delta$ and $\Gamma$ be groups commensurable with $G(\ZZ)$, and
suppose that $\Delta$ is a subgroup of $\Gamma$. Then we have
\begin{gather}\label{eqn:struapp:branch:sum_wGamma_is_index}
     \sum_{\co\in\Delta\backslash\cp}
     w_{\Gamma}(\co)
     =
     \#\Delta\backslash\Gamma
\end{gather}
for every cusp $\cp$ of $\Gamma$.
\end{lem}
\begin{proof}
Let $\cp$ be a cusp of $\Gamma$ and choose a scaling coset
$\Sigma_{\cp}$ for $\Gamma$ at $\cp$, and let us temporarily write
$c_{\cp}$ for the map $\Delta\sigma\mapsto\Delta\sigma\cdot\infty$
of (\ref{eqn:struapp:branch:coset to_cusp}). We claim that
$w_{\Gamma}(\co)$ is the cardinality of the preimage of $\co$ under
$c_{\cp}$, for each $\co\in \Delta\backslash\cp$. The identity
(\ref{eqn:struapp:branch:sum_wGamma_is_index}) follows from this
because the cardinality $\#\Delta\backslash\Sigma_{\cp}$ of the
source of $c_{\cp}$ is just the number of cosets of $\Delta$ in
$\Gamma$. To compute $\# c_{\cp}^{-1}(\co)$ suppose
$\Delta\sigma\cdot\infty=\Delta\sigma'\cdot\infty$ for some
$\sigma,\sigma'\in\Sigma_{\cp}$. Then $\sigma^{-1}\delta\sigma'\in
(\Sigma_{\cp}^{-1}\Sigma_{\cp})_{\infty}=B(\ZZ)$ for some $\delta\in
\Delta$, so that $\Delta\sigma'=\Delta\sigma T^k$ for some
$k\in\ZZ$. Now $\Delta\sigma T$ is an element of
$\Delta\backslash\Sigma_{\cp}$ whenever $\Delta\sigma$ is, for
$\Sigma_{\cp}$ is a union of left cosets of $B(\ZZ)$ by
construction. So we have $c_{\cp}^{-1}(\co)=\{\Delta\sigma T^k\mid
k\in\ZZ\}$ for any $\sigma\in\Sigma_{\cp}$ satisfying
$\co=\Delta\sigma\cdot\infty$. If $w_{\Gamma}(\co)=n$, so that
$(\sigma^{-1}\Delta\sigma)_{\infty}=\lab T^n\rab$ then $\Delta\sigma
T^k=\Delta\sigma T^l$ if and only if $n|(k-l)$. This establishes the
claim that $\# c_{\cp}^{-1}(\co)=w_{\Gamma}(\co)$, and completes the
proof.
\end{proof}
We can now compute the coefficients $b_{\co}^{n}$ in
(\ref{eqn:struapp:branch:Decomp_QS_Gamma_over_Delta}) explicitly,
and thus obtain a branching theorem for the modified Rademacher
sums. Even though we formulated the expression
(\ref{eqn:struapp:branch:Decomp_QS_Gamma_over_Delta}) under the
assumption that $\atp$ be negative, our methods will apply for all
$\atp\in \ZZ$ such that $\atp\leq 0$.
\begin{thm}\label{thm:struapp:branch:branching_thm}
Let $\Delta$ and $\Gamma$ be groups commensurable with $G(\ZZ)$.
Suppose that $\Delta$ is a subgroup of $\Gamma$, and suppose that
$\Delta$ and $\Gamma$ both have width one at infinity. Let $\cp$ be
a cusp of $\Gamma$ and let $\Sigma_{\cp}$ be a scaling coset for
$\Gamma$ at $\cp$. Let $\atp,m\in \ZZ$ such that $\atp\leq 0$ and
$m<0$. Then there exists a system of scaling cosets
$\{\Sigma_{\co}\mid\co\in\cP_{\Delta}\}$ for $\Delta$ for which we
have
\begin{gather}\label{eqn:struapp:branch:branching_thm}
     \QS{\Gamma,\cp}{\atp}{m}(\zz)
     =
     \sum_{\co\in\Delta\backslash\cp}
     w_{\Gamma}(\co)\QS{\Delta,\co}{\atp}{mw_{\Gamma}(\co)}(\zz).
\end{gather}
\end{thm}
\begin{proof}
Let $\Delta$ and $\Gamma$ be as in the statement of the lemma. Under
the assumption that $\Gamma$ has width one at infinity, and the
convention that $\Sigma_{\Gamma\cdot\infty}=\Gamma$ in this case
(cf. \S\ref{sec:conven:scaling}), we have
\begin{gather}\label{eqn:struapp:branch:QS_Gamma_is_sum_QS_cosets_Delta}
     \QS{\Gamma,\cp}{\atp}{m}(\zz)
     =
     \QS{\lBZ\Sigma_{\cp}^{-1}\rBZh}{\atp}{m}(\zz)
     =
     \sum_{\Delta\sigma\in\Delta\backslash\Sigma_{\cp}}
     \QS{\lBZ\sigma^{-1}\Delta\rBZh}{\atp}{m}(\zz),
\end{gather}
and we may attempt to write each summand
$\QS{\lBZ\sigma^{-1}\Delta\rBZh}{\atp}{m}(\zz)$ in terms of the
modified Rademacher sums $\QS{\Delta,\co}{\atp}{n}(\zz)$ associated
to $\Delta$. Let $\Delta\sigma\in\Delta\backslash\Sigma_{\cp}$ and
set $\co=\Delta\sigma\cdot\infty$. Supposing that
$w_{\Gamma}(\co)=n$, so that
$(\sigma^{-1}\Delta\sigma)_{\infty}=\lab T^{n}\rab$, we set
$\Sigma_{\co}=\Delta\sigma[n]$. Then we have
\begin{gather}\label{eqn:struapp:branch:QS_sigma_inv_Delta_to_scaling}
     \QS{\lBZ\sigma^{-1}\Delta\rBZh}{\atp}{m}(\zz)
     =
     \QS{\lBZ[n][1/n]\sigma^{-1}\Delta\rBZh}{\atp}{m}(\zz)
     =
     \QS{\lBZ\Sigma_{\co}^{-1}\rBZh}{\atp}{mn}(\zz),
\end{gather}
since
$\QS{\lBZ[n]\chi\rBZh}{\atp}{m}(\zz,\spp)=\QS{\lBZ\chi\rBZh}{\atp}{mn}(\zz,\spp)$
(cf. \S\ref{sec:modradsum:constr}). Under the assumption that
$\Delta$ has width one at infinity, so that
$\Sigma_{\Delta\cdot\infty}=\Delta$, we may rewrite the right most
term in (\ref{eqn:struapp:branch:QS_sigma_inv_Delta_to_scaling}) as
$\QS{\Delta,\co}{\atp}{mn}(\zz)$. We see then that a coset
$\Delta\sigma$ of $\Delta$ in $\Sigma_{\cp}$ contributes a term
$\QS{\Delta,\co}{\atp}{mw_{\Gamma}(\co)}(\zz)$ to the right hand
side of (\ref{eqn:struapp:branch:branching_thm}), where
$\co=\Delta\sigma\cdot\infty$. According to the proof of Lemma
\ref{lem:struapp:branch:sum_wGamma_is_index} there are exactly
$w_{\Gamma}(\co)$ cosets $\Delta\sigma$ in
$\Delta\backslash\Sigma_{\cp}$ satisfying
$\co=\Delta\sigma\cdot\infty$. The required identity
(\ref{eqn:struapp:branch:branching_thm}) now follows.
\end{proof}
From the proof of Theorem \ref{thm:struapp:branch:branching_thm} we
see that the existence of the scaling cosets of the conclusion is
verified constructively: if $\co=\Delta\sigma\cdot\infty$ for some
coset $\Delta\sigma\in \Delta\backslash\Sigma_{\cp}$, then we may
take $\Sigma_{\co}=\Delta\sigma[w_{\Gamma}(\co)]$.

By utilizing the inner products
$\lab\cdot\,,\cdot\rab_{\Delta}^{\atp}$ of
\S\ref{sec:struapp:iprods} we can reformulate Theorem
\ref{thm:struapp:branch:branching_thm} in a way that is independent
of scaling coset choices.
\begin{thm}\label{thm:struapp:branch:branching_thm_ips}
Let $\Delta$ and $\Gamma$ be groups commensurable with $G(\ZZ)$.
Suppose that $\Delta$ is a subgroup of $\Gamma$, and suppose that
$\Delta$ and $\Gamma$ both have width one at infinity. Let $\cp$ be
a cusp of $\Gamma$, let $\atp\,min \ZZ$ such that $\atp\leq 0$ and
$m<0$. Then we have
\begin{gather}\label{eqn:struapp:branch:branching_thm_ips}
     \left|
     \left\langle
     \QS{\Gamma,\cp}{\atp}{m},
     \QS{\Delta,\co}{\atp}{n}
     \right\rangle_{\Delta}^{\atp}
     \right|^2
     =
     w_{\Gamma}(\co)
     \delta_{\Gamma,\cp|\Gamma\co}
     \delta_{mw_{\Gamma}(\co),n}
\end{gather}
for all $\co\in \cP_{\Delta}$ and $n<0$, for any choice of
scaling coset systems for $\Delta$ and $\Gamma$.
\end{thm}

To conclude this section we comment on the branching of Rademacher
sums in the case that $\Gamma$ both contains and normalizes
$\Delta$. Suppose then that $\Delta$ and $\Gamma$ have width one at
infinity, and consider the case that $\cp$ is the infinite cusp
$\Gamma\cdot\infty$ in (\ref{eqn:struapp:branch:coset to_cusp}).
Then we are speaking of the map
\begin{gather}\label{eqn:struapp:branch:coset to_cusp_infinite}
     \begin{split}
     \Delta\backslash\Gamma
     &\to
     \Delta\backslash\Gamma\cdot\infty\\
     \Delta\gamma
     &\mapsto
     \Delta\gamma\cdot\infty,
     \end{split}
\end{gather}
and this map is injective (and thus bijective) when $\Gamma$
normalizes $\Delta$. Indeed, since $\Delta$ is supposed to have
width one at infinity we have
$(\gamma^{-1}\Delta\gamma)_{\infty}=\Delta_{\infty}=B(\ZZ)$, so that
every coset $\Delta\gamma\in\Delta\backslash\Gamma$ is indeed a
scaling coset for $\Delta$ at $\co=\Delta\gamma\cdot\infty$, and the
injectivity of (\ref{eqn:struapp:branch:coset to_cusp_infinite})
follows from Lemma
\ref{lem:struapp:branch:coset_to_cusp_inj_iff_cosets_scale}.
Applying Theorem \ref{thm:struapp:branch:branching_thm} now with
$\cp=\Gamma\cdot\infty$ we obtain the following result.
\begin{prop}\label{prop:struapp:branch:Gamma_Norm_Delta_QSGamma_sum_QSDelta_rs}
Suppose that $\Delta$ and $\Gamma$ are groups commensurable with
$G(\ZZ)$ that both have width one at infinity, and suppose that
$\Gamma$ contains and normalizes $\Delta$. Let $\atp,m\in \ZZ$ such
that $\atp\leq 0$ and $m<0$. Then we have
\begin{gather}\label{eqn:struapp:branch:Gamma_Norm_Delta_QSGamma_sum_QSDelta_rs}
     \QS{\Gamma}{\atp}{m}(\zz)
     =
     \sum_{\co\in\Delta\backslash\cp}
     \QS{\Delta,\co}{\atp}{m}(\zz)
\end{gather}
when the scaling cosets for $\Delta$ at the cusps $\co\in
\Delta\backslash\Gamma\cdot\infty$ are taken to lie in
$\Delta\backslash\Gamma$.
\end{prop}
Theorem \ref{thm:struapp:branch:branching_thm_ips} now implies the
following reformulation of Proposition \ref{prop:struapp:branch:Gamma_Norm_Delta_QSGamma_sum_QSDelta_rs} which is independent
of scaling coset choices.
\begin{prop}\label{prop:struapp:branch:Gamma_Norm_Delta_iprod QSGamma_QSDelta_rs}
Suppose that $\Delta$ and $\Gamma$ are groups commensurable with
$G(\ZZ)$ that both have width one at infinity, and suppose that
$\Gamma$ contains and normalizes $\Delta$. Then we have
\begin{gather}\label{eqn:struapp:branch:Gamma_Norm_Delta_iprod QSGamma_QSDelta_rs}
     \left|
     \left\langle
     \QS{\Gamma}{\atp}{-1},
     \QS{\Delta,\co}{\atp}{-1}
     \right\rangle_{\Delta}^{\atp}
     \right|^2
     =
     \delta_{\Gamma,\cp|\Gamma\co}
\end{gather}
for all $\co\in \cP_{\Delta}$, for any choice of scaling coset
system for $\Delta$.
\end{prop}

\subsection{Fractional orders}\label{sec:struapp:frac}

Let $\Gamma$ be a group commensurable with $G(\ZZ)$, let
$\phi,\psi\in G(\QQ)$ and set $U=\lBZ\phi^{-1}\Gamma\psi\rBZh$. Let
us consider the problem of writing the modified Rademacher sum
$\QS{U}{\atp}{m}(\zz)$ in terms of Rademacher sums of the form
$\QS{\Gamma',\cp'}{\atp}{m'}(\zz)$ for some group $\Gamma'$ and some
cusp $\cp'\in \cP_{\Gamma'}$. According to
\S\ref{sec:modradsum:constr} we have
$\QS{\Gamma,\cp|\cq}{\atp}{m}(\zz)=\QS{\Gamma',\cp'}{\atp}{m}(\zz)$
for $\Gamma'=\Gamma^{\cq}$ and $\cp'=\cp^{\cq}$ (cf.
(\ref{eqn:modradsum:constr:QS_twocusps_to_onecusp})), so we are done
if we can write $\QS{U}{\atp}{m}(\zz)$ in terms of
$\QS{\Gamma,\cp|\cq}{\atp}{n}(\zz)$ for some cusps $\cp,\cq\in
\cP_{\Gamma}$ for $\Gamma$. Set $\cp=\Gamma\phi\cdot\infty$ and
$\cq=\Gamma\psi\cdot\infty$. By the discussion of
\S\ref{sec:conven:scaling} there exist unique $\mu,\nu\in \QQp$ such
that the cosets $\Sigma_{\cp}=\Gamma\phi[\mu]$ and
$\Sigma_{\cq}=\Gamma\psi[\nu]$ (cf.
(\ref{eqn:conven:isomhyp:Defn_[alpha]})) are scaling cosets for
$\Gamma$ at $\cp$ and $\cq$, respectively. We then have
$U=\lBZ[\mu]\Sigma_{\cp}^{-1}\Sigma_{\cq}[1/\nu]\rBZh$. Let $\chi\in
\Sigma_{\cq}^{-1}\Sigma_{\cq}$ and consider the contribution
$\QS{\lBZ[\mu]\chi[1/\nu]\rBZh}{\atp}{m}(\zz,\spp)$ of the coset
$\lBZ[\mu]\chi[1/\nu]\rBZh\in U$ to the modified continued
Rademacher sum $\QS{U}{\atp}{m}(\zz,\spp)$. As the following lemma
demonstrates, the factor $[1/\nu]$ induces a re-scaling of the input
variable $\zz$.
\begin{lem}\label{lem:struapp:frac:Effect_1/nu}
Let $\chi\in G(\QQ)$ and let $\nu\in\QQp$. Let $\atp,m\in \ZZ$ such that $\atp\leq 0$ and $m<0$. Then we have
\begin{gather}
     \QS{\lBZ\chi[1/\nu]\rBZh}{\atp}{m}(\zz,\spp)
     =
          \frac{1}{\nu^{\atp}}
     \QS{\lBZ\chi\rBZh}{\atp}{m}
     \left(\frac{\zz}{\nu},\spp\right).
\end{gather}
\end{lem}
\begin{proof}
We recall from (\ref{eqn:modradsum:constr:Defn_QS_chi}) that the
modified continued Rademacher component function
$\QS{\lBZ\chi[1/\nu]\rBZh}{\atp}{m}(\zz,\spp)$ is, by definition,
the difference
$\TS{\lBZ\chi[1/\nu]\rBZh}{\atp}{m}(\zz,\spp)-\TSa{\lBZ\chi[1/\nu]\rBZh}{\atp}{m}(\spp)$,
where the functions $\TS{\lBZ\chi[1/\nu]\rBZh}{\atp}{m}(\zz,\spp)$
and $\TSa{\lBZ\chi[1/\nu]\rBZh}{\atp}{m}(\spp)$ are defined by
(\ref{eqn:modradsum:constr:Defn_TS_chi}) and
(\ref{eqn:modradsum:constr:Defn_TSa_chi}), respectively. We compute
$\jac(\lBZ\chi[1/\nu]\rBZh,\zz)=\jac(\lBZ\chi\rBZh,\zz/\nu)/\nu$,
and use the fact that $[1/\nu]$ fixes both $\infty$ and $0$ to
verify that
\begin{gather}
     \TS{\lBZ\chi[1/\nu]\rBZh}{\atp}{m}(\zz,\spp)
     =
     \frac{1}{\nu^{\atp}}
     \TS{\lBZ\chi\rBZh}{\atp}{m}
     \left(\frac{\zz}{\nu},\spp\right),\quad
     \TSa{\lBZ\chi[1/\nu]\rBZh}{\atp}{m}(\spp)
     =
     \frac{1}{\nu^{\atp}}
     \TSa{\lBZ\chi\rBZh}{\atp}{m}(\spp).
\end{gather}
The claim follows from these identities.
\end{proof}
We now seek to describe the effect of the factor $[\mu]$ in
$\QS{\lBZ[\mu]\chi\rBZh}{\atp}{m}(\zz,\spp)$, for arbitrary $\chi\in
G(\QQ)$. Inspecting (\ref{eqn:modradsum:constr:Defn_TS_chi}) we find
that
\begin{gather}\label{eqn:struapp:frac:Subst_muchi_in_TS_chi}
     \TS{\lBZ[\mu]\chi\rBZh}{\atp}{m}(\zz,\spp)
     =
     {\mu^{\atp}}
     \ee(m\mu\chi\cdot\zz)
     \Treg{\atp}(m\mu,\chi,\zz,\spp)
     \jac(\chi,\zz)^{\atp},
\end{gather}
which suggests that we generalize the notion of Rademacher sum so as
to allow for fractional orders. Recall from
(\ref{eqn:conven:isomhyp:Defn_B(alphaZZ)}) that $B(\alpha\ZZ)$
denotes the subgroup of $B_u(\QQ)$ generated by $T^{\alpha}$ (cf.
(\ref{eqn:conven:isomhyp:Defn_T^alpha})). For general $\mu\in\QQp$
the right hand side of
(\ref{eqn:struapp:frac:Subst_muchi_in_TS_chi}) will not be invariant
under the replacement of $\chi$ by $T\chi$, so a Rademacher sum of
fractional order is not naturally defined by collections of cosets
of $B(\ZZ)$, but rather, by cosets of $B(h\ZZ)$ for a suitably
chosen positive integer $h\in \ZZp$.

Given $h\in \ZZp$ and $\chi\in G(\QQ)$ let us write $\lBhZ\chi\rBZh$
as a shorthand for $B(h\ZZ)\chi$, and for $X\subset G(\QQ)$ let us
write $\lBhZ X\rBZh$ as a shorthand for the set of right cosets of
$B(h\ZZ)$ determined by elements of $X$, so that
\begin{gather}\label{struapp:frac:Defn_lBhZXrBZh}
     \lBhZ X\rBZh
     =
     \left\{B(h\ZZ)\chi\mid\chi\in X\right\}.
\end{gather}
Let $\atp\in \ZZ$ such that $\atp\leq 0$. We define modified and
normalized fractional Rademacher sums as follows,
in analogy with the constructions of \S\ref{sec:modradsum:constr}.
For $\mu=-g/h$ with $(g,h)=1$ and $g,h\in \ZZp$, and for
$\lBhZ\chi\rBZh\in \lBhZ G(\QQ)\rBZh$, we define the {\em continued
fractional Rademacher component function of weight $2\atp$ and order
$\mu$ associated to $\lBhZ\chi\rBZh$}, denoted $(\zz,\spp)\mapsto
\TS{\lBhZ\chi\rBZh}{\atp}{\mu}(\zz,\spp)$, by setting
\begin{gather}\label{eqn:struapp:frac:Defn_frac_TS_chi}
     \TS{\lBhZ\chi\rBZh}{\atp}{\mu}(\zz,\spp)
     =
     \ee(\mu\lBhZ\chi\rBZh\cdot\zz)
     \Treg{\atp}(\mu,\lBhZ\chi\rBZh,\zz,\spp)
     \jac(\lBhZ\chi\rBZh,\zz)^{\atp},
\end{gather}
where $\jac(\lBhZ\chi\rBZh,\zz)=\jac(\chi,\zz)$, and
$\Treg{\atp}(\mu,\lBhZ\chi\rBZh,\zz,\spp)$ is the {\em continued
fractional Rademacher regularization factor of weight $2\atp$},
which is in turn given by
\begin{gather}\label{eqn:struapp:frac:Defn_frac_Treg}
     \Treg{\atp}(\mu,\lBhZ\chi\rBZh,\zz,\spp)
     =
     \Phi(\spp-2\atp,1+\spp-2\atp,
     \mu\lBhZ\chi\rBZh\cdot\infty
     -
     \mu\lBhZ\chi\rBZh\cdot\zz
     )
     (\tpi
     (
     \mu\lBhZ\chi\rBZh\cdot\zz
     -
     \mu\lBhZ\chi\rBZh\cdot\infty
     )
     )^{\spp-2\atp}
\end{gather}
in case $\lBhZ\chi\rBZh\in \lBhZ G(\QQ)\rBZh^{\times}$, and
$\Treg{\atp}(\mu,\lBhZ\chi\rBZh,\zz,\spp)=1$ otherwise. We define a
function $\spp\mapsto \TSa{\lBhZ\chi\rBZh}{\atp}{\mu}(\spp)$ by
setting
\begin{gather}\label{eqn:struapp:frac:Defn_frac_TSa_chi}
     \begin{split}
     \TSa{\lBhZ\chi\rBZh}{\atp}{\mu}(\spp)
     =&
     \ee(\mu\lBhZ\chi\rBZh\cdot\infty)
     \left(
     \tpi
     (
     \mu\lBhZ\chi\rBZh\cdot 0
     -
     \mu\lBhZ\chi\rBZh\cdot\infty
     )
     \right)^{(\spp-2\atp)}
     {\jac(\lBhZ\chi\rBZh,0)^{\atp}}\\
     &-\ee(\spp/2)
     \ee(\mu\lBhZ\chi\rBZh\cdot\infty)
     \left(
     \tpi
     (
     \mu\lBhZ\chi\rBZh\cdot \infty
     -
     \mu\lBhZ\chi\rBZh\cdot 0
     )
     \right)^{(\spp-2\atp)}
     {\jac(\lBhZ\chi\rBZh,0)^{\atp}}
     \end{split}
\end{gather}
in case $\lBhZ\chi\rBZh\in \lBZ G(\QQ)\rBZh^{\times\times}$, and by
setting $\TSa{\lBhZ\chi\rBZh}{\atp}{\mu}(\spp)=0$ otherwise, and we
define the {\em modified continued fractional Rademacher component
function of weight $2\atp$ and order $\mu$ associated to
$\lBhZ\chi\rBZh$}, denoted $(\zz,\spp)\mapsto
\QS{\lBhZ\chi\rBZh}{\atp}{\mu}(\zz,\spp)$, by subtracting
$\TSa{\lBhZ\chi\rBZh}{\atp}{\mu}(\zz,\spp)$ from
$\TS{\lBhZ\chi\rBZh}{\atp}{\mu}(\zz,\spp)$.
\begin{gather}\label{eqn:struapp:frac:Defn_frac_QS_chi}
     \QS{\lBhZ\chi\rBZh}{\atp}{\mu}(\zz,\spp)
     =
     \TS{\lBhZ\chi\rBZh}{\atp}{\mu}(\zz,\spp)
     -
     \TSa{\lBhZ\chi\rBZh}{\atp}{\mu}(\zz,\spp)
\end{gather}
For $U\subset\lBhZ G(\QQ)\rBZh$ we now define the {\em continued
fractional Rademacher sum of weight $2\atp$ and order $\mu$
associated to $U$}, denoted $\TS{U}{\atp}{\mu}(\zz,\spp)$, and the
{\em modified continued fractional Rademacher sum of weight $2\atp$
and order $\mu$ associated to $U$}, denoted
$\QS{U}{\atp}{\mu}(\zz,\spp)$, by setting
\begin{gather}
     \TS{U}{\atp}{\mu}(\zz,\spp)
     \label{eqn:struapp:frac:Defn_frac_TS_U_mu}
     =
     \sum_{\lBhZ\chi\rBZh\in U}
     \TS{\lBhZ\chi\rBZh}{\atp}{\mu}(\zz,\spp),\\
     \QS{U}{\atp}{\mu}(\zz,\spp)
     \label{eqn:struapp:frac:Defn_frac_QS_U_m}
     =
     \sum_{\lBhZ\chi\rBZh\in U}
     \QS{\lBhZ\chi\rBZh}{\atp}{\mu}(\zz,\spp),
\end{gather}
when these sums are absolutely locally uniformly convergent, and given such circumstances we define $\TS{U}{\atp}{\mu}(\zz)$ and $\QS{U}{\atp}{\mu}(\zz)$
by taking the limit as $\spp$ tends to $1$ in
$\TS{U}{\atp}{\mu}(\zz,\spp)$ and $\QS{U}{\atp}{\mu}(\zz,\spp)$,
respectively, so long as these limits exist.
\begin{gather}
     \TS{U}{\atp}{\mu}(\zz)\label{eqn:struapp:frac:Defn_frac_TS_s=1}
     =
     \lim_{s\to 1^+}\TS{U}{\atp}{\mu}(\zz,\spp)\\
     \QS{U}{\atp}{\mu}(\zz)\label{eqn:struapp:frac:Defn_frac_QS_s=1}
     =
     \lim_{s\to 1^+}\QS{U}{\atp}{\mu}(\zz,\spp)
\end{gather}
We call $\TS{U}{\atp}{\mu}(\zz)$ the {\em normalized fractional
Rademacher sum of weight $2\atp$ and order $\mu$ associated to $U$},
and we call $\QS{U}{\atp}{\mu}(\zz)$ the {\em modified fractional
Rademacher sum of weight $2\atp$ and order $\mu$ associated to $U$}.

We have the following vanishing result for fractional Rademacher
sums.
\begin{prop}\label{prop:struapp:frac:vanishing}
Let $X$ be a union of cosets of $B(\ZZ)$ in $G(\QQ)$, let $\mu=-g/h$
for some $g,h\in \ZZp$ and $(g,h)=1$, and set $U=\lBhZ X\rBZh$.
Suppose that the normalized fractional Rademacher sum
$\TS{U}{\atp}{\mu}(\zz)$ and the modified fractional Rademacher sum
$\QS{U}{\atp}{\mu}(\zz)$ converge. Then they vanish identically
unless $h=1$.
\end{prop}
\begin{proof}
Suppose that $h\neq 1$. By hypothesis we have a disjoint
decomposition $X=\bigcup_i \lBZ\chi_i\rBZh$ for some
$\{\chi_i\}\subset X$. This implies a disjoint decomposition
$X=\bigcup_i\bigcup_{k=0}^{h-1} \lBhZ T^k\chi_i\rBZh$, so that we
have
\begin{gather}
     \begin{split}
     \TS{U}{\atp}{\mu}(\zz,\spp)
     &
     =
     \sum_i\sum_{k=0}^{h-1}
     \ee\left(-\frac{g}{h}T^k\chi_i\cdot\zz\right)
     \Treg{\atp}\left(-\frac{g}{h},T^k\chi_i,\zz,\spp\right)
     \jac(T^k\chi_i,\zz)^{\atp}\\
     &
     =
     \sum_i
     \left(
     \sum_{k=0}^{h-1}
     \ee\left(-\frac{gk}{h}\right)
     \right)
     \ee\left(-\frac{g}{h}\chi_i\cdot\zz\right)
     \Treg{\atp}\left(-\frac{g}{h},\chi_i,\zz,\spp\right)
     \jac(\chi_i,\zz)^{\atp}
     \end{split}
\end{gather}
since the continued fractional regularization factor is unaffected
when the second argument is multiplied by an element of $B(\ZZ)$ on
the left. The sum $\sum_{k=0}^{h-1}\ee(-gk/h)$ vanishes for $g$
coprime to $h$ unless $h=1$. This shows that the continued
fractional Rademacher sum $\TS{U}{\atp}{\mu}(\zz,\spp)$ vanishes for
all $\zz$ and $\spp$ when $h\neq 1$. A directly analogous
computation shows that the function $\TSa{U}{\atp}{\mu}(\spp)$
vanishes in case $h\neq 1$. We conclude that both the continued
fractional Rademacher sum $\TS{U}{\atp}{\mu}(\zz,\spp)$ and the
modified continued fractional Rademacher sum
$\QS{U}{\atp}{\mu}(\zz,\spp)$ vanish identically in case $h\neq 1$.
The claim of the proposition follows.
\end{proof}

We return now to the identification of the Rademacher sum
$\QS{U}{\atp}{m}(\zz)$, in the case that
$U=\lBZ[\mu]\Sigma_{\cp}^{-1}\Sigma_{\cq}[1/\nu]\rBZh$.
\begin{thm}\label{thm:struapp:frac:LeftMult_mu}
Let $\Gamma$ be a group commensurable with $G(\ZZ)$, let $\cp,\cq\in
\cP_{\Gamma}$ be cusps for $\Gamma$, and let $\Sigma_{\cp}$ and
$\Sigma_{\cq}$ be scaling cosets for $\Gamma$ at $\cp$ and $\cq$,
respectively. Let $\mu,\nu\in\QQp$ and set
$U=\lBZ[\mu]\Sigma_{\cp}^{-1}\Sigma_{\cq}[1/\nu]\rBZh$. Let
$\atp,m\in\ZZ$ such that $\atp\leq 0$ and $m<0$. Let $h$
be the smallest positive integer such that $h\mu\in \ZZ$. If $h$
divides $m$ then we have
\begin{gather}
     \QS{U}{\atp}{m}(\zz)
     =h
     \frac{\mu^{\atp}}{\nu^{\atp}}
     \QS{\Gamma,\cp|\cq}{\atp}{m\mu}
     \left(\frac{\zz}{\nu}\right),
\end{gather}
and if $h$ does not divide $m$ then $\QS{U}{\atp}{m}(\zz)$ vanishes
identically.
\end{thm}
\begin{proof}
Set $U'=\lBZ[\mu]\Sigma_{\cp}^{-1}\Sigma_{\cq}\rBZh$ and
$U''=\lBZ\Sigma_{\cp}^{-1}\Sigma_{\cq}\rBZh$, so that
$U=U'[1/\nu]=[\mu]U''[1/\nu]$. By Lemma
\ref{lem:struapp:frac:Effect_1/nu} we have
$\QS{U}{\atp}{m}(\zz)=\QS{U'}{\atp}{m}({\zz}/{\nu})/\nu^{\atp}$, so
it suffices for us to show that $\QS{U'}{\atp}{m}(\zz)$ is
$h\mu^{\atp}\QS{\Gamma,\cp|\cq}{\atp}{m\mu}(\zz)$ or vanishing,
according as $h$ divides $m$ or not.

Set $X=\Sigma_{\cp}^{-1}\Sigma_{\cq}\subset G(\QQ)$. Then $X$ is a
union of right cosets of $B(\ZZ)$, so we have a disjoint
decomposition $X=\bigcup_i \lBZ\chi_i\rBZh$ for some $\chi_i\in
G(\QQ)$. By the choice of $h$ we have $\mu=g/h$ for some $g\in \ZZp$
with $(g,h)=1$. Then for the set $[\mu]X$ we have
\begin{gather}
     [\mu]X
     =\bigcup_i
     [\mu]\lBZ\chi_i\rBZh
     =\bigcup_i\bigcup_{k=0}^{h-1}
     [\mu]T^k\lBhZ\chi_i\rBZh
     =\bigcup_i\bigcup_{k=0}^{h-1}
     T^{k\mu}\lBgZ[\mu]\chi_i\rBZh
\end{gather}
with all the unions disjoint, from which we conclude that
$\lBZ[\mu]X\rBZh$ admits the disjoint decomposition
$\lBZ[\mu]X\rBZh=\bigcup_{k=0}^{h-1}\bigcup_iT^{k\mu}\lBZ[\mu]\chi_i\rBZh$.
We have
\begin{gather}
     \QS{T^{k\mu}\lBZ[\mu]\chi_i\rBZh}{\atp}{m}(\zz,\spp)
     =
     \ee(m k\mu)
     \QS{\lBZ[\mu]\chi_i\rBZh}{\atp}{m}(\zz,\spp).
\end{gather}
If $h$ divides $m$ then $m\mu\in \ZZ$ and $\ee(mk\mu)=1$ for all
$k$. Further, $\QS{\lBZ[\mu]\chi_i\rBZh}{\atp}{m}(\zz,\spp)$ coincides with $\mu^{\atp}\QS{\lBZ\chi_i\rBZh}{\atp}{m\mu}(\zz,\spp)$
in this case, and so we deduce the required identity
$\QS{U'}{\atp}{m}(\zz)=h\mu^{\atp}\QS{\Gamma,\cp|\cq}{\atp}{m\mu}(\zz)$.
On the other hand, if $h$ does not divide $m$, so that $m\mu$ is not
an integer, then the sum $\sum_{k=0}^{h-1}\ee(m k\mu)$ vanishes
and this implies the vanishing of the modified Rademacher sum
$\QS{U'}{\atp}{m}(\zz)$. The proof is complete.
\end{proof}
Taking $\cq=\Gamma\cdot\infty$ in Theorem
\ref{thm:struapp:frac:LeftMult_mu} we obtain the following result.
\begin{thm}\label{thm:struapp:frac:Union_LeftCosets_is_AI}
Let $\Gamma$ be a group commensurable with $G(\ZZ)$, let $Z$ be a
finite union of left cosets of $\Gamma$ in $G(\QQ)$ and set $U=\lBZ
Z\rBZh$. Let $\atp,m\in \ZZ$ such that $\atp\leq 0$ and $m<0$. Then the modified Rademacher sum $\QS{U}{\atp}{m}(\zz)$ is an
automorphic integral of weight $2\atp$ for $\Gamma$.
\end{thm}

\subsection{Hecke operators}\label{sec:struapp:hops}

For the modified Rademacher sums, in contrast to the classical Rademacher sums, it is natural to consider not only
their variance with respect to group actions, but also how they vary
with respect to Hecke operators. The discrepancy in constant terms
between the classical and modified Rademacher sums illustrated by
Proposition \ref{prop:modradsum:conver:Relate_QS_s=1_RS} is a
stubborn barrier to an exposition of the interaction between Hecke operators
and the classical Rademacher sums.

Recall that $I_{\atp}(\Gamma)$ denotes the space of automorphic
integrals of weight $2\atp$ for $\Gamma$, in the sense of
\S\ref{sec:conven:autfrm}, and recall the operator
$I_{\atp}(\Gamma)\to \mc{O}(\HH)$ of
(\ref{eqn:conven:autfrm:Defn_cop_on_ai}), denoted $f\mapsto
f\cop{\Gamma}{\atp}X$, and defined for a right coset $X\in
\Gamma\backslash G(\QQ)$ by setting
$f\cop{\Gamma}{\atp}X=(f-\JO{\chi\cdot\infty}{\atp}g)\sop{\atp}\chi$
for any representative $\chi\in X$ where $g\in S_{1-\atp}(\Gamma)$
is the cusp form associated to $f$ (cf.
(\ref{eqn:conven:autfrm:AutInt_Xform})). Let us generalize this
operator by setting
\begin{gather}\label{eqn:struapp:hops:deco_HeckeOps_gen}
     f\cop{\Gamma}{\atp}X
     =\sum_{i}
     (f-\JO{\chi\cdot\infty}{\atp}g)\sop{\atp}\chi_i
\end{gather}
in case $X$ is a finite union of right cosets of $\Gamma$ in
$G(\QQ)$ and the $\chi_i$ furnish a transversal
$X=\bigcup_i\Gamma\chi_i$ for $X$ over $\Gamma$. The next lemma
verifies that the operator $f\mapsto f\cop{\Gamma}{\atp}X$ makes
sense in case $X$ is a double coset $X=\Gamma\sigma\Gamma$ of
$\Gamma$ in $G(\QQ)$.
\begin{lem}\label{lem:struapp:hops:Dbl_Coset_Decomp}
Let $\Gamma$ be a group commensurable with $G(\ZZ)$, let $\sigma\in
G(\QQ)$, set $\Delta=\Gamma\cap\sigma^{-1}\Gamma{\sigma}$ and
$\Delta'=\sigma\Gamma\sigma^{-1}\cap\Gamma$, and suppose that
\begin{gather}
     \Gamma=\bigcup_i\lambda_i'\Delta',
     \quad
     \Gamma=\bigcup_i\Delta\rho_i,
\end{gather}
are left and right transversals for $\Gamma$ over $\Delta'$ and
$\Delta$, respectively. Then the double coset $\Gamma\sigma\Gamma$
admits disjoint decompositions
\begin{gather}\label{eqn:struapp:hops:DblCosetDecomp}
     \Gamma\sigma\Gamma
     =\bigcup_i\lambda_i'\sigma\Gamma,
     \quad
     \Gamma\sigma\Gamma
     =\bigcup_i\Gamma\sigma\rho_i,
\end{gather}
into left and right cosets for $\Gamma$.
\end{lem}
\begin{proof}
Since $G(\QQ)$ is the commensurator of $G(\ZZ)$, the intersection
$\Delta=\Gamma\cap \sigma^{-1}\Gamma{\sigma}$ has finite index in
both $\Gamma$ and $\sigma^{-1}\Gamma\sigma$, so there are only
finitely many $\lambda_i'$ and $\rho_i$. Observe next that
$\sigma\Delta=\Delta'\sigma=\sigma\Gamma\cap\Gamma\sigma$.
Consequently, we have $\Delta'\sigma\subset\sigma\Gamma$ and
$\sigma\Delta\subset\Gamma\sigma$, so that
\begin{gather}
     \Gamma\sigma\Gamma
     =\bigcup \lambda_i'\Delta'\sigma\Gamma
     \subset \bigcup\lambda_i'\sigma\Gamma\Gamma
     =\bigcup\lambda_i'\sigma\Gamma,\\
     \Gamma\sigma\Gamma
     =\bigcup \Gamma\sigma\Delta\rho_i
     \subset \bigcup\Gamma\Gamma\sigma\rho_i
     =\bigcup\Gamma\sigma\rho_i.
\end{gather}
The reverse inclusions hold since all the $\lambda_i'$ and $\rho_i$
lie in $\Gamma$. The unions are disjoint, for if
$\Gamma\sigma\rho_i=\Gamma\sigma\rho_j$ say, then
$\rho_i\rho_j^{-1}\in\sigma^{-1}\Gamma{\sigma}$, but
$\rho_i\rho_j^{-1}\in\Gamma$ by our choice of the $\rho_i$, so
$\rho_i\rho_j^{-1}\in\Delta$, and this implies $i=j$. A similar
argument applies to the cosets $\lambda_i'\sigma\Gamma$.
\end{proof}
The right hand identity of (\ref{eqn:struapp:hops:DblCosetDecomp})
shows that the operator $f\mapsto
f\cop{\Gamma}{\atp}\Gamma\sigma\Gamma$ is well-defined for $f\in
I_{\atp}(\Gamma)$ and $\sigma\in G(\QQ)$. The left hand identity of
(\ref{eqn:struapp:hops:DblCosetDecomp}) shows that the function
$f\cop{\Gamma}{\atp}\Gamma\sigma\Gamma$ again lies
in $I_{\atp}(\Gamma)$. Indeed, if the cusp form associated to $f\in
I_{\atp}(\Gamma)$ is $g\in S_{1-\atp}(\Gamma)$ then the cusp form
associated to $f\cop{\Gamma}{\atp}\Gamma\sigma\Gamma$ is
$g\cop{\Gamma}{\atp}\Gamma\sigma\Gamma$. We call the operator
$I_{\atp}(\Gamma)\to I_{\atp}(\Gamma)$ given by $f\mapsto
f\cop{\Gamma}{\atp}\Gamma\sigma\Gamma$ the {\em weight $2\atp$ Hecke
operator associated to $\sigma$}. In order to ease notation we set
\begin{gather}\label{eqn:struapp:hops:Defn_hop}
     f\hop{\Gamma}{\atp}\sigma
     =
     f\cop{\Gamma}{\atp}\Gamma\sigma\Gamma
\end{gather}
for $\sigma\in G(\QQ)$. We have
$f\hop{\Gamma}{\atp}\sigma=\sum_{i}f\sop{\atp}\chi_i$ when
$\{\chi_i\}\subset G(\QQ)$ is a right transversal
$\Gamma\sigma\Gamma=\bigcup_i\Gamma\chi_i$ for $\Gamma$ in
$\Gamma\sigma\Gamma$. If $\sigma$ belongs to the normalizer of
$\Gamma$ then we have $\Gamma\sigma\Gamma=\Gamma\sigma$ and hence
$f\hop{\Gamma}{\atp}\sigma=(f-\JO{\sigma\cdot\infty}{\atp}g)\sop{\atp}\sigma$
for any $f\in I_{\atp}(\Gamma)$. More generally, Lemma
\ref{lem:struapp:hops:Dbl_Coset_Decomp} yields for us the formula
\begin{gather}
     f\hop{\Gamma}{\atp}\sigma
     =\sum_{i}(f-\JO{\sigma\rho_i\cdot\infty}{\atp}g)\sop{\atp}(\sigma\rho_i)
\end{gather}
in the case that $\{\rho_i\}$ is a right transversal for $\Gamma$
over the intersection $\Delta=\Gamma\cap\Gamma^{\sigma}$.

The action of the operator $f\mapsto f\cop{\Gamma}{\atp}X$ on
holomorphic Poincar\'e series can be described without reference to
transversals. Indeed, if $U$ is the set of right cosets of $B(\ZZ)$
determined by a union of left cosets of some group $\Gamma$
commensurable with $G(\ZZ)$, and if $X$ is a finite union
$\bigcup_j\Gamma\chi_j$ say, of right cosets of $\Gamma$, then,
taking $\atp>1$ to ensure absolute (and locally uniform) convergence, we have
\begin{gather}\label{eqn:struapp:hops:cop_on_PS}
     \PS{U}{\atp}{m}\cop{\Gamma}{\atp}X
     =\sum_j
     \PS{U}{\atp}{m}\sop{\atp}\chi_j
     =\sum_j
     \PS{U\chi_j}{\atp}{m}
     =\PS{UX}{\atp}{m},
\end{gather}
so that the Poincar\'e series are stable under the action of the
operators $f\mapsto f\cop{\Gamma}{\atp}X$. In particular, for the
Hecke operator $f\mapsto f\hop{\Gamma}{\atp}\sigma$ we have
$\PS{U}{\atp}{m}\hop{\Gamma}{\atp}\sigma=\PS{U\sigma\Gamma}{\atp}{m}$.
This identity $\PS{U}{\atp}{m}\cop{\Gamma}{\atp}X=\PS{UX}{\atp}{m}$
extends naturally to the
modified Rademacher sums.
\begin{prop}\label{prop:struapp:hops:cop_on_QS}
Let $\Gamma$ be a group commensurable with $G(\ZZ)$, let $Z$ be a
finite union of left cosets of $\Gamma$ in $G(\QQ)$ and set $U=\lBZ
Z\rBZh$. Let $\atp,m\in \ZZ$ such that $\atp\leq 0$ and $m<0$. Then for $X$ a finite union of right cosets of $\Gamma$ in
$G(\QQ)$ we have
\begin{gather}
     \QS{U}{\atp}{m}\cop{\Gamma}{\atp}X
     =
     \QS{UX}{\atp}{m}.
\end{gather}
\end{prop}
\begin{proof}
The modified Rademacher sum $\QS{U}{\atp}{m}(\zz)$ is an automorphic
integral of weight $2\atp$ for $\Gamma$ according to Theorem
\ref{thm:struapp:frac:Union_LeftCosets_is_AI}, so it has an
associated cusp form $g\in S_{1-\atp}(\Gamma)$ say. We may express
$X$ as a disjoint union $X=\bigcup_{i}\Gamma\chi_i$ for some
finitely many $\chi_i\in G(\QQ)$. We then compute
\begin{gather}
     \QS{U}{\atp}{m}\cop{\Gamma}{\atp}X
     =
     \sum_{i}
     (\QS{U}{\atp}{m}-\JO{\chi_i\cdot\infty}{\atp}g)
     \sop{\atp}\chi_i
     =
     \sum_{i}
     \QS{U\chi_i}{\atp}{m}
     =
     \QS{UX}{\atp}{m}.
\end{gather}
This proves the claim.
\end{proof}

In the case that $\Gamma$ is the modular group $G(\ZZ)$ we have the
classical Hecke operators $\HO(n)$, defined for $n\in\ZZp$ by
setting
\begin{gather}\label{eqn:invar:hops:deco_ClassHeckeOps}
     n^{1-\atp}
     (\HO(n)f)(\zz)=
                    \sum_{
                    \substack{
                    ad=n\\
                    0\leq b<d}
                         }
                    f\left(
                    \frac{a\zz+b}{d}
                    \right)
                    \frac{a^{\atp}}{d^{\atp}}
\end{gather}
for $f\in M_{\atp}({\Gamma})$. We deduce the relation between the
operators $\HO(n)$ and $f\mapsto f\hop{\Gamma}{\atp}\sigma$, for
$\Gamma=G(\ZZ)$, by observing the following coincidence of disjoint
unions of cosets and double cosets of $\Gamma$.
\begin{gather}\label{eqn:invar:hops:deco_DecompM(n)}
      \bigcup_{\substack{
                    ad=n\\
                    0\leq b<d}
                    }
          \Gamma
          \left[
            \begin{array}{cc}
              a & b \\
              0 & d \\
            \end{array}
          \right]
          =\bigcup_{\substack{
                    ad=n\\
                    d|a}
                    }
          \Gamma
          \left[
            \begin{array}{cc}
              a & 0 \\
              0 & d \\
            \end{array}
          \right]
          \Gamma
\end{gather}
Both sides of (\ref{eqn:invar:hops:deco_DecompM(n)}) are
decompositions of the image in $G(\QQ)$ of the set of $2\times 2$
matrices with integral entries and determinant $n$. From these
decompositions we deduce the following result.
\begin{lem}\label{lem:invar:hops:deco_ClassHOsAndHOs}
For $n\in\ZZp$ and $\atp\in \ZZ$ we have
\begin{gather}
     n^{1-\atp}
     (\HO(n)f)
     =\sum_{d\in\ZZp,\,d^2|n}
     f\hop{\Gamma}{\atp}[n/d^2]
\end{gather}
for any $f\in I_{\atp}(\Gamma)$.
\end{lem}
As a partial converse to Lemma
\ref{lem:invar:hops:deco_ClassHOsAndHOs}, observe that for
$\Gamma=G(\ZZ)$ the action of any Hecke operator $f\mapsto
f\hop{\Gamma}{\atp}\sigma$ can be expressed in terms of the Hecke
operator associated to a diagonal element $[\mu]\in B(\QQ)$. Indeed,
if $\sigma\in  B(\QQ)$ then, according to
(\ref{eqn:invar:hops:deco_DecompM(n)}), we have
$\Gamma\sigma\Gamma=\Gamma[\mu]\Gamma$ for some $\mu\in \QQp$. In
case $\sigma$ does not lie in $B(\QQ)$ we have
$\sigma\cdot\infty=\cpr\neq\infty$ for some $\cpr\in \QQ$. Since the
modular group acts transitively on $\hat{\QQ}$ there is some
$\sigma_{\cpr}\in \Gamma$ with $\sigma_{\cpr}\cdot\infty=\cpr$, and
then $\Gamma\sigma\Gamma=\Gamma\tilde{\sigma}\Gamma$ where
$\tilde{\sigma}=\sigma_{\cpr}^{-1}\sigma$ evidently lies in
$B(\QQ)$.
\begin{lem}
Let $\Gamma=G(\ZZ)$ and let $\sigma\in G(\QQ)$. Then there exists
$\mu\in \QQp$ such that
\begin{gather}
     f\hop{\Gamma}{\atp}\sigma=f\hop{\Gamma}{\atp}[\mu]
\end{gather}
for all $\atp\in \ZZ$ and all $f\in I_{\atp}(\Gamma)$.
\end{lem}

We now consider the action of the Hecke operators $\hat{T}(n)$ on
the modified Rademacher sum
$\QS{\Gamma}{}{-1}(\zz)=\QS{\lBZ\Gamma\rBZh}{}{-1}(\zz)$ of weight $0$
and order $-1$ associated to the modular group $\Gamma=G(\ZZ)$. From
Theorem \ref{thm:modradsum:var:Basis_QS_p_m} we have that
$\QS{\lBZ\Gamma\rBZh}{}{-m}(\zz)$ is an automorphic integral of
weight $0$ (that is, an abelian integral) for $\Gamma$ for $m>0$. Since there
are no non-zero cusp forms of weight $2$ for the modular group, we
see from Corollary \ref{cor:modradsum:var:Relate_QS_inv_to_PS_van}
that $\QS{\lBZ\Gamma\rBZh}{}{-m}(\zz)$ is in fact a
$\Gamma$-invariant function on $\HH$ for all $m\in \ZZp$.
Consequently we have
$\QS{\lBZ\Gamma\rBZh}{}{-1}\hop{\Gamma}{0}[n]=\QS{\lBZ\Gamma[n]\Gamma\rBZh}{}{-1}$
for $n\in \ZZp$, by Proposition \ref{prop:struapp:hops:cop_on_QS}
and the definition (\ref{eqn:struapp:hops:Defn_hop}) of
$f\hop{\Gamma}{\atp}\sigma$. Suppose that $n$ is square-free. Then,
according to Lemma \ref{lem:invar:hops:deco_ClassHOsAndHOs}, the
action of the operator $f\mapsto f\hop{\Gamma}{0}[n]$ on
$M_0(\Gamma)$ coincides with that of $n\hat{T}(n)$. We anticipate an
application of Lemma \ref{lem:struapp:hops:Dbl_Coset_Decomp}. For
$\sigma=[n]$ we have
$\Delta'=\sigma\Gamma\sigma^{-1}\cap\Gamma=\Gamma^0(n)$, and for a
left transversal of $\Delta'=\Gamma^0(n)$ in $\Gamma$ we may take
\begin{gather}
     \Gamma
     =
     \bigcup_{e\| n}
     \bigcup_{k=0}^{n/e-1}
     T^kST^e\Delta'
\end{gather}
where the first union is over exact divisors of $n$ (cf.
\S\ref{sec:conven:fns}). We deduce that
\begin{gather}
     \Gamma[n]\Gamma
     =
     \bigcup_{e\| n}
     \bigcup_{k=0}^{n/e-1}
     T^kST^e[n]\Gamma
\end{gather}
by Lemma \ref{lem:struapp:hops:Dbl_Coset_Decomp}. Now $\lBZ
T^kX\rBZh=\lBZ X\rBZh$ for any subset $X\subset G(\QQ)$ and any
$k\in\ZZ$, so we find that
\begin{gather}
     n(\hat{T}(n)\QS{\lBZ\Gamma\rBZh}{}{-1})
     =\QS{\lBZ\Gamma[n]\Gamma\rBZh}{}{-1}
     =
     \sum_{e\| n}
     \QS{\lBZ ST^e[n]\Gamma\rBZh}{}{-1}.
\end{gather}
\begin{lem}
Let $e$ be an exact divisor of $n$. Then we have $\lBZ
ST^e[n]\Gamma\rBZh=\lBZ [n/e^2]\Gamma\rBZh$ when $\Gamma$ is the
modular group $G(\ZZ)$.
\end{lem}
\begin{proof}
We compute
\begin{gather}
     ST^e[n]
     =
     \left[
       \begin{array}{cc}
         0 & -1 \\
         n & e \\
       \end{array}
     \right],\quad
     ST^e[n]ST^{n/e}S=
     \left[
       \begin{array}{cc}
         n/e & -1 \\
         0 & e \\
       \end{array}
     \right],
\end{gather}
and observe that $T^kAT^l=[n/e^2]$ for $k,l\in \ZZ$ such that
$ke+ln/e=1$, where $A$ is given by $A=ST^e[n]ST^{n/e}S$.
\end{proof}
We can now write
\begin{gather}\label{eqn:struapp:hops:CHO_on_RS_with_mu}
     n(\hat{T}(n)\QS{\lBZ\Gamma\rBZh}{}{-1})
     =\sum_{e\|n}\QS{\lBZ [n/e^2]\Gamma\rBZh}{}{-1}
\end{gather}
in the case that $n$ is square-free, and we can employ the methods
of \S\ref{sec:struapp:frac} to rewrite
(\ref{eqn:struapp:hops:CHO_on_RS_with_mu}) in terms of the
fractional Rademacher sums. Observe that
$B(g\ZZ)[g/h]=[g/h]B(h\ZZ)$. We thus have
\begin{gather}
     \lBZ[\mu]\chi\rBZh
     =\bigcup_{k=0}^{g-1}
     T^k\lBgZ[\mu]\chi\rBZh
     =\bigcup_{k=0}^{g-1}
     T^k[\mu]\lBhZ\chi\rBZh
\end{gather}
for any $\chi\in G(\QQ)$ when $\mu=g/h$ with $g,h\in \ZZp$ and
$(g,h)=1$. We rewrite the right hand side of
(\ref{eqn:struapp:hops:CHO_on_RS_with_mu}), taking $g=n/h$ and
$h=e$, as
\begin{gather}
     \sum_{h\|n}
     \QS{\lBZ [n/h^2]\Gamma\rBZh}{}{-1}
     =
     \sum_{h\|n}
     \QS{[n/h^2]\lBhZ \Gamma\rBZh}{}{-1}
     =
     \sum_{h\|n}
     \QS{\lBhZ \Gamma\rBZh}{}{-n/h^2},
\end{gather}
and thus obtain the following result, relating the actions of Hecke
operators to the fractional Rademacher sums.
\begin{thm}\label{thm:struapp:hops:CHO_on_QS}
Let $n\in\ZZp$ be square-free and let $\Gamma=G(\ZZ)$. Then we have
\begin{gather}\label{eqn:struapp:hops:CHO_on_QS}
     n(\hat{T}(n)\QS{\lBZ\Gamma\rBZh}{}{-1})
     =
     \sum_{h\|n}
     \QS{\lBhZ \Gamma\rBZh}{}{-n/h^2}
\end{gather}
where the sum is over the exact divisors of $n$.
\end{thm}
It is interesting to compare the result
(\ref{eqn:struapp:hops:CHO_on_QS}) of Theorem
\ref{thm:struapp:hops:CHO_on_QS} with the identity
\begin{gather}
     n\left(
     \hat{T}(n)\QS{\lBZ\Gamma\rBZh}{}{-1}
     \right)(\zz)
     =
     \QS{\lBZ\Gamma\rBZh}{}{-n}(\zz),
\end{gather}
which holds since both sides are holomorphic on $\HH$, and have the
same singular terms in their Fourier expansion at
$\Gamma\cdot\infty$. This proves that the Rademacher sum
$\QS{U}{}{n/h^2}(\zz)$ vanishes for $U=\lBhZ\Gamma\rBZh$ whenever
$h$ is an exact divisor of $n$, and thus recovers a special case of
the vanishing result of Theorem \ref{thm:struapp:frac:LeftMult_mu}.

%------------------------------------------------------------------%
\section{Moonshine}\label{sec:moon}
%------------------------------------------------------------------%

Monstrous moonshine associates a group $\Gamma_g$ commensurable with
$G(\ZZ)$ and having width one at infinity to each (conjugacy class
of) element(s) $g$ in the Monster group $\MM$. In the article
\cite{ConMcKSebDiscGpsM} the set of groups $\{\Gamma_g\mid g\in
\MM\}$ is characterized in purely group theoretic terms. In this
section we describe a reformulation of this characterization in
terms of normalized Rademacher sums, and a certain family of moduli
problems for solid tori with conformal structure on the boundary.

\subsection{Genera}\label{sec:moon:genus}

The main theorem of \cite{ConMcKSebDiscGpsM} gives four conditions
which, taken together, characterize the groups $\Gamma_{g}$ for
$g\in \MM$. The first of these conditions is the requirement that a
group $\Gamma$ have genus zero. Our first result in this section is
a reformulation of the genus zero condition in terms of Rademacher
sums.
\begin{thm}\label{thm:moon:genus:Genus_zero_iff_TS1_inv}
Let $\Gamma$ be a group commensurable with $G(\ZZ)$ that has width
one at infinity and let $\cp$ be a cusp for $\Gamma$. Then the
normalized Rademacher sum $\TS{\Gamma,\cp}{}{-1}(\zz)$ defines a
$\Gamma$-invariant function on $\HH$ if and only if $\Gamma$ has
genus zero.
\end{thm}
\begin{proof}
Let $\Gamma$ be as in the statement of the theorem and let
$\cp,\cq\in\cP_{\Gamma}$ be cusps of $\Gamma$. By Theorem
\ref{thm:modradsum:conver:Relate_TS_FR} we have
\begin{gather}\label{eqn:moon:genus:Genus_zero_iff_TS1_inv}
     \TS{\Gamma,\cp|\cq}{}{-1}(\zz)
     =
     \delta_{\Gamma,\cp|\cq}\ee(-\zz)
     +\FR{\Gamma,\cp|\cq}{}{-1}(\zz)_{\rm van}.
\end{gather}
Taking $\cq=\Gamma\cdot\infty$ in
(\ref{eqn:moon:genus:Genus_zero_iff_TS1_inv}) we see that
$\TS{\Gamma,\cp}{}{-1}(\zz)$ is holomorphic on $\HH$, and letting
$\cq$ range over $\cP_{\Gamma}$ we see from
(\ref{eqn:moon:genus:Genus_zero_iff_TS1_inv}) that the only pole of
$\TS{\Gamma,\cp}{}{-1}(\zz)$ is a simple pole at the cusp $\cp$. We
see then that if the function $\TS{\Gamma,\cp}{}{-1}(\zz)$ is
$\Gamma$-invariant then it defines a morphism $\phi$ say, of Riemann
surfaces $\phi:{\sf X}_{\Gamma}\to \hat{\CC}$ (cf.
\S\ref{sec:conven:groups}) which has degree one since the preimage
of $\infty\in \hat{\CC}$ under $\phi$ is the single point $\cp\in
{\sf X}_{\Gamma}$ and $\phi$ is unramified here since $\Gamma$ is assumed to have width one at infinity. We conclude that $\phi$ is an isomorphism, so
that $\Gamma$ indeed has genus zero.

Conversely, if $\Gamma$ has genus zero then, since the space
$S_1(\Gamma)$ of cusp forms of weight $2$ for $\Gamma$ is isomorphic
to the space of holomorphic differentials on ${\sf X}_{\Gamma}$, we
have $S_1(\Gamma)=\{0\}$, so that any automorphic integral of weight
$0$ for $\Gamma$ is in fact a $\Gamma$-invariant function, by
Corollary \ref{cor:modradsum:var:Relate_QS_inv_to_PS_van}. The
modified Rademacher sum $\QS{\Gamma,\cp}{}{-1}(\zz)$ is an
automorphic integral for $\Gamma$ by Theorem
\ref{thm:modradsum:var:Basis_QS_p_m}. The normalized Rademacher sum
$\TS{\Gamma,\cp}{}{-1}(\zz)$ differs from $\QS{\Gamma,\cp}{}{-1}(\zz)$
by a constant function by Proposition
\ref{prop:modradsum:conver:Relate_QS_TS_s=1}, and thus
$\TS{\Gamma,\cp}{}{-1}(\zz)$ is $\Gamma$-invariant whenever
$\QS{\Gamma,\cp}{}{-1}(\zz)$ is. We conclude that
$\TS{\Gamma,\cp}{}{-1}(\zz)$ is $\Gamma$-invariant in case $\Gamma$
has genus zero. This completes the proof.
\end{proof}
From the first part of the proof of Theorem
\ref{thm:moon:genus:Genus_zero_iff_TS1_inv} we see that the
normalized Rademacher sum $\TS{\Gamma,\cp}{}{-1}(\zz)$ associated to
$\Gamma$ at a cusp $\cp$ defines an isomorphism ${\sf X}_{\Gamma}\to
\hat{\CC}$ mapping $\cp$ to $\infty\in \hat{\CC}$ in case $\Gamma$
is a group of genus zero.
\begin{thm}\label{thm:moon:genus:genus_zero_implies_TS_p_ind_isom}
Let $\Gamma$ be a group commensurable with $G(\ZZ)$ that has width
one at infinity and let $\cp\in \cP_{\Gamma}$ be a cusp of $\Gamma$.
If $\Gamma$ has genus zero then the normalized Rademacher sum
$\TS{\Gamma,\cp}{}{-1}(\zz)$ associated to $\Gamma$ at the cusp $\cp$
induces an isomorphism ${\sf X}_{\Gamma}\to \hat{\CC}$ mapping $\cp$
to $\infty$.
\end{thm}
Applying Theorem
\ref{thm:moon:genus:genus_zero_implies_TS_p_ind_isom} with
$\cp=\Gamma\cdot\infty$ the infinite cusp we see that the normalized
Rademacher sum $\TS{\Gamma}{}{-1}(\zz)$ induces an isomorphism ${\sf
X}_{\Gamma}\to\hat{\CC}$ mapping the infinite cusp to $\infty\in
\hat{\CC}$. Applying Proposition
\ref{prop:modradsum:conver:Relate_TS_FR_s=1} with
$\cp=\cq=\Gamma\cdot\infty$ we see that the Fourier expansion of
$\TS{\Gamma}{}{-1}(\zz)$ has vanishing constant term. This shows that
the normalized Rademacher sum $\TS{\Gamma}{}{-1}(\zz)$ associated to
$\Gamma$ at the infinite cusp is the normalized hauptmodul for
$\Gamma$ (cf. \S\ref{sec:conven:groups}) whenever this statement makes sense; i.e. whenever $\Gamma$ has genus zero.
\begin{thm}\label{thm:moon:genus:genus_zero_implies_TS_normzd_haupt}
Let $\Gamma$ be a group commensurable with $G(\ZZ)$ that has width
one at infinity. If $\Gamma$ has genus zero then the normalized
Rademacher sum $\TS{\Gamma}{}{-1}(\zz)$ is the normalized hauptmodul
for $\Gamma$.
\end{thm}
From the second part of the proof of Theorem
\ref{thm:moon:genus:Genus_zero_iff_TS1_inv} we see that if $\Gamma$
has genus zero then not only the first order but in fact all the
higher order normalized Rademacher sums $\TS{\Gamma,\cp}{}{-m}(\zz)$ (for $m>0$)
associated to $\Gamma$ at an arbitrary cusp $\cp\in \cP_{\Gamma}$
are $\Gamma$-invariant.
\begin{thm}\label{thm:moon:genus:Genus_zero_implies_TSm_inv}
Let $\Gamma$ be a subgroup of $G(\QQ)$ that is commensurable with
$G(\ZZ)$ and has width one at infinity and let $\cp\in\cP_{\Gamma}$
be a cusp of $\Gamma$. Let $m\in \ZZp$. If $\Gamma$ has genus zero
then the normalized Rademacher sum $\TS{\Gamma,\cp}{}{-m}(\zz)$ is
$\Gamma$-invariant.
\end{thm}

\subsection{Moduli}\label{sec:moon:moduli}

The second condition of the main theorem of \cite{ConMcKSebDiscGpsM}
is that a group $\Gamma$ be of $n\|h$-type (cf.
\S\ref{sec:conven:groups}). In this section we provide a geometric
description of the groups of $n\|h$-type by furnishing a family of
moduli problems for which the corresponding moduli spaces may be
realized as quotients of the form ${\sf Y}_{\Gamma}=\Gamma\backslash
\HH$ where $\Gamma$ is a group of $n\|h$-type.

Consider pairs $(E,C)$ where $E$ is an elliptic curve over $\CC$ and
$C$ is an oriented subgroup of $E$ isomorphic to $S^1$. We call such
a pair a {\em solid torus}. For $(E,C)$ a solid torus, we call $E$
the {\em underlying elliptic curve}, and we call $C$ the {\em
underlying primitive cycle}. Note that for $(E,C)$ a solid torus,
the group $C$ determines a primitive element of the first homology
group of (the smooth real surface underlying) the elliptic curve
$E$.

Given $\zz\in \HH$ let us write $\Lambda_{\zz}$ for the lattice
$\ZZ\zz+\ZZ\subset\CC$. Observe that the lattices $\Lambda_{\zz}$
and $\Lambda_{\zz'}$ coincide if any only if $\zz-\zz'\in \ZZ$, so
the assignment $\zz\mapsto \Lambda_{\zz}$ descends naturally to the
orbit space $B(\ZZ)\backslash\HH$. If we agree to write
$\lBZ\zz\rBZh$ as a shorthand for the $B(\ZZ)$ orbit determined by
$\zz\in \HH$, then we may unambiguously write
$\Lambda_{\lBZ\zz\rBZh}$ for the lattice $\ZZ\zz+\ZZ$. We write
$E_{\lBZ\zz\rBZh}$ for the corresponding elliptic curve
$E_{\lBZ\zz\rBZh}=\CC/\Lambda_{\lBZ\zz\rBZh}=\CC/\ZZ\zz+\ZZ$. Then
the pair $(E_{\lBZ\zz\rBZh},C_{\lBZ\zz\rBZh})$ is a solid torus when
we take the subgroup $C_{\lBZ\zz\rBZh}$ to be the image of $\RR$ in
$E_{\lBZ\zz\rBZh}$ under the composition $\RR\to \CC\to
E_{\lBZ\zz\rBZh}$. Observe that for any solid torus $(E,C)$ there
exists $\zz\in \HH$ and a unique isomorphism of elliptic curves
$E\to E_{\lBZ\zz\rBZh}$ inducing an orientation preserving
isomorphism $C\to C_{\lBZ\zz\rBZh}$. %This observation is the content of the following proposition.
%\begin{prop}
%The quotient space $B(\ZZ)\backslash \HH$ is a fine moduli space for
%solid tori.
%\end{prop}

For $(E,C)$ a solid torus and $n\in \ZZp$ we write $E[n]$ for the
group of $n$-division points of $E$, and we write $C[n]$ for the
intersection $C\cap E[n]$. Since $C$ is oriented, each group $C[n]$
comes equipped with a distinguished generator; viz., the point
corresponding to $1/n+\LL_{\lBZ\zz\rBZh}$ under an isomorphism $E\to
E_{\lBZ\zz\rBZh}$ inducing an orientation preserving isomorphism
$C\to C_{\lBZ\zz\rBZh}$. If $K$ is a subgroup of $E[n]$ for some $n$ then the quotient $\bar{E}=E/K$ is again an elliptic curve, and the image $\bar{C}$ say, of $C$ under the natural map $E\to \bar{E}$ defines a primitive cycle of $\bar{E}$, so that the pair $(\bar{E},\bar{C})$ is again a solid torus. We call this the {\em quotient of $(E,C)$ by $K$} and denote it $(E,C)/K$.

For $n\in\ZZp$ define an {\em $n$-compatible isogeny of solid tori} $(E',C')\to
(E,C)$ to be an isogeny $E'\to E$ of elliptic curves that maps
$C'[n]$ to a subgroup of $C[n]$. Then a $1$-compatible isogeny of solid tori is
just an isogeny of the underlying elliptic curves. We define an {\em
isogeny of solid tori} $(E',C')\to (E,C)$ to be an isogeny $E'\to E$
of elliptic curves that restricts to an orientation preserving map
$C'\to C$ on the underlying primitive cycles. For $(E,C)$ a solid torus and $n\in \ZZp$, we may now interpret the canonical map
$E\to E/C[n]$ as defining an isogeny $(E,C)\to (E,C)/C[n]$ of solid
tori. We call this isogeny the {\em canonical $n$-fold quotient of
$(E,C)$.}

Observe that the elliptic curve $E/C[n]$ underlying the quotient
$(E,C)/C[n]$ both receives a natural map from $E$ and maps naturally
to $E$, for the quotient $E/E[n]$ is naturally isomorphic to $E$,
and for the map $E/C[n]\to E$ we may take the composition $E/C[n]\to
E/E[n]\xrightarrow{\sim}E$ where the first map is the natural
projection and the second map is the isomorphism just mentioned.
This map $E/C[n]\to E$ defines an isogeny $(E,C)/C[n]\to (E,C)$ of
solid tori; indeed, it restricts to an isomorphism on the underlying
primitive cycles. We call the isogeny $(E,C)/C[n]\to (E,C)$ of solid
tori the {\em canonical $n$-fold cover of $(E,C)$.}

Say an $n$-compatible isogeny of solid tori $(E',C')\to (E,C)$, for $n\in\ZZp$,
is an {\em $n$-compatible isomorphism of solid tori} in case it admits an
inverse $n$-compatible isogeny; i.e. an isogeny $E\to E'$ mapping $C[n]$ to a
subgroup of $C'[n]$ such that the compositions $E'\to E\to E'$ and
$E\to E'\to E$ are the identity maps on $E'$ and $E$, respectively.
Simply put, an $n$-compatible isomorphism of solid tori $(E',C')\to(E,C)$ is an
isomorphism of the underlying elliptic curves $E'\to E$ that induces
an isomorphism of groups $C'[n]\to C[n]$. An {\em isomorphism of
solid tori} $(E',C')\to(E,C)$ is an isomorphism $E'\to E$ of
elliptic curves that restricts to an orientation preserving
isomorphism $C'\to C$ on the underlying primitive cycles.

Suppose $(E,C)$ and $(E',C')$ are solid tori. For $n$ a positive
integer and $e$ an exact divisor of $n$, say $(E,C)$ and $(E',C')$
are {\em $n+e$-related} if there is an $n$-compatible isogeny $(E',C')\to
(E,C)$ that induces $n/e$-compatible isomorphisms $(E',C')/C'[e]\to (E,C)$ and
$(E',C')\to (E,C)/C[e]$. More precisely, we require that the
$n$-compatible isogeny $(E',C')\to (E,C)$ factor through the canonical $e$-fold
maps $(E',C')\to (E',C')/C'[e]$ and $(E,C)/C[e]\to (E,C)$, defining
$n/e$-compatible isomorphisms of the specified type.
\begin{gather}\label{diag:moon:moduli:n+e_rel}
     \xy
     (-20,10)*+{(E',C')}="TL";
     (20,10)*+{(E,C)/C[e]}="TR";
     {\ar@{->}_{n/e}^{\sim} "TL"; "TR"};
     (-20,-10)*+{(E',C')/C'[e]}="BL";
     (20,-10)*+{(E,C)}="BR";
     {\ar@{->}_{n/e}^{\sim} "BL"; "BR"};
     {\ar@{->} "TR"; "BR"};
     {\ar@{->} "TL"; "BL"};
     \endxy
\end{gather}
Then solid tori $(E,C)$ and $(E',C')$ are $n+1$-related just if they
are $n$-compatible isomorphic, which is the case just if there is an
isomorphism of elliptic curves $E'\to E$ mapping $C'[n]$ onto
$C[n]$. In particular, the notion of being $n+1$-related is an
equivalence relation on solid tori. For $S$ a subset of the set of
exact divisors of $n$, say solid tori $(E,C)$ and $(E',C')$ are {\em
$n+S$-related} if $(E,C)$ and $(E',C')$ are {$n+e$-related} for some
$e\in S$.

Recall from \S\ref{sec:conven:div} that the set of exact divisors of $n$, denoted $\Ex(n)$, admits a naturally defined group structure.
\begin{lem}\label{lem:moon:moduli:n+S_reln_equiv}
Let $n\in \ZZp$ and let $S\subset \Ex(n)$. Then the $n+S$-relation
is an equivalence relation on solid tori if and only if $S$ is a
subgroup of $\Ex(n)$.
\end{lem}
In light of Lemma \ref{lem:moon:moduli:n+S_reln_equiv}, we say that
solid tori $(E,C)$ and $(E',C')$ are {\em $n+S$-equivalent} if they
are $n+S$-related and $S$ is a subgroup of the group of exact
divisors of $n$.

In a slight modification of the notation of \cite{ConNorMM}, we
write $\Gamma_0(n)+S$ for the group formed by taking the union of
the Hecke congruence group $\Gamma_0(n)$ and the Atkin--Lehner
involutions $W_e(n)$ for $e$ in $S$, when $S$ is a subgroup of
$\Ex(n)$.
\begin{gather}
     W_e(n)=
     \left\{
     \left[
       \begin{array}{cc}
         ae & b \\
         cn & de \\
       \end{array}
     \right]
     \in G(\QQ)
     \mid
     a,b,c,d\in \ZZ,\,
     ade-bcn/e=1
     \right\}
\end{gather}

Let $n$ be a positive integer and let $h$ be a divisor of $n$. Then
an $n$-compatible isomorphism $\phi:(E',C')\to (E,C)$ naturally determines an
$nh$-compatible isogeny $\hat{\phi}:(E',C')\to (E,C)$ which factors through an
$n/h$-compatible isomorphism $\bar{\phi}:(E',C')/C'[h]\to (E,C)/C[h]$ on the
canonical $h$-fold quotients, via the canonical maps $(E',C')\to
(E',C')/C'[h]$ and $(E,C)/C[h]\to (E,C)$, as depicted in
(\ref{diag:moon:moduli:n-isog_to_n/h-isom_on_h-fold_quots}).
\begin{gather}\label{diag:moon:moduli:n-isog_to_n/h-isom_on_h-fold_quots}
     \xy
     (-30,10)*+{(E',C')}="TL";
     (30,10)*+{(E,C)}="TR";
     {\ar@{->}_{nh}^{\hat{\phi}} "TL"; "TR"};
     (-20,-10)*+{(E',C')/C'[h]}="BL";
     (20,-10)*+{(E,C)/C[h]}="BR";
     {\ar@{->}_{n/h}^{\bar{\phi}} "BL"; "BR"};
     {\ar@{->} "BR"; "TR"};
     {\ar@{->} "TL"; "BL"};
     \endxy
\end{gather}
We may take this $nh$-compatible isogeny $\hat{\phi}$ to be the composition
$E'\to E\to E/E[h]\to E$ where the first map is the given map
$\phi$, the second map is the natural projection, and the third map
is the natural isomorphism. Since $\phi$ is an $n$-compatible isomorphism the
composition $E'\to E\to E/E[h]$ factors through the natural
projection $E/C[h]\to E/E[h]$, and the kernel of the induced map
$E'\to E/C[h]$ is just $C'[h]$, and so we arrive at an isomorphism
$E'/C'[h]\to E/C[h]$ of elliptic curves, which defines the required
$n/h$-compatible isomorphism $\bar{\phi}:(E',C')/C'[h]\to (E,C)/C[h]$ of the
corresponding solid tori.
\begin{gather}\label{diag:moon:moduli:n/h-isom_on_h-fold_quotients}
     \xy
     (-40,10)*+{(E',C')}="TLL";
     (-15,10)*+{(E,C)}="TL";
     (10,10)*+{(E,C)/E[h]}="TR";
     (40,10)*+{(E,C)}="TRR";
     {\ar@{->}_{n}^{\phi} "TLL"; "TL"};
     {\ar@{->} "TL"; "TR"};
     {\ar@{->}^{\sim} "TR"; "TRR"};
     (-20,-10)*+{(E',C')/C'[h]}="BL";
     (20,-10)*+{(E,C)/C[h]}="BR";
     {\ar@{->}_{n/h}^{\bar{\phi}} "BL"; "BR"};
     {\ar@{->} "BR"; "TR"};
     {\ar@{->} "BR"; "TRR"};
     {\ar@{->} "TLL"; "BL"};
     \endxy
\end{gather}
More generally, and by the same argument, an $n$-compatible isogeny $(E',C')\to
(E,C)$ naturally defines an $nh$-compatible isogeny $(E',C')\to (E,C)$ which
factors through an $n/h$-compatible isogeny $(E',C')/C'[h]\to (E,C)/C[h]$, via
the canonical maps $(E',C')\to (E',C')/C'[h]$ and $(E,C)/C[h]\to
(E,C)$.

Not every $n/h$-compatible isogeny $\bar{\phi}:(E',C')/C'[h]\to (E,C)/C[h]$
arises in this way from an $n$-compatible isogeny $\phi:(E',C')\to (E,C)$, but,
given an $n/h$-compatible isogeny $\bar{\phi}:(E',C')/C'[h]\to (E,C)/C[h]$, we
always have the $nh$-compatible isogeny $\hat{\phi}:(E',C')\to (E,C)$ obtained
as the composition
\begin{gather}\label{eqn:moon:moduli:induced_isog}
     (E',C') \to(E',C')/C'[h]\to (E,C)/C[h]\to(E,C)
\end{gather}
where the first map is the canonical $h$-fold quotient, and the
third map is the canonical $h$-fold cover. We call this composition
$\hat{\phi}$ the {\em $nh$-compatible isogeny induced from $\bar{\phi}$}.

For $n\in \ZZp$ and $h$ a divisor of $n$ say solid tori $(E,C)$ and
$(E',C')$ are {\em $n|h$-related} if the canonical $h$-fold
quotients $(E,C)/C[h]$ and $(E',C')/C'[h]$ are $n/h$-compatible isomorphic. For
$e$ an exact divisor of $n/h$, say $(E,C)$ and $(E',C')$ are {\em
$n|h+e$-related} if the quotients $(E,C)/C[h]$ and
$({E'},{C'})/C'[h]$ are $n/h+e$-related. For $S$ a subset of
$\Ex(n/h)$ say $(E,C)$ and $(E',C')$ are {\em $n|h+S$-related} if
they are $n|h+e$-related for some $e\in S$.
\begin{lem}\label{lem:moon:moduli:n|h+S_reln_equiv}
Let $n\in \ZZp$, let $h$ be a divisor of $n$, and let $S\subset
\Ex(n/h)$. Then the $n|h+S$-relation is an equivalence relation on
solid tori if and only if $S$ is a subgroup of $\Ex(n/h)$.
\end{lem}

Let $N$ be a positive integer, set $h=((N,24)$ (cf. \S\ref{sec:conven:div}),  and set $n=N/h$. Then for $S$ a subgroup of the group of exact
divisors of $n/h$ we arrive at the notion of $n\|h+S$-equivalence of
solid tori in the following way. Suppose $(E,C)$ and $(E',C')$ are $\Gamma_0(n|h)+S$-equivalent. Then we have isomorphisms $(E,C)\to (E_{\lBZ\zz\rBZh},C_{\lBZ\zz\rBZh})$ and $(E',C')\to (E_{\lBZ\zz'\rBZh},C_{\lBZ\zz'\rBZh})$ for some $\zz,\zz'\in\HH$. Let $S'$ be the image of $S$ under the natural injection $\Ex(n/h)\to\Ex(N)$ (cf. \S\ref{sec:conven:div}), and define elements $x$, $y$ and $w_{e'}$ for $e'\in S'$ by setting
\begin{gather}
  x=
     \left[
       \begin{array}{cc}
         1 & 1/h \\
         0 & 1 \\
       \end{array}
     \right],\quad
     y=
     \left[
       \begin{array}{cc}
         1 & 0 \\
         n & 1 \\
       \end{array}
     \right],\quad
     w_{e'}
     =
     \left[
       \begin{array}{cc}
         ae' & b \\
         cN & de' \\
       \end{array}
     \right],
\end{gather}
where the $a$, $b$, $c$ and $d$ in the definition of $w_{e'}$ are arbitrary integers for which we have $ade'-bcN/e'=1$. Then there is a sequence $(\gamma_0,\gamma_1,\cdots,\gamma_k)$ with $\gamma_0\in\Gamma_0(N)$ and $\gamma_i\in \{x,y,w_{e'}\mid e'\in S'\}$ for $1\leq i\leq k$ for which we have $\zz'=\gamma_k\cdots\gamma_1\gamma_0\cdot\zz$. We say now that $(E,C)$ and $(E',C')$ are {\em $\Gamma_0(n\|h)+S$-equivalent} in the case that the sum $\sum_{i=1}^k\lambda(\gamma_i)$ vanishes in $\ZZ/h$, where $\lambda(x)=1$, and $\lambda(w_{e'})=0$ for all $e'\in S'$, and $\lambda(y)$ is $-1$ or $1$ according as $\Gamma_0(n|h)+S$ is Fricke or not.
\begin{thm}\label{thm:moon:moduli:nh+S_groups_moduli}
Let $N\in \ZZp$, set $h=((N,24)$ and set $n=N/h$. Let $S$ be a subgroup of $\Ex(n/h)$, and set $\Gamma=\Gamma_0(n\|h)+S$. Then the quotient $\Gamma\backslash\HH$
is in natural correspondence with $n\|h+S$-equivalence classes of solid tori.
\end{thm}

\subsection{Exponents}\label{sec:moon:exp}

The third of the four conditions of the main theorem of
\cite{ConMcKSebDiscGpsM} states that if $\Gamma$ is of the form
$\Gamma_0(n\|h)+S$ then the quotient group $\Gamma/\Gamma_0(nh)$
should be a group of exponent two. In the present section we
investigate the relationship between this exponent two condition and
properties of scaling cosets.
\begin{lem}\label{lem:moon:exp:scalg_coset_sq_implies_cusp_commute}
Suppose that $\Delta$ is a group commensurable with $G(\ZZ)$ that
has width one at infinity and let $\co\in \cP_{\Delta}$ be a cusp of
$\Delta$. Let $\atp,m\in \ZZ$ such that $\atp\leq 0$ and $m<0$. If $\Sigma_{\co}$ is a scaling coset for $\Delta$ at $\co$
with the property that $\Sigma_{\co}^2=\Delta$ then
\begin{gather}
     \QS{\Delta,\co}{\atp}{m}(\zz)
     =
     \QS{\Delta|\co}{\atp}{m}(\zz),
\end{gather}
and $\Sigma_{\co}$ is contained in the normalizer of $\Delta$.
\end{lem}
\begin{proof}
Set $U=\lBZ\Sigma_{\co}^{-1}\rBZh$ and set
$U'=\lBZ\Sigma_{\co}\rBZh$. Then, under the assumption that $\Delta$
has width one at infinity, we have
$\QS{\Delta,\co}{\atp}{m}(\zz)=\QS{U}{\atp}{m}(\zz)$ and
$\QS{\Delta|\co}{\atp}{m}(\zz)=\QS{U'}{\atp}{m}(\zz)$. Let
$\sigma\in\Sigma_{\co}$. Then $\Sigma_{\co}^2=\Delta$ implies
$\sigma\Delta\sigma=\Delta$, which is equivalent to the identity
$\Delta\sigma=\sigma^{-1}\Delta$. We conclude that
$\Sigma_{\co}^{-1}=\Sigma_{\co}$, so that $U=U'$, so that
$\QS{\Delta,\co}{\atp}{m}(\zz)=\QS{\Delta|\co}{\atp}{m}(\zz)$. The
identity $\sigma\Delta\sigma=\Delta$ also implies $\sigma^2\in
\Delta$, so that $\sigma\Delta\sigma=\sigma^{-1}\Delta\sigma$. We
conclude that $\sigma$ normalizes $\Delta$, so that
$\Sigma_{\co}=\Delta\sigma\in\Delta\backslash N(\Delta)$. This
completes the proof.
\end{proof}
Lemma \ref{lem:moon:exp:scalg_coset_sq_implies_cusp_commute} states
that if a scaling coset at a cusp $\co$ of a group $\Delta$ can be
chosen so that its square is the trivial coset then we have a kind
of commutativity for the Rademacher sums associated to $\Delta$;
viz. the expansion at infinity of the modified Rademacher sum
associated to $\Delta$ and the cusp $\co$ coincides with the
expansion at $\co$ of the modified Rademacher sum associated to
$\Delta$ and the cusp at infinity.

We also have the following converse to Lemma
\ref{lem:moon:exp:scalg_coset_sq_implies_cusp_commute}.
\begin{prop}\label{prop:moon:exp:cusp_commute_implies_scalg_coset_sq}
Suppose that $\Delta$ is a group commensurable with $G(\ZZ)$ that
has width one at infinity and let $\co\in \cP_{\Delta}$ be a cusp of
$\Delta$. Let $\atp,m\in\ZZ$ such that $\atp\leq 0$ and $m<0$. If $\Sigma_{\co}$ is a scaling coset for $\Delta$ at $\co$
that is contained in the normalizer of $\Delta$ and
\begin{gather}
     \QS{\Delta,\co}{\atp}{m}(\zz)
     =
     \QS{\Delta|\co}{\atp}{m}(\zz)
\end{gather}
then $\Sigma_{\co}^2=\Delta$.
\end{prop}
\begin{proof}
As in the proof of Lemma
\ref{lem:moon:exp:scalg_coset_sq_implies_cusp_commute} we set
$U=\lBZ\Sigma_{\co}^{-1}\rBZh$ and $U'=\lBZ\Sigma_{\co}\rBZh$. Then
since $\Delta$ is assumed to have width one at infinity we have
$\QS{\Delta,\co}{\atp}{m}(\zz)=\QS{U}{\atp}{m}(\zz)$ and
$\QS{\Delta|\co}{\atp}{m}(\zz)=\QS{U'}{\atp}{m}(\zz)$. Let
$\sigma\in\Sigma_{\co}$. Then we have $U=\lBZ\sigma^{-1}\Delta\rBZh$
and $U'=\lBZ\Delta\sigma\rBZh=\lBZ\sigma\Delta\rBZh$. Set
$\co'=\Delta\sigma^{-1}\cdot\infty$. Then
$\QS{U'}{\atp}{m}(\zz)=\QS{\Delta,\co'}{\atp}{m}(\zz)$, so the
coincidence $\QS{U}{\atp}{m}(\zz)=\QS{U'}{\atp}{m}(\zz)$ implies
that $\co=\co'$. This in turn implies $\sigma^2\in \Delta$, so that
$\Sigma_{\co}^2=\Delta$, as we required to show.
\end{proof}

\subsection{Cusps}\label{sec:moon:cusp}

The forth condition of the main theorem of \cite{ConMcKSebDiscGpsM}
is the following. Supposing that a group $\Gamma$ is of the form
$\Gamma=\Gamma_0(n\|h)+S$ for some subgroup $S$ of the group of
exact divisors of $n/h$, for each $\cpr\in \hat{\QQ}$ there should
exist an element $\tilde{\sigma}_{\cpr}\in G(\RR)$ such that
\begin{gather}\label{eqn:moon:cusp:CMS_cusp_cond}
     \infty
     =
     \tilde{\sigma}_{\cpr}
     \cdot
     \cpr
     ,
     \quad
     (\tilde{\sigma}_{\cpr}\Gamma\tilde{\sigma}_{\cpr}^{-1})_{\infty}
     =B(\ZZ)
     ,
     \quad
     \tilde{\sigma}_{\cpr}\Gamma\tilde{\sigma}_{\cpr}^{-1}
     \supset
     \Gamma_0(nh).
\end{gather}
Observe that if we set $\sigma_{\cpr}=\tilde{\sigma}_{\cpr}^{-1}$,
then the conditions of (\ref{eqn:moon:cusp:CMS_cusp_cond}) translate
into
\begin{gather}\label{eqn:moon:cusp:CMS_cusp_cond_inv}
     \sigma_{\cpr}
     \cdot
     \infty
     =
     \cpr
     ,
     \quad
     (\sigma_{\cpr}^{-1}\Gamma\sigma_{\cpr})_{\infty}
     =B(\ZZ)
     ,
     \quad
     \sigma_{\cpr}^{-1}\Gamma\sigma_{\cpr}
     \supset
     \Gamma_0(nh),
\end{gather}
the first two of which are just the conditions
(\ref{eqn:conven:scaling:scaling_elt_conds}) of Lemma
\ref{lem:conven:scaling:scaling_elts_exist} except that we allow
$\sigma_{\cpr}$ to lie in $G(\RR)$ in
(\ref{eqn:moon:cusp:CMS_cusp_cond_inv}) but insist that
$\sigma_{\cpr}$ belong to $G(\QQ)$ in
(\ref{eqn:conven:scaling:scaling_elt_conds}). Actually, an element
$\sigma_{\cpr}\in G(\RR)$ satisfying
(\ref{eqn:moon:cusp:CMS_cusp_cond_inv}) must lie in $G(\QQ)$, as the
following result demonstrates.
\begin{lem}\label{lem:moon:cusp:sigma_in_G(R)_is_in_G(Q)}
Let $\Delta$ and $\Gamma$ be groups commensurable with $G(\ZZ)$ and
suppose that $\Delta$ has width one at infinity. Let $\cpr\in
\hat{\QQ}$ and suppose that $\sigma\in G(\RR)$ satisfies
$\sigma\cdot\infty=\cpr$ and $\sigma^{-1}\Gamma\sigma\supset
\Delta$. Then $\sigma\in G(\QQ)$.
\end{lem}
\begin{proof}
Suppose that $\sigma\in G(\RR)$ satisfies the hypotheses of the
lemma. Then $(\sigma^{-1}\Gamma\sigma)_{\infty}$ is an infinite
cyclic group containing $\Delta_{\infty}$ and
$\Delta_{\infty}=B(\ZZ)$ since $\Delta$ has width one at infinity.
It must be then that $(\sigma^{-1}\Gamma\sigma)_{\infty}$ is
generated by $T^{1/n}$ for some $n\in \ZZp$, so that
$([n]\sigma^{-1}\Gamma\sigma[1/n])_{\infty}=B(\ZZ)$. According to
Lemma \ref{lem:conven:scaling:scaling_elts_exist} there exists
$\sigma'\in G(\QQ)$ such that $\sigma'\cdot\infty=\cpr$ and
$((\sigma')^{-1}\Gamma\sigma')_{\infty}=B(\ZZ)$. Now
$[n]\sigma^{-1}\sigma'$ fixes $\infty$ and so lies in $B(\RR)$.
Since $\sigma'$ and $\sigma[1/n]$ both conjugate $\Gamma$ to a group
with width one at infinity we must in fact have
$[n]\sigma^{-1}\sigma'\in B_u(\RR)$, where $B_u(\RR)$ consists of
all the elements $T^{\alpha}$ (cf.
(\ref{eqn:conven:isomhyp:Defn_T^alpha})) with $\alpha\in \RR$. We
conclude that $\sigma=\sigma''T^{\alpha}$ for some $\sigma''\in
G(\QQ)$ and $\alpha\in \RR$. Let $\gamma\in \Gamma$ such that
$\sigma^{-1}\gamma\sigma\in \Delta^{\times}$, so that
$\sigma^{-1}\gamma\sigma\cdot\infty=q$ for some $q\in \QQ$. Then
$q=q''-\alpha$ for $q''=(\sigma'')^{-1}\gamma\sigma''\cdot\infty$.
Since both $q$ and $q''$ lie in $\QQ$ we conclude that $\alpha$ also
lies in $\QQ$. Then the identity $\sigma=\sigma''T^{\alpha}$ implies
that $\sigma\in G(\QQ)$, as we required to show.
\end{proof}
Recall from \S\ref{sec:conven:scaling} that for $\Gamma$ a group
commensurable with $G(\ZZ)$, for $\cp\in \cP_{\Gamma}$ a cusp of
$\Gamma$, and for $\Sigma_{\cp}\in\Gamma\backslash G(\QQ)$ a scaling
coset for $\Gamma$ at $\cp$, we write $\Gamma^{\cp}$ as a shorthand
for the group $\Sigma_{\cp}^{-1}\Sigma_{\cp}$. In light of Lemma
\ref{lem:moon:cusp:sigma_in_G(R)_is_in_G(Q)} we may reformulate the
fourth condition of the main theorem of \cite{ConMcKSebDiscGpsM} as
follows.
\begin{lem}\label{lem:moon:cusp:reform_CMS_cusp_cond_elt_to_coset}
Let $\Gamma$ be a group of $n\|h$-type, so that
$\Gamma=\Gamma_0(n\|h)+S$  for some positive integers $n$ and $h$,
and some subgroup $S$ of the group of exact divisors of $n/h$. Then
the fourth condition of the main theorem of \cite{ConMcKSebDiscGpsM}
is satisfied if and only if for each cusp $\cp\in \cP_{\Gamma}$ of
$\Gamma$ there exists a scaling coset
$\Sigma_{\cp}\in\Gamma\backslash G(\QQ)$ for $\Gamma$ at $\cp$ such
that $\Gamma^{\cp}$ contains $\Gamma_0(nh)$.
\end{lem}
We conclude this section by relating the reformulation of Lemma
\ref{lem:moon:cusp:reform_CMS_cusp_cond_elt_to_coset} to the
normalized Rademacher sums associated to genus zero groups of
$n\|h$-type.
\begin{prop}\label{prop:moon:cusp:scaling_coset_cont_Delta_iff_TSpq_inv}
Let $\Gamma$ be a group of $n\|h$-type, so that
$\Gamma=\Gamma_0(n\|h)+S$ for some positive integers $n$ and $h$,
and some subgroup $S$ of the group of exact divisors of $n/h$, and
suppose that $\Gamma$ has genus zero. Let $\cp,\cq\in \cP_{\Gamma}$
be cusps of $\Gamma$ and let $\Sigma_{\cp}$ and $\Sigma_{\cq}$ be
scaling cosets for $\Gamma$ at $\cp$ and $\cq$, respectively. Then
the normalized Rademacher sum $\TS{\Gamma,\cp|\cq}{}{-1}(\zz)$ is
$\Gamma_0(nh)$-invariant if and only if $\Gamma^{\cq}$ contains
$\Gamma_0(nh)$.
\end{prop}
\begin{proof}
Let us set $\Delta=\Gamma_0(nh)$. Observe that if we set
$\cp^{\cq}=\Sigma_{\cq}^{-1}\Sigma_{\cp}\cdot\infty$ then
$\cp^{\cq}$ is a cusp of
$\Gamma^{\cq}=\Sigma_{\cq}^{-1}\Sigma_{\cq}$, and
$\Sigma_{\cq}^{-1}\Sigma_{\cp}$ is a scaling coset for
$\Gamma^{\cq}$ at $\cp^{\cq}$, so the function
$\TS{\Gamma,\cp|\cq}{}{-1}(\zz)$ may be identified with the
normalized Rademacher sum $\TS{\Gamma^{\cq},\cp^{\cq}}{}{-1}(\zz)$
associated to $\Gamma^{\cq}$ at the cusp $\cp^{\cq}$, for
$\Sigma_{\cp^{\cq}}=\Sigma_{\cq}^{-1}\Sigma_{\cp}$. Under the
assumption that $\Gamma$ has genus zero, $\Gamma^{\cq}$ also has
genus zero, and has width one at infinity by the defining properties
(cf. Lemma \ref{lem:conven:scaling:scaling_cosets_exist}) of
$\Sigma_{\cq}$. By Theorem
\ref{thm:moon:genus:genus_zero_implies_TS_normzd_haupt} then the
normalized Rademacher sum $\TS{\Gamma,\cp|\cq}{}{-1}(\zz)$ is, up to
a constant function, the expansion at $\cp^{\cq}$ of the normalized
hauptmodul of the group $\Gamma^{\cq}$. In particular,
$\TS{\Gamma,\cp|\cq}{}{-1}(\zz)$ defines an isomorphism of Riemann
surfaces ${\sf X}_{\Gamma^{\cq}}\to \hat{\CC}$, and so we have
$\TS{\Gamma,\cp|\cq}{}{-1}(\gamma\cdot\zz)=\TS{\Gamma,\cp|\cq}{}{-1}(\zz)$
for all $\zz\in \HH$ if and only if $\gamma\in \Gamma^{\cq}$. It
follows then that $\TS{\Gamma,\cp|\cq}{}{-1}(\zz)$ is
$\Delta$-invariant if and only if $\Delta$ is contained in
$\Gamma^{\cq}$, which is what we required to show.
\end{proof}

\subsection{Rademacher sums and the Monster}\label{sec:moon:monster}

We will now use Rademacher sums to reformulate the characterization of the groups of monstrous moonshine due to Conway--McKay--Sebbar (cf. \cite{ConMcKSebDiscGpsM}).
\begin{thm}\label{thm:moon:monster:char_thm}
Let $\Gamma$ be a subgroup of $G(\RR)$. Then we have
$\Gamma=\Gamma_g$ for some $g\in \MM$ if and only if the following
conditions are satisfied:
\begin{itemize}
\item
$\Gamma$ is the group defining $n\|h+S$-equivalence of solid tori
for some positive integers $n$ and $h$, and some subgroup
$S<\Ex(n/h)$;
\item
the normalized Rademacher sum $\TS{\Gamma}{}{-1}(\zz)$ is
$\Gamma$-invariant;
\item
there exists a system $\{\Sigma_{\cp}\mid\cp\in \cP_{\Gamma}\}$ of
scaling cosets for $\Gamma$ such that $\TS{\Gamma|\cp}{}{-1}(\zz)$ is
$\Delta$-invariant for every cusp $\cp\in \cP_{\Gamma}$ of $\Gamma$,
where $\Delta=\Gamma_0(nh)$;
\item
for every cusp $\co\in\cP_{\Delta}$ of $\Delta=\Gamma_0(nh)$ that is
contained in $\Gamma\cdot\infty$ we have
$\TS{\Delta,\co}{}{-1}(\zz)=\TS{\Delta|\co}{}{-1}(\zz)$ when the
scaling coset $\Sigma_{\co}$ for $\Delta$ at $\co$ is taken to lie
in $\Delta\backslash\Gamma$.
\end{itemize}
\end{thm}
\begin{proof}
According to the main theorem of \cite{ConMcKSebDiscGpsM} it
suffices to show that a group $\Gamma$ satisfies each of the four
conditions of Theorem \ref{thm:moon:monster:char_thm} if and only if
it satisfies the four conditions of the main theorem of
\cite{ConMcKSebDiscGpsM}. As a first step in establishing this
equivalence, observe that the first condition of Theorem
\ref{thm:moon:monster:char_thm} is exactly the same as the second
condition of the main theorem of \cite{ConMcKSebDiscGpsM}.

Suppose that $\Gamma$ satisfies the four conditions of the main
theorem of \cite{ConMcKSebDiscGpsM}. Then, in particular, it
satisfies the first condition of Theorem
\ref{thm:moon:monster:char_thm}, and has genus zero. We may suppose
then that $\Gamma=\Gamma_0(n\|h)+S$ for some $n,h\in\ZZp$, and
$S<\Ex(n/h)$, so that $\Gamma$ is commensurable with $G(\ZZ)$, has
width one at infinity, and contains and normalizes
$\Delta=\Gamma_0(nh)$. According to Theorem
\ref{thm:moon:genus:Genus_zero_iff_TS1_inv} the normalized
Rademacher sum $\TS{\Gamma}{}{-1}(\zz)$ is $\Gamma$-invariant, so the
second condition of Theorem \ref{thm:moon:monster:char_thm} is
satisfied. According to Lemma
\ref{lem:moon:cusp:reform_CMS_cusp_cond_elt_to_coset} there exists a
system $\{\Sigma_{\cp}\mid\cp\in\cP_{\Gamma}\}$ of scaling cosets
for $\Gamma$ such that $\Gamma^{\cp}$ contains $\Delta$ for each
cusp $\cp\in \cP_{\Gamma}$. Applying Proposition
\ref{prop:moon:cusp:scaling_coset_cont_Delta_iff_TSpq_inv} we see
that the normalized Rademacher sum $\TS{\Gamma|\cp}{}{-1}(\zz)$ is
$\Delta$-invariant when the scaling coset $\Sigma_{\cp}$ of the
above scaling coset system is chosen for the definition of
$\TS{\Gamma|\cp}{}{-1}(\zz)$. This confirms that the third condition
of Theorem \ref{thm:moon:monster:char_thm} is satisfied. The fourth
condition of the main theorem of \cite{ConMcKSebDiscGpsM} is that
the quotient $\Gamma/\Delta$ have exponent two. Let $\co\in
\cP_{\Delta}$ be a cusp of $\Delta$ that is contained in
$\Gamma\cdot\infty$, and let $\Sigma_{\co}$ be the unique right
coset of $\Delta$ in $\Gamma$ such that
$\co=\Sigma_{\co}\cdot\infty$. Then the exponent two condition
implies $\Sigma_{\co}^2=\Delta$, so that the fourth condition of
Theorem \ref{thm:moon:monster:char_thm} follows from an application
of Lemma \ref{lem:moon:exp:scalg_coset_sq_implies_cusp_commute}.

Suppose now that $\Gamma$ satisfies the four conditions of Theorem
\ref{thm:moon:monster:char_thm}. Then $\Gamma$ satisfies the second
condition of \cite{ConMcKSebDiscGpsM}, and so we have
$\Gamma=\Gamma_0(n\|h)+S$ for some $n,h\in\ZZp$, and $S<\Ex(n/h)$,
and $\Gamma$ is a group commensurable with $G(\ZZ)$ that has width
one at infinity. Applying Theorem
\ref{thm:moon:genus:Genus_zero_iff_TS1_inv} to the second condition
of Theorem \ref{thm:moon:monster:char_thm} we conclude that $\Gamma$
has genus zero, and so satisfies the first condition of
\cite{ConMcKSebDiscGpsM}. Set $\Delta=\Gamma_0(nh)$, so that
$\Gamma$ contains and normalizes $\Delta$. Applying Proposition
\ref{prop:moon:cusp:scaling_coset_cont_Delta_iff_TSpq_inv} to the
third condition of Theorem \ref{thm:moon:monster:char_thm} we see
that scaling cosets $\{\Sigma_{\cp}\mid\cp\in \cP_{\Gamma}\}$ can be
chosen for $\Gamma$ so that $\Gamma^{\cp}$ contains $\Delta$ for
each cusp $\cp\in \cP_{\Gamma}$. Applying Lemma
\ref{lem:moon:cusp:reform_CMS_cusp_cond_elt_to_coset} to this fact
we conclude that $\Gamma$ satisfies the fourth condition of the main
theorem of \cite{ConMcKSebDiscGpsM}. It remains to check that the
quotient $\Gamma/\Delta$ has exponent two. Let $\gamma\in \Gamma$
and set $\co=\Delta\gamma\cdot\infty$. If we take
$\Sigma_{\co}=\Delta\gamma$ then the fourth condition of Theorem
\ref{thm:moon:monster:char_thm} states that
$\TS{\Delta|\co}{}{-1}(\zz)=\TS{\Delta,\co}{}{-1}(\zz)$. Proposition
\ref{prop:moon:exp:cusp_commute_implies_scalg_coset_sq} now implies
that $\Sigma_{\co}^2=\Delta$, which in turn implies $\gamma^2\in
\Delta$. This argument applies to arbitrary $\gamma\in \Gamma$ so we
conclude that the quotient $\Gamma/\Delta$ has exponent two. This
completes the proof.
\end{proof}
Perhaps the most technical condition of Theorem
\ref{thm:moon:monster:char_thm} is the last one. Beyond the
monstrous moonshine conjectures, there are the generalized moonshine
conjectures of Norton (cf. \cite{Mas_GnzdmoonshineAppNorton}) which
associate genus zero groups to commuting pairs of elements in the
Monster. A number of the groups appearing in generalized moonshine
do not satisfy the last condition of Theorem
\ref{thm:moon:monster:char_thm}, but they all satisfy the second
property, and a slight weakening of the first (cf.
\cite{Fer_Genus0prob}), and we do not know of any examples that fail
to satisfy the third condition of Theorem
\ref{thm:moon:monster:char_thm}. It is an interesting question then
to determine if the first three conditions of Theorem
\ref{thm:moon:monster:char_thm} furnish a
characterization of the groups of generalized moonshine.

%------------------------------------------------------------------%
\section{Gravity}\label{sec:gravity}
%------------------------------------------------------------------%

In this section we consider applications of the normalized
Rademacher sums to chiral 3d quantum gravity. We should note that the notion of chiral 3d quantum gravity has not yet been defined, so the applications we make, and in particular the conjectures we formulate, are intended to shed light on the very problem of formulating a definition, in addition to elucidating important structural properties that a certain distinguished (and conjectural) example might satisfy. In the absence of a precise definition of 3d quantum gravity some of our discussion in this section is necessarily more speculative.

\subsection{First conjecture}\label{sec:gravity:conj1}

It is clear from our results that the Rademacher sums are
particularly convenient for understanding the special
characteristics of the McKay--Thompson series; in particular, their
crucial genus zero property. In order to fully explain the moonshine
phenomena, one has to relate these sums to the structure of the
vertex operator algebra $\vn$. In search for the new relation, one
might look again for hints from physics. Very recently, Witten,
revisiting 3d quantum gravity, has formulated a
number of results and observations \cite{Wit_3DGravRev} including a
conjecture about the existence of the 3d quantum
gravities with central charges $c_L=c_R$ proportional to $24$. In
particular, he asserted that the simplest in his list of 3d quantum gravities should be equivalent to the
2d conformal field theory $\vn\otimes(\vn)^*$. Manschot
\cite{Man_AdS3PFnsRecon} then suggested to consider chiral 3d quantum gravities, the simplest of which, with $c_L=24$
and $c_R=0$, should be equivalent exactly to $\vn$. Furthermore, Li,
Song and Strominger \cite{LiSonStr_ChGrav3D} have argued that the
chiral gravity possesses stability and consistency, the necessary
properties of a sound physical theory. A very recent work
\cite{MalSonStr_ChiGravLogGravExtCFT} by Maloney, Song and
Strominger provided further support for the existence of chiral
gravities with $c_L$ proportional to $24$. These results and
observations from physics taken together suggest that there exists
an alternative construction of the vertex operator algebra $\vn$
that may be viewed as a rigorous version of the simplest chiral
3d quantum gravity with $c=24$, in the same way as
the original construction of $\vn$ was interpreted as a rigorous
version of the chiral 2d conformal field theory with
the partition function $J(\zz)$. In 3d quantum
gravity one expects to obtain the partition function as a sum over
minimum points of a 3d quantum gravity action; i.e.
over all 3d hyperbolic structures on a solid torus
with genus one boundary, whose conformal structure corresponds to
the point $\zz$ on the moduli space. Since all such structures are
naturally parameterized by $\Gamma_{\infty}\backslash\PSL_2(\ZZ)$
(cf. \cite{DijMalMooVer_BlckHoleFryTale}), the partition function
should be a kind of Rademacher sum. In fact Manschot and Moore
\cite{ManMoo_ModFryTail} argued that the subtraction of constants in
(\ref{eqn:intro:radsum:Rad_Sum_J}) can be explained by a
regularization of the partition function of 3d
gravity. Our continued Rademacher sums of \S\ref{sec:modradsum}, leading naturally to the modified and normalized Rademacher sums, provide a good candidate for
such a regularization.

To obtain the McKay--Thompson series
(\ref{eqn:intro:moon:Four_Exp_MTg}) from the chiral 3d quantum
gravity is a more challenging problem even at the heuristic level.
However, one may look at these series from a slightly different
point of view. It has been shown in \cite{DoLiMaTrOrbThy} that for
any $g\in\MM$ there exists a unique simple $g$-twisted $\vn$-module
$\vn_g$, whose partition function is equal to $c(g)\MT_g(-1/\zz)$
where $c(g)$ is a constant depending on $g\in \MM$. The heuristic
analysis of the twisted sector $\vn_g$ in
\cite{Tui_MonsMoonsUniqMoonsMod} strongly supports the general
assumption that $c(g)=1$ for all $g\in \MM$. (For more on this assumption please see the discussion of ``Hypothesis $A_g$'' in \cite{Car_GMII}.)

We now state our first conjecture.
\begin{conj}\label{conj:conseq:conj1_conj1}
There exists a family of twisted chiral 3d quantum
gravities at central charge $c=24$ associated with elements of the
Monster $g\in\MM$ whose partition functions are naturally given by
sums over geometries parameterized by
$\Gamma_{\infty}\backslash\Gamma_g$, and these partition functions
coincide with the normalized Rademacher sum
$\TS{\Gamma}{}{-1}(-1/\zz)$, for $\Gamma=\Gamma_g$. Moreover, the
untwisted ($g=e$) chiral 3d quantum gravity has a VOA
structure isomorphic to $\vn$, and the twisted 3d
quantum gravity corresponding to $g\in\MM$ has a structure of
$g$-twisted $\vn$-module isomorphic to that of $\vn_g$.
\end{conj}
The special case $g=e$ of Conjecture \ref{conj:conseq:conj1_conj1},
corresponding to the untwisted chiral 3d quantum
gravity, is strongly supported by the physics literature mentioned
above. We will now make a few remarks about the general twisted
case. The principle of the 3d quantum gravity/2d CFT correspondence suggests that all the structures in
either theory should have appropriate counterparts in the other. The
$g$-twisted sector $\vn_g$ is an intrinsic part of the extended
theory of the chiral 2d CFT associated to $\vn$, and,
according to the physical principle, should therefore have analogues
in chiral 3d quantum gravity. Basic information about
these twisted chiral 3d quantum gravities can be
extracted from the Rademacher sums $\GD_g(\zz)=\MT_g(-1/\zz)$. We'll
illustrate the general case with the level two examples, which are
the cases that $\Gamma=\Gamma_0(2)$ or $\Gamma_0(2)+$. (In the
notation of \S\ref{sec:conven:groups} the symbols $\Gamma_0(2)+$ are
a shorthand for $\Gamma_0(2)+\Ex(2)$.) As in the case that $g=e$ one
expects that each term of the Rademacher sum comes from a classical
solution of the corresponding twisted chiral quantum gravity. Every
Rademacher sum $\GD_g(\zz)$ contains the term $\ee(1/\zz)$, which
corresponds to the BTZ black hole solution (cf.
\cite{BanTeiZan_BTZBlckHle}). The other terms arising for
$\Gamma=\Gamma_0(2)$ can be characterized among all those arising
for $\Gamma=G(\ZZ)$ as the solutions that have the same spin
structure as the BTZ black hole on the boundary. The Rademacher sum
$\GD_g(\zz)$ in this case coincides with the partition function of
the Ramond sector in supergravity (cf.
\cite[\S7]{MalWit_QGravPartFns3D}). The group $\Gamma=\Gamma_0(2)+$
is no longer a subgroup of the modular group $G(\ZZ)$. In this case
we have to consider orbifold solutions in addition to smooth
solutions; we consider orbifold solutions with a codimension $2$
singularity along the defining circle of the solid torus which looks
locally like $\CC/\lab\ee(1/2)\rab$ (cf.
\cite[\S2]{MalWit_QGravPartFns3D}), and we still impose the same
spin structure boundary condition. Note that on any space $X$, choices of spin structure on $X$ (when they exist) are in bijective
correspondence with a particular family of double covers of $X$ (cf.
e.g. \cite{LawMic_SpinGeom}). Such a double cover smooths the
$\ZZ/2$-orbifold singularity. One may view the spin structure as a
$\ZZ/2$-structure, and for a general level $N$ group $\Gamma$, spin
structures are replaced with $\ZZ/N$-structures. One can also allow
singularities along the defining circle of the solid torus that look
locally like $\CC/\lab\ee(1/N)\rab$. In general, a choice of twisted
gravity with $\ZZ/N$-structure group should impose an $n\|h+S$
equivalence for solid tori, for some $n$ and $h$ with $N=nh$ and $h$
a divisor of $24$ (cf. Theorem
\ref{thm:moon:moduli:nh+S_groups_moduli}), which yields the
Rademacher sum $\GD_g(\zz)$, for a corresponding $g\in \MM$, as the
saddle point approximation of the twisted quantum gravity partition
function.

We expect that the conjectural twisted chiral 3d
quantum gravity construction of $\vn_g$ will imply the
$\Gamma_g$-invariance of its partition function, in analogy with the
way in which 2d CFTs are found to have modular
invariant partition functions (cf. \cite{Wit_PhysGeom}). Then our
conjecture, in combination with our results on Rademacher sums, will
naturally imply the genus zero conjecture of Conway--Norton (cf.
\cite{ConNorMM}), and most importantly, will reveal the geometric
nature of moonshine. In particular, the solution to Conjecture \ref{conj:conseq:conj1_conj1} will answer the genus zero question: why the discrete groups attached to the Monster via monstrous moonshine all have genus zero.

We also expect that the conditions on Rademacher sums in our reformulation Theorem
\ref{thm:moon:monster:char_thm}, of the group theoretic
characterization of the groups of monstrous moonshine due to Conway--McKay--Sebbar, will also find an interpretation as consistency conditions for twisted chiral 3d quantum gravities. In particular,
the completeness of the family of twisted chiral 3d
quantum gravities at central charge $24$ associated with elements of
the Monster group $\MM$ should have a deep meaning in quantum
gravity.

Our Conjecture \ref{conj:conseq:conj1_conj1} should have
significance for the future development of analytic number theory also.
We already know from 2d CFT that various constructions
yield remarkable number theoretic identities. Our conjecture implies
that the development of the theory of 3d quantum
gravity might encompass a whole new family of number theoretic
results such as the theory of Rademacher sums studied in this work.
In \S\ref{sec:gravity:conj2} we'll show that the Hecke operators
also admit an interpretation in terms of quantum gravity. In
preparation for this we will recall in the next section some
results about a certain class of generalized Kac--Moody algebras and
their representations.

\subsection{Monstrous Lie algebras}\label{sec:gravity:mlas}

In his paper \cite{BorMM}, Borcherds constructed a remarkable
generalized Kac--Moody algebra ({\em GKM algebra}) $\mla$, called
the {\em Monster Lie algebra}, which plays a key r\^ole in his
approach to monstrous moonshine. He also defined GKM superalgebras
$\mla_g'$, for each $g\in \MM$, by using the McKay--Thompson series
$\MT_g(\zz)$ to specify the simple roots. In this section we'll
consider another family of GKM algebras $\mla_g$, parameterized by
elements $g\in \MM$, with simple roots specified by the functions
$\GD_g(\zz)=\MT_g(-1/\zz)$. Since the Fourier coefficients of the $\GD_g(\zz)$ are
non-negative integers, the algebras $\mla_g$ are purely even. We
will also study Verma modules for the GKM algebras $\mla_g$.

In the following section \S\ref{sec:gravity:conj2}, we will explain
how these algebraic structures arise from the Rademacher sums, and
the conjectural twisted chiral 3d quantum gravities.
This will lead us to a further extension of our first conjecture.

The algebras $\mla_g$ were introduced and studied by Carnahan in
\cite{Car_Phd} and were discovered independently by the second
author. We will now recall the relevant results of \cite{Car_Phd}.

Let $L=II_{1,1}$ denote a copy of the unique even self-dual
Lorentzian lattice of rank $2$. We identify $L$ with the group
$\ZZ\times\ZZ$ and set the norm of the pair $(m,n)$ to be $-2mn$.
For any positive integer $N$ we denote the sublattice $\ZZ\times
N\ZZ$ by $L(N)$, and we let $L(N)^{\vee}$ denote the dual lattice
$\frac{1}{N}\ZZ\times\ZZ$. Let $V_L$ denote the vertex operator
algebra ({VOA}) associated to the lattice $L$ (cf. \cite{FLM}),
\begin{gather}
     V_L=\bigoplus_{(m,n)\in L}V_L^{(m,n)},
\end{gather}
and let $h_N$ denote the automorphism of $V_L$ which acts as
multiplication by $\ee(n/N)$ on the subspace $V_L^{(m,n)}$. Then
$V_{L(N)}$ is a vertex operator subalgebra of $V_{L(N)}$ fixed by
$h_N$. Following \cite{DonLepGVAs} each coset $L+(k/N,0)$ of $L$ in
$L(N)^{\vee}$ defines an $h_N^k$-twisted module for $V_L$. We denote
this twisted module by $V_{L+(k/N,0)}$, and regard it as graded in
the natural way by $L+(k/N,0)$.

For $g\in \MM$ such that the level of $\Gamma_g$ is $N$ (i.e. $N=nh$ in the notation of \cite{ConNorMM}) define $\vnh_g$ to be the following
$L(N)^{\vee}$-graded space invariant under $\lab g\rab$.
\begin{gather}
     \vnh_g=\bigoplus_{k=0}^{N-1}
     \left(
     \vn_{g^k}\otimes V_{L+(k/N,0)}
     \right)^{\lab g\rab}
\end{gather}
The summand corresponding to $k=0$ has a natural VOA structure,
while the other summands are naturally modules for this VOA. There
also exists a unique (up to scalar) intertwining operator between
the product of the $k$-th and $l$-th summands and the $(k+l)$-th
summand. Carnahan asserts that there is a consistent choice of these
constants such that one has
\begin{thm}
The space $\vnh_g$ naturally admits a VOA structure.
\end{thm}
By construction $\vnh_g$ has rank $26$ and one can define the
semi-infinite cohomology of the Virasoro algebra with coefficients
in $\vnh_g$. The the first semi-infinite cohomology group acquires a
Lie algebra structure according to \cite{LiaZuc_MoonsCoh}. In this
way we obtain a {\em monstrous Lie algebra} $\mla_g$ for each $g\in
\MM$.
\begin{gather}
     \mla_g=H^{\infty/2+1}(\vnh_g)
\end{gather}
An alternative construction of $\mla_g$ may be based on the no-ghost
theorem of the $26$-dimensional bosonic string (cf.
\cite{Fre_RepKMAlgDualResMdl}, \cite{BorMM}). For $g=e$, the Lie
algebra $\mla_e$ is the original Monster Lie algebra of Borcherds
\cite{BorMM}.

The Lie algebra $\mla_g$ inherits an $L(N)^{\vee}$-grading
\begin{gather}
     \mla_g=\bigoplus_{m,n\in\ZZ}\mla_g^{(m/N,n)}.
\end{gather}
The same argument as in the case $g=e$ applied to the semi-infinite
cohomology or no-ghost theorem construction of $\mla_g$ yields
natural isomorphisms
\begin{gather}\label{eqn:conseq:mlas:Nat_Isom_mlag_vng}
     \mla_g^{(m/N,n)}
     \cong
     \vn_{g^{-m},n/N}(mn/N)
\end{gather}
where $\vn_{g^{-m},n/N}$ is the subspace of $\vn_{g^{-m}}$ upon
which $g$ acts by $\ee(n/N)$ and $\vn_{g^{-m},n/N}(mn/N)$ is the
$L_0$-eigenspace of $\vn_{g^{-m},n/N}$ with eigenvalue $mn/N+1$.

\begin{rmk}
Let us write $V^{\LL}$ for the vertex operator algebra associated to the Leech lattice. In his paper \cite{Tui_MonsMoonsUniqMoonsMod} Tuite gave strong
evidence that the genus zero property of monstrous moonshine is
equivalent to the following duality isomorphisms,
\begin{gather}
     \vn_{g^m,n/N}\cong\vn_{g^n,m/N},\label{eqn:conseq:mlas:Dual_Isom_Fricke}\\
     \vn_{g^m,n/N}\cong V^{\LL}_{h^n,m/N},\label{eqn:conseq:mlas:Dual_Isom_non_Fricke}
\end{gather}
the first holding in the case that $g$ is Fricke, and the second in
the case that $g$ is non-Fricke. In the second case we may take $g$ to lie in a subgroup $2^{1+24}.\Co_1$ of $\MM$ and the $h$ in
(\ref{eqn:conseq:mlas:Dual_Isom_non_Fricke}) denotes an element of $2^{24}.2.\Co_1$, a subgroup of $\Aut(V^{\LL})$, such that $h$ and $g$ project to the same conjugacy class of $\Co_1$. Thus the Tuite duality in the Fricke case implies the symmetry
\begin{gather}
     \mla_g^{(m/N,n)}\cong\mla_g^{(n/N,m)}
\end{gather}
of the monstrous Lie algebra $\mla_g$, as is known to hold in the
case $g=e$ studied in \cite{BorMM}. In the non-Fricke case the Tuite
duality suggests an alternative construction of the VOA $\vnh_g$ as
\begin{gather}
     \vnh_g=\bigoplus_{k=0}^{N-1}
     \left(
     V^{\LL}_{h^k}\otimes V_{\tilde{L}+(0,k)}
     \right)^{\lab h\rab}
\end{gather}
where $\tilde{L}=\frac{1}{N}\ZZ\times N\ZZ$.
\end{rmk}
We now re-scale the grading in $\vnh_g$ and $\mla_g$ by
interchanging $(m/N,n)$ with $(m,n/N)$, and we set
\begin{gather}
     \mla_g^m=\bigoplus_{n\in\frac{1}{N}\ZZ}\mla_g^{(m,n/N)}
\end{gather}
for $m\in\ZZ$. Then (\ref{eqn:conseq:mlas:Nat_Isom_mlag_vng})
implies isomorphisms of graded spaces
\begin{gather}
     \mla_g^1\cong\vn_{g^{-1}},\quad
     \mla_g^{-1}\cong\vn_{g},
\end{gather}
and we also have
\begin{gather}
     \mla_g^0=\mla_g^{(0,0)}=\CC c\oplus \CC d,
\end{gather}
where $c$ and $d$ are the degree operators for the re-scaled
grading.

In \cite{Car_Phd} Carnahan shows that $\mla_g$ is a GKM algebra,
thus generalizing a result of \cite{BorMM}. When $g$ is Fricke,
$\mla_g$ has a structure very similar to that of the original
Borcherds Monster Lie algebra $\mla=\mla_e$. In particular, it has
one real simple root and all simple roots correspond to a basis for
$\vn_g$. This implies that $\mla_g$ can be reconstructed from the
subalgebra
\begin{gather}\label{eqn:conseq:mlas:Loc_Subalg_mlag}
     \mla_g^1\oplus\mla_g^0\oplus\mla_g^{-1}
\end{gather}
which is called the {\em local subalgebra of $\mla_g$} (cf.
\cite{FeiFre_HypKMAlgSglMdlrFms}).
\begin{rmk}
When $g$ is non-Fricke $\mla_g$ does not have real simple roots, and
in general it cannot be reconstructed from its local subalgebra
(\ref{eqn:conseq:mlas:Loc_Subalg_mlag}). However the Tuite duality
suggests that if one considers the second grading of $\mla_g$ then
the local subalgebra consists of $V^{\LL}_h$ and
$V^{\LL}_{h^{-1}}$ and the $h$-twisted Heisenberg algebra. This
local subalgebra is expected to generate $\mla_g$ in the case that
$g$ is non-Fricke.
\end{rmk}
In his thesis \cite{Car_Phd} Carnahan also obtained remarkable
generalizations of the Borcherds identities for each of the
monstrous Lie algebras $\mla_g$; viz.,
\begin{gather}\label{eqn:conseq:mlas:Denom_Id_mlag_TJ}
     \vp(\MT_g(\ww)-\GD_g(\zz))
     =\prod_{m\in\ZZp,n\in\ZZ}
     (1-\vp^m\vq^{n/N})^{c_{g^m,n/N}(mn/N)}
\end{gather}
where we have set
\begin{gather}\label{eqn:conseq:mlas:Defn_cgmnNs}
     c_{g^m,n/N}(mn/N)=\dim \vn_{g^m,n/N}(mn/N).
\end{gather}
When $g$ is Fricke the invariance of $\MT_g(\ww)$ under
the Fricke involution and a re-scaling $\ww\mapsto N\ww$ yields
\begin{gather}\label{eqn:conseq:mlas:Denom_Id_mlag_JJ}
     \vp(\GD_g(N\ww)-\GD_g(N\zz))
     =\prod_{m\in\ZZp,n\in\ZZ}
     (1-\vp^m\vq^{n})^{c_{g^m,n/N}(mn/N)}.
\end{gather}
Dividing both sides of (\ref{eqn:conseq:mlas:Denom_Id_mlag_JJ}) by
$(1-\vp\vq^{-1})$ we obtain expressions which are invariant under
the transposition of $\vp$ with $\vq$. This entails the identity
\begin{gather}
     c_{g^m,n/N}(mn/N)=c_{g^m,m/N}(mn/N)
\end{gather}
which is in agreement with the Tuite duality of
(\ref{eqn:conseq:mlas:Dual_Isom_Fricke}). The identity
(\ref{eqn:conseq:mlas:Denom_Id_mlag_TJ}) implies an unexpected
formula for the graded dimension of the Verma module with trivial
character $\mlav_g$ associated to $g$.
\begin{gather}
     \mlav_g=\mc{U}(\mla_g^-)
\end{gather}
In fact the standard product formula
\begin{gather}\label{eqn:conseq:mlas:Prod_Form_gdim_mlavg}
     \gdim\mlav_g
     =\prod_{m\in\ZZp,n\in\ZZ}
     (1-\vp^m\vq^{n/N})^{-c_{g^m,n/N}(mn/N)}
\end{gather}
for the graded dimension of $\mlav_g$ follows from
(\ref{eqn:conseq:mlas:Nat_Isom_mlag_vng}) and
(\ref{eqn:conseq:mlas:Defn_cgmnNs}). The identity
(\ref{eqn:conseq:mlas:Denom_Id_mlag_TJ}) yields the alternative
expression
\begin{gather}\label{eqn:conseq:mlas:gdim_mlav_TJ}
     \gdim\mlav_g
     =\frac{1}{\vp(\MT_g(\ww)-\GD_g(\zz))}.
\end{gather}
This expression allows us to view the bi-graded dimension as a
meromorphic function on $\HH\times\HH$. Thanks to the properties of
principal moduli for curves of genus zero, the expression
(\ref{eqn:conseq:mlas:gdim_mlav_TJ}) is singular at the point
$(\ww,\zz)$ precisely when $\Gamma_g\cdot
\ww=\Gamma_g\cdot(-1/\zz)$. In \S\ref{sec:gravity:conj2} we will
give an interpretation of this fact in the setting of three
dimensional quantum gravity.

With the goal of making such an interpretation in mind we propose to rewrite the first expression (\ref{eqn:conseq:mlas:Prod_Form_gdim_mlavg}) for the bi-graded dimension of the Verma module $\mlav_g$ as
\begin{gather}\label{eqn:conseq:mlas_GDVermagHeckeOps}
     \gdim\mlav_g=\exp
     \left(
          \sum_{m\in\ZZp}\HO_g(m)\GD_g(\zz)\vp^m
     \right)
\end{gather}
using generalized Hecke operators $\HO_g(m)$ such as those discussed in \S\ref{sec:struapp:hops}; we have verified there the existence of such operators for $g=e$ and we conjecture that suitable $\HO_g(m)$ exist for every $g\in\MM$. Comparing then with the generating functions
\begin{gather}
     Z_g(\vp,\vq)=\sum_{m\in\ZZp}m\HO_g(m)\GD_g(\zz)\vp^m
\end{gather}
of the higher order Rademacher sums $\GD^{(-m)}_{g}(\zz)$ we obtain
the following expression.
\begin{gather}
     Z_g(\vp,\vq)=\vp\partial_{\vp}\log(\gdim\mlav_g)
\end{gather}
An identity of this form is well-known in the theory of symmetric
functions (cf. \cite{Mac_SymFnsHallPolys}), where it serves to
relate the generating function of the power symmetric functions
\begin{gather}\label{eqn:conseq:mlas_PowSymFns}
     p_n=\sum_{1\leq i} x_i^n
\end{gather}
with that of the complete symmetric functions
\begin{gather}\label{eqn:conseq:mlas_ComSymFns}
     h_n=
     \sum_{1\leq i_1\leq \cdots\leq i_n}
     \prod_{k=1}^nx_{i_k}.
\end{gather}
Thus the graded dimensions of the Verma modules $\mlav_g$ may be
viewed as {\em complete Rademacher sums}, where the r\^ole of the
variables $x_i$, for $i\in \ZZp$, in
(\ref{eqn:conseq:mlas_PowSymFns}) and
(\ref{eqn:conseq:mlas_ComSymFns}) is taken up by the exponential
expressions $\ee(\lBZ\gamma\rBZh\cdot \zz')$ for $\zz'=-1/\zz$ and
$\lBZ\gamma\rBZh\in\lBZ\Gamma_g\rBZh$.

Since all the coefficients appearing in the bi-graded dimension
$Z_g(\vp,\vq)$ are non-negative integers, one may expect to naturally find a
bi-graded vector space, and possibly even an
$\mla_g$-module, with bi-graded dimension given by $Z_g(\vp,\vq)$.
The next result suggests that one can expect to find such a space
within $\mlav_g$.

Given two elements $F=\sum F_{m,n}p^mq^n$ and $G=\sum G_{m,n}p^mq^n$
in $\ZZ\llb\vq\rrb[[\vp]]$, let us write $F\leq G$ in the case that
$F_{m,n}\leq G_{m,n}$ for all $m,n\in\ZZ$.

\begin{prop}\label{prop:conseq:mlas:gdim_mlavg_Ineq}
Let $g\in \MM$. If $g$ is Fricke then we have $Z_g(\vp,\vq)\leq
\gdim\mlav_g$.
\end{prop}
\begin{proof}
Dividing both sides of (\ref{eqn:conseq:mlas:Denom_Id_mlag_JJ}) by
$(1-\vp\vq^{-1})$ we obtain
\begin{gather}
     1-\sum_{m,n\in\ZZp}
     c_{g,\frac{m+n-1}{N}}(\frac{m+n-1}{N})\vp^m\vq^n
     =\prod_{m,n\in\ZZp}
     (1-\vp^m\vq^n)^{c_{g^m,n/N}(mn/N)}.
\end{gather}
We denote the sum in the left hand side by $\Sigma^+$ and the
product in the right hand side by $\Pi^+$. Then we have
\begin{gather}
     Z_g(\vp,\vq)
     =
     -\vp\partial_{\vp}\log(1-\vp\vq^{-1})\Pi^+
     =\frac{\vp\vq^{-1}}{1-\vp\vq^{-1}}
     +\frac{\vp\partial_{\vp}\Sigma^+}{1-\Sigma^+}.
\end{gather}
On the other hand
\begin{gather}
     \gdim\mlav_g
     =\frac{1}{1-\vp\vq^{-1}}\frac{1}{1-\Sigma^+}
     =\frac{1}{1-\vp\vq^{-1}}
     +\frac{\Sigma^+}{(1-\vp\vq^{-1})(1-\Sigma^+)}.
\end{gather}
Thus it is sufficient to show that
\begin{gather}\label{eqn:conseq:mlas:Sigma+_Ineq}
     \vp\partial_{\vp}(\Sigma^+)
     \leq \frac{1}{1-\vp\vq^{-1}}\Sigma^+,
\end{gather}
but this follows from the observation that both sides of
(\ref{eqn:conseq:mlas:Sigma+_Ineq}) have nonnegative coefficients
and the left hand side may be obtained from the right hand side by
deleting the terms which are singular or constant with respect to
$\vq$.
\end{proof}

To formulate a conjecture on the nature of the subspaces of
$\mlav_g$ with bi-graded dimensions $Z_g(\vp,\vq)$, as well as the
origin of the higher order Rademacher sums, we turn again, in the
next section, to the structures of twisted chiral 3d
quantum gravity.

\subsection{Second conjecture}\label{sec:gravity:conj2}

We have shown in \S\ref{sec:struapp:hops} that the higher order
Rademacher sums recover the action of the Hecke operators on the
first order Rademacher sums. In view of the conjectural relation
between first order Rademacher sums and chiral 3d
quantum gravities at central charge $c=24$, we may consider an
analogous interpretation for the higher order Rademacher sums given
by $m(\HO_g(m)\GD_g)(\zz)$ for $g\in \MM$ and ${m\in\ZZp}$. Since all
the Fourier coefficients of these expressions are non-negative
integers, one might guess that there exist families of twisted
chiral 3d quantum gravities for all central charges
$c=24m$, for $m\in\ZZp$. In particular, the untwisted ($g=e$) chiral
3d quantum gravities possess the structure of certain
extremal vertex algebras with the Monster symmetry.

The aforementioned class of extremal vertex algebras, with the
addition of a Virasoro element, has been originally conjectured by
Witten \cite{Wit_3DGravRev}. However, it was shown in
\cite{Gai_MonsSymmExtCFTS} and \cite{Hoe_SDVOSALgeMinWt} that the
addition of a Virasoro element precludes a non-trivial action of the
Monster group. The higher order Rademacher sums also point to the
partition functions without elements of spin $2$. We may modify
Witten's conjecture by asking for the existence of extremal vertex
algebras with partition functions given by the function
$J^{(-m)}(\zz)$ for $m>0$, related to the higher order Rademacher sums via
$J^{(-m)}(\zz)=\QS{\Gamma}{}{-m}(\zz)-\fc{\Gamma}{}(-m,0)$, with
$\Gamma=\PSL_2(\ZZ)$. The validity of this conjecture remains open.

Note that the Virasoro algebra, even when not represented by any
actual state, may still act on a vertex algebra with partition
function $J^{(-m)}(\zz)$. Regardless of whether or not spaces
$V^{(-m)}$ with partition function $J^{(-m)}(\zz)$ admit natural
vertex algebra structures, they certainly naturally inherit actions
of the Monster group, since the actions of Hecke operators $m\HO(m)$
on the Fourier coefficients of $J(\zz)$ may be interpreted as
actions on representations of $\MM$ (cf.
\cite{JurLepWil_RlznsMonsLieAlg}).

It is natural to enquire as to the meaning of the spaces
$V^{(-m)}_g$, for $m\in\ZZp$ and $g\in \MM$, with graded dimension
$\GD^{(-m)}_g(\zz)$. In light of our First Conjecture, it is unlikely
that the spaces $V^{(-m)}$ can represent states of a three
dimensional quantum gravity with $c=24m$, since a Virasoro element
is generally not present. Thus we have to conclude that the three
dimensional quantum gravity $V^{(-1)}=\vn$ is the only viable
candidate. Then one might view the spaces $V^{(-m)}$ for $m>1$ as
some higher overtones of the basic $m=1$ theory.

In fact, we have shown in \S\ref{sec:struapp:hops} that the action
of the classical Hecke operator $\HO(m)$ on the Rademacher sum
(\ref{eqn:intro:radsum:Rad_Sum_J}) yields a sum over
$\Gamma_{\infty}\backslash M(m)$, where $M(m)$ denotes the (image in
$G(\QQ)$ of the) set of $2\times 2$ matrices with integral entries
and determinant $m$. This sum may be viewed as the {\em
$m$-instanton correction} of the partition function in three
dimensional quantum gravity. By {\em $n$-instanton} in this context
one understands an elliptic curve which admits a holomorphic map of
degree $n$ into a given elliptic curve $E_{\lBZ\zz\rBZh}$ (cf.
\cite{DijMooVerVer_EllGenSndQntStrgs}). Thus the sum over
$\Gamma_{\infty}\backslash M(n)$ becomes a sum over all three
dimensional hyperbolic structures on a solid torus with genus one
boundary whose conformal structure corresponds to an $n$-instanton
on $E_{\lBZ\zz\rBZh}$. The same phenomena is expected for all the twisted
3d quantum gravities corresponding to elements $g\in
\MM$.

Collecting all the $m$-instanton contributions in one generating
function we then obtain a partition function
\begin{gather}\label{eqn:conseq:conj2_InstCorrPartnFn}
     Z_g(\vp,\vq)
     =
     \sum_{m\in\ZZp}
     m(\HO_g(m)\GD_g)(\zz)\vp^m
\end{gather}
depending on two variables, for each $g\in \MM$. This partition
function $Z_g(\vp,\vq)$ may be viewed as a part of the full
partition function
\begin{gather}\label{eqn:conseq:conj2_2ndQuantPartnFn}
     \tilde{Z}_g(\vp,\vq)=\exp\left(
     \sum_{m\in\ZZp}
     (\HO_g(m)\GD_g)(\zz)\vp^m
          \right)
\end{gather}
of the stringy second quantization of $\vn_g$, introduced in
\cite{DijMooVerVer_EllGenSndQntStrgs} in the untwisted case $g=e$.
(See also \cite{Tui_MonsGenMoonsPermOrbs}.) We will now generalize
their construction to an arbitrary twisted module $\vn_g$.

Let $V$ be an arbitrary vertex operator algebra; in our case
$V=\vn$. For $n$ a positive integer, let $V^{\otimes n}$ denote the
tensor product of $n$ copies of $V$, and let $Z_{n}$ be the group of
cyclic permutations of the factors generated by an $n$-cycle
$\sigma_n=(1,2,\cdots,n)$. Let $g$ be an automorphism of $V$ of
order $N$, then $g\times \sigma_n$ is an automorphism of $V^{\otimes
n}$ of order $nN/d$ where $d=(n,N)$. To the pair $(V^{\otimes
n},g\times\sigma_n)$ is canonically associated a twisted module (cf.
\cite{BarDonMas_TwSecsTensProdVOAsPermGps}), which we denote
$V_{(g,n)}$. The group $Z_{g,n}=\lab g\times\sigma_n\rab$ acts
naturally on $V_{(g,n)}$. We write $V_{(g,n)}^{Z_{g,n}}$ for the
$Z_{g,n}$-invariant subspace. The {\em stringy second quantization
of the twisted module $V_g$}, to be denoted $\squan V_g$, is, by
definition, the space
\begin{gather}\label{eqn:conseq:conj2:Defn_2nd_Quant_Vg}
     \squan V_g=
     \bigoplus_{\lambda\in\mc{P}}
     \bigotimes_{r>0}
     S^{m_r}V_{(g,r)}^{Z_{g,r}},
\end{gather}
where the sum is taken over all partitions $\lambda$ with $m_r$
parts of length $r>0$. The space $\squan V_g$ is doubly graded: by
the degrees of products of twisted sectors and by the value of
$|\lambda|$. For $g=e$ our definition of the stringy second
quantization coincides with that of
\cite{DijMooVerVer_EllGenSndQntStrgs}.

Next we establish a relation between the second quantization $\squan
\vn_g$ and the Verma module $\mlav_g$. First note the canonical
isomorphism
\begin{gather}
     \mc{U}(\mla_g^-)\cong S(\mla_g^-).
\end{gather}
Then we obtain the following result.
\begin{thm}\label{thm:conseq:conj2:Isom_Sec_Quant_Sym_mlag}
There is a canonical isomorphism of bi-graded vector spaces
\begin{gather}\label{eqn:conseq:conj2:Isom_Sec_Quant_Sym_mlag}
     \squan \vn_g\cong S(\mla_g^-).
\end{gather}
\end{thm}
\begin{proof}
The isomorphism (\ref{eqn:conseq:conj2:Isom_Sec_Quant_Sym_mlag})
follows from the existence of isomorphisms
\begin{gather}
      (\vn_{(g,n)})^{Z_{g,n}}\cong\mla_g^{-n}
\end{gather}
of bi-graded vector spaces for each $n\in \ZZp$. To establish the
existence of these we will use the isomorphism
\begin{gather}\label{eqn:conseq:conj2:Isom_BDM}
     \vn_{(g,n)}\cong \vn_{g^n}
\end{gather}
of \cite{BarDonMas_TwSecsTensProdVOAsPermGps}.

We wish to show that the subspace of $\vn_{g^n}$ fixed by the
operator
\begin{gather}\label{eqn:conseq:conj2:Operator_Proof_Isom_Sec_Quant_Sym_mlag}
     g\times \ee((L_0-1)/n)
\end{gather}
is naturally isomorphic to $\mla_g^n$. Consider the invariant
subspace of $\vn_{g^n}$ with respect to the $N$-th power of the
operator
(\ref{eqn:conseq:conj2:Operator_Proof_Isom_Sec_Quant_Sym_mlag}),
which may be expressed as $\ee((L_0-1)N/n)$. Its action on the
graded subspace $\vn_{g^n}(kd/N)$ for $k\in\ZZ$ is multiplication by
$\ee(kd/n)$, and this scalar is $1$ if and only if $k=(n/d)m$ for
some $m\in\ZZ$. Thus we are only concerned with the graded subspaces
$\vn_{g^n}(nm/N)$ where $m\in\ZZ$. To find the invariant subspaces
with respect to the operator
(\ref{eqn:conseq:conj2:Operator_Proof_Isom_Sec_Quant_Sym_mlag}) we
consider its action on the subspaces of the form
$\vn_{g^n,m'/N}(mn/N)$. The action is scalar multiplication by
$\ee(-m'/N)\ee(m/N)$ and is therefore trivial if and only if
$m\equiv m' \pmod{N}$. Thus the full subspace of $\vn_{g^n}$
invariant under
(\ref{eqn:conseq:conj2:Operator_Proof_Isom_Sec_Quant_Sym_mlag}) is
given by
\begin{gather}
     \bigoplus_{m\in\ZZ}\vn_{g^n,m/N}(mn/N)
     \cong\mla_g^{-n}.
\end{gather}
To complete the proof we note that the group generated by the action
of the operator
(\ref{eqn:conseq:conj2:Operator_Proof_Isom_Sec_Quant_Sym_mlag}) has
order $nN/d$, and under the isomorphism
(\ref{eqn:conseq:conj2:Isom_BDM}) its action recovers that of
$Z_{g,n}$.
\end{proof}
The isomorphism of Theorem
\ref{thm:conseq:conj2:Isom_Sec_Quant_Sym_mlag} suggests that the
stringy second quantization $\squan \vn_g$ admits an action by the
Lie algebra $\mla_g$ and thus provides a Fock space realization of
the Verma module $\mlav_g$. When $\mla_g$ is generated by its local
subalgebra (\ref{eqn:conseq:mlas:Loc_Subalg_mlag}), as happens in
the Fricke case, it is sufficient to describe the action of brackets
$[\mla_g^{\pm 1},\mla_g^{-n}]\subset\mla_g^{-n\pm 1}$ via the action
of $\vn_{g^{\pm 1}}$ on the twisted sector for $((\vn)^{\otimes
n},g\times \sigma_n)$. Note also that the removal of the subspaces
corresponding to the terms which are singular or constant with
respect to $\vq$, as in the proof of Proposition
\ref{prop:conseq:mlas:gdim_mlavg_Ineq}, will yield an embedding of
the instanton subspace $\mc{I}_g$ inside the stringy second
quantization of $\vn_g$.

Our interpretation of the higher order Rademacher sums and the
complete Rademacher sums, and their relation to the second
quantization of the moonshine vertex operator algebra and its
twisted modules, as well as the monstrous Lie algebras and their
Verma modules, suggests the following extension of our first
conjecture, Conjecture \ref{conj:conseq:conj1_conj1}.
\begin{conj}\label{conj:conseq:conj2_conj2}
Assuming the existence of a family of $g$-twisted chiral 3d quantum gravities for $g\in \MM$, having the properties
stated in Conjecture \ref{conj:conseq:conj1_conj1}, there also
exists a natural geometric interpretation of the family of twisted
monstrous Lie algebras $\mla_g$, the instanton spaces $\mc{I}_g$,
and the denominator formulas for all $g\in \MM$ via the second
quantization of the corresponding $g$-twisted chiral three
3d quantum gravities.
\end{conj}
First of all, the partition function of the second quantized
$g$-twisted chiral 3d quantum gravity should depend
on two modular parameters $\ww$ and $\zz$, and should respect the
symmetry that interchanges $\ww$ with $-1/\zz$. Second, there should
be a quantum gravity theoretic explanation for why these partition
functions are singular precisely when $\ww$ and $-1/\zz$ belong to
the same orbit of $\Gamma_g$. This will explain the remarkable
denominator formulas (\ref{eqn:conseq:mlas:Denom_Id_mlag_TJ}) of
Carnahan, and ultimately, the fundamental r\^ole of the principal
moduli in the moonshine Conjectures.

Where can we find an appropriate setting for all the structures that
appear in our Second Conjecture? The first answer that might come to
mind is the 26 dimensional bosonic string theory. It was known for
a long time that the physical space has a Lie algebra structure (cf.
\cite{Fre_RepKMAlgDualResMdl}, \cite{GodOli_AlgsLatsStgs}), which
can also be recovered from the semi-infinite cohomology (cf.
\cite{FreGarZuc_SemiInfCohStrThy}, \cite{LiaZuc_MoonsCoh}). However,
in this way one can only get a fake Monster Lie algebra (cf.
\cite{BorMM}) and there is no (straight forward) geometric way to
turn it into the real Monster Lie algebra that appears in the three
dimensional quantum gravity approach. Also, the 26 dimensions of
the bosonic string have nothing to do with the 3 dimensions of
the quantum gravity. It is still possible that some constructions of
string theory can be applied in the 3d quantum
gravity setting. In particular, one can expect to identify the
negative part $\mla^-$ of the Monster Lie algebra in the space of
the second quantization of $\vn$ with a certain version of BPS
states (cf. \cite{HarMoo_AlgsBPSStgs}).

Clearly, our results on Rademacher sums, their relation to the
moonshine module, monstrous Lie algebras, and 3d
quantum gravity, admit a super-counterpart \cite{DunFre_RadSumsSM}.
In this case the promise of relationships with the structures of the
10 dimensional superstring is even more tempting, since the latter
has played such a prominent r\^ole in physics over the past 25
years. However, in spite of the of the remarkable mathematical
similarity, the 10 of the ten dimensional superstring and the 3
of the 3d quantum supergravity emerging from the
super-counterpart of the Rademacher sums have different geometric
meaning and are not related by a compactification of any kind.

The most fascinating fact about the 3d quantum
supergravity is that while it is similar, but not directly related,
to the initial 3d quantum gravity, the corresponding
second quantized theories are directly related. The higher order
Rademacher sums point to a larger second quantized space in the
super-case; a space which contains the second quantized space we
considered above.

Then what can these second quantized 3d quantum
gravities mean in physics? Na\"ively, any quantization can be
interpreted as a categorification, which lifts a given theory one
dimension up. One can then wonder if the second quantization in our
case might point to an extreme sector of certain four dimensional
quantum gravities, such as the extreme Kerr black hole, which was
recently found in \cite{GuiHarSonStr_KerrCFTCorr} also to be dual to
a chiral 2d CFT. In this case the Monster (or
moonshine) might be the answer to the perpetual question of what is
behind the letter {M} in the theory that has not yet revealed its
true name.

\section*{Acknowledgement}

We are very grateful to the referees for many helpful comments and suggestions.

%------------------------------------------------------------------%
%\bibliographystyle{alpha}
%\bibliography{catrad}
%------------------------------------------------------------------%
%\comment{
%------------------------------------------------------------------%

%------------------------------------------------------------------%
%}

\end{document}